\documentclass[12pt, leqno]{amsart}

\usepackage{graphicx}
\usepackage{amssymb,amsmath,amsthm,amscd}
\usepackage{mathrsfs}
\usepackage{lscape}
\usepackage{enumerate}
\usepackage{upgreek}
\usepackage[usenames,dvipsnames]{color}
\usepackage[colorlinks=true, pdfstartview=FitV,
 linkcolor=blue,citecolor=blue,urlcolor=blue]{hyperref}
\usepackage{tikz}
\tikzset{dynkdot/.style={circle,draw,scale=.38}}
\usepackage{comment}
\usepackage{xspace}
\usepackage{marginnote}
\usepackage[normalem]{ulem}  
\usepackage[all]{xy}

\DeclareUnicodeCharacter{2212}{-}

\allowdisplaybreaks[3]

\newcommand{\nc}{\newcommand}
\numberwithin{equation}{section}

\newenvironment{red}{\relax\color{red}}{\relax}

\newenvironment{green}{\relax\color{green}}{\hspace*{.5ex}\relax}

\newcommand{\gbr}{\begin{green}}
\newcommand{\egr}{\end{green}}

\newcommand{\berm}[1]{\begin{red}{}\marginnote{\fbox{\scshape\lowercase{M}%
}}#1}  

\newcommand{\bero}[1]{\begin{red}{}\marginnote{\fbox{\scshape\lowercase{O}}}%
#1}

\newcommand{\berno}[1]{\begin{red}{}\marginnote{\fbox{\scshape\lowercase{newO}}}}

\nc{\qR}[1]{q_{\mspace{-2mu}\raisebox{-.8ex}{${\scriptstyle{#1}}$}}}
\newcommand{\isoto}[1][]{\mathop{\xrightarrow%
[{\raisebox{.3ex}[0ex][.3ex]{$\scriptstyle{#1}$}}]%
{{\raisebox{-.6ex}[0ex][-.6ex]{$\mspace{2mu}\sim\mspace{2mu}$}}}}}

\theoremstyle{plain}
\newtheorem{lemma}{Lemma}[section]

\newtheorem{proposition}[lemma]{Proposition}
\newtheorem{theorem}[lemma]{Theorem}

\newtheorem{corollary}[lemma]{Corollary}

\newtheorem{convention}{Convention}

\theoremstyle{definition}
\newtheorem{remark}[lemma]{Remark}
\newtheorem{example}[lemma]{Example}

\newtheorem{definition}[lemma]{Definition}

\newcommand{\srt}[1]{{\langle #1 \rangle}}
\newcommand{\ssrt}[1]{{( #1 )}}

\newcommand{\head}{{\operatorname{hd}}}
\newcommand{\soc}{{\operatorname{soc}}}
\newcommand{\dg}{{\operatorname{deg}}}

\newcommand{\ssqcup}{\mathop{\mbox{\normalsize$\bigsqcup$}}\limits}
\newcommand{\hconv}{\mathbin{\scalebox{.9}{$\nabla$}}}

\newcommand{\sconv}{\mathbin{\scalebox{.9}{$\Delta$}}}

\newcommand{\seteq}{\mathbin{:=}}
\newcommand{\conv}{\mathop{\mathbin{\mbox{\large $\circ$}}}}
\newcommand{\soplus}{\mathop{\mbox{\normalsize$\bigoplus$}}\limits}

\newcommand{\tens}{\mathop\otimes}

\newcommand{\Mod}{\text{-}\mathrm{Mod}}
\newcommand{\gmod}{\text{-}\mathrm{gmod}}

\newcommand{\smod}{\text{-}\mathrm{mod}}
\newcommand{\Rep}{{\rm Rep}}

\newcommand{\rmQ}{\mathrm{Q}}
\newcommand{\rmR}{\mathrm{R}}

\newcommand{\g}{\mathfrak{g}}
\newcommand{\h}{\mathfrak{h}}
\newcommand{\n}{\mathfrak{n}}

\newcommand{\R}{\mathbb{R}}
\newcommand{\C}{\mathbb{C}}
\newcommand{\Q}{\mathbb{Q}}
\newcommand{\Z}{\mathbb{Z}}

\newcommand{\al}{\alpha}
\newcommand{\ep}{\epsilon}
\newcommand{\la}{\lambda}
\newcommand{\be}{\beta}
\newcommand{\ga}{\gamma}
\newcommand{\ve}{\varepsilon}
\newcommand{\La}{\Lambda}

\newcommand{\upep}{\upepsilon}
\newcommand{\upal}{\upalpha}
\newcommand{\upbe}{\upbeta}
\newcommand{\upga}{\upgamma}

\newcommand{\upve}{\upvarepsilon}

\newcommand{\wt}{{\rm wt}}
\newcommand{\hd}{{\rm hd}}

\newcommand{\tr}{{\rm tr}}
\newcommand{\rev}{{\rm rev}}
\newcommand{\het}{{\rm ht}}
\newcommand{\mul}{{\rm mul}}
\newcommand{\supp}{{\rm supp}}
\newcommand{\de}{\mathfrak{d}}
\newcommand{\Dynkin}{\triangle}
\newcommand{\Dynkinv}{\Dynkin_{   \vee}}
\newcommand{\Dynkins}{{\Dynkin_{  \sigma}}}
\newcommand{\Dynkintv}{\Dynkin_{   \widetilde{\vee}}}
\newcommand{\bDynkin}{{\mathop{\mathbin{\mbox{\large $\blacktriangle$}}}}} 
\newcommand{\lf}{[\hspace{-0.3ex}[}
\newcommand{\rf}{]\hspace{-0.3ex}]}

\newcommand{\im}{\imath}
\newcommand{\jm}{\jmath}

\newcommand{\lt}{{\mathrm{l}}}
\newcommand{\rt}{{\mathrm{r}}}

\newcommand*\circled[1]{ \fontsize{6}{6}\selectfont \tikz[baseline=(char.base)]{
  \node[shape=circle,draw,inner sep=0.4pt] (char) {#1};} \fontsize{12}{12}\selectfont }

\newcommand{\Hom}{\operatorname{Hom}}

\newcommand{\res}{\mathrm{res}}

\newcommand{\HOM}{\mathrm{H{\scriptstyle OM}}}

\newcommand{\tw}{\widetilde{w}}

\newcommand{\tde}{\widetilde{\de}}

\newcommand{\tsfC}{{\widetilde{\osfC}}} 

\newcommand{\tuB}{\widetilde{\usfB}}
\newcommand{\tfb}{\widetilde{\sfb}}

\newcommand{\tLa}{\widetilde{\Lambda}}

\newcommand{\hPhi}{\widehat{\Phi}}

\newcommand{\hg}{{\widehat{\g}}}

\newcommand{\ov}[1]{{ \overline{#1}   }}

\newcommand{\oQ}{{\overline{Q}}}

\newcommand{\osfB}{\overline{\mathsf{B}}}
\newcommand{\usfB}{\underline{\mathsf{B}}}
\newcommand{\osfC}{\mathsf{C{\scriptstyle q}}} 
\newcommand{\oxi}{\overline{\xi}}

\newcommand{\oim}{{\overline{\im}}}

\newcommand{\usfC}{{\mathsf{C\ms{1mu}{\scriptstyle t}}}}
\newcommand{\uw}{{\underline{w}}}

\newcommand{\us}{\underline{s}}

\newcommand{\up}{\underline{p}}
\newcommand{\um}{{\underline{m}}}

\newcommand{\uQ}{\underline{Q}}
\newcommand{\uxi}{\underline{\xi}}

\newcommand{\frakp}{\mathfrak{p}}

\newcommand{\pibe}{\frakp_{i,\be}}

\newcommand{\sfC}{\mathsf{C}}
\newcommand{\sfS}{\mathsf{S}}
\newcommand{\sfc}{\mathsf{c}}
\newcommand{\sfn}{\mathsf{n}}

\newcommand{\sfD}{\mathsf{D}}
\newcommand{\sfP}{\mathsf{P}}
\newcommand{\sfQ}{\mathsf{Q}}
\newcommand{\sfW}{\mathsf{W}}

\newcommand{\sfh}{{\ms{1mu}\mathsf{h}}}

\newcommand{\sfV}{\mathsf{V}}
\newcommand{\sfM}{\mathsf{M}}
\newcommand{\sfN}{\mathsf{N}}

\newcommand{\sfb}{\mathsf{b}}
\newcommand{\sfg}{\mathsf{g}}
\newcommand{\sfd}{\ms{1mu}\mathsf{d}}

\newcommand{\bbA}{\mathbb{A}}

\newcommand{\bfB}{\mathbf{B}}
\newcommand{\bfk}{\mathbf{k}}

\newcommand{\bfg}{\mathbf{g}}
\newcommand{\bfn}{\mathbf{n}}

\newcommand{\calQ}{\mathcal{Q}}

\newcommand{\calI}{\mathcal{I}}

\newcommand{\calW}{\mathcal{W}}

\newcommand{\calS}{\mathcal{S}}

\newcommand{\scrC}{\mathscr{C}}
\newcommand{\scrF}{\mathscr{F}}

\newcommand{\ttb}{\mathtt{b}}

\newcommand{\tto}{\mathtt{o}}
\newcommand{\ttp}{\mathtt{p}}

\newcommand{\ttO}{\mathtt{O}}
\newcommand{\To}[1][{\hspace{2ex}}]{\xrightarrow{\,#1\,}}

\newlength{\mylength}
\newcommand{\zm}{z_{\Ma}}
\newcommand{\Ma}{\widehat{\sfM}}
\newcommand{\Na}{\widehat{\sfN}}
\newcommand{\Rnorm}{R^{{\rm norm}}}
\newcommand{\Rr}{\mathbf{r}}

\newcommand{\pair}[1]{ \boldsymbol{\langle  \hspace{-0.6ex} \langle} {#1} \boldsymbol{ \rangle\hspace{-0.6ex} \rangle}   }
\newcommand{\seq}[1]{ \st{#1}_{\beta\in\Phi^+}   }

\newcommand{\hDynkin}{\widehat{\Dynkin}}

\newcommand{\sxi}{{}^\upsigma\hspace{-.4ex}\xi}
\newcommand{\vxi}{{}^\vee\hspace{-.4ex}\xi}
\newcommand{\tvxi}{{}^{\widetilde{\vee}}\hspace{-.4ex}\xi}
\newcommand{\sg}{{}^\upsigma \hspace{-.4ex}\sfg}
\newcommand{\vg}{{}^\vee \hspace{-.4ex}\sfg}
\newcommand{\sI}{{}^\upsigma \hspace{-.4ex}I}
\newcommand{\vuxi}{{}^\vee  \hspace{-.4ex}  \uxi }

\newcommand{\cm}{\sfC}
\newcommand{\wl}{\sfP}
\newcommand{\rl}{\sfQ}

\newcommand{\cwl}{\wl^\vee}

\newcommand{\weyl}{\sfW}
\newcommand{\lan}{\langle}
\newcommand{\ran}{\rangle}

\newcommand{\sprt}[1]{ \fontsize{8}{8}\selectfont \left( \begin{matrix} #1  \end{matrix} \right) \fontsize{12}{12}\selectfont }

\newcommand{\ee}{\end{enumerate}}
\newcommand{\ben}{\begin{enumerate}[{\rm (1)}]}
\newcommand{\bnum}{\begin{enumerate}[{\rm (i)}]}
\newcommand{\bnump}{\begin{enumerate}[{\rm (i)$'$}]}
\newcommand{\bna}{\begin{enumerate}[{\rm (a)}]}
\newcommand{\bnA}{\begin{enumerate}[{\rm (A)}]}
\newcommand{\bc}{\begin{cases}}
\newcommand{\ec}{\end{cases}}

\renewcommand{\preceq}{\preccurlyeq}
\renewcommand{\ge}{\geqslant}
\renewcommand{\le}{\leqslant}

\nc{\bl}{\bigl(}
\nc{\br}{\bigr)}
\nc{\st}[1]{\left\{#1\right\}}
\nc{\ake}[1][1ex]{\rule[-#1]{0ex}{1ex}}
\nc{\akew}[1][1ex]{\rule[-1ex]{#1}{0ex}}
\nc{\akeu}[1][1ex]{\rule[#1]{0ex}{1ex}}
\nc{\cl}{\colon}
\nc{\qt}[1]{\quad\text{#1}}
\nc{\cor}{\bfk}
\nc{\ms}{\mspace}
\nc{\ro}{{\rm(}}
\nc{\rfm}{{\rm)}\xspace}

\newenvironment{myequation}
{\relax\setlength{\arraycolsep}{1pt}\begin{eqnarray}}
{\end{eqnarray}}
\newenvironment{myequationn}
{\relax\setlength{\arraycolsep}{1pt}\begin{eqnarray*}}
{\end{eqnarray*}}

\nc{\cc}{{\ms{1mu}\mathbf{c}}}
\nc{\eq}{\begin{myequation}}
\nc{\eneq}{\end{myequation}}
\nc{\eqn}{\begin{myequationn}}
\nc{\eneqn}{\end{myequationn}}
\nc{\pr}[1][{\cc}]{$#1$-pair\xspace}
\nc{\prq}[1][{[Q]}]{$#1$-pair\xspace}
\nc{\prs}[1][{\cc}]{$#1$-pairs\xspace}
\nc{\prqs}[1][{[Q]}]{$#1$-pairs\xspace}

\nc{\ex}{exponent\xspace}
\nc{\exs}{exponents\xspace}
\nc{\Ex}{\Z_{\ge0}^{\ms{8mu}\Phi^+}}
\nc{\on}{\operatorname}
\nc{\Proof}{\begin{proof}}
\nc{\QED}{\end{proof}}
\nc{\wdt}{\widetilde}
\nc{\oi}{\wdt{i}}
\nc{\oj}{\wdt{j}}
\nc{\ok}{\wdt{k}}

\newenvironment{myarray}
{\relax\renewcommand{\arraystretch}{1.3}\begin{array}}
{\end{array}}

\nc{\ba}{\begin{myarray}}
\nc{\ea}{\end{myarray}}
\nc{\Lemma}{\begin{lemma}}
\nc{\enlemma}{\end{lemma}}
\nc{\pbw}[1][{\cc}]{r_{#1}}
\nc{\qtq}[1][{and}]{\quad\text{#1}\quad}
\nc{\sym}{\mathfrak{S}}
\nc{\monoto}{\rightarrowtail}
\nc{\vphi}{\varphi}
\renewcommand{\Im}{\on{Im}}
\nc{\noi}{\noindent}
\nc{\snoi}{\smallskip\noi}
\nc{\id}{{\mathrm{id}}}

\usepackage[colorinlistoftodos]{todonotes}

\title[$t$-quantized Cartan matrix and R-matrices]{$t$-quantized Cartan matrix and R-matrices for cuspidal modules over quiver Hecke algebras}

\date{2023, February 15}

\author[M. Kashiwara]{Masaki Kashiwara}
\thanks{The research of M.\ Kashiwara
was supported by Grant-in-Aid for Scientific Research (B)
20H01795, Japan Society for the Promotion of Science.}
\address[M. Kashiwara]{
Kyoto University Institute for Advanced Study,
Research Institute for Mathematical Sciences, Kyoto University,
Kyoto 606-8502, Japan \& Korea Institute for Advanced Study, Seoul 02455, Korea }
\email{masaki@kurims.kyoto-u.ac.jp}

\author[S.-j.~Oh]{Se-jin Oh}
\thanks{ The research of S.-j.\ Oh was supported by the Ministry of Education of the Republic of Korea and the National Research Foundation of Korea (NRF-2022R1A2C1004045).}
\address[S.-j.~Oh]{Ewha Womans University Seoul, 52 Ewhayeodae-gil, Daehyeon-dong, Seodaemun-gu, Seoul, South Korea}
\email{sejin092@gmail.com}

\begin{document}

\begin{abstract} As every simple module of a quiver Hecke algebra appears
as the image of the $\rmR$-matrix defined on the convolution product of certain cuspidal modules,
knowing the $\Z$-invariants of the $\rmR$-matrices between cuspidal modules is quite significant.
In this paper, we prove that the $(q,t)$-Cartan matrix specialized at $q=1$ of \emph{any} finite type, called the \emph{$t$-quantized Cartan matrix},
inform us of the invariants of R-matrices. To prove this, we use combinatorial AR-quivers associated with Dynkin quivers and their properties as crucial ingredients.
\end{abstract}

\setcounter{tocdepth}{1}

\maketitle \tableofcontents

\section{Introduction}

Let $\g$ be a finite-dimensional simple Lie algebra, $\Dynkin$ its Dynkin diagram, $\sfC =(\sfc_{i,j})_{i,j \in I}$ its Cartan matrix.
In \cite{FR98}, Frenkel-Reshetikhin introduced the two-parameters deformation $\sfC(q,t)$ of the Cartan matrix $\sfC$
to define a deformation $\calW_{q,t}(\g)$ of the $\calW$-algebra $\calW$ associated with $\g$. Interestingly, it comes to light that $\calW_{q,t}(\g)$ \emph{interpolates}
the representation theories of the quantum affine algebra $U'_q(\widehat{\g})$ of untwisted type and its Langlands dual $U'_q(\widehat{\g}^L)$ in the following sense. Let
$\scrC_\hg$ be the category of finite dimensional $U_q(\widehat{\g})$-modules and $K(\scrC_\hg)$ its Grothendieck ring. Then the specialization of
$\calW_{q,t}(\g)$ at $t=1$ recovers $K(\scrC_\hg)$ while the one at $q=\exp(\pi i/r)$ is expected to recover  $K(\scrC_{\hg^L})$, where $r$ is the racing number of $\g^L$ \cite{FR99,FM01,FH11}.

In the representation theory of $U'_q(\widehat{\g})$ over the 
 base field $\bfk =\overline{\Q(q)}$, the \emph{$\rmR$-matrices} play a central role as every simple module in $\scrC_\hg$ appears as the image of the $\rmR$-matrix defined on a certain tensor product of a special family of simple modules, called the \emph{fundamental modules} \cite{AK97, Kas02, VV02}. Let us explain this in a more precise way.
The $\rmR$-matrices are constructed as intertwining operators between modules in $\scrC_\hg$ satisfying the Yang-Baxter equation.  For a simple module $M \in \scrC_\hg$, there exists a unique finite family
$ \st{(i_k,a_k) \in I \times \bfk^\times }_{1 \le k \le r}$ and the corresponding fundamental modules $\{ V_{i_k}(a_k) \}_{1 \le k \le r}$ such that $M$ is an image of the composition of $\rmR$-matrices
$$
\Rr
:  V_{i_r}(a_r) \otimes \cdots \otimes V_{i_2}(a_2) \otimes V_{i_1}(a_1) \to V_{i_1}(a_1) \otimes V_{i_2}(a_2) \otimes \cdots \otimes V_{i_r}(a_r),
$$
where
$$
\Rr_{t,s} :  V_{i_t}(a_t) \otimes V_{i_s}(a_s) \to  V_{i_s}(a_s) \otimes V_{i_t}(a_t)
$$
 denotes the $\rmR$-matrix between  $V_{i_t}(a_t)$ and $V_{i_t}(a_t)$,
\begin{align*}
\Rr =   (\Rr_{2, 1}) \circ \cdots \circ ( \Rr_{r-1, 1}  \circ \cdots \circ \Rr_{r-1, r-2} )  \circ ( \Rr_{r, 1}  \circ \cdots \circ \Rr_{r, r-1} ),
\end{align*}
and
$a_{s}/a_t$ is \emph{not} a root of the \emph{denominator} of the \emph{normalized $\rmR$-matrix}  $\Rnorm_{V_{i_t}(1),V_{i_s}(1)_z}$, which is a polynomial in $\bfk[z]$ and denoted by $d_{i_t,i_s}(z)$.
Furthermore, the denominator of $\Rnorm_{V_{i_t}(1),V_{i_s}(1)_z}$ indicates whether $ V_{i_t}(a_t) \otimes V_{i_s}(a_s)$ is simple or not (we refer~\cite{KKOP19C} for more detail about the normalized $\rmR$-matrix  and its denominator).
Thus the study of the $\rmR$-matrices and denominators between fundamental modules is one of the first step to investigate the representation theory of quantum affine algebras.  
Surprisingly,
 it is proved that the matrix $(d_{i,j}(z) )_{i,j \in I}$ can read from the inverse  $\tsfC(q) \seteq \osfC(q)^{-1}$ of $\osfC(q) \seteq \sfC(q,t)|_{t=1}$
by using the geometry of the graded quiver varieties in type $ADE$ \cite{Fuj19} and using the root system for a simply-laced type $\sfg$
(see~\eqref{eq: Gg} below)  in any finite type $\g$ \cite{HL15,OS19A,FO21}.
We remark here that $( d_{i,j}(z) )_{i,j \in I}$ was explicitly calculated in \cite{KKK13b,DO94,KMN2,Oh15,OS19,Fuj19}.

We call $\osfC(q)$  the \emph{quantum Cartan matrix}. Note that $\tsfC(q)$ is ubiquitously utilized
in the representation theory of $U_q'(\widehat{\g})$. For instances, it plays crucial roles in the $q$-character theory invented in \cite{FR99,FM01} and in
the construction of the quantum Grothendieck ring $K_{t}(\scrC_\hg)$ invented in \cite{Nak04,VV02A,H04}. As we mentioned, to compute $\tsfC(q)$ of type $\g$, we need to consider its corresponding simply-laced Lie algebra $\sfg$ given in the below:
\begin{align} \label{eq: Gg}
(\g,\sfg) =  (ADE_n,ADE_{n}),  \  (B_n,A_{2n-1}),
\ (C_n,D_{n+1}),\  (F_4,E_6),\  (G_2,D_4).
\end{align}
Then it is revealed that the nature of the representation theory of $U_q'(\widehat{\g})$ is simply-laced even though $\g$ is not simply-laced. For instances,
the quantum Grothendieck rings $K_{t}(\scrC^\calQ_\hg)$ of the \emph{heart subcategories} $\scrC_\hg^\calQ$  of $\scrC_\hg$ are isomorphic to
the integral form of the \emph{unipotent quantum coordinate ring} $A_q(\sfn)$ of
$U_q(\sfg)$ \cite{HL15,KO19,OS19},  and the blocks of $\scrC_\hg$ is parameterized by the root lattice $\sfQ_\sfg$ of $\sfg$ \cite{CM05, KKOP22,FO21}.
Here, (A) the heart subcategory $\scrC_\hg^\calQ$ depends on the choice of a $\rmQ$-datum $\calQ=(\Dynkin_\sfg,\sigma,\xi)$ of $\g$, which consists of (i)
the Dynkin diagram $\Dynkin_\sfg$ of type $\sfg$, (ii) the Dynkin diagram automorphism $\sigma$ which yields $\Dynkin_\g$ as the $\sigma$-orbits of
$\Dynkin_\sfg$ and (iii) a height function $\xi:I \to \Z$ satisfying certain conditions (\cite[Definition 3.5]{FO21}),  (B) $\scrC_\hg^\calQ$ is   characterized by the smallest subcategory of $\scrC_\hg$
(i) containing $\sfV_\calQ \seteq \{ V_{i}((-q)^{p}) \}$, where $(i,p)$
ranges over a certain subset of $I \times \Z$ which
is the set of vertices $(\Gamma_\calQ)_0$ of \emph{combinatorial AR-quiver} $\Gamma_\calQ$ of $\calQ$, and (ii) stable by taking subquotients, extensions and tensor products \cite{HL15,KO19,OS19}.

 \medskip

The \emph{quiver Hecke algebra} $R$, introduced by Khovanov-Lauda \cite{KL09} and Rouquier \cite{R08} independently, is a $\Z$-graded $\bfk$-algebra, whose finite-dimensional graded module category $R\gmod$ also categorifies the integral form of $A_q(\bfn)$ of the quantum group $U_q(\bfg)$ associated with a symmetrizable Kac-Moody algebra $\bfg$. When $\bfg$ is symmetric and  the base field
$\bfk$ is of characteristic 0, the self-dual simple modules in $R\gmod$ categorifies the distinguished basis $\bfB^{{\rm up}}$ \cite{VV09,R11}, which is called the \emph{upper global basis}
(equivalent to Lusztig's \emph{dual canonical basis}).  For a finite simple Lie algebra $\g$,
every simple module can be obtained as the image of $\rmR$-matrices of a certain ordered \emph{convolution product $\conv$} of \emph{cuspidal modules} (\S~\ref{subsec: cuspidal}), where the cuspidal modules categorify the dual PBW-vectors \cite{Kato12,McNa15} (see also \cite{KR09,HMD12}). Let us explain this
more precisely.
For a \emph{commutation class} $\cc$ of reduced expressions of $\uw_0$ of the longest element $w_0$   of Weyl group $\weyl_\g$ (\S~\ref{subsec: convex}), one can define the associated set  of cuspidal modules $\sfS_{\cc} = \{ S_{\cc}(\be) \}_{\be \in \Phi_\g^+}$, which consists of \emph{affreal} simple modules (Definition~\ref{def: affreal}) and is parameterized by the set of positive roots $\Phi_\g^+$.
Then, every simple $R$-module $M$  can be obtained as the image of a distinguished homomorphism, which is also called the $\rmR$-matrix and satisfy the Yang-Baxter equation (see \cite{KKK18A} for more detail), defined an ordered convolution product of modules in $\sfS_{\cc}$.
Namely, there exists a unique $\um =(m_\ell,\ldots,m_1) \in \Z_{\ge 0}^{\Phi_\g^+}$ such that
$M$ is the image of the composition of $\rmR$-matrices  (\S~\ref{subsec: cuspidal})
$$
\Rr
:  S_{\cc}(\be_\ell)^{ \circ m_\ell} \conv \cdots \conv S_{\cc}(\be_2)^{\circ m_2} \conv S_{\cc}(\be_1)^{\circ m_1}
 \to S_{\cc}(\be_1)^{\circ m_1} \conv S_{\cc}(\be_2)^{\circ m_2} \conv \cdots \conv S_{\cc}(\be_\ell)^{\circ m_\ell},
$$
where
$$
\Rr_{t,s} :  S_{\cc}(\be_t)    \conv S_{\cc}(\be_s) \to  S_{\cc}(\be_s)  \conv S_{\cc}(\be_t)
$$
 denotes the $\rmR$-matrix between  $S_{\cc}(\be_t)$ and  $S_{\cc}(\be_s)$, and
\begin{align*}
\Rr =   (\Rr_{2, 1}^{\circ m_{2} m_{1}}) \circ \cdots \circ ( \Rr_{r-1, 1}^{\circ m_{r-1} m_{1}}  \circ \cdots \circ \Rr_{r-1, r-2}^{\circ m_{r-1} m_{r-2}} )  \circ ( \Rr_{r, 1}^{\circ m_r m_{1}}  \circ \cdots \circ \Rr_{r, r-1}^{\circ m_r m_{r-1}} ).
\end{align*}
Here the homogeneous degree $\La(S_{\cc}(\be_t)  ,S_{\cc}(\be_s))$ $(t>s)$ of the $\rmR$-matrix $\Rr_{t,s}$ coincides with $-(\be_t,\be_s)$. As in quantum affine algebra
cases, for each $\Rr_{t,s}$, we have the $\Z_{\ge 0}$-invariant of the $\rmR$-matrix
$$\de(S_{\cc}(\be_t)  ,S_{\cc}(\be_s)) \seteq \frac{1}{2}\bl\La(S_{\cc}(\be_t)  ,S_{\cc}(\be_s)) + \La(S_{\cc}(\be_s)  ,S_{\cc}(\be_t)) \br$$
indicating whether $S_{\cc}(\be_t) \conv S_{\cc}(\be_s)$ is simple or not. Note that, in general, computing $\de(M,N)$ is quite difficult, when
$M$ and $N$ are simple $R$-modules such that one of them is affreal.

 \medskip

In \cite{KKK18A}, Kang-Kashiwara-Kim constructed the \emph{generalized Schur-Weyl functor} $\scrF_J$
from the category $R^J\gmod$ of the quiver Hecke algebra $R^J$
to the category $\scrC_\hg$.
Here the quiver Hecke algebra $R^J$ and the monoidal functor $\scrF_J$
are determined by the \emph{duality datum} (\cite[Section 4]{KKOP21C}), a set of real simple modules $\{ V^j \}_{j \in J}$ in $\scrC_\hg$, and the denominators of $\rmR$-matrices $(\Rr_{V^i,V^j})_{i,j \in J}$.
When the duality datum \emph{arises} from a $\rmQ$-datum $\calQ$ for $\g$, the quiver Hecke algebra $R^\calQ$ is associated with $\sfg$ and the functor $\scrF_\calQ$ gives an equivalence of categories $\scrC^\calQ_\hg$ and $R^\sfg\smod$ 
for each pair $(\g,\sfg)$ in~\eqref{eq: Gg} and any $\rmQ$-datum $\calQ$ of $\g$ \cite{KKK13b,KKKO16,KO19,OS19,Fuj17,Nao21,KKOP21C}. Note that $\scrF_\calQ$ (i)
preserves $\rmR$-matrices in both categories, their invariants and (ii) sends
$\scrC_\hg$ to the set of  fundamental modules $\sfV_\calQ$ to
$\sfS_\calQ  \seteq  \{ S_\calQ(\be)\seteq S_{[\calQ]}(\be) \}_{\be \in \Phi^+_\g}$,
where $[\calQ]$ is a commutation class of $w_0$ \emph{adapted to $\calQ$}. Hence we can obtain information about $\de( S_{\calQ}(\al), S_{\calQ}(\be) )$ in $R^\sfg\gmod$ from $\tsfC(q)$ of $\g$. Thus, roughly speaking, we can understand $\tsfC(q)$ as the denominator  formulae among the cuspidal modules in $\sfS_\calQ$.

To sum up,  the quantum Cartan matrices $\osfC(q)$, which are the specialization of $\sfC(q,t)$ at $t=1$, or rather their inverses $\tsfC(q)$ \emph{control}  the representation theories of quantum affine algebras $U'_q(\hg)$   for all $\g$ and the ones of quiver Hecke algebras $R^\sfg$ for simply-laced types $\sfg$:
 \begin{align} \label{eq: right wing}
\raisebox{2.3em}{ \xymatrix@!C=9.3ex@R=1ex{  &&&&  \scrC_\hg  & \text{ for any $\g$}\\
&&\sfC(q,t) \ar@{.>}[r]^{t=1} & \osfC(q) \ar@{~>}[dr]_{\text{controls}}\ar@{~>}[ur]^{\text{controls}}\\
 &&&&\Rep(R^\sfg)  & \text{ only for $\sfg$}.
 }}
\end{align}

In this paper, \emph{regardless of type,}
we prove that (a) the matrix  $\sfC(q,t)|_{q=1}$, denoted by $\usfC(t)$ and called the \emph{$t$-quantized Cartan matrix},  
and (b) a (combinatorial) AR-quiver $\Gamma_Q$ associated with any Dynkin quiver $Q=(\Dynkin_\g,\xi)$ inform us of the $\Z$-invariants of $\rmR$-matrices and related properties for pairs of cuspidal modules  $(S_{Q}(\al), S_{Q}(\be))$ over the quiver Hecke algebra $R^\g$, where $[Q]$ is the commutation class of $w_0 \in \weyl_\g$ \emph{adapted to $Q$}
(Definition~\ref{def: adapted}).
Since $\usfC(t)$ can be obtained from $\osfC(q)$ by just replacing $q$ with $t$
when $\g$ is simply-laced, we can say that $\usfC(t)$ controls $\scrC_{\widehat{\sfg}}$ and hence  we can complete the left wing of~\eqref{eq: right wing} as follows:
 \begin{align} \label{eq: 2 sided wing}
\raisebox{2.3em}{ \xymatrix@!C=9.3ex@R=1ex%
{\Rep(R^\g) &&&&  \scrC_\hg  & \text{ for any $\g$}\\
&\ar@{~>}[dl]^-(.4){\;\text{controls}}\ar@{~>}[ul]_(.4){\text{controls}} \usfC(t)  &\sfC(q,t) \ar@{.>}[l]_{q=1}\ar@{.>}[r]^{t=1} & \osfC(q) \ar@{~>}[dr]_(.4){\text{controls}}\ar@{~>}[ur]^{\text{controls\;}}\\
\scrC_{\widehat{\sfg}}  &&&&\Rep(R^\sfg)  & \text{ only for $\sfg$}.
 }}
\end{align}

As we explained, when $\g$ is of simply-laced type and $Q$ is a Dynkin quiver of the same type, $\usfC(t) = \osfC(q)|_{q=t}$ indicates whether
$S_{Q}(\al) \conv S_{Q}(\be)$  is simple or not. More precisely, via the bijection $ I \times \Z \supset \hDynkin_0 \overset{\phi_Q}{\longrightarrow} \Phi_\g^+ \times \Z$, we have
$$ \de( S_{Q}(\be) \conv S_{Q}(\al))
= \text{ the coefficient of $t^{|p-s|-1}$ in $\tsfC(t)_{i,j}$},
$$
where $\phi_Q^{-1}(\al,0)=(i,p)$ and $\phi_Q^{-1}(\be,0)=(j,s)$. That is, by only reading (I) relative positions of $\al$ and $\be$ in $\Gamma_Q$ and (II)
the coefficients of $\tsfC(t)_{i,j}$, we can  obtain the $\Z$-invariant $\de(S_{Q}(\be) ,S_{Q}(\al))$. This kind of strategy was achieved in
\cite{Oh16,Oh17,Oh18,OS19A} by developing statistics on $\Gamma_\calQ$, which also depends only on the relative positions of $\al$ and $\be$ in $\Gamma_\calQ$.

\smallskip

We apply the aforementioned strategy to non simply-laced type $\g$ also.
For the purpose, in the previous paper \cite{KO22} of the authors, we extended the notion for Dynkin quivers $Q$ to the non simply-laced types.
We then constructed the combinatorial AR-quiver $\Gamma_Q$  for $Q$ of any type $\g$,
which generalizes the classical AR-quiver $\Gamma_Q$ of the path algebra $\C Q$ for a simply-laced type $Q$ in an aspect of combinatorics. More precisely,
we proved that (i) there are reduced expressions $\uw_0$'s of $w_0$ adapted to $Q$ and such reduced expressions form the commutation class $[Q]$, and (ii) we introduced
the combinatorial AR-quiver $\Gamma_Q$ which realizes the \emph{convex partial order} $\prec_Q$ on the set of positive roots $\Phi^+_\g$ (see~\eqref{eq: convex partial order}
and Theorem~\ref{thm: OS17}).
We also understand $\Gamma_Q$  as a \emph{heart full-subquiver} by constructing the \emph{repetition quiver} $\hDynkin=(\hDynkin_0,\hDynkin_1)$, which depends \emph{only} on $\g$ and whose set of vertices $\hDynkin_0$ is in bijection with $\Phi_\g^+ \times \Z$. We remark here that, when $\g$ is of simply-laced type, $\hDynkin$
is isomorphic to the AR-quiver of the bounded derived category of the path algebra $\C Q$ \cite{Ha87} and  $\hDynkin_0$ is a subset of $I \times \Z$.
One of the main theorems in \cite{KO22} is that the inverse $\tuB(t)$ of $\usfB(t) \seteq \usfC(t) \times \sfD^{-1}$  can be read from any AR-quiver $\Gamma_Q$  of the same type $\g$.  Here $\sfD={\rm diag}(\sfd_i \ | \ i \in I)$
is a diagonal matrix symmetrizing the Cartan matrix $\sfC$.

\smallskip

In this paper, we first show that
$\tuB(t)$ of type $\g$ can be obtained from $\tuB(t)$ of type $\sfg$ via the Dynkin diagram folding $\sigma$ on $\Dynkin_\sfg$ yielding $\Dynkin_\g$ (Theorem~\ref{thm: tbij folding}). Since  $\tuB(t)$ of type $\sfg$ can be read from any $\Gamma_Q$ of the same type $\sfg$, Theorem~\ref{thm: tbij folding} tells us that
combinatorial properties of $\Gamma_Q$ of type $\g$ are related  to the ones of AR-quivers of simply-laced type. In Section~\ref{sec: Labeling BCFG}, we show that
the labeling of $\Gamma_Q$ of type $\g=B_n$, $C_n$, $F_4$ and $G_2$ via $\Phi^+_\g$ can be obtained from the one of $\Gamma_Q$ of type
$\sfg' = D_{n+1}$, $A_{2n-1}$, $E_6$ and $D_4$  via $\Phi^+_{\sfg'}$ by simple surgery, respectively:
\begin{align} \label{eq: Gg'}
(\g,\sfg') =    \  (C_n,A_{2n-1}), \ (B_n,D_{n+1}),\  (F_4,E_6),\  (G_2,D_4).
\end{align}
Note that $\g$ is realized as a non-trivial Lie subalgebra of $\sfg'$ via the Dynkin diagram folding $\sigma$ (\cite[Proposition 7.9]{Kac}). Based on the combinatorial results in Section~\ref{sec: Labeling BCFG}, we show that the statistics, defined in \cite{Oh16,Oh17,Oh18} for simply-laced type,
is also well-defined and does not depend on the choice of Dynkin quivers for any type. Hence we can obtain the set of polynomials, called the \emph{degree polynomials} $(\de_{i,j}(t))_{i,j \in I} \subset \Z_{\ge 0}[t]$ of any type $\g$ (Proposition~\ref{eq: d well-defined}).   Then we show in  Theorem~\ref{thm:Main} that
\begin{align}\label{eq: relation}
\de_{i,j}(t) + \delta_{i,j^*}\,  \sfd_i\, t^{\sfh-1} = \tuB_{i,j}(t) + t^\sfh \tuB_{i,j^*}(t),
\end{align}
when $\g$ is of classical type or  type $E_6$ (see Remark~\ref{rmk: remained type} for the remained types). Here $\sfh$ is the Coxeter number of $\g$ and $*:I \to I$
is an involution on $I$ induced by the longest element $w_0$.

\smallskip

\noindent \textbf{Main Theorem}\;(Theorem~\ref{thm: determimant}).\ Let $Q$ be a Dynkin quiver of any type $\g$. Then, for any $\al,\be \in \Phi^+_\g$, we have
\begin{align} \label{eq: de = tde}
\de(S_{Q}(\al),S_{Q}(\be)) = \text{ the coefficient of $t^{|p-s|-1}$ in $\tuB_{i,j}(t)$},
\end{align}
where $\phi_Q(i,p)=(\al,0)$ and $\phi_Q(j,s)=(\be,0)$. Furthermore, we have the following:
\bna
\item The head of $S_{Q}(\al) \conv S_{Q}(\be)$ is simple and isomorphic to
a convolution product of mutually commuting cuspidal modules in $\sfS_{Q}$, under the assumption that
$\be \not\prec_Q \al$.
\item If the coefficient of $t^{|p-s|-1}$ in $\de_{i,j}(t)$ is strictly larger than $\max(\sfd_i,\sfd_j)$, then the composition length $\ell(S_{Q}(\al) \conv S_{Q}(\be))$ of $S_{Q}(\al) \conv S_{Q}(\be)$ is strictly larger than $2$.
\item If the coefficient of $t^{|p-s|-1}$ in $\de_{i,j}(t)$ is equal to $\max(\sfd_i,\sfd_j)$, $\ell(S_{Q}(\al) \conv S_{Q}(\be))$ is  $2$.
\item If the coefficient of $t^{|p-s|-1}$ in $\de_{i,j}(t)$ is equal to $0$, $S_{Q}(\al) \conv S_{Q}(\be)$ is  simple.
\ee

\smallskip

We prove \textbf{Main Theorem} by using (i) the combinatorial properties of AR-quivers investigated in the sections~\ref{Sec: classical AR},~\ref{sec: Labeling BCFG}, (ii) the study
on dual PBW-vectors and their categorifications in \cite{BKM12} and (iii) the theories related to $\rmR$-matrices developed in \cite{KKKO18,KKOP18}  by authors and their collaborators.

\smallskip

Note that \textbf{Main Theorem} tells that, for a Dynkin quiver $Q=(\Dynkin,\xi)$ with a \emph{source} $i$, we have
\begin{align} \label{eq: reflection}
\de(S_{Q}(\al),S_{Q}(\be)) = \de(S_{s_iQ}(s_i\al),S_{s_iQ}(s_i\be)) \quad \text{
if $\al\not=\al_i$ and $\be \not= \al_i$,}
\end{align}
where $\al_i$ is the $i$-th simple root and $s_iQ=(\Dynkin,s_i\xi)$ is \emph{another} Dynkin quiver obtained from $Q$ (see~\eqref{eq: si height}).
Since $\sfS_{Q}$ and $\sfS_{s_iQ}$ categorify two different families of dual PBW-vectors, which are connected via the braid group action of Lusztig and Saito, we can expect an existence of
a functor proving~\eqref{eq: reflection}. When $\g$ is of simply-laced type, Kato constructed such a functor in \cite{Kato12,Kato17}, called \emph{Saito reflection functor} by using a geometric approach.
We remark here that, comparing with simply-laced type,
there is no fully-developed geometry for non simply-laced type $\g$, and hence we can not apply a geometric approach in \cite{VV09,Kato12,Kato17,Fuj19} at this moment.
Thus we adopt an algebro-combinatorial approach in this paper.

\medskip

This paper is organized as follows. In Section~\ref{Sec: backgroud}, we review the necessary backgrounds of this paper. In particular, we recall the notion of Dynkin quiver and the statistics defined on the pairs of positive roots with respect to the convex partial orders on $\Phi^+$. In Section~\ref{Sec: t-Cartan}, We first review the results in \cite{KO22} about the computation of $\tuB(t)=(\tuB_{i,j}(t))$ as a matrix with entries in $\Z( \hspace{-.3ex} (t) \hspace{-.3ex}  )$.
Then we prove that  $\tuB(t)$ of non simply-laced type $\g$ can be obtained from
the one of simply-laced type $\sfg$.  In Section~\ref{Sec: classical AR}, we review the results in~\cite{Oh16,Oh17,Oh18} about the statistics and the degree polynomials
of AR-quivers $\Gamma_Q$'s when $Q$ is of simply-laced type. In Section~\ref{sec: Labeling BCFG}, we develop the BCFG-analogues of Section~\ref{Sec: classical AR}
and observe the similarities among $\Gamma_Q$'s of different types. In particular, we show that the labeling of AR-quivers of non simply-laced type via $\Phi^+$
can be obtained from the one of  simply-laced type by simple surgery. In Section~\ref{sec: Degree poly}, we define the notion of degree polynomials for each Dynkin quiver $Q$
and show that it does not depend on the choice of $Q$ and hence well-defined for any finite type $\g$. Then we show that~\eqref{eq: relation} holds for classical type and $E_6$.
We allocate a lot of parts of this paper to Section~\ref{Sec: classical AR},~\ref{sec: Labeling BCFG} and ~\ref{sec: Degree poly} to exhibit many computations and examples.
 In Section~\ref{Sec: quiver Hecke}, we recall the quiver Hecke algebra, the $\rmR$-matrix and related $\Z$-invariants related to $\rmR$-matrix. Also we review the categorification of PBW-vectors via cuspidal modules.   In Section~\ref{sec: Cuspidal}, we prove \textbf{Main Theorem}  by a case-by-case method
based on the results in the previous sections. More precisely, we classify pairs $(\al,\be)$ by relative positions in $\Gamma_Q$ and hence statistics, and prove~\eqref{eq: de = tde}
according to the  relative position and the type of $\g$.

\begin{convention}  Throughout this paper, we follow the following conventions.
\ben
\item For a statement $\mathtt{P}$, $\delta(\mathtt{P})$ is $1$ or $0$
 whether $\mathtt{P}$ is true or not.  In particular, we set $\delta_{i,j}
\seteq \delta(i = j)$ $($Kronecker's delta$)$.
\item For a finite set $A$, we denote by $|A|$ the number of elements in $A$.
\item We sometimes use $\bDynkin$
for a Dynkin diagram of non-simply-laced type
to distinguish it from a Dynkin diagram of simply-laced type.

\end{enumerate}
\end{convention}

\section{Backgrounds} \label{Sec: backgroud}
In this section, we review the necessary backgrounds of this paper.

\subsection{Finite Cartan data}  Let $I$ be an index set. A \emph{Cartan datum} is a quintuple
$$(\sfC,\sfP,\Pi,\sfP^\vee,\Pi^\vee)$$ consisting of
\ben
\item a generalized symmetrizable Cartan matrix $\sfC=(\sfc_{i,j})_{i,j\in I}$,
\item a free abelian group $\sfP$, called the \emph{weight lattice},
\item $\Pi=\{ \al_i \ | i \in I \}$, called the set of  \emph{simple roots} ,
\item $\sfP^\vee \seteq \Hom_{\Z}(\wl,\Z)$, called the \emph{coweight lattice}  and
\item $\Pi^\vee=\{ h_i \ | i \in I \}$, called  the set of  \emph{simple coroots}
\ee
satisfying
\bnum
\item $\lan h_i,\al_j\ran = \sfc_{i,j}$ for $i,j\in I$,
\item $\Pi$ is linearly independent over $\Q$,
\item for each $i \in I$, there exists $\varpi_i \in \wl$, called the \emph{fundamental weight}, such that $\lan h_i,\varpi_j \ran=\delta_{i,j}$ for all $j \in I$ and
\item there exists a $\Q$-valued symmetric bilinear form $( \cdot,\cdot)$ on $\wl$ such that $\lan h_i,\la\ran = \dfrac{2 (\al_i , \la) }{ (\al_i , \al_i ) }$ and
$(\al_i,\al_i) \in \Q_{>0}$ for any $i$ and $  \la  \in \wl$.
\ee

We set  $\h \seteq \Q \otimes_\Z \cwl$,  $\rl \seteq \bigoplus_{ i \in I} \Z\al_i$ and $\rl^+ \seteq \sum_{i \in I} \Z_{\ge 0} \al_i$,
  $\wl^+ \seteq \{ \la \in P \ | \ \lan h_i,\la \ran \ge 0 \text{ for any } i \in I\}$ and $\het(\be)=\sum_{i \in I} k_i$ for $\be = \sum_{i \in I} k_i\al_i \in \rl^+$.
We denote by $\Phi$ the set of \emph{roots}, by $\Phi^\pm$ the set of \emph{positive roots} (resp.\ \emph{negative roots}).

\begin{definition} \hfill
\bna
\item
The \emph{multiplicity} $\mul(\be)$ of $\be =\sum_{i \in I}n_i\al_i  \in \rl^+$ is defined by
$$\mul(\be) \seteq \max\st{n_i \mid i \in I}.$$
\item  The \emph{support} $\supp(\be)$ of $\be =\sum_{i \in I}n_i\al_i  \in \rl^+$  is defined by
$$\supp(\be) \seteq\{  i \in I \ | \  n_i \ne 0 \}.$$
\item  For $i \in I$ and $\be =\sum_{i \in I}n_i\al_i  \in \rl^+$, we set $\supp_i(\be) \seteq n_i$.
\item For $\al,\be \in \Phi$, we set
$$
p_{\be,\al} \seteq \max\st{p\in \Z \mid \be -p \al \in \Phi }.
$$
 \ee
\end{definition}

  \emph{Throughout this paper,
we only consider a finite Cartan datum} and   realize the \emph{root lattice} $\rl  \subset \bigoplus_{s \in J} \R \ep_s$ with an orthonormal basis $\{ \ep_s \}_{s\in J}$. For instance, if $\rl$  is of type $A_{n}$,
$\rl \subset \soplus_{i=1}^{n+1} \Z \ep_i$ and $\al_i = \ep_i-\ep_{i+1}$ for $1 \le i \le n$.

For each finite Cartan datum, we take the non-degenerate symmetric bilinear form $(\cdot, \cdot)$ on $\h^*$
such that $(\al,\al)$=2 for short roots $\al$ in $\Phi^+$.
In this paper, a {\em Dynkin diagram} $\Dynkin$ associated with a Cartan datum is
a graph with $I$ as the set of vertices and
the set of edges with edges between $i$ and $j$
such that $(\al_i,\al_j)<0$. We assign $\sfd_i \seteq   (\al_i , \al_i) /{2}  = (\al_i, \varpi_i) \in\Z_{>0}$ to each vertex $i\in I$.
We denote by $\Dynkin_0$ the set of vertices and $\Dynkin_1$ the set of edges.
Here is the list of Dynkin diagrams:
\eq
&&\label{fig:Dynkin}\\
&&\ba{ccc}
&A_n  \   \xymatrix@R=0.5ex@C=4ex{ *{\circ}<3pt> \ar@{-}[r]_<{ 1 \ \   } & *{\circ}<3pt> \ar@{-}[r]_<{ 2 \ \  } & *{\circ}<3pt> \ar@{.}[r]
&*{\circ}<3pt> \ar@{-}[r]_>{ \ \ n} &*{\circ}<3pt> \ar@{}[l]^>{   n-1} }, \quad
 B_n  \    \xymatrix@R=0.5ex@C=4ex{ *{\circled{$\circ$}}<3pt> \ar@{-}[r]_<{ 1 \ \  } & *{\circled{$\circ$}}<3pt> \ar@{-}[r]_<{ 2 \ \  } & *{\circled{$\circ$}}<3pt> \ar@{.}[r]
&*{\circled{$\circ$}}<3pt> \ar@{-}[r]_>{ \ \ n} &*{\circ}<3pt> \ar@{}[l]^>{   n-1} }, \quad
C_n  \    \xymatrix@R=0.5ex@C=4ex{ *{\circ}<3pt> \ar@{-}[r]_<{ 1 \ \  } & *{\circ}<3pt> \ar@{-}[r]_<{ 2 \ \  } & *{\circ}<3pt> \ar@{.}[r]
&*{\circ}<3pt> \ar@{-}[r]_>{ \ \ n} &*{\circled{$\circ$}}<3pt> \ar@{}[l]^>{   n-1} }, \  \allowdisplaybreaks\\[2ex]
&D_n  \ \  \raisebox{1em}{ \xymatrix@R=2ex@C=4ex{  &&& *{\circ}<3pt> \ar@{-}[d]_<{ n-1  }  \\
*{\circ}<3pt> \ar@{-}[r]_<{ 1 \ \  } & *{\circ}<3pt> \ar@{-}[r]_<{ 2 \ \  } & *{\circ}<3pt> \ar@{.}[r]
&*{\circ}<3pt> \ar@{-}[r]_>{ \ \ n} &*{\circ}<3pt> \ar@{}[l]^>{   n-2} }},  \
E_{6}  \ \  \raisebox{1em}{   \xymatrix@R=2ex@C=4ex{  &&  *{\circ}<3pt> \ar@{-}[d]^<{ 2\ \ }   \\
*{\circ}<3pt>  \ar@{-}[r]_<{ 1 \ \  } & *{\circ}<3pt> \ar@{-}[r]_<{ 3 \ \  } & *{\circ}<3pt> \ar@{-}[r]_<{ 4 \ \  }
&*{\circ}<3pt> \ar@{-}[r]_>{ \ \ 6} &*{\circ}<3pt> \ar@{}[l]^>{   5} } },  \
E_{7}  \ \  \raisebox{1em}{    \xymatrix@R=2ex@C=3ex{  &&  *{\circ}<3pt> \ar@{-}[d]^<{ 2\ \ }   \\
 *{\circ}<3pt>  \ar@{-}[r]_<{ 1 \ \  } & *{\circ}<3pt>  \ar@{-}[r]_<{ 3 \ \  } & *{\circ}<3pt> \ar@{-}[r]_<{ 4 \ \  } & *{\circ}<3pt> \ar@{-}[r]_<{ 5 \ \  }
&*{\circ}<3pt> \ar@{-}[r]_>{ \ \ 7} &*{\circ}<3pt> \ar@{}[l]^>{   6} } },  \allowdisplaybreaks \\[4ex]
&  E_{8}  \ \  \raisebox{1em}{    \xymatrix@R=2ex@C=3ex{ &&  *{\circ}<3pt> \ar@{-}[d]^<{ 2\ \ }   \\
 *{\circ}<3pt>  \ar@{-}[r]_<{ 1 \ \  } & *{\circ}<3pt>  \ar@{-}[r]_<{ 3 \ \  } & *{\circ}<3pt> \ar@{-}[r]_<{ 4 \ \  } & *{\circ}<3pt> \ar@{-}[r]_<{ 5 \ \  }   & *{\circ}<3pt> \ar@{-}[r]_<{ 6 \ \  }
&*{\circ}<3pt> \ar@{-}[r]_>{ \ \ 8} &*{\circ}<3pt> \ar@{}[l]^>{   7} } }, \quad
F_{4}   \ \   \xymatrix@R=0.5ex@C=4ex{    *{\circled{$\circ$}}<3pt> \ar@{-}[r]_<{ 1 \ \  } & *{\circled{$\circ$}}<3pt> \ar@{-}[r]_<{ 2 \ \  }
&*{\circ}<3pt> \ar@{-}[r]_>{ \ \ 4} &*{\circ}<3pt> \ar@{}[l]^>{   3} }, \quad
 G_2  \ \   \xymatrix@R=0.5ex@C=4ex{  *{\circ}<3pt> \ar@{-}[r]_<{ 1 \ \  }
& *{\circled{$\odot$}}<3pt> \ar@{-}[l]^<{ \ \ 2  } }.
\ea\nonumber
\eneq

Here $ \circ_{j}$ implies $(\al_j,\al_j)=2$, $\circled{$\circ$}_j$ implies $(\al_j,\al_j)=4$, and
$\circled{$\odot$}_j$ implies $(\al_j,\al_j)=6$.
For $i,j \in  I$, $d(i,j)$ denotes the number of edges between $i$ and $j$ in $\Dynkin$ as a graph.

By the diagonal matrix $\sfD = {\rm diag}(\sfd_i \ | \ i \in I)$, $\sfC$  is \emph{symmetrizable}; i.e.,
\begin{align*}
\text{the matrices $\osfB\seteq\sfD\cm=\bl(\al_i,\al_j)\br_{i,j\in I}$ and   $\usfB\seteq\cm\sfD^{-1}=\bl(\al^\vee_i,\al^\vee_j)\br_{i,j\in I}$ are symmetric, }
\end{align*}
where $\al^\vee_i=(\sfd_i)^{-1}\al_i$.

For $\be \in \Phi$, we set $\sfd_\be \seteq (\be , \be) /{2}$.
For $\al,\be\in\Phi$ such that $\ga=\al+\be \in \Phi$, we have
\begin{align} \label{eq: palbe}
p_{\be,\al} = \bc
2 & \text{ if } \sfd_\ga=3 \text{ and } \sfd_{\al}= \sfd_{\be} =1,\\
1 & \text{ if } \sfd_\ga=2 \text{ and } \sfd_{\al}= \sfd_{\be} =1,\\
1 & \text{ if $\g$ is of type $G_2$ and }  \sfd_\al=\sfd_{\be}= \sfd_{\ga} =1,\\
0 & \text{otherwise.}
\ec
\end{align}

\subsection{Convex orders} \label{subsec: convex}
We denote by $\weyl$ the Weyl group associated to the finite Cartan datum. It is generated by the simple reflections
$\{ s_i  \mid i \in I \}$: $s_i \la = \la - \lan  h_i,\la \ran \al_i$ $(\la \in \wl)$.

Note that there exists a unique element $w_0 \in \weyl$ whose length $\ell(w_0)$ is   the largest,
and $w_0$ induces the Dynkin diagram automorphism
$^*\cl I \to I$ sending $i \mapsto i^*$ , where $w_0(\al_i)=-\al_{i^*}$.

Let $\uw_0 \seteq s_{i_1} \cdots s_{i_\ell}$ be a reduced expression of $w_0 \in \weyl$ and define
\begin{align}\label{eq: beta_k uw_0}
\be^{\uw_0}_k \seteq s_{i_1} \cdots s_{i_{k-1}}(\al_{i_k})   \quad \text{ for } k=1,\ldots,\ell.
\end{align}
Then we have  $\Phi^+ = \{ \be^{\uw_0}_k \ | \ 1 \le k \le l\}$ and $|\Phi^+|=\ell(w_0)$. It is well-known that the total order $<_{\uw_0}$  on $\Phi^+$, defined by
$\be^{\uw_0}_a <_{\uw_0} \be^{\uw_0}_b $ for $a <b$,
is \emph{convex} in the following sense: if $\al,\be\in\Phi^+$
satisfy $\al<_{\uw_0}\beta$ and $\al+\be \in \Phi^+ $, then we have
$ \al <_{\uw_0} \al+\be <_{\uw_0} \be $.

Two reduced expressions   $\uw_0$ and $\uw'_0$ of $w_0$   are said to be \emph{commutation equivalent}, denoted by $\uw_0 \sim \uw'_0$,  if
$\uw'_0$ can be obtained from $\uw_0$ by applying the \emph{commutation relations} $s_is_j=s_js_i$ $(d(i,j)>1)$. Note that this relation  $\sim$ is an   equivalence   relation,
and an   equivalence   class  under $\sim$ is called a \emph{commutation class}. We denote by $[\uw_0]$ the  commutation  class of $\uw_0$.
For a commutation class $\cc$ of $w_0$, we define the \emph{convex partial order} $\prec_{\cc}$ on $\Phi^+$ by:
\begin{align} \label{eq: convex partial order}
\al \prec_{\cc} \be  \quad \text{ if and only if }  \quad \al <_{\uw_0} \be  \quad \text{for any $ \uw_0 \in\cc$.}
\end{align}

For $[\uw_0]$ of $w_0$ and $\al \in \Phi^+$, we define $[\uw_0]$-residue of $\al$, denoted by $\res^{[\uw_0]}(\al)$, to be $i_k \in I$
if $\be^{\uw_0}_k=\al$ with $\uw_0=s_{i_1} \cdots s_{i_\ell}$. Note that this notion is well-defined; i.e, for any $\uw_0' = s_{j_1} \cdots s_{j_\ell} \in [\uw_0]$
with $\be^{\uw'_0}_t=\al$, we have $j_t=i_k$. Note that $$ (\al , \al) = (\al_i ,\al_i) \quad \text{ if } i= \res^{[\uw_0]}(\al).$$

For a reduced expression $\uw_0=s_{i_1}s_{i_2} \cdots s_{i_\ell}$, it is known that the expression $\uw_0'\seteq s_{i_2} \cdots s_{i_\ell}s_{i_1^*}$ is also a reduced expression
of $w_0$. This operation is sometimes referred to as a \emph{combinatorial reflection functor} and we write $r_{i_1}\uw_0 = \uw'_0$. Also it induces the operation on commutation classes of $w_0$ (i.e., $r_{i_1}[\uw_0] = [r_{i_1}\uw_0]$
is well-defined if there exists a reduced expression $\uw_0' =s_{j_1}s_{j_2}\cdots s_{j_\ell} \in [\uw_0]$ such that $j_1=i_1$).
The relations $[\uw] \overset{r}{\sim} [r_i\uw]$ for $i \in I$ generate an equivalence relation, called the \emph{reflection equivalent relation} $\overset{r}{\sim}$, on the
set of commutation classes of $w_0$. For $\uw_0$ of $w_0$, the family of commutation classes $\lf \uw_0 \rf \seteq \{ [\uw_0']\ \mid \ [\uw_0'] \overset{r}{\sim}  [\uw_0] \}$
is called an \emph{$r$-cluster point}.

\subsection{Statistics} \label{subsec: stat}
An element $\um=\seq{\um_\beta}$ of $\Ex$ is called an \emph{\ex}.
In this subsection, we give several notions concerning \exs,
which we need in the PBW description of simple $R$-modules
(e.g., see Theorem~\ref{thm: cuspidal} below).

For an \ex $\um$, we set $\wt(\um) \seteq \sum_{\beta\in\Phi^+} \um_\beta \be \in
\rl^+$.

\begin{definition}[{cf.\ \cite{McNa15,Oh18}}] \label{def: bi-orders}
For a reduced expression $\uw_0=s_{i_1}\ldots s_{\ell}$ and a commutation class $\cc$,
we define the partial orders $<^\ttb_{\uw_0}$ and $\prec^\ttb_{\cc}$ on $\Ex$ as follows:
\begin{enumerate}[{\rm (i)}]
\item $<^\ttb_{\uw_0}$ is the bi-lexicographical partial order induced by $<_{\uw_0}$. Namely, $\um<^\ttb_{\uw_0}\um'$ if
\begin{itemize}
\item $\wt(\um)=\wt(\um')$.
\item there exists $\al\in\Phi^+$ such that
$\um_\al<\um'_\al$ and $\um_\be=\um'_\be$ for any $\be$ such that $\be<_{\uw_0}\al$,
\item there exists $\eta\in\Phi^+$ such that
$\um_\eta<\um'_\eta$ and $\um_\zeta=\um'_\zeta$ for any $\zeta$ such that $\eta<_{\uw_0}\zeta$.
\end{itemize}
\item \label{eq: crazy order} For sequences $\um$ and $\um'$, we define $\um \prec^\ttb_{\cc} \um'$ if the following conditions are satisfied:
$$\um<^\ttb_{\uw_0} \um'\qt{for all $\uw_0 \in \cc$.}$$
\end{enumerate}
\end{definition}

\smallskip

We say that an \ex $\um=\seq{\um_\be} \in \Z_{\ge0}^{\Phi^+}$ is \emph{$\cc$-simple} if it is minimal with respect to the partial order $\prec^\ttb_{\cc}$. For a given $\cc$-simple \ex
$\us=\seq{\us_\beta} \in \Z_{\ge0}^{\Phi^+}$, we call a cover\footnote{Recall that a \emph{cover} of $x$ in a poset $P$ with partial order $\prec$ is an element $y \in P$ such that $x \prec  y$ and there
does not exist  $y' \in P$ such that $x \prec y' \prec y$.}
of $\us$ under $\prec^{\ttb}_{\cc}$ a \emph{$\cc$-minimal \ex for $\us$}. The \emph{$\cc$-degree} of an \ex $\um$,  denoted by $\dg_{\cc}(\um)$,  is the largest integer $k \ge 0$ such that
\begin{align} \label{eq: dist}
 \um^{(0)} \prec^\ttb_{\cc} \um^{(1)} \prec^\ttb_{\cc} \cdots \prec^\ttb_{\cc} \um^{(k-1)} \prec^\ttb_{\cc} \um^{(k)} = \um.
\end{align}

A pair $\pair{\al,\be}$ of $\al,\be\in\Phi^+$ is called a {\em \pr} if
$\be\not\preceq_\cc\al$
(i.e., $\al<_{[\uw_0]}\be$ for some $\uw_0\in\cc$).
We regard a \pr  $\up\seteq\pair{\al,\be}$ as an \ex
by $\up_\al=\up_\be=1$ and $\up_\gamma=0$ for
$\gamma\not=\al,\be$.
When there exists a unique $\cc$-simple \ex $\us$ satisfying
$\us \preceq^\ttb_{\cc} \up$, we call $\us$ the   \emph{$\cc$-head} of $\up$ and
denote it by $\head_{\cc}(\up)$\footnote{In \cite{Oh18}, we have used \emph{socle} instead of head, because we use a different convention
 in Theorem~\ref{thm: cuspidal} below}.

  We regard a positive root $\al\in\Phi^+$ as an \ex $\um$ by
$\um_{\be}=\delta_{\be,\al}$.
\begin{remark} 
The definition of
${\rm dist}_{[\tw_0]}(\up)$ for a \pr $\up$ in~\cite{Oh18} is different
from $\dg_{[\tw_0]}(\up)$ in~\eqref{eq: dist}, which was defined as follows:
\begin{align}
\um^{(0)} \prec^\ttb_{\cc} \up^{(1)} \prec^\ttb_{\cc}  \cdots \prec^\ttb_{\cc} \up^{(k-1)}  \prec^\ttb_{\cc} \up^{(k)} = \up,
\end{align}
for \prs $\up^{(i)}$ and  $\um^{(0)}$ is $\cc$-simple.
However, for a Dynkin quiver $Q$ of simply-laced type, the statistics
$\dg_{[Q]}(\up)$ for every \pr $\up$ still coincides with the
statistics ${\rm dist}_{[Q]}(\up)$ in  \cite{Oh18} when $Q$ is not
of type  $E_7$ and $E_8$ (see the next subsection for Dynkin quivers).
\end{remark}

\subsection{Dynkin quivers}
A \emph{Dynkin quiver} $Q$ of $\Dynkin$ is an oriented graph whose
underlying graph is $\Dynkin$ (see Example~\ref{ex: sigma-fixed} below). For each Dynkin quiver of $Q$, we can
associate a function $\xi\cl \Dynkin_0 \to \Z$ , called   a
\emph{height function} of $Q$  which satisfies: 
\begin{align*}
 \xi_i = \xi_j +1  \qquad \text{ if } 
i \to j \text{ in } Q.
\end{align*}
Conversely, for a Dynkin diagram $\Dynkin$ and a function $\xi\cl
\Dynkin_0 \to \Z$ satisfying $|\xi_i-\xi_j|=1$ for $i,j\in I$ with
$d(i,j)=1$, we can associate a Dynkin quiver $Q$ in a canonical way.
Thus we can identify a Dynkin quiver $Q$ with a pair
$(\Dynkin,\xi)$. Note that, since $\Dynkin$ is connected, height
functions of $Q$ differ by integers.

For a Dynkin quiver $Q=(\Dynkin,\xi)$, the function $\xi^*: \Dynkin_0 \to \Z$ given by
$   (\xi^*)_i  = \xi_{i^*}$ is also a height function. Thus
we set $Q^*=(\Dynkin,\xi^*)$.

For a Dynkin quiver $Q$, we call $i \in \Dynkin_0$ a \emph{source} of $Q$ if $\xi_i > \xi_j$ for all $j \in \Dynkin_0$ with $d(i,j)=1$.
For a Dynkin quiver $Q=(\Dynkin,\xi)$ and its source $i$, we denote by $s_iQ$ the Dynkin quiver $(\Dynkin,s_i\xi)$ where
$s_i\xi$ is the height function defined by
\begin{align} \label{eq: si height}
(s_i\xi)_j = \xi_j-2\delta_{i,j}.
\end{align}

\begin{definition} \label{def: adapted}
Let  $Q=(\Dynkin,\xi)$ be a Dynkin quiver.
\bnum
\item  A reduced expression $\uw=s_{i_1}\cdots s_{i_l}$ of an element of $\weyl_{\Dynkin}$
is said to be \emph{adapted to} $Q$ (or \emph{$Q$-adapted})  if
$$ \text{ $i_k$ is a source of } s_{i_{k-1}}s_{i_{k-2}}\ldots s_{i_1}Q \text{ for all } 1 \le k \le l.$$
\item  A \emph{Coxeter element} $\tau$ of $\weyl_{\Dynkin}$ is a product of all simple reflections; i.e., there exists a reduced expression $s_{i_1} \cdots s_{i_{|I|}}$ of $\tau$ such that $\{ i_1,\ldots,i_{|I|} \}=I$.
\item We denote by $\sfh$ the Coxeter number of $\Dynkin$
which is the order of a Coxeter element.
\ee
\end{definition}

\begin{theorem}[\cite{OS19B,KO22}]
For each Dynkin quiver $Q=(\Dynkin,\xi)$, there exists a $Q$-adapted reduced expression $\uw_0$ of $w_0 \in \weyl_\Dynkin$, and the set of all $Q$-adapted reduced expressions  is a commutation class of $w_0$.
\end{theorem}
We denote by $[Q]$ the commutation class of $w_0$ consisting of all $Q$-adapted reduced expressions.
It is known that, for a fixed Dynkin diagram $\Dynkin$,  the set of commutation classes $\{ [Q] \ | \ Q=(\Dynkin,\xi) \}$ forms an $r$-cluster point, denoted by $\lf \Dynkin \rf$.

It is  also well-known that all of reduced expressions of a Coxeter element $\tau$ form a single commutation class and they are adapted to some Dynkin quiver $Q$. Conversely, for each Dynkin quiver $Q$,
there exists a unique Coxeter element $\tau_Q$ all of whose reduced expressions are adapted to $Q$.
All Coxeter elements are conjugate in $\weyl$, and their common order is called the \emph{Coxeter number} and denoted by $\sfh$.

\smallskip

A  bijection $\sigma$ from $\Dynkin_0$ to itself is said to be a \emph{Dynkin diagram automorphism}
if $\lan h_i,\al_j \ran= \lan h_{\sigma(i)},\al_{\sigma(j)} \ran$ for all $i,j\in \Dynkin_0$. Throughout this paper, we
assume that  Dynkin diagram automorphisms $\sigma$ satisfy the following condition:
\begin{align} \label{eq: auto cond}
\text{ there is no $i \in \Dynkin_0$ such that $d(i,\sigma(i))=1$}.
\end{align}

 For a Dynkin diagram of type $A_{2n-1}$, $D_n$ and $E_6$, there exists a unique non-identity Dynkin diagram automorphism $\vee$ of order $2$ (except $D_4$-type
in which case there are   three   automorphisms of order $2$ and
two automorphisms $\widetilde{\vee}$ and $\widetilde{\vee}{\akew[.2ex]}^2$ of order $3$) satisfying the condition in~\eqref{eq: auto cond}.
The automorphisms $\vee$ and $\widetilde{\vee}$ can be depicted as follows:
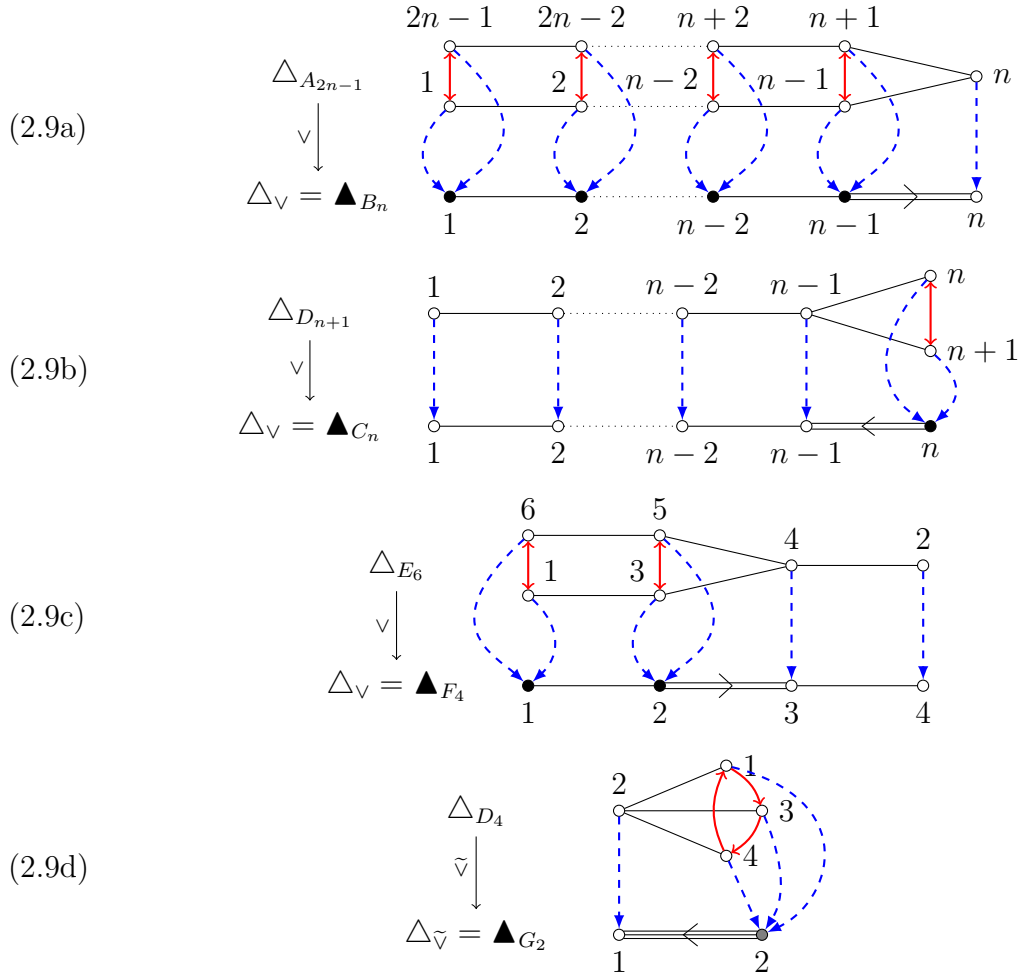
\begin{subequations}
\label{eq: diagram foldings}
\begin{gather}
\label{eq: B_n}
\begin{tikzpicture}[xscale=1.75,yscale=.8,baseline=0]
\node (A2n1) at (0,1) {$\Dynkin_{A_{2n-1}}$};
\node[dynkdot,label={above:$n+1$}] (A6) at (4,1.5) {};
\node[dynkdot,label={above:$n+2$}] (A7) at (3,1.5) {};
\node[dynkdot,label={above:$2n-2$}] (A8) at (2,1.5) {};
\node[dynkdot,label={above:$2n-1$}] (A9) at (1,1.5) {};
\node[dynkdot,label={above left:$n-1$}] (A4) at (4,0.5) {};
\node[dynkdot,label={above left:$n-2$}] (A3) at (3,0.5) {};
\node[dynkdot,label={above left:$2$}] (A2) at (2,0.5) {};
\node[dynkdot,label={above left:$1$}] (A1) at (1,0.5) {};
\node[dynkdot,label={right:$n$}] (A5) at (5,1) {};
\path[-]
 (A1) edge (A2)
 (A3) edge (A4)
 (A4) edge (A5)
 (A5) edge (A6)
 (A6) edge (A7)
 (A8) edge (A9);
\path[-,dotted] (A2) edge (A3) (A7) edge (A8);
\path[<->,thick,red] (A1) edge (A9) (A2) edge (A8) (A3) edge (A7) (A4) edge (A6);
\def\Foffset{-1}
\node (Bn) at (0,\Foffset) {$\Dynkinv = \bDynkin_{B_n}$};
\foreach \x in {1,2}
{\node[dynkdot, fill= black,label={below:$\x$}] (B\x) at (\x,\Foffset) {};}
\node[dynkdot, fill= black,label={below:$n-2$}] (B3) at (3,\Foffset) {};
\node[dynkdot, fill= black,label={below:$n-1$}] (B4) at (4,\Foffset) {};
\node[dynkdot,label={below:$n$}] (B5) at (5,\Foffset) {};
\path[-] (B1) edge (B2)  (B3) edge (B4);
\draw[-,dotted] (B2) -- (B3);
\draw[-] (B4.30) -- (B5.150);
\draw[-] (B4.330) -- (B5.210);
\draw[-] (4.55,\Foffset) -- (4.45,\Foffset+.2);
\draw[-] (4.55,\Foffset) -- (4.45,\Foffset-.2);
 \draw[-latex,dashed,color=blue,thick]
 (A1) .. controls (0.75,\Foffset+1) and (0.75,\Foffset+.5) .. (B1);
 \draw[-latex,dashed,color=blue,thick]
 (A9) .. controls (1.75,\Foffset+1) and (1.25,\Foffset+.5) .. (B1);
 \draw[-latex,dashed,color=blue,thick]
 (A2) .. controls (1.75,\Foffset+1) and (1.75,\Foffset+.5) .. (B2);
 \draw[-latex,dashed,color=blue,thick]
 (A8) .. controls (2.75,\Foffset+1) and (2.25,\Foffset+.5) .. (B2);
 \draw[-latex,dashed,color=blue,thick]
 (A3) .. controls (2.75,\Foffset+1) and (2.75,\Foffset+.5) .. (B3);
 \draw[-latex,dashed,color=blue,thick]
 (A7) .. controls (3.75,\Foffset+1) and (3.25,\Foffset+.5) .. (B3);
 \draw[-latex,dashed,color=blue,thick]
 (A4) .. controls (3.75,\Foffset+1) and (3.75,\Foffset+.5) .. (B4);
 \draw[-latex,dashed,color=blue,thick]
 (A6) .. controls (4.75,\Foffset+1) and (4.25,\Foffset+.5) .. (B4);
 \draw[-latex,dashed,color=blue,thick] (A5) -- (B5);
\draw[->] (A2n1) -- (Bn);
\node (A2n1) at (-0.1,-0.1) {$^{\vee}$};
\end{tikzpicture}
\allowdisplaybreaks \\
\label{eq: C_n}
\begin{tikzpicture}[xscale=1.65,yscale=1.25,baseline=-25]
\node (Dn1) at (0,0) {$\Dynkin_{D_{n+1}}$};
\node[dynkdot,label={above:$1$}] (D1) at (1,0){};
\node[dynkdot,label={above:$2$}] (D2) at (2,0) {};
\node[dynkdot,label={above:$n-2$}] (D3) at (3,0) {};
\node[dynkdot,label={above:$n-1$}] (D4) at (4,0) {};
\node[dynkdot,label={right:$n$}] (D6) at (5,.4) {};
\node[dynkdot,label={right:$n+1$}] (D5) at (5,-.4) {};
\path[-] (D1) edge (D2)
  (D3) edge (D4)
  (D4) edge (D5)
  (D4) edge (D6);
\draw[-,dotted] (D2) -- (D3);
\path[<->,thick,red] (D6) edge (D5);
\def\Coffset{-1.2}
\node (Cn) at (0,\Coffset) {$\Dynkinv=\bDynkin_{C_n}$};
\foreach \x in {1,2}
{\node[dynkdot,label={below:$\x$}] (C\x) at (\x,\Coffset) {};}
\node[dynkdot,label={below:$n-2$}] (C3) at (3,\Coffset) {};
\node[dynkdot,label={below:$n-1$}] (C4) at (4,\Coffset) {};
\node[dynkdot, fill= black,label={below:$n$}] (C5) at (5,\Coffset) {};
\draw[-] (C1) -- (C2);
\draw[-,dotted] (C2) -- (C3);
\draw[-] (C3) -- (C4);
\draw[-] (C4.30) -- (C5.150);
\draw[-] (C4.330) -- (C5.210);
\draw[-] (4.55,\Coffset+.1) -- (4.45,\Coffset) -- (4.55,\Coffset-.1);
\path[-latex,dashed,color=blue,thick]
 (D1) edge (C1)
 (D2) edge (C2)
 (D3) edge (C3)
 (D4) edge (C4);
\draw[-latex,dashed,color=blue,thick]
 (D6) .. controls (4.55,-.25) and (4.55,-0.8) .. (C5);
\draw[-latex,dashed,color=blue,thick]
 (D5) .. controls (5.25,\Coffset+.5) and (5.25,\Coffset+.3) .. (C5);
 \draw[->] (Dn1) -- (Cn);
\node (Dn1) at (-0.1,-0.65) {$^{\vee}$};
\end{tikzpicture}
\allowdisplaybreaks \\
\label{eq: F_4}
\begin{tikzpicture}[xscale=1.75,yscale=.8,baseline=0]
\node (E6desc) at (0,1) {$\Dynkin_{E_6}$};
\node[dynkdot,label={above:$2$}] (E2) at (4,1) {};
\node[dynkdot,label={above:$4$}] (E4) at (3,1) {};
\node[dynkdot,label={above:$5$}] (E5) at (2,1.5) {};
\node[dynkdot,label={above:$6$}] (E6) at (1,1.5) {};
\node[dynkdot,label={above left:$3$}] (E3) at (2,0.5) {};
\node[dynkdot,label={above right:$1$}] (E1) at (1,0.5) {};
\path[-]
 (E2) edge (E4)
 (E4) edge (E5)
 (E4) edge (E3)
 (E5) edge (E6)
 (E3) edge (E1);
\path[<->,thick,red] (E3) edge (E5) (E1) edge (E6);
\def\Foffset{-1}
\node (F4desc) at (0,\Foffset) {$\Dynkinv=\bDynkin_{F_4}$};
\foreach \x in {1,2}
{\node[dynkdot,fill= black,label={below:$\x$}] (F\x) at (\x,\Foffset) {};}
\node[dynkdot,label={below:$3$}] (F3) at (3,\Foffset) {};
\node[dynkdot, label={below:$4$}] (F4) at (4,\Foffset) {};
\draw[-] (F1.east) -- (F2.west);
\draw[-] (F3) -- (F4);
\draw[-] (F2.30) -- (F3.150);
\draw[-] (F2.330) -- (F3.210);
\draw[-] (2.55,\Foffset) -- (2.45,\Foffset+.2);
\draw[-] (2.55,\Foffset) -- (2.45,\Foffset-.2);
\path[-latex,dashed,color=blue,thick]
 (E2) edge (F4)
 (E4) edge (F3);
\draw[-latex,dashed,color=blue,thick]
 (E1) .. controls (1.25,\Foffset+1) and (1.25,\Foffset+.5) .. (F1);
\draw[-latex,dashed,color=blue,thick]
 (E3) .. controls (1.75,\Foffset+1) and (1.75,\Foffset+.5) .. (F2);
\draw[-latex,dashed,color=blue,thick]
 (E5) .. controls (2.75,\Foffset+1) and (2.25,\Foffset+.5) .. (F2);
\draw[-latex,dashed,color=blue,thick]
 (E6) .. controls (0.25,\Foffset+1) and (0.75,\Foffset+.5) .. (F1);
\draw[->] (E6desc) -- (F4desc);
\node (E6desc) at (-0.1,-0.1) {$^{\vee}$};
\end{tikzpicture}
\allowdisplaybreaks \\
\label{eq: G_2}
\begin{tikzpicture}[xscale=1.9,yscale=1.5,baseline=-25]
\node (D4desc) at (0,0) {$\Dynkin_{D_{4}}$};
\node[dynkdot,label={right:$1$}] (D1) at (1.75,.4){};
\node[dynkdot,label={above:$2$}] (D2) at (1,0) {};
\node[dynkdot,label={right:$3$}] (D3) at (2,0) {};
\node[dynkdot,label={right:$4$}] (D4) at (1.75,-.4) {};
\draw[-] (D1) -- (D2);
\draw[-] (D3) -- (D2);
\draw[-] (D4) -- (D2);
\path[->,red,thick]
(D1) edge[bend left=20] (D3)
(D3) edge[bend left=20] (D4)
(D4) edge[bend left=20] (D1);
\def\Coffset{-1.1}
\node (G2desc) at (0,\Coffset) {$\Dynkintv=\bDynkin_{G_2}$};
\node[dynkdot,label={below:$1$}] (G1) at (1,\Coffset){};
\node[dynkdot, fill= gray,label={below:$2$}] (G2) at (2,\Coffset) {};
\draw[-] (G1) -- (G2);
\draw[-] (G1.40) -- (G2.140);
\draw[-] (G1.320) -- (G2.220);
\draw[-] (1.55,\Coffset+.1) -- (1.45,\Coffset) -- (1.55,\Coffset-.1);
\path[-latex,dashed,color=blue,thick]
 (D2) edge (G1);
\draw[-latex,dashed,color=blue,thick]
 (D1) .. controls (2.7,0.1) and (2.5,-0.8) .. (G2);
\draw[-latex,dashed,color=blue,thick]
 (D3) .. controls (2.1,-0.3) and (2.2,-0.7) .. (G2);
\draw[-latex,dashed,color=blue,thick] (D4) -- (G2);
\draw[->] (D4desc) -- (G2desc);
\node (D4desc) at (-0.1,-0.55) {$^{\widetilde{\vee}}$};
\end{tikzpicture}
\end{gather}
\end{subequations}

\begin{definition}\label{def:folding}
Let $\Dynkin$ be a simply-laced Dynkin diagram and $\sigma$ a Dynkin diagram automorphism satisfying \eqref{eq: auto cond}.
\bnum
\item
Let $\Dynkins$ be the folding of $\Dynkin$ by $\sigma$.
Namely, the set $(\Dynkins)_0$ of vertices of $\Dynkins$
is the set of $\sigma$-orbits. Let
$\pi\cl \Dynkin_0\to(\Dynkins)_0$ be the projection.
The edges of $\Dynkins$ are given by the Cartan matrix
$$\mathsf{c}^{\Dynkins}_{\pi(i),\pi(j)}=\sum_{j'\in\pi^{-1}\pi(j)}\mathsf{c}^\Dynkin_{i,j'}.$$
Hence $\sfd_{\pi(i)}=|\pi^{-1}\pi(i)|$.
\item
For a simple Lie algebra $\sfg$ of simply-laced type associated to $\Dynkin$,
we denote by $\sfg_\sigma$  the simple Lie algebra  whose
Dynkin diagram is $\Dynkins$.
\ee
\end{definition}

Note that we have
the following relations of Coxeter numbers:
\eq
&&\sfh^{\Dynkin_\sigma}=\sfh^{\Dynkin}.\label{eq:Coxeter}
\eneq

\begin{definition}
Let $Q=(\Dynkin,\xi)$ be a Dynkin quiver with $\Dynkin$ and $\sigma$
of a Dynkin diagram automorphism of $\Dynkin$.
We say that $Q$ is \emph{$\sigma$-fixed}
if
$$   \xi_{\sigma(i)} = \xi_{i}  \qquad  \text{ for all $i \in \Dynkin_0$}.  $$
For a $\sigma$-fixed Dynkin quiver $Q$, we sometimes denote its height function by ${}^\sigma\xi$ instead of $\xi$, to emphasize.
\end{definition}

\begin{example} \label{ex: sigma-fixed}
Here are several examples of $\sigma$-fixed Dynkin quivers $Q$.
\begin{subequations}
\begin{gather}
\label{it: A5 fixed}
\xymatrix@R=0.5ex@C=6ex{ *{\circ}<3pt> \ar@{->}[r]^<{ _{\underline{3}} \ \  }_<{1 \ \ } & *{\circ}<3pt>  \ar@{->}[r]^<{ _{\underline{2}} \ \  }_<{2 \ \ }
&*{\circ}<3pt>    \ar@{<-}[r]^<{ _{\underline{1}} \ \  }_<{3 \ \ } & *{\circ}<3pt>  \ar@{<-}[r]^<{ _{\underline{2}} \ \  }_<{4 \ \ } & *{\circ}<3pt>    \ar@{}[l]_<{ \quad  \  _{\underline{3}} \ \  }^<{\quad 5 }
} \qt{for $Q=(\Dynkin_{A_5}, \vxi)$.}
\\
\label{it: D4 fixed vee}
 \raisebox{1.2em}{ \xymatrix@R=0.5ex@C=6ex{ && *{\circ}<3pt> \ar@{}[dl]_<{ \quad \ \  _{\underline{1}} \  }^<{  \ 3   } \\
*{\circ}<3pt> \ar@{->}[r]^<{ _{\underline{3}} \ \  }_<{1 \ \ } & *{\circ}<3pt>    \ar@{->}[ur]^<{ _{\underline{2}} \  }_<{2 \ \quad }  \\
&& *{\circ}<3pt>  \ar@{<-}[ul]_<{  \ \ _{\underline{1}} \  }^<{  \quad 4  }
}}\qt{for $Q=(\Dynkin_{D_4},\vxi)$.}
\\
\label{it: E6 fixed}
\raisebox{1.2em}{ \xymatrix@R=2ex@C=6ex{  && *{\circ}<3pt> \ar@{<-}[d]^<{ _{\underline{0}} \ \  }_<{2} \\
*{\circ}<3pt> \ar@{->}[r]^<{ _{\underline{3}} \ \  }_<{1 \ \ } & *{\circ}<3pt>  \ar@{->}[r]^<{ _{\underline{2}} \ \  }_<{3 \ \ }
&*{\circ}<3pt>    \ar@{<-}[r]^<{ _{\underline{1}} }_<{4 \ \ } & *{\circ}<3pt>  \ar@{<-}[r]^<{ _{\underline{2}} \ \  }_<{5\ \ } & *{\circ}<3pt>    \ar@{}[l]_<{ \quad  \  _{\underline{3}} \ \  }^<{\quad 6 }
}}\qt{for $Q=(\Dynkin_{E_6},\vxi)$.}
\\
\label{it: D4 fixed wvee}
 \raisebox{1.2em}{ \xymatrix@R=0.5ex@C=6ex{ && *{\circ}<3pt> \ar@{}[dl]_<{ \quad \ \  _{\underline{1}} \  }^<{  \ 3   } \\
*{\circ}<3pt> \ar@{<-}[r]^<{ _{\underline{1}} \ \  }_<{1 \ \ } & *{\circ}<3pt>    \ar@{->}[ur]^<{ _{\underline{2}} \  }_<{2 \ \quad }  \\
&& *{\circ}<3pt>  \ar@{<-}[ul]_<{  \ \ _{\underline{1}} \  }^<{  \quad 4  }
}}\qt{for $Q=(\Dynkin_{D_4}, \tvxi)$.}
\end{gather}
\end{subequations}
Here
\bnum
\item an underline integer $\underline{*}$ is the value $\xi_i$ at each vertex $i \in \Dynkin_0$,
\item an arrow $\xymatrix@R=0.5ex@C=4ex{  *{\circ }<3pt> \ar@{->}[r]_<{ i \ \  }   & *{\circ}<3pt> \ar@{-}[l]^<{ \ \ j  }}$ means that
$\xi_i >\xi_j$ and $d(i,j)=1$.
\ee
\end{example}

The following lemma is obvious.
\begin{lemma}
Let $\Dynkin$ be a Dynkin diagram, and $\sigma$
a Dynkin diagram automorphism of $\Dynkin.$
Then, the number of $\sigma$-fixed Dynkin quivers
\ro as an oriented graph\rfm \  is
$2^r$, where $r$ is the number of edges of $\Dynkin_\sigma$.
In particular, the number of $\vee$-fixed
Dynkin quivers  is
$$
\bc
2^{n-1} & \text{ if  Dynkin is of type $A_{2n-1}$ or $D_{n+1}$,} \\
8 & \text{ if  Dynkin is of type $E_6$.}
\ec
$$
The number of $\widetilde{\vee}$-fixed Dynkin quivers of $D_4$ is $2$.

\end{lemma}

\section{$t$-quantized Cartan matrix} \label{Sec: t-Cartan}
In this section, we briefly recall the $(q,t)$-Cartan matrix $\sfC(q,t)$ introduced by Frenkel-Reshetikhin in~\cite{FR98}, and the results about the
relationship between $\sfC(1,t)^{-1}$ and Dynkin quivers
in~\cite{HL15,KO22}.
Then we shall see that inverse of  $\sfC(1,t)\sfD^{-1}$ of type $\Dynkin_\sigma$ can be obtained from the one of $\sfC(1,t)$ of type $\Dynkin$ via the \emph{folding} by
$\sigma \ne {\rm id}$.

\medskip

For an indeterminate $x$ and $k \in \Z$, we set,
$$   [k]_x \seteq \dfrac{x^k-x^{-k}}{x-x^{-1}}.$$
For an indeterminate $q$ and $i$, we set $q_i \seteq q^{\sfd_i}$. For instance, when $\g$ is of type $G_2$, we have
$q_2 = q^{3}$.

For a given finite Cartan datum, we define
the \emph{adjacent matrix}  $\calI=(\calI_{i,j})_{  i,j\in I }$ of $\cm$ by
\begin{align}\label{eq: adjacent}
 \calI_{i,j} = 2\delta_{i,j} - \sfc_{i,j} \qquad \text{ for } i,j \in I.
\end{align}

In~\cite{FR98}, the $(q,t)$-deformation of Cartan matrix  $\cm(q,t) =(\sfc_{i,j}(q,t))_{  i,j\in I } $ is introduced:
$$
\sfc_{i,j}(q,t) \seteq (q_it^{-1}+q_i^{-1}t)\delta_{i,j}-[\calI_{i,j}]_{q}.
$$
The specialization of  $\cm(q,t)$ at $t=1$, denoted by $\osfC(q)\seteq \sfC(q,1)$, is usually called the \emph{quantum Cartan matrix}.

\begin{definition}
For each finite Cartan datum, we set
$$ \usfC(t) \seteq \cm(1,t).$$
and call it \emph{$t$-quantized Cartan matrix}. We also set
$$ \usfB(t) \seteq \usfC(t) \sfD^{-1} = (\usfB_{i,j}(t))_{i,j\in I} \quad  \text{ and }   \quad   \osfB(t) \seteq \sfD \usfC(t) = (\osfB_{i,j}(t))_{i,j\in I}.$$
\end{definition}

Hence we have
\eqn
&&\usfB_{i,j}(t)=\bc
\sfd_i^{-1}(t+t^{-1})&\text{if $i=j$,}\\
(\al_i^\vee,\al_j^\vee)&\text{if $i\not=j$,}
\ec
\eneqn
where $\al_i^\vee\seteq\sfd_i^{-1}\al_i$ is the coroot.
\begin{example} Note that, for simply-laced types, we have  $\usfB(t)=\usfC(t)$. The followings are $\usfB(t)$ of non-simply-laced types:
\begin{align*}
& \usfB(t)_{B_n}=\scriptsize{ \left(\begin{matrix}
\frac{t+t^{-1}}{2} & -\frac{1}{2} & 0& 0 &\cdots & 0 \\
-\frac{1}{2}&\frac{t+t^{-1}}{2} & -\frac{1}{2} &   0 & \cdots & 0 \\
\vdots & \vdots & \ddots  &   \ddots & \cdots & 0  \\
0 & \cdots & \cdots  &   -\frac{1}{2} & \frac{t+t^{-1}}{2}  & -1  \\
0 & \cdots & \cdots &  0 & -1 &  t+t^{-1}
\end{matrix}\right)}, \allowdisplaybreaks \\
& \usfB(t)_{C_n}=\scriptsize{\left(\begin{matrix}
t+t^{-1} & -1 & 0& 0 &\cdots & 0 \\
-1 & t+t^{-1}  & -1 &   0 & \cdots & 0 \\
\vdots & \vdots & \ddots  &   \ddots & \cdots & 0  \\
0 & \cdots & \cdots  &   -1 & t+t^{-1} & -1  \\
0 & \cdots & \cdots &  0 & -1 & \frac{t+t^{-1}}{2}
\end{matrix}\right) }, \allowdisplaybreaks \\
& \usfB(t)_{F_4} =
\scriptsize{\left(\begin{array}{rrrr}
 \frac{t+t^{-1}}{2}& -\frac{1}{2} & 0 & 0 \\
-\frac{1}{2} &  \frac{t+t^{-1}}{2} & -1 & 0 \\
0 & -1 &t+t^{-1} & -1 \\
0 & 0 & -1 & t+t^{-1}
\end{array}\right)}, \allowdisplaybreaks\\
&  \usfB(t)_{G_2}=\scriptsize{\left(\begin{array}{rr}
t+t^{-1} & -1 \\
-1 & \frac{t+t^{-1}}{3}
\end{array}\right)}.
\end{align*}
\end{example}

Note that $\usfB(t)\vert_{t=1}=\usfB \in {\rm GL}_{I}(\Q)$. We regard $\usfB(t)$ as an element of ${\rm GL}_{I}(\Q(t))$ and denote its inverse
by $\tuB(t)\seteq\bl\usfB(t)\br^{-1}=(\tuB_{i,j}(t))_{i,j\in I}$. Let
$$ \tuB_{i,j}(t) =\sum_{u\in\Z} \tfb_{i,j}(u)t^u$$
be the Laurent expansion of $ \tuB_{i,j}(t)$ at $t=0$. Note that $ \tuB_{i,j}(t) = \tuB_{j,i}(t)$ for all $i,j \in I$.

\smallskip

Recall that, for each Dynkin quiver $Q=(\Dynkin, \xi)$ of an arbitrary finite type,
there exists a unique Coxeter element $\tau_Q \in W_{\Dynkin}$ such that
all of its reduced expressions are adapted to $Q$.

\begin{definition}[\cite{HL15,FO21,KO22}] 
For a Dynkin quiver $Q$ and $i,j \in I$, we define a function $\eta_{i,j}^Q\cl\Z \to \Z$ by
$$
\eta_{i,j}^Q(u) \seteq \bc
 (\varpi_i , \tau_Q^{(u+\xi_j-\xi_i-1)/2}(\ga_j^Q) ) & \text{ if } u+\xi_j-\xi_i-1\in 2\Z, \\
0& \text{ otherwise}.
\ec
$$
Here $\ga_j^Q\seteq(1-\tau_Q)\varpi_{j} \in \Phi^+$. 
\end{definition}

\begin{lemma} [\cite{HL15,KO22}] Let $Q'$ be another Dynkin quiver of the same type of $Q$. Then we have
$$ \eta_{i,j}^Q = \eta_{i,j}^{Q'} \quad \text{and hence $\eta_{i,j}^\Dynkin$ is well-defined.}$$
\end{lemma}

\begin{theorem}[\cite{HL15,KO22}]  \label{thm: inv}
For each $i,j \in I$ and $u \in \Z_{\ge 0}$,  we have $$\tfb_{i,j}(u)=\eta_{i,j}(u).$$ In other words, we can compute
$\tuB(t)$ by using $\Gamma_Q$ of any Dynkin quiver $Q$ as follows:
$$\tfb_{i,j}(u) = \bc
(\varpi_i , \tau_Q^{(u+\xi_j-\xi_i-1)/2}(\ga_j^Q))   & \text{if $u+\xi_j-\xi_i-1\in2\Z$   and $u\ge0$,  } \\
0 & \text{ otherwise.}
\ec
$$
\end{theorem}

\begin{corollary}[\cite{HL15,FO21,FM21,KO22}]  
\label{cor: bij property}
The coefficients $\{ \tfb_{i,j}(u) \ | \ i,j \in I, \ u \in \Z_{\ge 0} \}$ enjoy the following properties:
\ben
\item \label{it: vanish}  $\tfb_{i,j}(u)=0$   for $u\le 0$, and  $\tfb_{i,j}(1)=\sfd_i\, \delta(i=j)$.
\item \label{it: additive} $\tfb_{i,j}(u-1)+ \tfb_{i,j}(u+1) = \displaystyle\sum_{k; \ d(k,j)=1}  -\lan h_k,\al_j \ran \tfb_{i,k}(u)$ for $u \ge 1$.
\item \label{it: h nega} $\tfb_{i,j}(u+\sfh)=-\tfb_{i,j^*}(u)$  and  $\tfb_{i,j}(u+2\sfh)=\tfb_{i,j}(u)$ for $u \ge 0$.
\item $\tfb_{i,j}(\sfh-u)=\tfb_{i,j^*}(u)$ for $0 \le u \le \sfh$ and $\tfb_{i,j}(2\sfh-u)=-\tfb_{i,j}(u)$ for $0 \le u \le 2\sfh$.
\item\label{it: tfb positive} $\tfb_{i,j}(u) \ge 0$ for $0 \le u \le \sfh$ and $\tfb_{i,j}(u)\le 0$ for $\sfh \le u \le 2\sfh$.
\ee
\end{corollary}

By Corollary~\ref{cor: bij property}~\eqref{it: h nega},
it suffices to compute the $\tfb_{i,j}(u)$ for $1 \le u \le \sfh$, to know $\tuB_{i,j}(t)$ for any $i,j \in \Dynkin_0$.
Thus we set
$$\tde_{i,j}(t)\seteq \tuB_{i,j}(t)+t^\sfh \tuB_{i,j^*}(t)=  \sum_{u=1}^{\sfh}\tfb_{i,j}(u)t^u $$
and
$$\tde_{i,j}[k]\seteq\tfb_{i,j}(k-1)   \quad \text{for $0 \le  k \le \sfh$.}
$$
Then, we have
$$\tuB_{i,j}(t)=\dfrac{\tde_{i,j}(t)-t^{2\sfh}\,\tde_{i,j}(t^{-1})}{1-t^{2\sfh}}.$$

\begin{theorem} \label{thm: tbij folding}
Let $\sigma$ be a Dynkin diagram automorphism,
$\Dynkins$ the folding of $\Dynkin$ by $\sigma$,
and $\pi\cl(\Dynkin)_0\to(\Dynkins)_0$ the projection
\ro see Definition~\ref{def:folding}\/\rfm.
Then we have the following relation:
$$  \tde^{\Dynkins}_{i,j}(t) =  \sum_{\substack{ \oi\in \pi^{-1}(i), \\ \oj\in \pi^{-1}(j)}}  \tde^{\Dynkin}_{\oi,\oj}(t)\qt{for any $i,j\in(\Dynkins)_0$.}$$
\end{theorem}
\Proof
First let us remark the relations of $\usfB$ for $\Dynkins$ and $\Dynkin$:
\eqn
&&(\usfB^\Dynkins)_{i,\pi(\oj)}=\sfd_{i}^{-1}
\sum_{\oi\in\pi^{-1}(i)}(\usfB^\Dynkin)_{\oi,\oj}
\qt{for any $i\in(\Dynkins)_0$ and $\oj\in\Dynkin_0$.}
\eneqn

Now set
\eqn
f_{\oi,\oj}^\Dynkin&&=\Bigl((\usfB^\Dynkin)^{-1}\Bigr)_{\oi,\oj}
=(1+t^\sfh)^{-1}\tde^{\Dynkin}_{\im,\jm}(t),\\
f_{i,j}^{\Dynkins}&&=\sum_{\oi\in\pi^{-1}(i),\;\oj\in\pi^{-1}(j)}
f_{\oi,\oj}^\Dynkin
\eneqn
for $\oi,\oj\in\Dynkin_0$ and $i,j\in(\Dynkins)_0$.

In order to see the theorem, it is enough to show that
\eq
&&\sum_{j\in(\Dynkins)_0}(\usfB^\Dynkins)_{i,j}f^\Dynkins_{j,k}=\delta_{i,k}
\qt{for any $i,k\in(\Dynkins)_0$.}\label{eq:req}
\eneq
We have
$$
\delta_{\oi,\ok}
=\sum_{\oj\in (\Dynkin)_0}(\usfB^\Dynkin)_{\oi,\oj}f^\Dynkin_{\oj,\ok}$$
for any $\oi,\ok\in\Dynkin_0$.
Hence for $i\in (\Dynkins)_0$ and $\ok\in\Dynkin_0$, we have
\eqn
\delta_{i,\pi(\ok)}&&=\sum_{\oi\in\pi^{-1}(i)}\delta_{\oi,\ok}
=\sum_{\oi\in\pi^{-1}(i),\;\oj\in (\Dynkin)_0}(\usfB^\Dynkin)_{\oi,\oj}f^\Dynkin_{\oj,\ok}\\
&&=\sum_{\oj\in (\Dynkin)_0}\sfd_{i}(\usfB^\Dynkins)_{i,\pi(\oj)}f^\Dynkin_{\oj,\ok}.
\eneqn
Hence, for any $i,k\in (\Dynkins)_0$, we have
\eqn
\sfd_k\delta_{i,k}
&&=\sum_{\ok\in\pi^{-1}(k)}\delta_{i,\pi(\ok)}
=\sum_{\ok\in\pi^{-1}(\ok),\;\oj\in (\Dynkin)_0}
\sfd_{i}(\usfB^\Dynkins)_{i,\pi(\oj)}f^\Dynkin_{\oj,\ok}\\
&&=\sum_{\ok\in\pi^{-1}(\ok),\;j\in (\Dynkins)_0,\;\oj\in\pi^{-1}(j)}
\sfd_{i}(\usfB^\Dynkins)_{i,j}f^\Dynkin_{\oj,\ok}\\
&&=\sum_{j\in (\Dynkins)_0}
\sfd_{i}(\usfB^\Dynkins)_{i,j}f^\Dynkins_{j,k}.
\eneqn
Thus we obtain \eqref{eq:req}.
\QED

The polynomials  $\tde_{i,j}(t)$  are calculated explicitly as
follows:

\begin{theorem}  [\cite{HL15,Fuj19,OS19,KO22}] \label{thm: BC denom}  Note that $\tde_{i,j}(t) = \tde_{j,i}(t)$  for $i,j \in I$ .
\ben
\item For $\Dynkin$ of type $A_{n}$, and $i,j\in I =\{1,\ldots,n\}$, $\tde_{i,j}(t)$ is given as follows:
\begin{align} \label{eq: A formula}
\tde_{i,j}(t)&   = \sum_{s=1}^{\min(i,j,n+1-i,n+1-j)}t^{|i-j|+2s-1}.
\end{align}
\item For $\Dynkin$ of type $D_{n+1}$, and $i,j\in I =\{1,\ldots,n,n+1\}$, $\tde_{i,j}(t)$ is given as follows:
\begin{align}\label{eq: D formula}
\tde_{i,j}(t) &  = \bc
\displaystyle\sum_{s=1}^{\min(i,j)} \bl t^{|i-j|+2s-1}+\delta(\max(i,j)<n)\,t^{2n-i-j+2s-1})
 & \text{ if } \min(i,j)<n,\\[3ex]
\ \displaystyle \sum_{s=1}^{\lfloor (n+ \delta_{i,j}) /2 \rfloor}  t^{4s-1 -2\ms{1mu}\delta(i,j)} & \text{ otherwise.}
\ec
\end{align}
\item
For $\Dynkin$ of type $B_n$ or $C_n$, and $i,j\in I =\{1,\ldots,n\}$, $\tde_{i,j}(t)$ is given as follows:
\begin{equation}\label{eq: BC formula}
\begin{aligned}
\tde_{i,j}(t)& =
\max(\sfd_i,\sfd_j) \displaystyle\sum_{s=1}^{\min(i,j)}\Bigl( t^{|i-j|+2s-1} +  \delta(\max(i,j)<n)\, t^{2n -i-j+2s-1} \Bigr).
\end{aligned}
\end{equation}
\item   For $\Dynkin$ of type $E_6$  and $i \le j\in I =\{1,\ldots,6\}$, $\tde_{i,j}(t)$ is given as follows:
\begin{align*} 
&\tde_{1,1}(t) =  t + t^{7},  && \tde_{1,2}(t) =  t^{4} + t^{8},   \allowdisplaybreaks\\
& \tde_{1,3}(t)  = t^{2} +t^{6}+t^{8}, && \tde_{1,4}(t)  = t^{3}+t^{5}+t^{7}+t^{9}, \allowdisplaybreaks\\
& \tde_{1,5}(t)  = t^{4}+t^{6}+t^{10}, &&  \tde_{1,6}(t)  = t^{5}+t^{11},  \allowdisplaybreaks\\
&  \tde_{2,2}(t)  = t^{1}+t^{5}+t^{7}+t^{11}, && \tde_{2,3}(t)  = t^{3}+t^{5}+t^{7}+t^{9},  \allowdisplaybreaks\\
& \tde_{2,4}(t)  = t^{2}+t^{4}+2t^{6}+t^{8}+t^{10},  && \tde_{3,3}(t) = t^{1}+t^{3}+t^{5}+2t^{7}+t^{9}, \allowdisplaybreaks\\
&\tde_{3,4}(t) = t^{2}+2t^{4}+2t^{6}+2t^{8}+t^{10},   && \tde_{3,5}(t) = t^{3}+2 t^{5}+ t^{7}+ t^{9}+ t^{11},\allowdisplaybreaks \\
& \tde_{4,4}(t) = t^{1}+ 2t^{3}+ 3t^{5}+ 3t^{7}+ 2t^{9}+ t^{11},  && \tde_{i,j}(t) =
t^\sfh\tde_{i,j^*}(t^{-1}) = \tde_{j,i}(t) =t^\sfh \tde_{j,i^*}(t^{-1}).
\end{align*}
\item For $E_7$ and $E_8$, see {\rm Appendix~\ref{appeA: tde}.}
\item  For $\Dynkin$ of type $F_4$ and $i \le j\in I =\{1,2,3,4\}$, $\tde_{i,j}(t)$ is given as follows:
\begin{align*} 
&\tde_{1,1}(t) = 2(t+t^{5}+t^{7}+t^{11} ),  && \tde_{1,2}(t)  = 2(t^{2}+t^{4}+2t^{6}+t^{8}+t^{10} ),  \allowdisplaybreaks\\
&\tde_{1,3}(t)  = 2(t^{3}+t^{5}+t^{7}+t^{9} ),  && \tde_{1,4}(t)  = 2(t^{4}+t^{8} ),  \allowdisplaybreaks\\
&\tde_{2,2}(t)  = 2(t+2t^{3}+3t^{5}+3t^{7}+2t^{9}+t^{11} ),   && \tde_{2,3}(t)  = 2(t^{2}+2t^{4}+2t^{6}+2t^{8}+t^{10} ),  \allowdisplaybreaks\\
&\tde_{2,4}(t)  = 2(t^{3}+t^{5}+t^{7}+t^{9} ),   && \tde_{3,3}(t)  = t+2t^{3}+3t^{5}+3t^{7}+2t^{9}+t^{11},  \allowdisplaybreaks\\
&\tde_{3,4}(t)  = t^{2}+t^{4}+2t^{6}+t^{8}+t^{10},   &&  \tde_{4,4}(t)  = t+t^{5}+t^{7}+t^{11}.
\end{align*}

\item  For $\Dynkin$ of type $G_2$   and $i \le j\in I =\{1,2\}$, $\tde_{i,j}(t)$ is given as follows:
$$
\tde_{1,1}(t) = t+2t^{3}+t^5,  \qquad \tde_{1,2}(t)  =3(t^{2}+t^{4}), \qquad \tde_{2,2}(t)  =3 \tde_{1,1}(t).
$$
\ee
\end{theorem}

\section{Auslander-Reiten (AR) quivers} \label{Sec: classical AR}

In this section, we first recall (combinatorial) Auslander-Reiten  quiver $\Gamma_Q$ associated to each Dynkin quiver $Q$. Then we
recall their combinatorial properties, including the simple algorithm for labeling them with the set of positive roots, and statistics in classical $A_{n}$ and $D_{n}$-cases.  In the next section, we will investigate the relationship among AR-quivers related to the Dynkin diagram folding $\sigma$.

\subsection{Quivers} For each reduced expression $\uw_0 =s_{i_1} \cdots s_{i_\ell}$ of $w_0 \in \weyl$, we associate a quiver $\Upsilon_{\uw_0}$ as follows \cite{OS19A}:
\bnum
\item The set of vertices is $\Phi^+   = \{ \be^{\uw_0}_k \ | \ 1 \le k \le \ell \}$.
\item We assign $(-\lan h_{i_k},\al_{i_l} \ran)$-many arrows from
$\be^{\uw_0}_k$ to $\be^{\uw_0}_l$ if and only if $1 \le l  <  k \le \ell$ and there is no $j$ such that
$l< j< k$ and $i_j \in \{i_k,i_l\}$.
\ee
 We say that a total order  $(\be_1< \cdots < \be_\ell)$ of $\Phi^+$ is a \emph{compatible reading} of $\Upsilon_{\uw_0}$ if we have $k < l$ whenever there is
an arrow from  $\be_l$ to $\be_k$ in $\Upsilon_{\uw_0}$.

\begin{theorem}[{\cite{OS19A,KO22}}] \label{thm: OS17}
Let $\cc$ be a commutation class.
\bnum
\item  Two reduced expressions $\uw_0$ and  $\uw'_0$ of $w_0$
are commutation equivalent if and only if
$\Upsilon_{\uw_0}=\Upsilon_{\uw'_0}$ as quivers.
 Hence $\Upsilon_{\cc}$ is well-defined.
\item For $\al,\be \in \Phi^+$,  $\al  \preceq_{\cc}  \be$ if and only if there exists a path from $\be$ to $\al$ in $\Upsilon_{\cc}$. In other words, the quiver
$\Upsilon_{\cc}$ \ro forgetting the number of arrows\rfm  is the Hasse quiver of the partial order $\preceq_{\cc}$.
\item \label{it: noncom}
If $\al,\be \in \Phi^+$ are not comparable with respect to $\preceq_{\cc}$, then we have $(\al,\be)=0$.
\item \label{it: comp reading}
If
$\Phi^+=\{\be_1< \cdots < \be_\ell\}$ is  a compatible reading of $\Upsilon_{\cc}$, then
there is a unique reduced expression $\uw_0=s_{i_1} \cdots s_{i_\ell}$
in $\cc$ such that
$\beta_k=\be^{\uw_0}_k$ for any $k$.
\ee
\end{theorem}

\begin{definition} [\cite{KO22}] For a Dynkin quiver $Q=(\Dynkin,\xi)$, the
{\em repetition quiver} $\hDynkin=(\hDynkin_0,\hDynkin_1)$ associated to
$Q$ is defined as follows:
\begin{equation}\label{eq: rep quiver}
\begin{aligned}
& \hDynkin_0 \seteq \{ (i,p) \in \Dynkin_0 \times \Z \bigm|  p -\xi_i \in 2 \Z\}, \\
& \hDynkin_1 \seteq \{ (i,p) \To[-\lan h_{i},\al_j \ran]  (j,p+1) )   \ | \  (i,p),(j, p+1) \in \hDynkin_0, \ d(i,j)=1 \}.
\end{aligned}
\end{equation}
Here $(i,p) \To[-\lan h_{i},\al_j \ran] (j,p+1)$ denotes that we assign $(-\lan h_{i},\al_j \ran)$-many arrows from $(i,p)$ to $(j,p+1)$.
Note that the definition of $\hDynkin$  depend only on the parity of the height function of $Q$.
\end{definition}

For any Dynkin quiver $Q = (\Dynkin,\xi)$, we have a bijection $\phi_Q\cl \hDynkin_0  \isoto \hPhi^+_\Dynkin \seteq \Phi^+_\Dynkin \times \Z$ as follows (see \cite{HL15,KO22}): set
$\ga_i^Q \seteq (1-\tau_Q)\varpi_i \in \Phi^+_\Dynkin$ and
\begin{eqnarray} &&
\parbox{75ex}{
\bnum
\item $\phi_Q(i,\xi_i)=(\ga_i^Q,0)$,
\item if $\phi_Q(i,p)=(\be,u)$, then we define
\bna
\item $\phi_Q(i,p\mp 2)=(\tau_Q^{\pm 1}(\be),u)$ if $\tau_Q^{\pm 1}(\be) \in \Phi^+_\Dynkin$,
\item $\phi_Q(i,p\mp 2)=(-\tau_Q^{\pm  1}(\be),u\pm1)$ if $\tau_Q^{\pm 1}(\be) \in \Phi^-_\Dynkin$.
\ee
\ee
}\label{eq: bijection}
\end{eqnarray}
We say   that $(i,p)$ with $\phi_Q(i,p)=(\be,u)$ is the \emph{$Q$-coordinate of $(\be,u)$ in $\hDynkin$}.

The combinatorial description  $\gamma^Q_i$ is given as follows \cite[\S 2.2]{HL15}:
\begin{align} \label{eq: gaQ}
\gamma^Q_i = \sum_{j \in B^Q(i)} \al_j,
\end{align}
where $B^Q(i)$ denotes the set of vertices $j$ such that there exists a path in $Q$ from $j$ to $i$.

For each Dynkin quiver $Q=(\Dynkin,\xi)$, we also denote by
$\Gamma_Q=\bl(\Gamma_Q)_0, (\Gamma_Q)_1 \br$ the full subquiver of
$\hDynkin$ whose set of vertices is $\phi_Q^{-1}(\Phi_{\Dynkin}^+
\times 0)$. We call $\Gamma_Q$ the \emph{AR-quiver associated to
$Q$}.

\begin{theorem} [\cite{HL15, OS19A,KO22}] Let $Q=(\Dynkin,\xi)$ be a Dynkin quiver.
\bnum
\item \label{it:range}  
We have $ \xi_i \equiv \xi_{i^*} - \sfh \bmod2$ for any $i\in I$ and
\begin{align}\label{eq: range}
(\Gamma_Q)_0=\{ (i,p) \in \hDynkin_0 \mid  \xi_i \ge p > \xi_{i^*} - \sfh \}.
\end{align}
In particular, we have
\begin{align*}
\phi_Q(i,p)=\bl\tau_Q^{(\xi_i-p)/2}(\ga_i^Q),0\br  \quad \text{for any $(i,p) \in (\Gamma_Q)_0$.}
\end{align*}
\item \label{it:si}
If $i$ is a source of $Q$, then we have
$$\phi_{s_iQ}(i^*,\xi_i-\sfh)=\phi_Q(i,\xi_i)=(\al_i,0)
\text{ and } \
 \Gamma_{s_iQ}=\left(\Gamma_Q\setminus\{(i,\xi_i) \} \right) \cup \{(i^*,\xi_i-\sfh)\}.$$
\item The map $\phi_Q$ induces a quiver isomorphism
$\Gamma_Q\isoto\Upsilon_{[Q]}$.
Recall that $[Q]$ is the commutation class of $Q$-adapted reduced expressions.
\ee
\end{theorem}

\begin{example} \label{ex: AR Q ADE}
Here are some examples of $\Gamma_Q$ for $Q$ of simply-laced types.
\ben
\item For $ Q=  \xymatrix@R=0.5ex@C=6ex{    *{\circ}<3pt> \ar@{->}[r]^<{ _{\underline{4}} \ \  }_<{1 \ \ } & *{\circ}<3pt> \ar@{->}[r]^<{ _{\underline{3}} \ \  }_<{2 \ \ }
&*{\circ}<3pt> \ar@{<-}[r]^<{ _{\underline{2}} \ \  }_<{3 \ \ }  & *{\circ}<3pt> \ar@{<-}[r]^<{ _{\underline{3}} \ \  }_<{4 \ \ }
& *{\circ}<3pt> \ar@{-}[l]_<{ \ \  _{\underline{4}}   }^<{ \ \ 5 }  }  $ of type $A_5$,  $\Gamma_Q$ can be depicted as
$$
 \raisebox{3.2em}{ \scalebox{0.7}
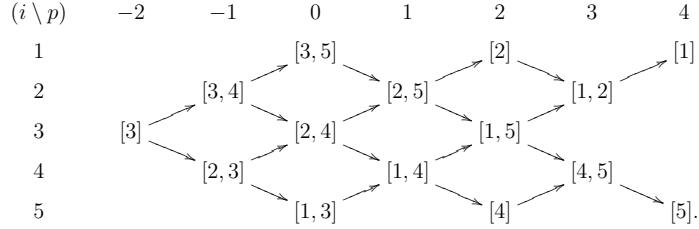
{\xymatrix@!C=5ex@R=0.5ex{
(i\setminus p) &  -2 & -1 & 0 & 1 & 2 & 3 & 4\\
1&&& [3,5]\ar@{->}[dr]&& [2] \ar@{->}[dr] && [1]   \\
2&& [3,4] \ar@{->}[dr]\ar@{->}[ur] && [2,5] \ar@{->}[dr] \ar@{->}[ur]  && [1,2] \ar@{->}[ur] \\
3& [3]  \ar@{->}[dr]\ar@{->}[ur] && [2,4]  \ar@{->}[dr]\ar@{->}[ur]  && [1,5] \ar@{->}[dr]\ar@{->}[ur]\\
4&& [2,3]  \ar@{->}[dr]\ar@{->}[ur] && [1,4] \ar@{->}[dr] \ar@{->}[ur]  && [4,5] \ar@{->}[dr] \\
5&&& [1,3]\ar@{->}[ur]&& [4] \ar@{->}[ur] && [5].   \\
}}}
$$
Here $[a,b] \seteq \sum_{k=a}^b \al_k$ for $1 \le a <b \le 5$ and $[a] \seteq \al_a$ for $1 \le a \le 5$.
\item \label{it: E6}
 For
$Q=\raisebox{1.2em}{ \xymatrix@R=2ex@C=6ex{  && *{\circ}<3pt> \ar@{<-}[d]^<{ _{\underline{0}} \ \  }_<{2} \\
*{\circ}<3pt> \ar@{->}[r]^<{ _{\underline{3}} \ \  }_<{1 \ \ } & *{\circ}<3pt>  \ar@{->}[r]^<{ _{\underline{2}} \ \  }_<{3 \ \ }
&*{\circ}<3pt>    \ar@{<-}[r]^<{ _{\underline{1}} }_<{4 \ \ } & *{\circ}<3pt>  \ar@{<-}[r]^<{ _{\underline{2}} \ \  }_<{5\ \ } & *{\circ}<3pt>    \ar@{}[l]_<{ \quad  \  _{\underline{3}} \ \  }^<{\quad 6 }
}}$ of type $E_6$, $\Gamma_Q$  can be depicted as
$$ \hspace{-6ex}
\raisebox{6em}{\scalebox{0.58}
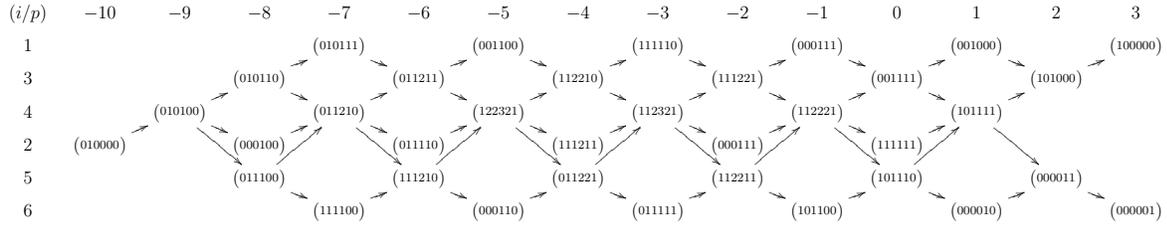
{\xymatrix@R=0.5ex@C=2ex{
(i/p) & -10 & -9 & -8 & -7 & -6 & -5 & -4 & -3 & -2 & -1 & 0 & 1 & 2 & 3    \\
1 &&&&  \sprt{010111}\ar[dr] &&  \sprt{001100}\ar[dr]&&  \sprt{111110}\ar[dr]&&  \sprt{000111}\ar[dr]&&  \sprt{001000}\ar[dr]&&  \sprt{100000} \\
3 &&&\sprt{010110}\ar[dr]\ar[ur]&&\sprt{011211}\ar[dr]\ar[ur]&&  \sprt{112210}\ar[dr]\ar[ur]&&  \sprt{111221}\ar[dr]\ar[ur]&&  \sprt{001111}\ar[dr]\ar[ur]&&  \sprt{101000}\ar[ur] \\
4 &&  \sprt{010100}\ar[dr]\ar[ddr]\ar[ur]&&  \sprt{011210}\ar[dr]\ar[ddr]\ar[ur]&&  \sprt{122321}\ar[dr]\ar[ddr]\ar[ur]&&  \sprt{112321}\ar[dr]\ar[ddr]\ar[ur]&&
\sprt{112221}\ar[dr]\ar[ddr]\ar[ur]&&  \sprt{101111} \ar[ddr]\ar[ur]\\
2 &  \sprt{010000}\ar[ur]&&  \sprt{000100}\ar[ur]&&  \sprt{011110}\ar[ur]&&  \sprt{111211}\ar[ur]&&  \sprt{000111}\ar[ur]&&  \sprt{111111}\ar[ur] \\
5 &&&  \sprt{011100}\ar[uur]\ar[dr]&&  \sprt{111210}\ar[uur]\ar[dr]&&  \sprt{011221}\ar[uur]\ar[dr]&&  \sprt{112211}\ar[uur]\ar[dr]&&  \sprt{101110}\ar[uur]\ar[dr]
&&  \sprt{000011} \ar[dr] \\
6 &&&&  \sprt{111100}\ar[ur] &&  \sprt{000110} \ar[ur]&&  \sprt{011111}\ar[ur]&&  \sprt{101100}\ar[ur]&&  \sprt{000010}\ar[ur]&&  \sprt{000001}
}}}
$$
Here $\sprt{a_1a_2a_3a_4a_5a_6} \seteq \ \sum_{i=1}^6 a_i\al_i$.
\ee
\end{example}

\begin{example} \label{ex: AR Q BCFG}
For $Q =  \xymatrix@R=0.5ex@C=6ex{    *{\circ}<3pt> \ar@{->}[r]^<{ _{\underline{4}} \ \  }_<{1} & *{\circ}<3pt> \ar@{->}[r]^<{  _{\underline{3}} \ \  }_<{2 \ \ } &*{\circled{$\circ$}}<3pt>
\ar@{-}[l]_>{  \qquad \qquad\qquad _{\underline{2}}}^<{\ \ 3   }  }$ of type $C_3$,  $\Gamma_Q$  can be depicted as
$$ \raisebox{3.2em}{ \scalebox{0.7}
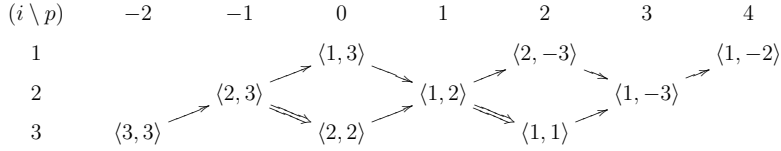
{\xymatrix@!C=6ex@R=0.5ex{
(i\setminus p) &  -2 & -1 & 0 & 1 & 2 & 3  & 4\\
1&&& \srt{1,3}\ar@{->}[dr]&& \srt{2,-3} \ar@{->}[dr] && \srt{1,-2}   \\
2&& \srt{2,3} \ar@{=>}[dr]\ar@{->}[ur] && \srt{1,2} \ar@{=>}[dr] \ar@{->}[ur]  && \srt{1,-3} \ar@{->}[ur] \\
3& \srt{3,3}  \ar@{->}[ur] && \srt{2,2} \ar@{->}[ur]  && \srt{1,1} \ar@{->}[ur]  }}}
$$
Here $\lan a, \pm b\ran \seteq \ep_a \pm \ep_b$ for $1 \le a < b \le 4$ and $\lan a,a\ran \seteq 2\ep_a$ for $1 \le a \le 3$.
\end{example}

\noindent
Note the followings:
\begin{eqnarray} &&
\parbox{80ex}{
\bna
\item The vertices in $\Gamma_{Q}$ are labeled by $\Phi^+_{\Dynkin}$ by the bijection $\phi_{Q}|_{(\Gamma_{Q})_0}$.
\item \label{it: reflection}
For a source $i$ of $Q$, $\Gamma_{s_iQ}$ can be obtained from $\Gamma_Q$ by following way:
\bnum
\item Each $\be \in \Phi^+ \setminus \{ \al_i \}$ is located at $(j,p)$ in $\Gamma_{s_iQ}$, if $s_i\be$ is at $(j,p)$ in $\Gamma_Q$.
\item $\al_i$ is located at the coordinate $(i^*,\xi_i-\sfh)$ in $\Gamma_{s_iQ}$, while $\al_i$ was at $(i,\xi_i)$ in $\Gamma_Q$.
\ee
\ee
}\label{eq: Qd properties}
\end{eqnarray}
We refer the operation in ~\eqref{eq: Qd properties}~\eqref{it: reflection} the \emph{reflection} from $\Gamma_Q$ to $\Gamma_{s_iQ}$.

Let $Q$ be a Dynkin diagram except $E$-types. We say an arrow $(i,p) \to (j,p+1)$ in $\Gamma_Q$ is a \emph{downward} (resp.\ \emph{upward}) arrow if $j > i$ (resp.\ $j<i$) by identifying vertices of $\Dynkin_0$ with integers as in \eqref{fig:Dynkin}.

For $\Dynkin$ of $E$-types, we give a total order $<$ on $\{ 1,2, \ldots, 8 \}$ as follows:
$$ 1 < 3 < 4 < 2 < 5 < 6 < 7 < 8 .$$
Then we can define a downward (resp.\ upward) arrow in $\Gamma_Q$ as in other cases (see Example~\ref{ex: AR Q ADE}~\eqref{it: E6} for $E_6$-case).
\begin{definition} \hfill
\bnum
\item A \prq $\pair{\al,\be} \in (\Phi^+)^2$ with $\phi_Q(i,p)=(\al,0)$ and $\phi_Q(j,s)=(\be,0)$ is \emph{sectional in $\Gamma_Q$} if 
$d(i,j)=p-s$.
\item A full subquiver $\rho$ of $\Gamma_Q$ is \emph{sectional} if every \prq $\pair{\al,\be}$ in $\rho$ is sectional.
\item A connected full subquiver $\varrho$ in $\Gamma_Q$ is said to be \emph{$S$-sectional} (resp.\ \emph{$N$-sectional}) path if it is concatenation of downward (resp.\ upward) arrows.
\item An $S$-sectional (resp.\ $N$-sectional)  path $\varrho$ is said to be  \emph{$N$-path} (resp.\ \emph{$S$-path}) if there is no longer connected sectional path in $\Gamma_Q$ consisting of downward (resp.\ upward) arrows and containing $\varrho$.
\ee
\end{definition}

\subsection{Classical AR-quiver $\Gamma_Q$} In this subsection, we review the labeling and the combinatorial properties of $\Gamma_Q$ for Dynkin quivers $Q$ of type $ADE$.
Throughout this subsection, $Q=(\Dynkin, \xi)$ denotes a Dynkin quiver of type $ADE$.

\begin{proposition}[\cite{B99,KKK13b}]  \label{prop: range descrption}
\ben
\item Let $Q$ be a Dynkin quiver of $ADE$-type, and $\sfh$ the Coxeter number.
For $i \in I$, let
$$   r^Q_i \seteq = \dfrac{\sfh +{\rm a}^Q_i -{\rm b}^Q_i}{2}  $$
where ${\rm a}^Q_i$ is the number of arrows in $Q$ between $i$ and $i^*$ directed toward $i^*$,
${\rm b}^Q_i$ is the number of arrows in $Q$ between $i$ and $i^*$ directed toward $i$ \ro ${\rm a}^Q_i -{\rm b}^Q_i=\xi_{i}-\xi_{i^*})$.
Then we have  $ r^Q_i\in\Z_{\ge0}$ and
$$(\Gamma_Q)_0 = \{ (i,\xi_i-2k ) \ | \  0 \le k < r_i^Q  \}.$$
\item For $i , j \in \Dynkin_0$ and any $Q$, we have
\begin{align*} 
(i,\xi_j-d(i,j)),\ (i,\xi_j-2r_j^Q+d(i,j))\in \Gamma_Q.\
\end{align*}
\ee
\end{proposition}

Note that, if $Q$ is $\sigma$-fixed $(\sigma\not=\id)$, we have
\begin{align}\label{eq: riq fixed}
 r_i^Q = \sfh/2 \quad \text{ for all $i$}.
\end{align}

\begin{proposition}[{\cite[Proposition 4.1]{Oh18}}] \label{prop: dir Q cnt}
Let $Q$ be a Dynkin quiver $Q$ of type $ADE$
and let $\pair{\al,\be}$ be a \prq  sectional in $\Gamma_Q$. Then, we have
\begin{enumerate}
\item[{\rm (a)}] the \prq $\pair{\al,\be}$ is $[Q]$-simple,
\item[{\rm (b)}] $ (\al , \be)=1$,
\item[{\rm (c)}] $\al-\be\in\Phi$,
\item[{\rm (d)}] either $\pair{\al,\be-\al}$ is a $[Q]$-minimal \prq for $\be$ or $\pair{\al-\be,\be}$ is a $[Q]$-minimal \prq for $\al$.
\end{enumerate}
\end{proposition}

\begin{definition} \label{def: Phi ijk}
For $i,j \in I$ and integer $k \in
\Z_{\ge 0}$, let $\Phi_Q(i,j)[k]$
be the set of \prs[{[Q]}] $\pair{\al,\be}$
such that 
$$   |p-s| = k \qt{   when $\{ \phi_{Q}(i,p), \phi_{Q}(j,s)\} = \{ (\al,0),(\be,0) \}$.   }$$
\end{definition}

\begin{lemma} [{\cite[Lemma 6.15]{Oh18}}] \label{lem: integer o}
For any $\pair{\al,\be}$, $\pair{\al',\be'} \in \Phi_Q(i,j)[k]$, we have
$$ \dg_{[Q]}(\pair{\al,\be}) =\dg_{[Q]}(\pair{\al',\be'}).$$
\end{lemma}

For $i,j\in I$ and $k\in\Z_{>0}$, we define $\tto_k^{Q}(i,j)$
by
\eq\tto_k^{Q}(i,j)\seteq
\bc\dg_{[Q]}(\pair{\al,\be})&\text{
if $\pair{\al,\be} \in \Phi_Q(i,j)[k]$,}\\
0&\text{if $\Phi_Q(i,j)[k]=\emptyset$.}
\ec\label{def:tto}\eneq
It is well-defined by the above lemma.

\begin{definition} \label{def: Distance 1}
Let $t$ be an indeterminate.
For $i,j \in I$ and a Dynkin quiver $Q$ of type ADE,
we define a polynomial $\de_{i,j}^Q(z) \in \Z_{\ge0}[t]$
$$    \de_{i,j}^Q(t) \seteq \sum_{k \in \Z_{\ge 0} } {\ttO_k^{Q}(i,j)} t^{k-1},$$
where
\begin{align} \label{eq: ttO}
\ttO_k^{Q}(i,j) \seteq \max(\tto_k^{Q}(i,j), \tto_k^{Q^*}(i,j)).
\end{align}
\end{definition}

\begin{proposition} [{\cite[Proposition 6.18]{Oh18}}] \label{prop: well-defined integer}
For any Dynkin quivers $Q$, $Q'$ of $\Dynkin$,
we have $$\ttO_k^{Q}(i,j) = \ttO_k^{Q'}(i,j).$$
Thus $\ttO_k^{\Dynkin}(i,j)$ is well-defined.
\end{proposition}

\begin{remark}
In \cite{Oh18}, $\ttO_k^{Q}(i,j)$ was defined as $\max(\tto_k^{Q}(i,j), \tto_k^{Q^\rev}(i,j))$, where  $Q^{\rev}$
is a Dynkin quiver obtained from $Q$ by reversing all arrows. However, by applying the argument in \cite{Oh18}, one can check
$$\max(\tto_k^{Q}(i,j), \tto_k^{Q^\rev}(i,j)) = \max(\tto_k^{Q}(i,j), \tto_k^{Q^*}(i,j)).$$
\end{remark}

By Proposition~\ref{prop: well-defined integer}, we can replace the superscript $^Q$ in  $\de_{i,j}^Q(t)$ with $^\Dynkin$. We call
$\de_{i,j}^\Dynkin(t)$  the \emph{degree polynomial of $\Dynkin$ at $(i,j)$}. For $k \in \Z$, we set
$$ \de_{i,j}[k]  \seteq  \ttO_k^\Dynkin(i,j).$$

In the next \S\,\ref{subsub:D} and \S\,\ref{subsub:A}
we recall the results in~\cite{Oh16,Oh17,Oh18}
on the combinatorial properties and statistics of AR-quivers of type $A_{n}$ and $D_n$,  and prove that
$\tde_{i,j}^\Dynkin(t)=\de_{i,j}^\Dynkin(t)+\delta_{i,j^*}t^{\sfh-1}$ in those types.

\subsubsection{$D_{n+1}$-case}\label{subsub:D}  Let $Q$ be a Dynkin quiver of type $D_{n+1}$.
The simple roots $\{\al_i \ | \ 1 \le i \le n+1 \}$ and $\Phi^+_{D_{n+1}}$ can be identified in $\rl^+ \subset   \soplus_{i=1}^{n+1} \R\ep_i$ as follows:
\begin{equation}\label{eq: PR D}
\begin{aligned}
& \al_i =\ep_i- \ep_{i+1} \quad \text{ for $i<n+1$} \quad\text{ and } \quad \al_{n+1} =\ep_n+\ep_{n+1}, \\
&\Phi_{D_{n+1}}^+  = \left\{  \ep_i -\ep_{j} = \sum_{i \le k <j}  \al_k \ | \ 1 \le i < j \le n \right\} \\
&\quad  \ssqcup \left\{  \ep_i-\ep_{n+1} = \sum_{k=i}^{n} \al_k    \ | \ 1 \le i \le n \right\} \ssqcup \left\{  \ep_i+\ep_{n+1} = \sum_{k=i}^{n-1} \al_k  + \al_{n+1} \ | \ 1 \le i   \le n \right\} \\
& \quad\ssqcup \left\{  \ep_i+\ep_{j} = \sum_{ i \le k <j} \al_k + 2\sum_{j \le k <n} \al_s +\big(\al_n +\al_{n+1} \big) \ | \ 1 \le i <j  < n+1 \right\}.
\end{aligned}
\end{equation}

\begin{convention}
For $2 \le a \le n+1$, we write $\ep_{-a}$ to denote $-\ep_a$. For $\be=\ep_a+\ep_b \in \Phi$ $(1 \le |a| < |b| \le n+1)$, we call $\ep_a$ and $\ep_b$ \emph{components} of $\be$.
\end{convention}

\begin{remark} \label{rmk: analy roots D}
Note the followings:
\bnum
\item For each $1 \le a \le n$, there are exactly $(2n+1-a)$-many positive roots having $\ep_a$ as their components.
\item For each $2 \le a \le n$, there are exactly $(a-1)$-many positive roots having $\ep_{-a}$ as their components.
\item There are exactly $n$-many positive roots having $\ep_{\pm (n+1)}$ as their components.
\ee
\end{remark}

\begin{definition} \label{def: swing}    A connected subquiver $\uprho$ in $\Gamma_Q$ is called a {\it swing} if it consists of vertices and arrows in the following way: There exist positive roots
$\al,\be \in \Phi^+$ and $r,s \le n$ such that
$$\raisebox{2em}{\xymatrix@C=2.5ex@R=0.1ex{ &&&&\be\ar@{->}[dr] \\S_r \ar@{->}[r] & S_{r+1} \ar@{->}[r] & \cdots \ar@{->}[r]& S_{n-1}
\ar@{->}[ur]\ar@{->}[dr] && N_{n-1} \ar@{->}[r]& N_{n-2} \ar@{->}[r] & \cdots \ar@{->}[r] & N_s,\\
&&&&\al\ar@{->}[ur] }} \text{ where }$$
\bna
\item $ \raisebox{1.8em}{\xymatrix@C=2.5ex@R=0.1ex{ &&&&\be
\\ S_{r} \ar@{->}[r] & S_{r+1} \ar@{->}[r] & \cdots \ar@{->}[r]& S_{n-1} \ar@{->}[ur]\ar@{->}[dr] \\
&&&&\al}}$ is an $S$-path ($\phi^{-1}_{Q}(S_l,0)=(l,k)$ for   some $k \in \Z$),
\item $\raisebox{1.8em}{\xymatrix@C=2.5ex@R=0.1ex{ \be \ar@{->}[dr] \\
& N_{n-1} \ar@{->}[r]& N_{n-2} \ar@{->}[r] & \cdots \ar@{->}[r] & N_{s}\\
\al\ar@{->}[ur]}}$ is an $N$-path ($\phi^{-1}_{Q}(N_l,0)=(l,k')$ for   some   $k' \in \Z$),
\item $\phi^{-1}_{Q}(\be,0)=(n,u)$ and $\phi^{-1}_{Q}(\al,0)=(n+1,u)$ for some $u \in \Z$
\ee
(see Example~\ref{ex: label D} below for examples of swings).
\end{definition}

\begin{proposition} [{\cite[Theorem 2.21]{Oh16}}]\label{prop: label1}
For each swing $\uprho$ of $\Gamma_Q$, there exists a unique $1 \le a \le n$ satisfying the following properties$\colon$
\bnum
\item Every vertex shares a component $\ep_a$.
\item It consists of exactly $(2n+1-a)$-many vertices.
\item Every positive roots having $\ep_a$ as its component appears in the swing.
\item $\min(r,s)=1$, where $r$, $s$ are given in Definition~\ref{def: swing}.
\ee
Thus we say the swing as an $a$-swing and denote it by $\uprho_a$.
\end{proposition}

By above proposition and Remark~\ref{rmk: analy roots D}, there are distinct $n$-swings in $\Gamma_Q$.

\begin{definition}
We say that a $S$-path (resp.\ $N$-path) is {\it shallow} if it ends (resp.\ starts) at residue less than $n$.
\end{definition}

\begin{proposition}[{\cite[Theorem 2.25]{Oh16}}]
\label{prop: label2} Let $\rho$ be a shallow $S$-path $($resp.\ $N$-path$)$. Then there exists a unique $2 \le a \le n$ such that
\bnum
\item $\rho$ consists of $(a-1)$-many vertices,
\item all vertices in $\rho$ share $\ep_{-a}$ as their components.
\ee Furthermore, there are exactly $(n-1)$-many shallow $X$-paths
$(X=S$ or $N)$ in $\Gamma_Q$. Thus we say the shallow $X$-path as
$(-a)$-shallow path and denote it by $\rho_{-a}$.
\end{proposition}

In this paper, we sometimes understand elements in $\sfW_{D_{n+1}}$ as automorphisms of the set $$ \{ \pm 1, \pm 2,\ldots,\pm (n+1) \},$$
i.e., $s_i\ep_{t}=\ep_{s_i(t)}$ for $t$ with $1 \le |t| \le n$.
Note that $s_ns_{n+1}= s_{n+1}s_n$, $s_ns_{n+1}(n)=-n$ and  $s_ns_{n+1}(n+1)=-n-1$.

\begin{lemma} \label{lem: tauQ fixed D}
Let $Q=(\Dynkin,\vxi)$ be a $\vee$-fixed Dynkin quiver of type $D_{n+1}$. 
\bna
\item If $\phi^{-1}_{Q}(\al,0)=(n,k)$ $($resp.\ $(n+1,k))$ for some $\al \in \Phi^+$, there exists $\be \in \Phi^+$ such that $\phi^{-1}_{Q}(\be,0)=(n+1,k)$ $($resp.\ $(n,k))$.
\item  For $\al,\be \in \Phi^+$ with $\phi^{-1}_{Q}(\al,0)=(n,k)$ and $\phi^{-1}_{Q}(\be,0)=(n+1,k)$, there exists a unique $1 \le a\le n$ such that
$$ \{ \al,\be\}  = \{ \ep_a+\ep_{n+1},  \ \ep_a-\ep_{n+1} \}.$$
\item \label{it: fixed D} $\tau_Q( \pm (n+1)) = \mp (n+1)$.
\ee
\end{lemma}

\begin{proof} The first assertion is  \cite[Lemma 2.14]{Oh16}.  Recall that any reduced expression of $\tau_Q$ is adapted to $Q$. Since $\xi_{n}=\xi_{n+1}$, there exists a reduced expression $\uw = s_{i_1}\cdots s_{i_{n+1}}$ of $\tau_Q$ and $u \in [1,n]$ such that
$\{ i_{u}, i_{u+1} \} = \{ n,n+1\}$. Then we have
$$   \tau_Q(n+1) =  s_{i_1}\cdots s_{i_{u-1}}s_{n}s_{n+1}(n+1) =  s_{i_1}\cdots s_{i_{u-1}}(-(n+1)) =-(n+1),$$
as desired.
\end{proof}

Using the results above,
we can label the $\Gamma_Q$ in terms of $\Phi^+$ without computation
$\be^{\uw_0}_k$'s or the bijection $\phi_{Q}|_{\Gamma_Q}$.

\begin{example} \label{ex: label D}
Let us consider the following $\vee$-fixed Dynkin quiver $Q$ of type $D_5$:
$$Q \ =\raisebox{1.2em}
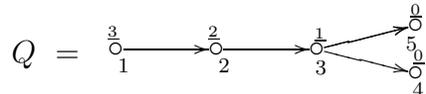
{ \xymatrix@R=0.5ex@C=6ex{ &&&*{\circ}<3pt> \ar@{<-}[dl]^<{5 }_<{ \ \ \ _{\underline{0}}} \\
*{\circ}<3pt> \ar@{->}[r]^<{ _{\underline{3}} \ \  }_<{1} & *{\circ}<3pt>  \ar@{->}[r]^<{  _{\underline{2}} \ \  }_<{2} & *{\circ}<3pt>  \ar@{->}[ur]  \ar@{->}[dr]^<{  _{\underline{1}} \ \  }_<{3} \\
&&&*{\circ}<3pt> \ar@{}[l]_>{ \qquad\qquad\qquad _{\underline{0}}}^<{ \ \ 4}  }}$$
By ~\eqref{eq: bijection}, the $\Gamma_Q$ without labeling are given as follows:
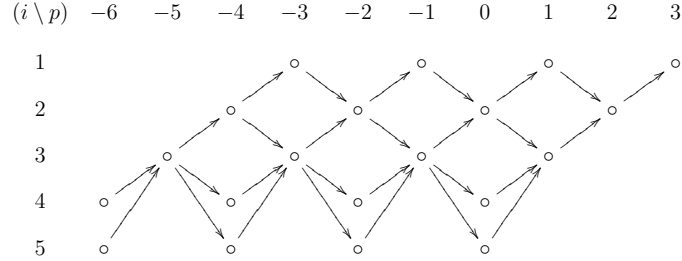
\begin{align*}
 \raisebox{3.2em}{ \scalebox{0.7}{\xymatrix@!C=2ex@R=2ex{
(i\setminus p) &-6  &-5  &-4 &-3 &  -2 & -1 & 0 & 1 & 2  & 3  \\
1&&&& \circ \ar[dr]  && \circ \ar[dr]&& \circ \ar[dr]&& \circ   \\
2&&& \circ \ar[dr]\ar[ur]  && \circ \ar[dr]\ar[ur] && \circ\ar[dr]\ar[ur] && \circ \ar[ur]   \\
3&& \circ \ar[ddr] \ar[dr]\ar[ur]  && \circ \ar[ddr]\ar[dr]\ar[ur] && \circ \ar[ddr]\ar[dr]\ar[ur]&& \circ \ar[ur]  \\
4& \circ \ar[ur]  && \circ \ar[ur]&& \circ \ar[ur]&& \circ \ar[ur]   \\
5& \circ  \ar[uur] && \circ  \ar[uur]&& \circ  \ar[uur] && \circ   \ar[uur]
}}}
\end{align*}
Using (i) Proposition~\ref{prop: label1} and (ii)  Lemma~\ref{lem: tauQ fixed D} and Proposition~\ref{prop: label2} in order, we can complete the label of $\Gamma_Q$
as follows:

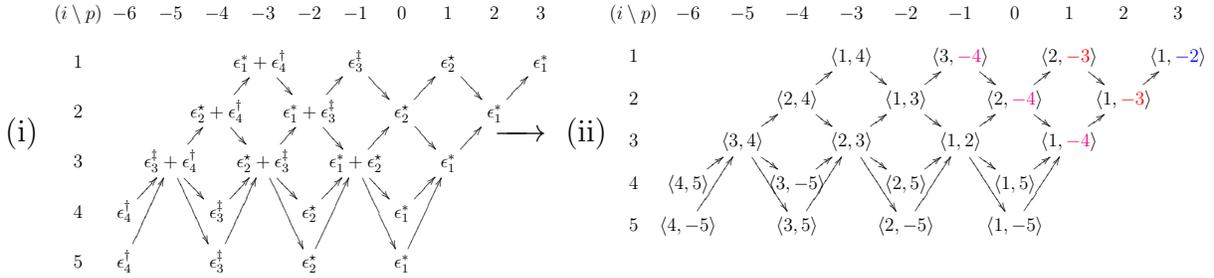
\begin{align*}
{\rm (i)} \raisebox{4.1em}{ \scalebox{0.6}{\xymatrix@!C=1ex@R=2ex{
(i\setminus p) &-6  &-5  &-4 &-3 &  -2 & -1 & 0 & 1 & 2  & 3  \\
1&&&&  \ep_1^*  +\ep_4^\dagger\ \ar[dr]  &&\ep_3^\ddagger\ar[dr]&& \ep_2 ^\star\ar[dr]&&\ep_1^*  \\
2&&& \ep_2 ^\star+\ep_4^\dagger\ar[dr]\ar[ur]  &&\ep_1^*+\ep_3^\ddagger \ar[dr]\ar[ur] &&\ep_2 ^\star\ar[dr]\ar[ur] &&\ep_1^*\ar[ur]   \\
3&&\ep_3^\ddagger+\ep_4^\dagger\ar[ddr] \ar[dr]\ar[ur]  && \ep_2 ^\star+\ep_3^\ddagger\ar[ddr]\ar[dr]\ar[ur] &&\ep_1^*+\ep_2 ^\star\ar[ddr]\ar[dr]\ar[ur]&&\ep_1^*\ar[ur]  \\
4& \ep_4^\dagger\ar[ur]  &&\ep_3^\ddagger\ar[ur]&& \ep_2 ^\star\ar[ur]&&\ep_1^*\ar[ur]   \\
5& \ep_4^\dagger \ar[uur] &&\ep_3^\ddagger \ar[uur]&& \ep_2 ^\star \ar[uur] &&\ep_1^*  \ar[uur]
}}} \hspace{-5ex}\To \
{\rm (ii)} \hspace{-1ex} \raisebox{4.1em}{ \scalebox{0.6}{\xymatrix@!C=2ex@R=1.5ex{
(i\setminus p) &-6  &-5  &-4 &-3 &  -2 & -1 & 0 & 1 & 2  & 3  \\
1&&&& \srt{1,4} \ar[dr]  && \srt{3,{\color{magenta}-4}} \ar[dr]&& \srt{2,{\color{red}-3}} \ar[dr]&& \srt{1, {\color{blue}-2}}   \\
2&&& \srt{2,4} \ar[dr]\ar[ur]  && \srt{1,3} \ar[dr]\ar[ur] && \srt{2,{\color{magenta}-4}}\ar[dr]\ar[ur] && \srt{1,{\color{red}-3}} \ar[ur]   \\
3&& \srt{3,4} \ar[ddr] \ar[dr]\ar[ur]  && \srt{2,3} \ar[ddr]\ar[dr]\ar[ur] && \srt{1,2} \ar[ddr]\ar[dr]\ar[ur]&& \srt{1,{\color{magenta}-4}} \ar[ur]  \\
4& \srt{4,5} \ar[ur]  && \srt{3,-5} \ar[ur]&& \srt{2,5} \ar[ur]&& \srt{1,5} \ar[ur]   \\
5& \srt{4,-5}  \ar[uur] && \srt{3,5}  \ar[uur]&& \srt{2,-5}  \ar[uur] && \srt{1,-5}   \ar[uur]
}}}
\end{align*}
Here, each symbol $^{\divideontimes}$ in (i) represents an $a$-swing in Proposition~\ref{prop: label1} and each color in (ii) represents an $(-a)$-shallow path in Proposition~\ref{prop: label2}.

\end{example}
Furthermore, in~\cite{Oh18}, the second named author proved
the following proposition by using
the labeling algorithm for $(\Gamma_Q)_0$ and considering all relative positions of $\pair{\al,\be}$'s.

\begin{proposition}[{\cite[Proposition 4.5]{Oh18} $($see also  \cite[Lemma 2.6]{BKM12}$)$}]
For any \pr $\up=\pair{\al,\be}$ and any Dynkin quiver $Q$ of type $D$, we have the followings:
\ben
\item   $\head_{[Q]}(\pair{\al,\be})$   is well-defined. In particular, if $\al+\be = \ga \in \Phi^+$, then    $\head_{[Q]}(\pair{\al,\be})=\ga.$
\item $\dg_{[Q]}(\pair{\al,\be}) \le 2$.
\ee
\end{proposition}

The relative positions for all \prs $\pair{\al,\be}$ with $\dg_Q(\pair{\al,\be}) \ne 0$ are classified and described, which do not depend on the choice of $Q$, as follows (\cite{Oh18}):
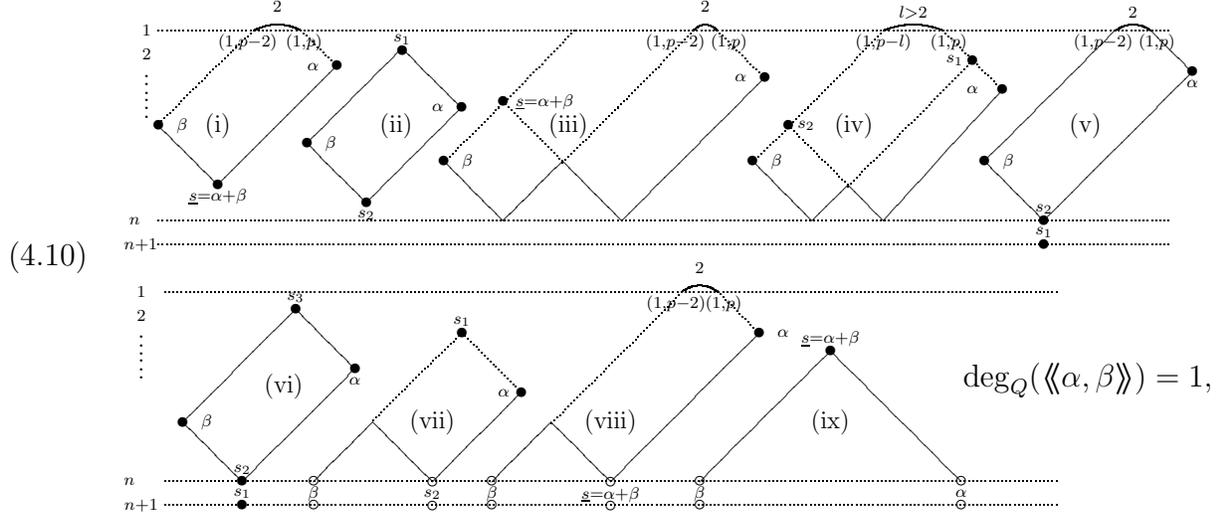
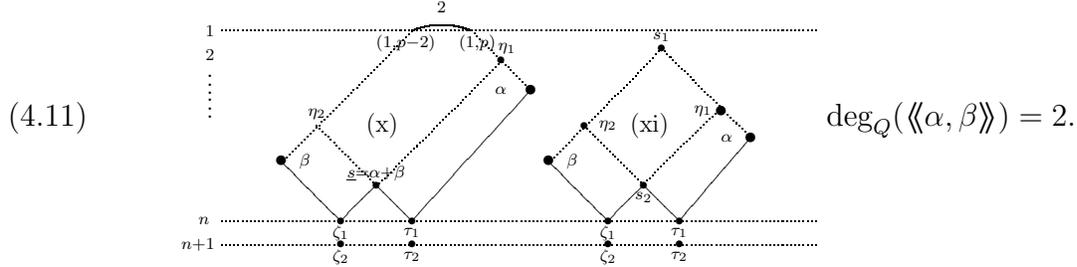
\begin{equation} \label{eq: dist al, be 1}
\begin{aligned}
& \scalebox{0.79}{{\xy
(-20,0)*{}="DL";(-10,-10)*{}="DD";(0,20)*{}="DT";(10,10)*{}="DR";
"DT"+(-30,-4); "DT"+(140,-4)**\dir{.};
"DD"+(-20,-6); "DD"+(150,-6) **\dir{.};
"DD"+(-20,-10); "DD"+(150,-10) **\dir{.};
"DT"+(-32,-4)*{\scriptstyle 1};
"DT"+(-32,-8)*{\scriptstyle 2};
"DT"+(-32,-12)*{\scriptstyle \vdots};
"DT"+(-32,-16)*{\scriptstyle \vdots};
"DT"+(-34,-36)*{\scriptstyle n};
"DT"+(-33,-40)*{\scriptstyle n+1};
"DL"+(-10,0); "DD"+(-10,0) **\dir{-};"DL"+(-10,0); "DT"+(-14,-4) **\dir{.};
"DT"+(-6,-4); "DR"+(-10,0) **\dir{.};"DR"+(-10,0); "DD"+(-10,0) **\dir{-};
"DL"+(-6,0)*{\scriptstyle \be};
"DL"+(0,0)*{{\rm(i)}};
"DR"+(-14,0)*{\scriptstyle \al};
"DL"+(-10,0)*{\bullet};
"DR"+(-10,0)*{\bullet};
"DT"+(-6,-4); "DT"+(-14,-4) **\crv{"DT"+(-10,-2)};
"DT"+(-5.5,-6)*{{\scriptstyle(1,p) }};
"DT"+(-15,-6)*{{\scriptstyle(1,p-2) }};
"DT"+(-10,0)*{\scriptstyle 2};
"DD"+(-10,0)*{\bullet};
"DD"+(-10,-2)*{\scriptstyle \us=\al+\be};
"DL"+(15,-3); "DD"+(15,-3) **\dir{-};
"DR"+(11,-7); "DD"+(15,-3) **\dir{-};
"DT"+(11,-7); "DR"+(11,-7) **\dir{-};
"DT"+(11,-7); "DL"+(15,-3) **\dir{-};
"DL"+(19,-3)*{\scriptstyle \be};
"DR"+(7,-7)*{\scriptstyle \al};
"DL"+(30,0)*{{\rm(ii)}};
"DL"+(15,-3)*{\bullet};
"DL"+(31,12.5)*{\bullet};
"DL"+(31,14.5)*{\scriptstyle s_1};
"DL"+(25,-13)*{\bullet};
"DL"+(25,-15)*{\scriptstyle s_2};
"DR"+(11,-7)*{\bullet};
"DD"+(28,4); "DD"+(38,-6) **\dir{-};
"DD"+(28,4); "DD"+(50,26) **\dir{.};
"DD"+(48,4); "DD"+(38,-6) **\dir{-};
"DD"+(48,4); "DD"+(70,26) **\dir{.};
"DD"+(82,18); "DD"+(74,26) **\dir{.};
"DD"+(82,18); "DD"+(58,-6) **\dir{-};
"DD"+(48,4); "DD"+(58,-6) **\dir{-};
"DD"+(82,18)*{\bullet};
"DD"+(78,18)*{\scriptstyle \al};
"DD"+(28,4)*{\bullet};
"DD"+(32,4)*{\scriptstyle \be};
"DD"+(74,26); "DD"+(70,26) **\crv{"DD"+(72,28)};
"DD"+(76,24)*{{\scriptstyle(1,p) }};
"DD"+(67,24)*{{\scriptstyle(1,p-2) }};
"DD"+(72,30)*{\scriptstyle 2};
"DD"+(38,14)*{\bullet};
"DD"+(45,14)*{\scriptstyle \us=\al+\be};
"DL"+(59,0)*{{\rm(iii)}};
"DD"+(48,4); "DD"+(38,14) **\dir{.};
"DD"+(80,4); "DD"+(90,-6) **\dir{-};
"DD"+(96,0); "DD"+(90,-6) **\dir{-};
"DD"+(96,0); "DD"+(102,-6) **\dir{-};
"DD"+(122,16); "DD"+(102,-6) **\dir{-};
"DD"+(80,4); "DD"+(102,26) **\dir{.};
"DD"+(122,16); "DD"+(112,26) **\dir{.};
"DD"+(96,0); "DD"+(117,21) **\dir{.};
"DD"+(96,0); "DD"+(86,10) **\dir{.};
"DD"+(117,21)*{\bullet};
"DD"+(114,21)*{\scriptstyle s_1};
"DD"+(86,10)*{\bullet};
"DD"+(89,10)*{\scriptstyle s_2};
"DD"+(102,26); "DD"+(112,26) **\crv{"DD"+(107,28)};
"DD"+(113,24)*{{\scriptstyle(1,p) }};
"DD"+(102,24)*{{\scriptstyle(1,p-l) }};
"DD"+(107,29)*{\scriptstyle l >2};
"DD"+(80,4)*{\bullet};
"DL"+(107,0)*{{\rm (iv)}};
"DD"+(84,4)*{\scriptstyle \be};
"DD"+(122,16)*{\bullet};
"DD"+(117,16)*{\scriptstyle \al};
%
"DD"+(119,4); "DD"+(129,-6) **\dir{-};
"DD"+(135,0); "DD"+(129,-6) **\dir{-};
"DD"+(119,4); "DD"+(141,26) **\dir{-};
"DD"+(154,19); "DD"+(147,26) **\dir{-};
"DD"+(135,0); "DD"+(154,19) **\dir{-};
"DD"+(119,4)*{\bullet};
"DL"+(146,0)*{{\rm (v)}};
"DD"+(123,4)*{\scriptstyle \be};
"DD"+(141,26); "DD"+(147,26) **\crv{"DD"+(144,28)};
"DD"+(148,24)*{{\scriptstyle(1,p) }};
"DD"+(139,24)*{{\scriptstyle(1,p-2) }};
"DD"+(154,19)*{\bullet};
"DD"+(154,17)*{\scriptstyle \al};
"DD"+(129,-10)*{\bullet};
"DD"+(129,-8)*{\scriptstyle s_1};
"DD"+(129,-6)*{\bullet};
"DD"+(129,-4)*{\scriptstyle s_2};
"DD"+(144,29)*{\scriptstyle 2};
\endxy}}
\allowdisplaybreaks \\
& \scalebox{0.79}{{\xy
(-20,0)*{}="DL";(-10,-10)*{}="DD";(0,20)*{}="DT";(10,10)*{}="DR";
"DT"+(-30,-4); "DT"+(120,-4)**\dir{.};
"DD"+(-20,-6); "DD"+(130,-6) **\dir{.};
"DD"+(-20,-10); "DD"+(130,-10) **\dir{.};
"DT"+(-34,-4)*{\scriptstyle 1};
"DT"+(-34,-8)*{\scriptstyle 2};
"DT"+(-34,-12)*{\scriptstyle \vdots};
"DT"+(-34,-16)*{\scriptstyle \vdots};
"DT"+(-36,-36)*{\scriptstyle n};
"DT"+(-34,-40)*{\scriptstyle n+1};
"DD"+(-17,4); "DD"+(-7,-6) **\dir{-};
"DD"+(-1,0); "DD"+(-7,-6) **\dir{-};
"DD"+(-17,4); "DD"+(2,23) **\dir{-};
"DD"+(12,13); "DD"+(2,23) **\dir{-};
"DD"+(-1,0); "DD"+(12,13) **\dir{-};
"DD"+(-17,4)*{\bullet};
"DD"+(2,23)*{\bullet};
"DL"+(10,0)*{{\rm (vi)}};
"DD"+(-13,4)*{\scriptstyle \be};
"DD"+(2,25)*{\scriptstyle s_3};
"DD"+(12,13)*{\bullet};
"DD"+(12,11)*{\scriptstyle \al};
"DD"+(-7,-10)*{\bullet};
"DD"+(-7,-8)*{\scriptstyle s_1};
"DD"+(-7,-6)*{\bullet};
"DD"+(-7,-4)*{\scriptstyle s_2};
"DD"+(15,4); "DD"+(25,-6) **\dir{-};
"DD"+(15,4); "DD"+(5,-6) **\dir{-};
"DD"+(15,4); "DD"+(30,19) **\dir{.};
"DD"+(40,9);"DD"+(25,-6) **\dir{-};
"DD"+(40,9);"DD"+(30,19) **\dir{.};
"DD"+(40,9)*{\bullet};
"DD"+(30,19)*{\bullet};
"DD"+(30,21)*{\scriptstyle s_1};
"DD"+(5,-6)*{\circ};
"DD"+(5,-10)*{\circ};
"DD"+(5,-8)*{\scriptstyle \be};
"DD"+(25,-6.2)*{\circ};
"DD"+(25,-8.2)*{\scriptstyle s_2};
"DD"+(25,-10.2)*{\circ};
"DD"+(37,9)*{\scriptstyle \al};
"DD"+(25,4)*{{\rm (vii)}};
"DD"+(45,4); "DD"+(55,-6) **\dir{-};
"DD"+(45,4); "DD"+(35,-6) **\dir{-};
"DD"+(45,4); "DD"+(67,26) **\dir{.};
"DD"+(80,19);"DD"+(55,-6) **\dir{-};
"DD"+(80,19);"DD"+(73,26) **\dir{.};
"DD"+(80,19)*{\bullet};
"DD"+(55,4)*{{\rm (viii)}};
"DD"+(84,19)*{\scriptstyle \al};
"DD"+(35,-6)*{\circ};
"DD"+(35,-10)*{\circ};
"DD"+(55,-6.2)*{\circ};
"DD"+(55,-8.2)*{\scriptstyle \us=\al+\be};
"DD"+(55,-10.2)*{\circ};
"DD"+(35,-8)*{\scriptstyle \be};
"DD"+(67,26); "DD"+(73,26) **\crv{"DD"+(70,28)};
"DD"+(74,24)*{{\scriptstyle(1,p) }};
"DD"+(66,24)*{{\scriptstyle(1,p-2) }};
"DD"+(70,30)*{\scriptstyle 2};
"DD"+(92,16); "DD"+(114,-6) **\dir{-};
"DD"+(92,16); "DD"+(70,-6) **\dir{-};
"DD"+(92,4)*{ {\rm (ix)}};
"DD"+(114,-6)*{\circ};
"DD"+(114,-10)*{\circ};
"DD"+(114,-8)*{\scriptstyle \al};
"DD"+(70,-6)*{\circ};
"DD"+(92,16)*{\bullet};
"DD"+(92,18)*{\scriptstyle \us=\al+\be};
"DD"+(70,-10)*{\circ};
"DD"+(70,-8)*{\scriptstyle \be};
\endxy}} \hspace{-3em} \dg_Q(\pair{\al,\be})=1,
\end{aligned}
\end{equation}
and
\begin{align} \label{eq: dist al be 2}
\scalebox{0.79}{{\xy
(-20,0)*{}="DL";(-10,-10)*{}="DD";(0,20)*{}="DT";(10,10)*{}="DR";
"DT"+(-30,-4); "DT"+(70,-4)**\dir{.};
"DD"+(-20,-6); "DD"+(80,-6) **\dir{.};
"DD"+(-20,-10); "DD"+(80,-10) **\dir{.};
"DT"+(-32,-4)*{\scriptstyle 1};
"DT"+(-32,-8)*{\scriptstyle 2};
"DT"+(-32,-12)*{\scriptstyle \vdots};
"DT"+(-32,-16)*{\scriptstyle \vdots};
"DT"+(-33,-36)*{\scriptstyle n};
"DT"+(-34,-40)*{\scriptstyle n+1};
"DD"+(-10,4); "DD"+(0,-6) **\dir{-};
"DD"+(6,0); "DD"+(0,-6) **\dir{-};
"DD"+(6,0); "DD"+(12,-6) **\dir{-};
"DD"+(32,16); "DD"+(12,-6) **\dir{-};
"DD"+(-10,4); "DD"+(12,26) **\dir{.};
"DD"+(32,16); "DD"+(22,26) **\dir{.};
"DD"+(6,0); "DD"+(27,21) **\dir{.};
"DD"+(6,0); "DD"+(-4,10) **\dir{.};
"DD"+(12,26); "DD"+(22,26) **\crv{"DD"+(17,28)};
"DD"+(23,24)*{{\scriptstyle(1,p) }};
"DD"+(11,24)*{{\scriptstyle(1,p-2) }};
"DD"+(17,30)*{\scriptstyle 2};
"DD"+(-10,4)*{\bullet};
"DL"+(17,0)*{{\rm(x)}};
"DD"+(-6,4)*{\scriptstyle \be};
"DD"+(-4,12)*{\scriptstyle \eta_2};
"DD"+(32,16)*{\bullet};
"DD"+(27,16)*{\scriptstyle \al};
"DD"+(28,23)*{\scriptstyle \eta_1};
"DD"+(27,21)*{\scriptstyle \bullet};
"DT"+(2,-36)*{\scriptstyle \bullet};
"DT"+(2,-38)*{\scriptstyle \tau_1};
"DT"+(2,-40)*{\scriptstyle \bullet};
"DT"+(2,-42)*{\scriptstyle \tau_2};
"DT"+(-10,-36)*{\scriptstyle \bullet};
"DT"+(-10,-38)*{\scriptstyle \zeta_1};
"DT"+(-10,-40)*{\scriptstyle \bullet};
"DT"+(-10,-42)*{\scriptstyle \zeta_2};
"DT"+(-4,-30)*{\scriptstyle \bullet};
"DT"+(-4,-28)*{\scriptstyle \us=\al+\be};
"DD"+(35,4); "DD"+(45,-6) **\dir{-};
"DD"+(51,0); "DD"+(45,-6) **\dir{-};
"DD"+(51,0); "DD"+(57,-6) **\dir{-};
"DD"+(69,8); "DD"+(57,-6) **\dir{-};
"DD"+(35,4); "DD"+(54,23) **\dir{.};
"DD"+(69,8); "DD"+(54,23) **\dir{.};
"DD"+(51,0); "DD"+(64,13) **\dir{.};
"DD"+(51,0); "DD"+(41,10) **\dir{.};
"DD"+(35,4)*{\bullet};
"DL"+(62,0)*{{\rm(xi)}};
"DD"+(39,4)*{\scriptstyle \be};
"DD"+(41,10)*{\scriptstyle \bullet};
"DD"+(45,10)*{\scriptstyle \eta_2};
"DD"+(69,8)*{\bullet};
"DD"+(64,12.5)*{\bullet};
"DD"+(61,12.5)*{\scriptstyle \eta_1};
"DD"+(65,8)*{\scriptstyle \al};
"DT"+(47,-36)*{\scriptstyle \bullet};
"DT"+(47,-38)*{\scriptstyle \tau_1};
"DT"+(35,-36)*{\scriptstyle \bullet};
"DT"+(35,-38)*{\scriptstyle \zeta_1};
"DT"+(47,-40)*{\scriptstyle \bullet};
"DT"+(47,-42)*{\scriptstyle \tau_2};
"DT"+(35,-40)*{\scriptstyle \bullet};
"DT"+(35,-42)*{\scriptstyle \zeta_2};
"DT"+(41,-32)*{\scriptstyle s_2};
"DT"+(44,-5)*{\scriptstyle s_1};
"DT"+(44,-7)*{\scriptstyle \bullet};
"DT"+(41,-30)*{\scriptstyle \bullet};
\endxy}} \ \dg_Q(\pair{\al,\be})=2.
\end{align}
Here we only exhibit the cases when $\res^{[Q]}(\al) \le \res^{[Q]}(\be)$ and
\ben
\item we set $\us= \seq{s_1,\ldots,s_r}=  \head_Q(\pair{\al,\be})$,
\item
 $\be$ in (vii), (viii) and (ix)  is one of two vertices $\circ$ and then $s_2$ in (vii) (resp.\ $\al+\be$ in (viii), $\al$ in (ix)) is determined by
the choice,
\item every \ex $\um \prec_Q^\ttb \pair{\al,\be}$ is also described in the above pictures ($\prec_Q \seteq \prec_{[Q]}$),
\item in~\eqref{eq: dist al, be 1}, the positive roots with residue $n$ or $n+1$ appear   in (v), (vi), (vii), (viii), (ix) (x) and (xi),
\item see Example~\ref{ex: Uding D} below for $ \dg_Q(\pair{\al,\be})=2$ cases.
\ee

In each case, the value $(\al,\be)$ does not depend on the choice of $Q$ and is given as follows:
\begin{table}[h]
\centering
{ \arraycolsep=1.6pt\def\arraystretch{1.5}
\begin{tabular}{|c||c|c|c|c|c|c|c|c|c|c|c|c|}
\hline
 &   (i) & (ii)  & (iii)  & (iv)  & (v)  & (vi)  & (vii)  & (viii)  & (ix)  & (ix)  & (x)  \\ \hline
$(\al,\be)$ &  $-1$ &  $0$ &  $-1$  &  $0$ &  $0$ &  $1$ &  $0$ &  $-1$ &  $-1$ &  $-1$ &  $0$ \\ \hline
\end{tabular}
}\\[1.5ex]
    \caption{$(\al,\be)$ for non $[Q]$-simple \prs of type $D_{n+1}$}
    \protect\label{table: al,be D}
\end{table}

\begin{example} \label{ex: Uding D}
Using $\Gamma_Q$ in Example~\ref{ex: label D},
\ben
\item $ \srt{1,-4} \prec_Q^\ttb \pair{\srt{1,-2},\srt{2,-4}}$ corresponds to (i),
\item $ \pair{\srt{2,-3},\srt{1,-4}} \prec_Q^\ttb \pair{\srt{1,-3},\srt{2,-4}}$ corresponds to (ii),
\item $ \srt{1,4}  \prec_Q^\ttb  \pair{\srt{1,-2},\srt{2,4}}$ corresponds to (iii),
\item $ \pair{\srt{1,4},\srt{2,-3}} \prec_Q^\ttb  \pair{\srt{1,-3},\srt{2,4}}$ corresponds to (iv),
\item $ \pair{\srt{1,5},\srt{1,-5}} \prec_Q^\ttb  \pair{\srt{1,-2},\srt{1,2}}$ corresponds to (v),
\item $ \seq{\srt{2,-3} ,\srt{1,5},\srt{1,-5}} \prec_Q^\ttb  \pair{\srt{1,-3},\srt{1,2}}$ corresponds to (vi),
\item $ \pair{\srt{3,-4}, \srt{1,5}} \prec_Q^\ttb  \pair{\srt{1,-4},\srt{3,5}}$ and $ \pair{\srt{3,-4} , \srt{1,-5}} \prec_Q^\ttb  \pair{\srt{1,-4},\srt{3,-5}}$  correspond to (vii),
\item $ \srt{1,5}  \prec_Q^\ttb  \pair{\srt{1,-3},\srt{3,5}}$ and $ \srt{1,-5}  \prec_Q^\ttb  \pair{\srt{1,-3},\srt{3,-5}}$  correspond to (viii),
\item $ \srt{1,3}  \prec_Q^\ttb  \pair{\srt{1,5},\srt{3,-5}}$ and $ \srt{1,3}  \prec_Q^\ttb  \pair{\srt{1,-5},\srt{3,5}}$  correspond to (ix),
\item $ \srt{1,2} \prec_Q^\ttb    \pair{\srt{1,-5},\srt{2,5}}, \pair{\srt{1,5},\srt{2,-5}} ,\pair{\srt{1,3},\srt{2,-3}}  \prec_Q^\ttb  \pair{\srt{1,-3},\srt{2,3}}$ corresponds to (x),
\item $  \pair{\srt{1,2},\srt{3,-4}} \prec_Q^\ttb   \begin{matrix}   \seq{\srt{1,-5},\srt{2,5},\srt{3,-4}} \\ \seq{\srt{1,5},\srt{2,-5},\srt{3,-4}} \\ \pair{\srt{1,3},\srt{2,-4}} \end{matrix}  \prec_Q^\ttb  \pair{\srt{1,-4},\srt{2,3}}$ corresponds to (xi).
\ee
\end{example}

\smallskip

For $\Dynkin$ of type $D_{n+1}$, the degree polynomials $\de_{i,j}^{D_{n+1}}(t)$ for $i,j \in \Dynkin_0$ are given in the following explicit form \cite{Oh16,Oh18,Fuj19}:
\begin{align}\label{eq: d poly D}
\de_{i,j}^{D_{n+1}}(t) + \delta_{i,j^*}t^{\sfh-1} = \bc
  \dfrac{\displaystyle \sum_{s=1}^{\min(i,j)} (t^{|i-j|+2s-1}+t^{2n-i-j+2s-1})}{1+ \delta(\max(i,j) \ge n)}  & \text{ if } \min(i,j)<n,\\[3ex]
\displaystyle \sum_{s=1}^{\lfloor (n+ \delta_{i,j}) /2 \rfloor}  t^{4s-1 -2\delta(i,j)} & \text{ otherwise.}
\ec
\end{align}

\begin{remark} \label{remark: dist 2}
The coefficient $t^{s-1}$ of $\de_{i,j}^\Dynkin(t)$ is $2$  if and only if
\begin{align}   \label{eq: double pole}
2 \le i,j \le n-1, \ \ i+j \ge n+1, \ \ \sfh+2 -i-j \le s \le  i+j \ \text{ and }  \ s\equiv_2 i+j
\end{align}
(see \cite[Lemma 3.2.4]{KKK13b} also).
Equivalently, if $\dg_Q(\pair{\al,\be})=2$, the relative positions of the \prs $\pair{\al,\be}$ in $\Gamma_Q$ are described as in~\eqref{eq: dist al be 2},
where
\bnum
\item[{\rm (x)}] happens when $\al+\be=\ga=\us$ with $\mul(\ga)=2$ and we have
$$      \ga \prec^\ttb_Q \pair{\eta_1,\eta_2},  \pair{\tilde{\tau}_1,\tilde{\zeta}_1}, \pair{\tilde{\tau}_2,\tilde{\zeta}_2} \prec^\ttb_{Q}  \pair{\al,\be} $$
where
$\{ \tilde{\tau}_1,\tilde{\tau}_2\} =\{\tau_1,\tau_2 \}$, $\{ \tilde{\zeta}_1,\tilde{\zeta}_2\} =\{\zeta_1,\zeta_2 \}$ and $\tilde{\tau}_i+\tilde{\zeta}_i =\ga$ $(i=1,2)$,
\item[{\rm (xi)}] happens when $\head_Q(\pair{\al,\be}) = \pair{\ga_1,\ga_2}$ and we have
$$     \pair{ \ga_1,\ga_2} \prec^\ttb_Q \pair{\eta_1,\eta_2},  \seq{\tilde{\tau}_1,\tilde{\zeta}_1,\ga_1}, \seq{\tilde{\tau}_2,\tilde{\zeta}_2,\ga_1} \prec^\ttb_{Q}  \pair{\al,\be} $$
where
$\{ \tilde{\tau}_1,\tilde{\tau}_2\} =\{\tau_1,\tau_2 \}$, $\{ \tilde{\zeta}_1,\tilde{\zeta}_2\} =\{\zeta_1,\zeta_2 \}$ and $\tilde{\tau}_i+\tilde{\zeta}_i =\ga_1$ $(i=1,2)$.
\ee
\end{remark}

\subsubsection{$A_{2n-1}$-case} \label{subsub:A} Let $Q$ be a Dynkin quiver of type $A_{2n-1}$.
The simple roots and $\Phi^+_{A_{2n-1}}$ can be identified in $\rl^+ \subset  \soplus_{i=1}^{2n} \Z\ep_i$ as follows:
\begin{equation}\label{eq: PR A}
\begin{aligned}
& \al_i =\ep_i- \ep_{i+1} \quad \text{ for $1 \le i \le 2n-1$}, \\
&\Phi_{A_{2n-1}}^+  = \left\{  \ep_i -\ep_{j} = \sum_{ i \le k <j}  \al_k \ \left| \ 1 \le i \le j \le 2n  \right. \right\}.
\end{aligned}
\end{equation}

\begin{remark} \label{rmk: vee fixed A}
Note that for $\vee$-fixed Dynkin quiver $Q=(\Dynkin, \vxi)$ of type $A_{2n-1}$,
\ben
\item there is no path $\mathtt{p}$ in $Q$ passing through $n$,
\item $n$ is always sink or source of $Q$,
\ee
since $\vxi_{n-1}=\vxi_{n+1}$ and $|\vxi_n - \vxi_{n \pm 1}|=1$.
\end{remark}

We identify $\be \in \Phi^+_{A_{2n-1}}$ with a \emph{segment} $[a,b]$ for $1 \le a \le b \le 2n-1$, where $\be = \sum_{k=a}^b \al_k$.
For $[a,b]$, we call $a$  the \emph{first component} and $b$ the \emph{second component} of $\be$. If $\be$ is simple,
we sometimes write $\be$ as $[a]$ instead of $[a,a]$.

\begin{proposition}\cite[Theorem 1.11]{Oh17} \label{prop :labeling A} \hfill
\ben
\item Every positive root in an $N$-path shares the same first component
and  every positive root in an $S$-path shares the same second component.
\item For each $1\le i \le 2n-i$, there are exactly one $N$-path $\varrho^N_i$ containing $(2n-i)$ vertices and each vertex in
 $\varrho^N_i$ has $i$ as its first component. Conversely, every positive root with the first component $i$ appears in $\varrho^N_i$.
\item For each $1\le i \le 2n-i$, there are exactly one $S$-path $\varrho^S_i$ containing $i$ vertices and each vertex in
 $\varrho^S_i$ has $i$ as its second component. Conversely, every positive root with the second component $i$ appears in $\varrho^S_i$.
\ee
\end{proposition}

As in $D_{n+1}$-case, we can label the $\Gamma_Q$ in
terms of $\Phi^+$ without computation $\be^{\uw_0}_k$'s or the
bijection $\phi_{Q}|_{\Gamma_Q}$ by using the result above and obtain the following proposition:

\begin{proposition}  \cite[Proposition 4.5]{Oh18} $($see also~\cite[Lemma 2.6]{BKM12}$)$
For any \pr $\up=\pair{\al,\be}$ and any Dynkin quiver $Q$ of type $A_n$, we have the followings:
\ben
\item $\head_{[Q]}(\pair{\al,\be})$ is well-defined. In particular, if $\al+\be = \ga \in \Phi^+$, then  $\head_{[Q]}(\pair{\al,\be}) =\ga.$
\item $\dg_{[Q]}(\pair{\al,\be}) \le 1$.
\ee
\end{proposition}

Understanding the non-trivial Dynkin diagram automorphism $\vee$ on
$A_{2n-1}$ as an involution on $\{ 1,\ldots,2n\}$, $\vee$ normalizes
$\sfW_{A_{2n-1}}$ via $\vee s_i \vee  = s_{\vee(i)}$.

\begin{lemma} \label{lem: vee comuute A}
For $\vee$-fixed Dynkin quiver $Q=(\Dynkin,\vxi)$ of type $A_{2n-1}$ and $i \in \Dynkin_0$, we have
$$   \ga_{\vee(i)}^Q    = \vee(\ga^Q_i).$$
Here we understand $\vee$ as an involution of $\Phi^+_{A_{2n-1}}$.
\end{lemma}

\begin{proof}
Recall $B^Q(i)$ in~\eqref{eq: gaQ}.
By the facts that $\vxi_i = \vxi_{\vee(i)}$ and Remark~\ref{rmk: vee fixed A},
 if there exists a path $j$ to $i$ in $Q$ for $i,j\le n$, then there exists a path $2n-i$ to $2n-j$ in $Q$.
Hence $j \in B^Q(i) \iff 2n-j \in B^Q(2n-i)$ for $i,j \le n$ and our assertion follows from~\eqref{eq: gaQ}.
\end{proof}

For $\Dynkin$ of type $A_{2n-1}$ with $\sfh=2n$, the degree polynomials $\de_{i,j}^{A_{2n-1}}(t)$ for $i,j \in \Dynkin_0$ are given in the following explicit form  \cite{Oh16,Oh18,Fuj19}:
\begin{align*}
\de_{i,j}^{A_{2n-1}}(t) + \delta_{i,j^*}t^{\sfh-1} = \sum_{s=1}^{\min(i,j,2n-i,2n-j)}t^{|i-j|+2s-1}.
\end{align*}
Note that every non-zero coefficient of $\de_{i,j}^{A_{2n-1}}(t)$ is $1$.

\subsection{Relation between $\de_{i,j}(t)$ and $\tde_{i,j}(t)$} For a Dynkin diagram $\Dynkin$ of type $A_{n}$, $D_{n}$ and $E_6$, we have the following theorem:

\begin{theorem}[\cite{Oh18}] \label{thm: ADE6 de=tde}
For $\Dynkin$ of type $A_{n}$, $D_{n}$ and $E_6$, and $i,j \in I$, we have
\begin{align} \label{eq: de=tde}
\tde_{i,j}(t) = \de_{i,j}(t) +\delta_{i,j^*} t^{\sfh-1}.
\end{align}
\end{theorem}

\begin{remark} \label{rmk: remained type}
In Section~\ref{sec: Degree poly}, we will see that ~\eqref{eq: de=tde} also holds for $B_n$ and $C_n$ (see Theorem~\ref{thm:Main}). However, in types $E_7,E_8$, $F_4$ and $G_2$,
~\eqref{eq: de=tde} do not hold (see Remark~\ref{rmk: F4 do not hold} and Remark~\ref{rmk: G2 do not hold} below for $F_4$ and $G_2$.)
\end{remark}

\section{Labeling of AR-quivers in type BCFG} \label{sec: Labeling BCFG}

For a Lie algebra $\sfg$ of simply-laced type associated to
$\Dynkin$ and a Dynkin diagram automorphism $\upsigma\; (\ne {\rm id})$
on $\Dynkin_\sfg$, we denote by $\sg$ the Lie subalgebra of $\sfg$
such that it is of non simply-laced type \cite[Proposition 7.9]{Kac}:
\begin{table}[h]
\centering
{ \arraycolsep=1.6pt\def\arraystretch{1.5}
\begin{tabular}{|c||c|c|c|c|}
\hline
 $(\Dynkin,  \upsigma)$ &  $(\Dynkin_{A_{2n-1}}, \vee)$& $(\Dynkin_{D_{n+1}}, \vee)$ & $(\Dynkin_{E_{6}},  \vee)$ & $(\Dynkin_{D_{4}},  \widetilde{\vee})$  \\ \hline
Type of $\sg$ & $C_n $ & $B_n$  & $F_4$  & $G_2$ \\ \hline
\end{tabular}
}\\[1.5ex]

\caption{$\sg$ for each $(\Dynkin,\sigma)$ with $\sigma \ne {\rm id}$}
\protect\label{g sigma}
\end{table}

Note that
the index $\sI$ for $\sg$ can be identified the image of
$I_\sfg$ under the following surjection $\overline{\upsigma}$: For
$\sg$ of type $B_n$ or $C_n$,
$$
\overline{\upsigma}(i) =
\bc
n  & \text{ if $\sg$  is of type $B_n$ and $i =n+1$}, \\
2n-i  & \text{ if $\sg$  is of type $C_n$ and $i \ge n+1$},  \\
i  & \text{ otherwise},
\ec
$$
and
\begin{align*}
&\overline{\upsigma}(1) = \overline{\upsigma}(6) =2, \quad \overline{\upsigma}(3) =\overline{\upsigma}(5) = 2, \quad  \overline{\upsigma}(4) =  2, \quad  \overline{\upsigma}(2) =1, &&\text{ if  $\sg$ is of  type $F_4$,} \\
&\overline{\upsigma}(1) = \overline{\upsigma}(3) =\overline{\upsigma}(4) =2, \qquad \overline{\upsigma}(2) =1, &&\text{ if  $\sg$ is of  type $G_2$.}
\end{align*}
Then, for a $\upsigma$-fixed Dynkin quiver $Q=(\Dynkin_{\sfg}, \sxi)$,
we can obtain a Dynkin quiver $\overline{Q}=(\Dynkin_{\sg}, \overline{\xi})$ such that
\begin{align}\label{eq: oQ}
\overline{\xi}_{\overline{\upsigma}(i)} =   \sxi_{i}.
\end{align}
 for any $i \in I$. We write $\sI = \{ \ov{\im} \seteq \overline{\upsigma}(\im) \ | \ \im \in I_\sfg\}$.

\begin{example}
Here are $\overline{Q}$'s corresponding to $Q$'s in Example~\ref{ex: sigma-fixed}.
\ben
\item \label{it: C3 induced}
$\overline{Q} = \xymatrix@R=0.5ex@C=6ex{ *{\circ}<3pt> \ar@{->}[r]^<{ _{\underline{3}} \ \  }_<{1 \ \ } & *{\circ}<3pt>  \ar@{->}[r]^<{ _{\underline{2}} \ \  }_<{2 \ \ }
&*{\circled{$\circ$}}<3pt>    \ar@{}[l]_<{ \ \   _{\underline{1}}  }^<{ \ \ 3 }
}$ of type $C_n$ for $Q$ of type $A_5$ in Example~\ref{ex: sigma-fixed}~\eqref{it: A5 fixed}.
\item \label{it: B3 induced}
$\overline{Q} = \xymatrix@R=0.5ex@C=6ex{ *{\circled{$\circ$}}<3pt> \ar@{->}[r]^<{ _{\underline{3}} \ \  }_<{1 \ \ } & *{\circled{$\circ$}}<3pt>  \ar@{->}[r]^<{ _{\underline{2}} \ \  }_<{2 \ \ }
&*{\circ}<3pt>    \ar@{}[l]_<{ \ \   _{\underline{1}}  }^<{ \ \ 3 }
}$ of type $B_3$ for $Q$ of type $D_4$ in Example~\ref{ex: sigma-fixed}~\eqref{it: D4 fixed vee}.
\item \label{it: F4 induced}
$\overline{Q} = \xymatrix@R=0.5ex@C=6ex{ *{\circled{$\circ$}}<3pt> \ar@{<-}[r]^<{ _{\underline{0}} \ \  }_<{1 \ \ } & *{\circled{$\circ$}}<3pt>  \ar@{<-}[r]^<{ _{\underline{1}} \ \  }_<{2 \ \ }
&*{\circ}<3pt>  \ar@{<-}[r]^<{ _{\underline{2}} \ \  }_<{3 \ \ }  &*{\circ}<3pt>     \ar@{}[l]_<{ \ \ _{\underline{3}}  }^<{ \ \ 4 }
}$ of type $F_4$
 for $Q$ of type $E_6$ in Example~\ref{ex: sigma-fixed}~\eqref{it: E6 fixed}.
\item \label{it: G2 induced}
$\overline{Q} = \xymatrix@R=0.5ex@C=6ex{   *{\circ}<3pt>  \ar@{->}[r]^<{ _{\underline{2}} \ \  }_<{1 \ \ }
&*{\circled{$\odot$}}<3pt>    \ar@{}[l]_<{ \ \   _{\underline{1}}  }^<{ \ \ 2 }
}$ of type $G_2$ for $Q$  of type $D_4$ in Example~\ref{ex: sigma-fixed}~\eqref{it: D4 fixed wvee}.
\ee
\end{example}

\medskip

In
this section, we shall show that we can obtain $\Gamma_{\ov{Q}}$
from $\Gamma_Q$ and hence the labeling  of  $\Gamma_{\ov{Q}}$ from
the ones of $\Gamma_Q$.
The following lemma holds for Dynkin quivers $Q$ of any finite type, which generalizes~\eqref{eq: gaQ}:
\begin{lemma} \label{lem: general gaQ}
For Dynkin quiver $Q=(\Dynkin, \xi)$ of any finite type and $i \in I$, we have
$$ \ga_i^Q  = \sum_{j \in B^Q(i)}  \left(  \prod_{k=1}^{l^j-1}  - \left\lan  h_{p^j_k},\al_{p^j_{k+1}} \right\ran  \right)\al_j $$
where $B^Q(i)$ denotes the set of indices $j$ such that there exists a
path $\ttp^j = (j=p^j_1 \to p^j_2 \to \cdots \to p^j_{l^j-1} \to
p^j_{l^j}=i)$ in $Q$. Here we understand $\left(  \prod_{k=1}^{l^j-1}  - \left\lan  h_{p^j_k},\al_{p^j_{k+1}} \right\ran  \right)=1$ when $i=j$ and hence $l^j=1$.
\end{lemma}

\begin{proof} Set $i_m=i$. Note that we have
\begin{align}\label{eq: gaQ 2}
 \ga_{i_m}^Q = (1-\tau_Q)\varpi_{i_m} = (s_{i_1} \cdots s_{i_{m-1}}- s_{i_1} \cdots s_{i_{m-1}}s_{i_m} )\varpi_{i_m}  =  s_{i_1} \cdots s_{i_{m-1}}(\al_{i_m}),
\end{align}
where $\tau_Q = s_{i_1} \cdots  s_{i_{|\Dynkin_0|}}$ and $1 \le m \le |\Dynkin_0|$. 
Let $j \not\in B^Q(i)$. Then there is no  \ex $1 \le t_1< \ldots < t_l = m$ such that $d(i_{t_x},i_{t_{x+1}})=1$. Hence $j$ can not be contained in $\supp(\ga_{i}^Q)$ by~\eqref{eq: gaQ 2}.

Let $j \in B^Q(i)$. Since $\Dynkin$ is of finite type, there exists a
unique path  $\ttp^j = (j=p^j_1 \to p^j_2 \to \cdots \to p^j_{l^j-1}
\to p^j_{l^j}=i)$ in $Q$. If the length of $\ttp^j$ is $2$, we have
$\xi_j=\xi_i+1$ and the assertion follows from~\eqref{eq: gaQ 2}. By
an induction on $l$, $p^j_2 \in \supp(\ga_{i}^Q)$ and
$\supp_{p^j_2}(\ga_{i}^Q) = \left(  \displaystyle\prod_{k=2}^{l^j-1}
- \left\lan  h_{p^j_k},\al_{p^j_{k+1}} \right\ran  \right)$. Since
$\xi_{j} = \xi_{p^j_2}+1$ and $\tau_Q$ is adapted to $Q$, we can
guarantee that $\supp_{j}(\ga_{i}^Q) \ge \left(
\displaystyle\prod_{k=1}^{l^j-1}  - \left\lan
h_{p^j_k},\al_{p^j_{k+1}} \right\ran  \right)$.

\noindent
(i) Assume that there exists $j' \in \Dynkin_0$ such that $d(j,j')=1$ and $\xi_{j'}=\xi_j+1$. Then $s_j'$ appears earlier than $s_j$ in every $Q$-adapted reduced expression of $\tau_Q$.

\noindent
(ii) Assume that there exists $j' \in \Dynkin_0$ such that $d(j,j')=1$ and $\xi_{j'}=\xi_j-1$. Then there is no path form $j'$ to $i$. Thus, even though $s_j'$ appears later than $s_j$ in every $Q$-adapted reduced expression of $\tau_Q$,  we have $j' \not \in \supp(\ga_i^Q)$.

By (i) and (ii), we have  $\supp_{j}(\ga_{i}^Q) = \left(  \displaystyle\prod_{k=1}^{l^j-1}  - \left\lan  h_{p^j_k},\al_{p^j_{k+1}} \right\ran  \right)$, as we desired.
\end{proof}

\begin{corollary}
For a Dynkin quiver $Q=(\Dynkin,\xi)$ of an arbitrary type and $i \in \Dynkin_0$, we have
$$\al_i - \sum_{\substack{d(i,j)=1, \\ \xi_j-\xi_i=1}}\sfc_{j,i} \ga_j^Q = \ga_i^Q.$$
\end{corollary}

\begin{convention}
In the rest of this section, we use
\bna
\item $\im,\jm$ for elements in the index set $I$ of type $ADE$ and $i,j$ for elements in  the index set $I$ of type $BCFG$,
\item $\al,\be,\ga$ for positive roots of type $ADE$ and $\upal,\upbe,\upga$ for positive roots of type $BCFG$,
\item $\ep,\ve$ for elements in $\R \otimes_\Z \rl^+$ of type $ADE$ and $\upep,\upve$ for elements in $\R \otimes_\Z \rl^+$ of type $BCFG$,
\ee
to avoid confusion.
\end{convention}

For subsections~\ref{subsec: Bn combinatorics} and~\ref{subsec: Cn combinatorics} below, we shall investigate the combinatorial properties and the labeling algorithm for AR-quivers of type $B_n$ and $C_n$, and observe the relationship with the ones for $D_n$ and $A_{2n-1}$, respectively. The results in those subsections will be used for obtaining the degree polynomials in the next section crucially.

\subsection{$B_{n}$-case} \label{subsec: Bn combinatorics}
Note that for $\sfg$ of type $D_{n+1}$, its corresponding $\vg$ is of type $B_n$. For $Q=(\bDynkin_{B_n}, \xi)$ of type $B_n$,
we denote by $\uQ = (\Dynkin_{D_{n+1}}, \vuxi)$ the $\vee$-fixed Dynkin quiver of type $D_{n+1}$ such that
\begin{align} \label{eq: uB}
 \vuxi_{\im} = \xi_{\ov{\im}}   \qquad \text{ for $\im \in \Dynkin_0$}.
\end{align}
Recall the notation $\oQ=(\bDynkin_{B_n}, \oxi)$ of type $B_n$ corresponding to $\vee$-fixed Dynkin quiver $Q=(\Dynkin_{D_{n+1}}, \vxi)$ of type $D_{n+1}$ in~\eqref{eq: oQ}.

The simple roots $\{ \upal_i \}$ and $\Phi^+$ of type $B_n$ can be identified in $\rl^+ \subset \soplus_{i=1}^n \R \upep_i$ as follows:
\begin{equation}\label{eq: PR B}
\begin{aligned}
\upal_i &=\upve_i- \upve_{i+1} \quad \text{ for $i<n$,} \quad \upal_n =\upve_n \quad\text{ and }  \quad \upve_i =\sqrt{2}\upep_i,\\
\Phi^+_{B_n} & =  \left\{  \upve_i-\upve_{j} = \sum_{k=i}^{j-1} \upal_k \ | \ 1 \le i <j \le n \right\} \ssqcup  \left\{  \upve_i = \sum_{k=i}^n \upal_k \ | \ 1 \le i \le n \right\}\\
& \quad \ssqcup \left\{  \upve_i+\upve_{j} = \sum_{k=i}^{j-1} \upal_k + 2\sum_{s=j}^n \upal_s \ | \ 1 \le i <j \le n \right\}.
\end{aligned}
\end{equation}

For $\upbe \in \Phi^{+}_{B_n}$,  we write
\begin{align}\label{eq: conv B root}  \upbe =
\bc
\lan i,\pm j \ran  & \text{ if $\be=\upve_{i} \pm \upve_j$ for some $1 \le i<j\le n$}, \\
\lan i \ran   & \text{ if $\be=\upve_{i}$ for some $1 \le i \le n$}
\ec
\end{align}

Understanding $\weyl_{B_n}$ as bijections on $\{ \pm1,\ldots,\pm n\}$, we have the following descriptions for simple reflections:
$$  s_i (\pm i)=\pm(i+1)  \quad \text{ for }i <n  \quad \text{ and }  \quad s_n(n)=-n.$$

Note that there is a surjection  $\psi\cl \Phi^+_{D_{n+1}} \to \Phi^+_{B_{n}}$ whose description can be described in two ways using $\ep_\im$'s, $\upve_i$'s, $\al_\im$'s and $\upal_i$'s:
\begin{align} \label{eq: D to B}
 \psi(\ep_\im) = \bc
\ \upve_{\oim} & \text{ if } \im \le n, \\
\ 0   & \text{ if } \im = n+1,
\ec
\qquad \left(\text{equivalently, } \
\psi(\al_\im) =  \upal_{\oim}
   \right)
\end{align}
and extends it linearly.

For each $i \in I_{B_n}$, let $\im$ be the index of $I_{D_n}$ such that $i = \im$ as an integer.
By~\eqref{eq: PR D} and~\eqref{eq: PR B}, we have
$$
\psi^{-1}(\upve_i)=\{ \ep_\im + \ep_{n+1}, \ep_\im - \ep_{n+1} \}  \quad \text{ and } \quad \psi^{-1}(\upve_i \pm \upve_j)= \ep_\im \pm \ep_\jm.
$$

\begin{proposition} \label{prop: surgery D to B}
For a $\vee$-fixed Dynkin quiver $Q$ of type $D_{n+1}$, we can obtain $\Gamma_\oQ$ from $\Gamma_Q$ by the following simple surgery:
\ben
\item Remove all vertices in $\Gamma_Q$ located at residue $(n+1)$.
\item Replace all labels at residue $n$ from $ \lan \im, \pm  (n+1) \ran$ to $\lan i \ran$.
\ee
\end{proposition}

To prove Proposition~\ref{prop: surgery D to B}, we need preparations:

\begin{lemma}  \label{lem: Weyl D and B}
For any $\vee$-fixed Dynkin quiver $Q$ of type $D_{n+1}$, we have
$$\tau_{Q}(\im)=\tau_{\oQ}(i)  \qquad \text{ for $1 \le \im \le n$.}$$
\end{lemma}

\begin{proof}
By Lemma~\ref{lem: tauQ fixed D}~\eqref{it: fixed D},  $\tau_Q(n+1) = -(n+1)$.
Since $Q$ is $\vee$-fixed, there exists a $1 \le k < n+1$ reduced expression $s_{\im_1} \ldots s_{\im_{n+1}}$ of $\tau_Q$ such that $\{ \im_k,\im_{k+1} \}=\{n,n+1\}$.
Since $\oxi_{i} =\vxi_\im$, $1 \le \im \le n$,  $\tau_{\oQ}$ has a reduced expression $s_{\ov{\im_1}} \ldots s_{\ov{\im_{k-1}}} s_{i^\dagger_k} s_{\ov{\im_{k+2}}} \ldots s_{\ov{\im_{n+1}}}$ such that $i^\dagger_k =n$.
Thus our assertion follows.
\end{proof}

The following lemma is obvious:

\begin{lemma} \label{eq: reflection commuting}
For $Q=(\bDynkin_{B_n}, \xi)$, we have
$$   \underline{s_iQ}  = \iota(s_i) \uQ   \quad \text{ for $1 \le i \le n$}, $$
where
$$  \iota(s_i)  \seteq \bc
s _\im  & \text{ if } i < n, \\
s _n s_{n+1} & \text{ if } i= n.
\ec $$
\end{lemma}

\begin{lemma} \label{lem: ga D and B}
For any $\vee$-fixed Dynkin quiver $Q$ of type $D_{n+1}$, we have
$$  \psi( \ga_\im^Q  )  = \upga_i^{\oQ}  \quad \text{ for $1\le \im \le n$} \quad \text{ and } \quad   \psi( \ga_{n+1}^Q  )  = \upga_n^{\oQ}.$$
\end{lemma}

\begin{proof}
By Lemma~\ref{lem: general gaQ} and~\eqref{eq: D to B},
\begin{equation}
\begin{aligned}
\psi( \ga_\im^Q  ) & =  \sum_{\jm \in B^Q(\im) \setminus \{ n+1 \} }    \left(   \prod_{k=1}^{l^\jm-1}   -   \left\lan  h_{p^\jm_k},\al_{p^\jm_{k+1}} \right\ran   \right)   \upal_j  \\
& \hspace{25ex} +  \delta(n + 1 \in B^Q(\im))  \left(  \prod_{k=1}^{l^{n+1}-1} - \left\lan  h_{p^{n+1}_k},\al_{p^{n+1}_{k+1}} \right\ran  \right) \upal_n .
\end{aligned}
\end{equation}

For any $j \in B^\oQ(i) \setminus \{ n \}$,  we have $|\sfc_{ji}^{B_n}| \le 1$.
Then we have
$\supp_j( \upga_i^\oQ ) = \supp_\jm( \upga_\im^Q )$ for any $j \in B^Q(i) \setminus \{ n \}$  by Lemma~\ref{lem: general gaQ}.

Assume that $n \in B^\oQ(i)$. Then we have $n,n+1 \in B^Q(\im)$ if $i<n$ and $n+1 \not \in B^Q(n)$ if $i=n$. Since $-\sfc_{n,n-1}=2$, we have
\begin{align*}
& \left(   \prod_{k=1}^{l^n-1}   -   \left\lan  h_{p^n_k},\al_{p^n_{k+1}} \right\ran   \right)\upal_n +\delta(n + 1 \in B^Q(\im))  \left(  \prod_{k=1}^{l^{n+1}-1} - \left\lan  h_{p^{n+1}_k},\al_{p^{n+1}_{k+1}} \right\ran  \right) \upal_n  \\
& \hspace{55ex}=  \left(   \prod_{k=1}^{l^n-1}   -   \left\lan  h_{p^n_k},\upal_{p^n_{k+1}} \right\ran   \right)\upal_n,
\end{align*}
which completes our assertion.
\end{proof}

\begin{proof}[Proof of Proposition~\ref{prop: surgery D to B}] By Lemma~\ref{lem: ga D and B},
we can obtain  $\{\upga^\oQ_i  \}$ by using $\{\ga^Q_\im  \}$ and the surjection $\psi$ in~\eqref{eq: D to B}.
Then our assertion follows from Lemma~\ref{lem: Weyl D and B} and~\eqref{eq: bijection}.
\end{proof}

\begin{definition}
A connected full subquiver $\ov{\uprho}$ in $\Gamma_\oQ$ is said to be
\emph{swing} if it is a concatenation of one $N$-path and one
$S$-path whose intersection is located at $(n,p) \in (\Gamma_Q)_0$
for some $p \in \Z$:  There exist a positive root
$\be \in \Phi^+$ and $r,s \le n$ such that
$$\xymatrix@C=2.5ex@R=0.5ex{  S_r \ar@{->}[r] & S_{r+1} \ar@{->}[r] & \cdots \ar@{->}[r]& S_{n-1}
\ar@{->}[r] & \be \ar@{=>}[r]  & N_{n-1} \ar@{->}[r]& N_{n-2} \ar@{->}[r] & \cdots \ar@{->}[r] & N_s} \text{ where }$$
\begin{itemize}
\item $\xymatrix@C=2.5ex@R=0.5ex{ S_{r} \ar@{->}[r] & S_{r+1} \ar@{->}[r] & \cdots \ar@{->}[r]& S_{n-1} \ar@{->}[r] &\be}$ is an $S$-path ($\phi^{-1}_{Q}(S_l,0)=(l,k)$ for   $k \in \Z$),
\item $\xymatrix@C=2.5ex@R=0.5ex{ \be \ar@{=>}[r] & N_{n-1} \ar@{->}[r]& N_{n-2} \ar@{->}[r] & \cdots \ar@{->}[r] & N_{s} }$ is an $N$-path ($\phi^{-1}_{Q}(N_l,0)=(l,k')$ for  $k' \in \Z$),
\item $\be$ is located at $(n,p)$ for some $p \in \Z$.
\end{itemize}

\end{definition}

By Proposition~\ref{prop: surgery D to B}, we have the $B_n$-analogue of Proposition~\ref{prop: label1}  and Proposition~\ref{prop: label2} as follows:

\begin{proposition} \label{prop: label12 B} \hfill
\ben
\item \label{it: swing B} For each swing in $\Gamma_\oQ$, there exists a unique $1 \le a \le n$ satisfying the following properties:
\bnum
\item Every vertex shares a component $\upve_a$.
\item If every vertex shares a component $\upve_a$, it consists of exactly $(2n-a)$-many vertices.
\item Every positive roots having $\upve_a$ as its component appears in the swing.
\item It starts at $1$ or ends at $1$.
\item There are only $n$-swings.
\ee
Thus we say the swing as $a$-swing and denote it by $\ov{\varrho}_a$.
\item \label{it: shallow B} Let $\ov{\rho}$ be a shallow $S$ $($resp.\ $N)$-path in $\Gamma_\oQ$. Then there exists a unique $2 \le a \le n$ such that
\bnum
\item $\ov{\rho}$ consists of $(a-1)$-many vertices,
\item all vertices in $\rho$ share $\upve_{-a}$ as their components.
\ee \ee Furthermore, there are exactly $(n-1)$-many shallow
$X$-paths $(X=S$ or $N)$ in $\Gamma_\oQ$. Thus we say the shallow
$X$-path as $(-a)$-path and denote it by $\ov{\rho}_{-a}$.
\end{proposition}

From~Proposition~\ref{prop: surgery D to B} and  Proposition~\ref{prop: label12 B}, we have the following main result of this subsection:

\begin{theorem} \label{thm: D to B}
For each $\sigma$-fixed $Q$ of type $D_{n+1}$ and $k \in \Z$, we have
$$   \psi\left( \tau_Q^k  (\gamma^Q_\im)  \right)  = \tau_\oQ^k  (\gamma^\oQ_i) \quad \text{ for $\im<n$} $$
and
$$   \psi\left( \tau_Q^k  (\gamma^Q_\im)  \right)  = \tau_\oQ^k  (\gamma^\oQ_n) \quad \text{ for $\im=n$ or $n+1$.} $$
\end{theorem}

\begin{example} \label{ex: Label B4}
Let us consider the following $\vee$-fixed Dynkin quiver $Q$ of type $B_4$:
$$Q \ =
 \xymatrix@R=0.5ex@C=4ex{ *{\circled{$\circ$}}<3pt> \ar@{<-}[r]_<{ 1 \ \  }^{_{\underline{2}} \qquad } & *{\circled{$\circ$}}<3pt> \ar@{->}[r]_<{ 2 \ \  }^{_{\underline{3}} \qquad }
&*{\circled{$\circ$}}<3pt> \ar@{<-}[r]_>{ \ \ 4}^{_{\underline{2}} \qquad \ }   &*{\circ}<3pt> \ar@{}[l]^>{  3 \ \ }_{  \ \qquad _{\underline{3}}   }   }
$$
By ~\eqref{eq: range}, the $\Gamma_Q$ without labeling are given as follows:
\begin{align*}
 \raisebox{3.2em}{ \scalebox{0.7}{\xymatrix@!C=2ex@R=2ex{
(i\setminus p)   &-4 &-3 &  -2 & -1 & 0 & 1 & 2  & 3  \\
1& \circ \ar[dr] && \circ \ar[dr]  && \circ \ar[dr]&& \circ \ar[dr]    \\
2&& \circ \ar[dr]\ar[ur]  && \circ \ar[dr]\ar[ur] && \circ\ar[dr]\ar[ur] && \circ     \\
3& \circ  \ar[dr]\ar[ur]  && \circ  \ar[dr]\ar[ur] && \circ  \ar[dr]\ar[ur]&& \circ \ar[ur] \ar[dr]  \\
4&& \circ \ar@{=>}[ur]&& \circ \ar@{=>}[ur]&& \circ \ar@{=>}[ur]  && \circ  \\
}}}
\end{align*}

By applying Proposition~\ref{prop: label12 B}~\eqref{it: swing B} and~\eqref{it: shallow B}  in order, we can complete the label of $\Gamma_Q$
as follows:
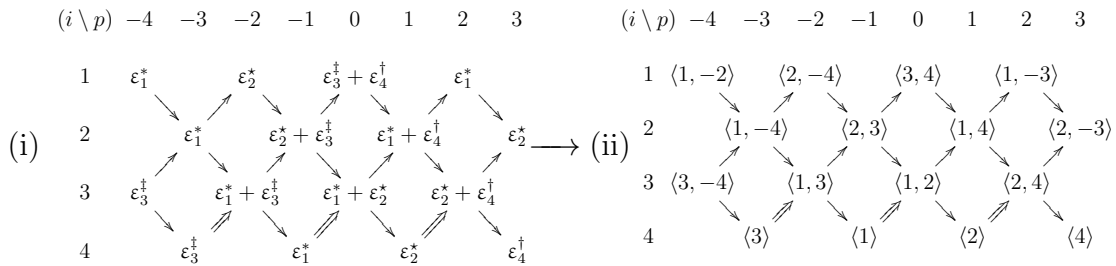
\begin{align*}
{\rm (i)}  \raisebox{4.3em}{ \scalebox{0.7}{\xymatrix@!C=1ex@R=2ex{
(i\setminus p)   &-4 &-3 &  -2 & -1 & 0 & 1 & 2  & 3  \\
1& \upve_1^* \ar[dr] && \upve_2^\star\ar[dr]  &&  \upve_3^\ddagger + \upve_4^\dagger  \ar[dr]&& \upve_1^* \ar[dr]    \\
2&& \upve_1^* \ar[dr]\ar[ur]  && \upve_2^\star+  \upve_3^\ddagger \ar[dr]\ar[ur] && \upve_1^* + \upve_4^\dagger\ar[dr]\ar[ur] && \upve_2^\star    \\
3& \upve_3^\ddagger\ar[dr]\ar[ur]  && \upve_1^*+ \upve_3^\ddagger \ar[dr]\ar[ur] && \upve_1^* + \upve_2^\star \ar[dr] \ar[ur]&& \upve_2^\star+\upve_4^\dagger\ar[ur] \ar[dr]  \\
4&&  \upve_3^\ddagger\ \ar@{=>}[ur]&& \upve_1^* \ar@{=>}[ur]&& \upve_2^\star\ar@{=>}[ur]  &&  \upve_4^\dagger  \\
}}}
 \hspace{-1ex}
\To
{\rm (ii)}   \hspace{-2ex}  \raisebox{4.3em}{ \scalebox{0.7}{\xymatrix@!C=1ex@R=2ex{
(i\setminus p)   &-4 &-3 &  -2 & -1 & 0 & 1 & 2  & 3  \\
1& \srt{1,-2}\ar[dr] && \srt{2,-4} \ar[dr]  && \srt{3,4} \ar[dr]&& \srt{1,-3} \ar[dr]    \\
2&& \srt{1,-4} \ar[dr]\ar[ur]  && \srt{2,3} \ar[dr]\ar[ur] && \srt{1,4}\ar[dr]\ar[ur] && \srt{2,-3}     \\
3& \srt{3,-4}  \ar[dr]\ar[ur]  && \srt{1,3}  \ar[dr]\ar[ur] && \srt{1,2}  \ar[dr]\ar[ur]&& \srt{2,4} \ar[ur] \ar[dr]  \\
4&& \srt{3} \ar@{=>}[ur]&& \srt{1} \ar@{=>}[ur]&& \srt{2} \ar@{=>}[ur]  && \srt{4}  \\
}}}
\end{align*}
\end{example}

\subsection{$C_{n}$-case}  \label{subsec: Cn combinatorics}
Note that for $\sfg$ of type $A_{2n-1}$, its corresponding $\vg$ is of type $C_n$.
The simple roots  $\{ \upal_i \}$  and $\Phi^+_{C_{n}}$ can be identified in $\rl^+ \subset \soplus_{i=1}^n \Z\upep_i$ as follows:
\begin{equation} \label{eq: PR C}
\begin{aligned}
\upal_i &=\upep_i- \upep_{i+1} \quad \text{ for $i<n$} \quad\text{ and } \quad \al_n =2\upep_n \allowdisplaybreaks \\
\Phi^+_{C_n} & =  \left\{  \upep_i-\upep_{j} = \sum_{k=i}^{j-1} \upal_k \ | \ 1 \le i <j \le n \right\} \ssqcup
\left\{ 2\upep_i = 2\sum_{k=i}^{n-1} \upal_k +\upal_n \ | \ 1 \le i \le n \right\} \allowdisplaybreaks\\
& \quad \ssqcup \left\{  \upep_i+\upep_{j} = \sum_{k=i}^{j-1} \upal_k + 2\sum_{s=j}^{n-1} \upal_s + \upal_n \ | \ 1 \le i <j \le n \right\}.
\end{aligned}
\end{equation}

For $\upbe \in \Phi^{+}_{C_n}$,  we write (cf. ~\eqref{eq: conv B root})
\begin{align}\label{eq: conv C root}  \upbe =
\bc
\lan i,\pm j \ran  & \text{ if $\upbe=\upep_{i} \pm \upep_j$ for some $1 \le i<j\le n$}, \\
\lan i, i \ran   & \text{ if $\upbe=2\upep_{i}$ for some $1 \le i \le n$}.
\ec
\end{align}

Note that $\sfW_{B_n} \simeq \sfW_{C_n}$. For a Dynkin quiver $Q=(\bDynkin_{B_n}, \xi)$ of type $B_n$, we denote by $Q^\tr=(\bDynkin_{C_n} ,\xi^{\tr})$
the Dynkin quiver of type $C_n$ such that
$$ \xi^{\tr}_i =\xi_i \qquad \text{ for } 1 \le i \le n.$$
Then we have a natural bijection $\psi^\tr\cl \Phi^+_{B_n} \to \Phi^+_{C_n}$ given by
$$ \psi^{\tr}(\upve_i\pm\upve_j) =  \upep_i\pm\upep_j  \ \  (1 \le i < j \le n) \ \quad \text{ and } \ \quad \phi^{\tr}(\upve_i)=2\upep_i.$$

\begin{proposition}
For all $1\le i\le n$, we have
$$   \psi^{\tr}(\upga_i^{Q} ) =\upga_i^{Q^{\tr}}.$$
\end{proposition}

\begin{proof}
Note that
$B^Q(i) = B^{Q^\tr}(i)$ for all $1 \le i \le n$.
Set
\bna
\item $a = \min( s \ |   \ s \in B^Q(i))$,
\item $b = \max( s \ |   \ s \in B^Q(i))$ if  $i \ne n$ and $n \not\in B^Q(i)$.
\ee

By Lemma~\ref{lem: general gaQ}, we have
\begin{align}
\upga_i^Q  
& = \bc
\displaystyle \sum_{j \in B^Q(n)}    \upal_j & \text{ if } i=n, \\
\displaystyle \sum_{j \in B^Q(i) \setminus\{n\}}    \upal_j +\delta(n \in B^Q(i)) 2\upal_n & \text{ if } i \ne n,
\ec \label{eq: ga B}\\
& =
 \bc
\upve_a & \text{ if } i=n, \\
\upve_a -\upve_{b+1}  & \text{ if } i \ne n  \text{ and } n \not\in B^Q(i), \\
 \upve_a+\upve_n,& \text{ if } i \ne n  \text{ and } n  \in B^Q(i),
\ec \nonumber
\end{align}
On the other hand, we have
\begin{align}
\upga_i^{Q^\tr}
& = \bc
\displaystyle \left( \sum_{j \in B^{Q}(n) \setminus\{ n \}}    2\upal_j \right) +  \upal_n & \text{ if } i=n, \qquad\qquad \qquad \\[4ex]
\displaystyle \sum_{j \in B^{Q}(i)}    \upal_j & \text{ if } i \ne n,
\ec \label{eq: ga C} \\
& =
 \bc
2\upep_a & \text{ if } i=n, \\
\upep_a -\upep_{b+1}  & \text{ if } i \ne n  \text{ and } n \not\in B^{Q}(i), \\
 \upep_a+\upep_n,& \text{ if } i \ne n  \text{ and } n  \in B^{Q}(i),
\ec \nonumber
\end{align}
which implies our assertion.
\end{proof}

 Note that every $Q$ of type $C_n$ is $Q^\tr$ for some $Q$ of type $B_n$. Hence we have the following theorem:

\begin{proposition} \label{prop: transpose}
The labeling of $\Gamma_{Q^\tr}$ in terms of notations $\lan i,\pm j \ran$ and $\lan i,i \ran$ can be obtained from the one of $\Gamma_{Q}$
by replacing only $\lan i \ran$ with $\lan i,i \ran$ for $1 \le i \le n$.
\end{proposition}

\begin{proof}
By the previous proposition, the replacement of notations works for $\{ \upga_i^Q\}$ and $\{ \upga_i^{Q^\tr}\}$. Then our assertion follows from the fact that $\weyl_{B_n}=\weyl_{C_n}$
and the bijection $\phi$ in~\eqref{eq: bijection}.
\end{proof}

\begin{example} \label{ex: Label C4}
Consider the Dynkin quiver $Q$ of type $B_4$ in Example~\ref{ex: Label B4}. Then its corresponding $Q^\tr$ of $C_4$ can be depicted as follows:
$$Q \ =
 \xymatrix@R=0.5ex@C=4ex{ *{\circ}<3pt> \ar@{<-}[r]_<{ 1 \ \  }^{_{\underline{2}} \qquad } & *{\circ}<3pt> \ar@{->}[r]_<{ 2 \ \  }^{_{\underline{3}} \qquad }
&*{\circ}<3pt> \ar@{<-}[r]_>{ \ \ 4}^{_{\underline{2}} \qquad \ }   &*{\circled{$\circ$}}<3pt> \ar@{}[l]^>{  3 \ \ }_{  \ \qquad _{\underline{3}}   }   }
$$

Using the labeling of $\Gamma_Q$ in Example~\ref{ex: Label B4} and Proposition~\ref{prop: transpose}, we have
$$\raisebox{3.7em}{ \scalebox{0.7}
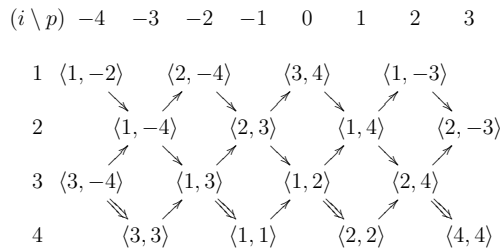
{\xymatrix@!C=1ex@R=2ex{
(i\setminus p)   &-4 &-3 &  -2 & -1 & 0 & 1 & 2  & 3  \\
1& \srt{1,-2}\ar[dr] && \srt{2,-4} \ar[dr]  && \srt{3,4} \ar[dr]&& \srt{1,-3} \ar[dr]    \\
2&& \srt{1,-4} \ar[dr]\ar[ur]  && \srt{2,3} \ar[dr]\ar[ur] && \srt{1,4}\ar[dr]\ar[ur] && \srt{2,-3}     \\
3& \srt{3,-4}  \ar@{=>}[dr]\ar[ur]  && \srt{1,3}  \ar@{=>}[dr]\ar[ur] && \srt{1,2}  \ar@{=>}[dr]\ar[ur]&& \srt{2,4} \ar[ur] \ar@{=>}[dr]  \\
4&& \srt{3,3} \ar[ur]&& \srt{1,1} \ar[ur]&& \srt{2,2} \ar[ur]  && \srt{4,4}  \\
}}}
$$
\end{example}

For each $i \in I_{C_n}$, let $\im$ be the index of $I_{A_{2n-1}}$ such that $i = \im$ as an integer.
Note that there is a surjection $\psi\cl \Phi^+_{A_{2n-1}} \to  \Phi^+_{C_{n}}$ which is described as follows:
\begin{align}\label{eq: psi A to C}
\psi(\al_\im) =
\upal_{\oim}   \qquad \left(\text{ equivalently, } \
\psi(\ep_\im) = \bc
\ \upep_{\oim} & \text{ if } \im \le n, \\
\ -\upep_{2n+1 -\oim}  & \text{ if } \im > n,
\ec  \right)
\end{align}
 and extends it linearly.

Then the inverse image of $\upbe \in \Phi^+_{C_{n}}$ can be described as follows:
\begin{equation} \label{eq: inverse psi C to A}
\begin{aligned}
\psi^{-1}(2\upep_i) &= \ep_\im - \ep_{2n+1-\im}  = \sum_{k=\im}^{2n-\im}  \al_k \quad \text{ for } 1 \le i \le n,\\
 \psi^{-1}(\upep_i -\upep_j) &= \left\{   \ep_\im - \ep_{\jm}=\sum_{k=\im}^{\jm-1}  \al_k,  \vee(\ep_\im - \ep_{\jm}) =  \ep_{2n+1-\jm} - \ep_{2n+1-\im}  =\sum_{k=2n+1-\jm}^{2n-\im}\al_k  \right\},\\
 \psi^{-1}(\upep_i +\upep_j)
&=\left\{   \ep_\im - \ep_{2n+1-\jm}=\sum_{k=\im}^{2n-\jm}  \al_k, \vee(\ep_\im - \ep_{2n+1-\jm}) = \ep_{\jm} - \ep_{2n+1-\im}  =\sum_{k=\jm}^{2n-\im}  \al_k  \right\}
\end{aligned}
\end{equation}
for $i<j \le n$.

Note that, for $\vee$-fixed Dynkin quiver $Q=(\Dynkin,  \vxi)$ of type $A_{2n-1}$,
since there exists a reduced expression  $s_{\im_1}\cdots s_{\im_{2n-1}}$ of $\tau_Q$ such that
\begin{align}\label{eq: red ex tauQ A}
\bc
\im_1=n, \hspace{3ex} (\im_{2s}, \im_{2s+1}) =(k_s,\vee(k_s)) &\text{ if $n$ is a source of $Q$}, \\
\im_{2n-1}=n, (\im_{2s-1},\im_{2s}) =(k_s,\vee(k_s)) &\text{ if $n$ is a sink of $Q$},
\ec
\end{align}
for  $1 \le s \le n-1$ and unique $k_s \in \{1,\ldots,n-1\}$.
Thus, for a $\vee$-fixed Dynkin quiver $Q$ of type $A_{2n-1}$, the Coxeter element $\tau_Q$ is also $\vee$-fixed in the sense that
\begin{align}\label{eq: sigma fix tau A}
 \vee \tau_Q  \vee =\tau_Q.
\end{align}

\begin{lemma} \label{lem: psi C}
For each $\vee$-fixed $Q$ of type $A_{2n-1}$ and $k \in \Z$, we have
\begin{align}\label{eq: ga: A to C}
 \psi(\ga^Q_\im) = \psi(\ga^Q_{2n-\im}) =\upga_i^\oQ \quad \text{ for }\im \le n.
\end{align}
\end{lemma}

\begin{proof}
The equation in~\eqref{eq: ga: A to C} follow from Lemma~\ref{lem: vee comuute A} and the surjection $\psi$ in~\eqref{eq: psi A to C}.

Since $n$ is a sink or source of $Q$, $\supp(\ga_\im^Q) \subset \{
1,\ldots, n\}$ for $\im < n$ and $\supp(\ga_\im^Q) \subset \{
n,\ldots, 2n-1\}$ for $\im > n$. Thus $n$ is a sink or source of
$\oQ$ and $\supp(\upga_i^Q) \subset \{ 1,\ldots, n\}$ for $i < n$
and
$$  \psi(\ga_\im^Q) = \displaystyle \sum_{\jm \in B^Q(\im)}   \upal_j .$$
Since $B^Q(\im)=B^\oQ(i)$ for $i <n$ as subsets of integers, the assertion for these cases follows by~\eqref{eq: ga C}.

For $i=n$,
if $i \in B^\oQ(n)$, we have $\im,2n-\im \in B^Q(n)$. Thus the assertion for $i=n$ also follows  by~\eqref{eq: ga C}.
\end{proof}

For $\vee$-fixed Dynkin quiver $Q$ of type $A_{2n-1}$,
the set of positive roots of residue $n$ in $\Gamma_Q$ form the set $\upeta \seteq \{ [i,2n-i] \ |  1 \le i \le n \}$, since $\ga_n^Q$ and $\tau_Q$ are $\vee$-fixed and $r_n^Q =n$ (see~\eqref{eq: riq fixed}).

We label the set $\upeta$
by using $\phi_Q^{-1}$ as follows:
$$\eta_k =[a_k,2n-a_k] =\ep_{a_k} -\ep_{2n+1-a_k} = \ep_{a_k} -\ep_{b_k} \in \Phi^+_{A_{2n-1}}  \ \ (1\le k \le n   \text{ and } 1 \le a_{k} \le n)  $$
where
$$(\eta_1,0) = \phi_{Q}(n,\xi_n), \  (\eta_2,0) = \phi_{Q}(n,\xi_n-2), \ldots,  (\eta_n,0) = \phi_{Q}(n,\xi_n-2(n-1)).$$
Then we have $\{ a_1,\ldots, a_n \} = \{1,\ldots,n\}$ and $\{ b_1,\ldots, b_n \} = \{n+1,\ldots,2n\}$. It also implies
$$\tau_Q(\ep_{a_k}) = \ep_{a_{k+1}} \text{ and }
\tau_Q(\ep_{b_k}) = \tau_Q(\ep_{2n+1-a_k})  = \ep_{b_{k+1}} \text{ for $1 \le i <n$}.$$ 

Recall that $n$ is source or sink of $Q$.
\ben
\item If $n$ is source of $Q$, $\ep_{a_1}-\ep_{b_1}=[n]=\ep_n-\ep_{n+1}$ and thus $\tau_Q(\ep_{a_n})=\ep_{n+1}$.
\item If $n$ is sink of $Q$, $\ep_{a_n}-\ep_{b_n}=[n]=\ep_n-\ep_{n+1}$ and thus $\tau_Q(\ep_{n})=\ep_{2n+1-a_1}$ and $\tau_Q(\ep_{n+1})=\ep_{a_1}$.
\ee

 \begin{proposition} \label{prop: tauQ A to C}
We have
$$
\tau_\oQ(\upep_{a_k}) = \upep_{a_{k+1}} \quad \text{ for $1 \le k < n$} \quad \text{ and } \quad \tau_\oQ(\upep_{a_n}) = -\upep_{a_1}
$$
\end{proposition}
\begin{proof}
Recall the reduced expression $\tau_Q=s_{\im_1} \cdots s_{\im_{2n-1}}$ in~\eqref{eq: red ex tauQ A}.

\noindent
(1) Assume $n$ is a source of $Q$.
Then $\tau_\oQ$ has a reduced expression $s_ns_{i_2}s_{i_4}\cdots s_{i_{2n}}$.
Since $a_k<n$ for $2\le k \le n$, we have $ \tau_Q^{k-1}(\ep_n)= \tau_Q^{k-1}(\ep_{a_1}) = \ep_{a_k}=s_n \ep_{a_k}$. Hence $\tau_\oQ(\upep_{a_k}) = \upep_{a_{k+1}}$ for $1 \le k <n$.
For $k=n$, we have $\tau_Q(\ep_{a_n}) = \ep_{n+1} = s_n (\ep_n)$ which imply $s_{i_2}s_{i_4}\cdots s_{i_{2n}}(\upep_{a_n}) = \upep_n$. Then we $\tau_\oQ(\upep_{a_n}) =s_n(\upep_n) =-\upep_n = - \upep_{a_1}$.

\noindent
(2) Assume $n$ is a sink of $Q$.
Then $\tau_\oQ$ has a reduced expression $s_{i_1}s_{i_3}\cdots s_{i_{2n-1}}s_n$.
Note that $ a_i \le n$ for $1\le i <n$ and $a_n=n$. Then, for $1 \le k \le n$, we have
$ \tau_\oQ^{k-1}(\upep_{a_1}) =  \tau_\oQ^{k-1}(\upep_{a_1})= \upep_{a_k}$.
Now let us compute $\tau_\oQ(\upep_{a_n})$. 
By Lemma~\ref{lem: psi C} and , we have
$$w_0(2\upep_{a_1}) =  \tau_\oQ^n(2\upep_{a_1}) = \tau_\oQ^n(\upga_n^\oQ) = \tau_\oQ(2\ep_{a_n})= -\ga_n^\oQ = -2\upep_{a_1},$$
which implies $\tau_\oQ(\upep_{a_n}) = -\upep_{a_1}$.
\end{proof}

\begin{theorem}
For each $\vee$-fixed $Q$ of type $A_{2n-1}$ and $k \in \Z$, we have
$$  \psi\left(  \tau_Q^k(\ga_{\im}^Q)  \right) = \psi\left(  \tau_Q^k(\ga_{2n-\im}^Q)  \right) = \tau_\oQ^k( \upga_i^\oQ).$$
\end{theorem}

\begin{proof}
By Proposition~\ref{eq: psi A to C}, ~\eqref{eq: psi A to C} and~\eqref{eq: sigma fix tau A}, we have
$$   \psi(\tau_Q(\al_\im)) = \tau_\oQ(\psi(\al_\im))  \qquad \text{ for all $1 \le  \im \le 2n-1$}.$$
Since $\psi(\ga_{\im}^Q)=\psi(\ga_{2n-\im}^Q)=\ga_i^\oQ$ by Lemma~\ref{lem: psi C}, our assertion follows.
\end{proof}

\subsection{$F_4$ and $G_2$ cases}
(1) The set of positive roots $\Phi^+_{E_6}$ can be described as
\begin{align*}
&(000001), (000010), (000100),(011111),(101110),(010100), \\
&(001000),(100000), (000011),(000110),(011211),(112221),\\
&(111210),(011100),(101000),(000111), (011221),(112321),\\
&(122321),(112210),(111100),(010111),(001110),(111211),\\
&(011110),(101100),(010000),(001111),(111211),(011210),\\
 &(112211),(111110),(101111),(010110),(001100),(111111),
\end{align*}
where $(a_1a_2a_3a_4a_5a_6) \seteq \sum_{i=1}^6 a_i\al_i$.

On the other hand, the set of positive roots $\Phi^+_{F_4}$ can be described as
\begin{align*}
&(0001), (0010), (0100),(1121),(0121),(1100),\\
& (0011),(0110),(1221),(1242),(1110), (0111), \\
& (1231),(1342),(2342), (1111),(0120),(1222),\\
&(1120),(1000),(1220),(1232),(0122),(1122),
\end{align*}
where $(b_1b_2b_3b_4) \seteq \sum_{i=1}^4 b_i\upal_i$.

Then one can easily check that there exists a bijection $\psi\cl \Phi^+_{E_{6}} \to  \Phi^+_{F_{4}}$   given by
\begin{align*}
\al_\im  \mapsto \upal_{\oim} \quad \text{ and extends it linearly}.
\end{align*}

Note that the positive roots $\upbe$ in $\Phi^+_{F_4}$ such that $\psi^{-1}(\upbe) =1$ or $2$, and
\begin{align*}
& \{ \upbe \in \Phi^+_{F_4} \ | \ \phi^{-1}(\upbe) =2 \}  = \left\{  \begin{matrix}(0001),(0010),(1221),(0011),(1231),(1111)\\ (0111),(0121),(1222),(1121),(1110),(0110)\end{matrix} \right\}.
\end{align*}

\noindent
(2) The positive roots $\Phi^+_{G_2}$ are listed as below:
\begin{align*}
\Phi^+_{G_2} = \{ \upal_1,\upal_2,\upal_1+\upal_2,2\upal_1+\upal_2,3\upal_1+\upal_2,3\upal_1+2\upal_2\}
\end{align*}
Then one can easily check that there exists a bijection $\widetilde{\psi}\cl \Phi^+_{D_{4}} \to  \Phi^+_{G_{2}}$  given by
\begin{align*}
\al_\im, \mapsto \upal_{\oim} \quad \text{ and extends it linearly}.
\end{align*}

Now one can check the following theorem holds:
\begin{theorem} \label{thm: ED to FG}
For each $\vee$-fixed  $($resp.\ $\widetilde{\vee}$-fixed$)$   $Q=(\Dynkin,\vxi)$ $($resp, $Q=(\Dynkin,\tvxi))$ of type $E_{6}$ $($resp.\ $D_4)$ and $k \in \Z$, we have
$$   \psi\left( \tau_Q^k  (\gamma^Q_\im)  \right) = \tau_\oQ^k (\upga_{\ov{i}}^\oQ ) \quad \text{ $($resp.\ $ \widetilde{\psi}\left( \tau_Q^k  (\gamma^Q_i)  \right) = \tau_\oQ^k (\upga_{\ov{i}}^\oQ ) )$} \quad \text{ for } i \in \Dynkin_0.$$
\end{theorem}

\section{Degree polynomials} \label{sec: Degree poly}
In this section, we  generalize the definition of degree polynomial to include
type $BCFG$ and compute the degree polynomials in those types by
using the results in the previous sections. Recall the statistics
reviewed in subsection~\ref{subsec: stat}. In this section, we
almost skip the proofs because their arguments are almost similar to
the ones in \cite{Oh18} which use the labeling algorithms described
in Proposition~\ref{prop: label12 B} and Proposition~\ref{prop:
transpose} as crucial ingredients.

\smallskip

By ~\eqref{eq: bijection} and the reflection operation in~\eqref{eq: Qd properties}~\eqref{it: reflection}, Lemma~\ref{lem: integer o} and Proposition~\ref{prop: well-defined integer}  can be generalized to Dynkin diagram $\Dynkin$ of any finite type:
\begin{proposition}  \label{eq: d well-defined} \hfill
\bna
\item Let $Q$ be a Dynkin quiver of any finite type. For any $\pair{\al,\be}$, $\pair{\al',\be'} \in \Phi_Q(i,j)[k]$, we have $ \dg_{[Q]}(\pair{\al,\be}) =\dg_{[Q]}(\pair{\al',\be'})$. Thus the integer   $\tto_k^{Q}(i,j)$
is well-defined by \eqref{def:tto}.
We then define $\ttO_k^Q(i,j)$ by \eqref{eq: ttO}.
\item  Let $\Dynkin$ be a Dynkin diagram of any finite type.  For any $[Q],[Q'] \in \lf \Dynkin \rf$, we have $\ttO_k^{Q}(i,j) = \ttO_k^{Q'}(i,j)$. Thus $\ttO_k^{\Dynkin}(i,j)$ is well-defined.
\ee
\end{proposition}

\begin{proof}
(a) By ~\eqref{eq: bijection},~\eqref{eq: range} and the assumption, there exists $0 \le a,a',b,b' < \sfh/2$ such that $\tau_Q^a(\ga^Q_i)=\al$, $\tau_Q^b(\ga^Q_j)=\be$, $\tau_Q^{a'}(\ga^Q_i)=\al'$, $\tau_Q^{b'}(\ga^Q_j)=\be'$, $s\seteq a-a'=b-b'$.
Thus  $\um \prec_Q^\ttb \pair{\al,\be}$ if and only if $\tau_Q^{-s}(\um) \prec_Q^\ttb \pair{\al',\be'}$. Here $\tau_Q^{-s}(\um)$ denotes
 the canonical exponents of $\um$ obtained by applying $\tau_Q^{-s}$ to $\um$.

\noindent
(b) It is enough to show that
$\ttO_k^{Q}(i,j) = \ttO_k^{s_tQ}(i,j)$ for a Dynkin quiver $Q=(\Dynkin,\xi)$ with a source $t \in \Dynkin_0$. Since $\sfh/2 > 2$, the second assertion follows from (a) and the reflection operation in~\eqref{eq: Qd properties}~\eqref{it: reflection}.
\end{proof}

\begin{definition} \label{def: distane poly 2}
Let $t$ be an indeterminate.
\ben
\item
For a Dynkin diagram $\Dynkin$ of any finite type and $i,j \in \Dynkin_0$, we define a polynomial $\de^\Dynkin_{i,j}(t)$ as follows:
\begin{align}    \label{eq: d poly BCFG}
\de_{i,j}^\Dynkin(t) \seteq \sum_{k \in \Z_{\ge 0} } \max(\sfd_i,\sfd_j) {\ttO_k^{\Dynkin}(i,j)} t^{k-1}.
\end{align}
We call  $\de_{i,j}^\Dynkin(t)$ the \emph{degree polynomial of $\Dynkin$ at $(i,j)$}.
\item
For $k \in \Z$, we set
$$ \de_{i,j}[k]  \seteq \max(\sfd_i,\sfd_j) \ttO_k^\Dynkin(i,j).$$
\ee
\end{definition}

\begin{remark} Definition~\ref{def: distane poly 2} coincides with Definition~\ref{def: Distance 1} when $Q$ is a Dynkin quiver of type $ADE$, since
$\sfd_i=1$ for all $i \in \Dynkin_0$.
\end{remark}

\begin{lemma} \label{lem: nonzero to  nonzero}
Let $\up=\pair{\al,\be}$ be a \prq in $(\Phi^+_{\Dynkin})^2$. If $\dg_Q(\pair{\al,\be})  \ne 0$, we have
$$\dg_\oQ(\pair{\psi(\al),\psi(\be)}) \ne 0.$$
\end{lemma}

\begin{proof}
Since $\dg_Q(\al,\be) \ne 0$, there exists $\um$ such that $\um
\prec^\ttb_Q \up$. Let $\psi(\um)$ be the \ex in
$\Phi^+_\blacktriangle$, which can be understood as a canonical image of
$\um$ via $\psi$. Then we have
$$       \psi(\um)  \prec_{\oQ}^\ttb  \pair{\psi(\al),\psi(\be)} $$
since  $\eta_1 \prec_Q \eta_2$ implies $  \psi(\eta_1) \prec_\oQ \psi(\eta_2)$. Thus the assertion follows.
\end{proof}

\begin{remark}
Note that even though $\al,\be \in \Phi^+_\Dynkin$ is not comparable with respect to $\prec_Q$, sometimes
$\psi(\al),\psi(\be) \in \Phi^+_\blacktriangle$ is comparable  with respect to $\prec_\oQ$. Moreover, there exists a pair $\pair{\al,\be} \in (\Phi^+_{\Dynkin})^2$
such that $\dg_Q(\al,\be) =0$ but  $\dg_\oQ(\psi(\al),\psi(\be)) \ne 0$.  For instance, consider  $\pair{[2,5],[2,3] }$ of type $A_5$ in Example~\ref{ex: AR Q ADE}. Then we have
$$ \dg_Q( [2,5],[2,3])=0,   \text{ while } \dg_\oQ( \psi([2,5]),\psi([2,3]))=\dg_\oQ( \lan 2,3 \ran,\lan 1,2 \ran)=1,$$
since $\head_\oQ ( \lan 2,3 \ran,\lan 1,2 \ran) = (\lan1,3\ran, \lan 2, 2  \ran)$ (see $\Gamma_Q$ of type $A_5$ in Example~\ref{ex: AR Q ADE} and $\Gamma_\oQ$ of type $C_3$ in Example~\ref{ex: AR Q BCFG}).
\end{remark}

 In the following subsections, we will give the proof of the following theorem by the case-by-case method (see Remark~\ref{rmk: remained type}).

\begin{theorem} \label{thm:Main}
Let $\Dynkin$ be a Dynkin diagram of classical type $A_n, B_n, C_n, D_n$ and $E_6$. Then we have
\eqn
&&\tde_{i,j}^\Dynkin(t)=
\de_{i,j}^\Dynkin(t) + \delta_{i,j^*}\,  \sfd_i\, t^{\sfh-1}.
\eneqn
\end{theorem}

Note that for cases when $A_n, D_n$ and $E_6$, the above theorem was proved in \cite{Oh16,Oh17,Oh18}.

\subsection{$B_n$-case}

Throughout this subsection, $\Dynkin$ denotes the Dynkin diagram of type $D_{n+1}$ and $\bDynkin$ denotes the one of type $B_n$.
\begin{theorem} \label{thm: B_n distance polynomial}
Let $\bDynkin$ be the Dynkin diagram of type $B_n$. Then we have
\begin{align*}
&\de_{i,j}^\blacktriangle(t) + \delta_{i,j}\, \sfd_i\,   t^{\sfh-1}  = \tde_{i,j}^\blacktriangle(t).
\end{align*}
\end{theorem}

By taking $\psi$ on~\eqref{eq: dist al, be 1}, we have the following pictures:
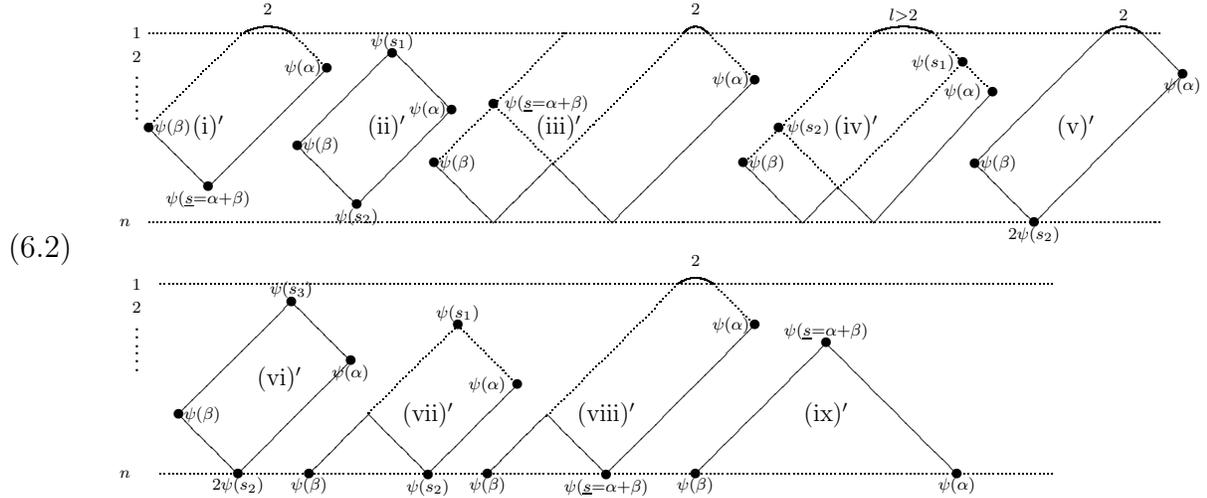
\begin{equation} \label{eq: dist al, be 1 p}
\begin{aligned}
& \scalebox{0.79}{{\xy
(-20,0)*{}="DL";(-10,-10)*{}="DD";(0,20)*{}="DT";(10,10)*{}="DR";
"DT"+(-30,-4); "DT"+(140,-4)**\dir{.};
"DD"+(-20,-6); "DD"+(150,-6) **\dir{.};
"DT"+(-32,-4)*{\scriptstyle 1};
"DT"+(-32,-8)*{\scriptstyle 2};
"DT"+(-32,-12)*{\scriptstyle \vdots};
"DT"+(-32,-16)*{\scriptstyle \vdots};
"DT"+(-34,-36)*{\scriptstyle n};
"DL"+(-10,0); "DD"+(-10,0) **\dir{-};"DL"+(-10,0); "DT"+(-14,-4) **\dir{.};
"DT"+(-6,-4); "DR"+(-10,0) **\dir{.};"DR"+(-10,0); "DD"+(-10,0) **\dir{-};
"DL"+(-6,0)*{\scriptstyle \psi(\be)};
"DL"+(0,0)*{{\rm(i)'}};
"DR"+(-14,0)*{\scriptstyle \psi(\al)};
"DL"+(-10,0)*{\bullet};
"DR"+(-10,0)*{\bullet};
"DT"+(-6,-4); "DT"+(-14,-4) **\crv{"DT"+(-10,-2)};
"DT"+(-10,0)*{\scriptstyle 2};
"DD"+(-10,0)*{\bullet};
"DD"+(-10,-2)*{\scriptstyle \psi(\us=\al+\be)};
"DL"+(15,-3); "DD"+(15,-3) **\dir{-};
"DR"+(11,-7); "DD"+(15,-3) **\dir{-};
"DT"+(11,-7); "DR"+(11,-7) **\dir{-};
"DT"+(11,-7); "DL"+(15,-3) **\dir{-};
"DL"+(19,-3)*{\scriptstyle \psi(\be)};
"DR"+(7,-7)*{\scriptstyle \psi(\al)};
"DL"+(30,0)*{{\rm(ii)'}};
"DL"+(15,-3)*{\bullet};
"DL"+(31,12.5)*{\bullet};
"DL"+(31,14.5)*{\scriptstyle \psi(s_1)};
"DL"+(25,-13)*{\bullet};
"DL"+(25,-15)*{\scriptstyle \psi(s_2)};
"DR"+(11,-7)*{\bullet};
"DD"+(28,4); "DD"+(38,-6) **\dir{-};
"DD"+(28,4); "DD"+(50,26) **\dir{.};
"DD"+(48,4); "DD"+(38,-6) **\dir{-};
"DD"+(48,4); "DD"+(70,26) **\dir{.};
"DD"+(82,18); "DD"+(74,26) **\dir{.};
"DD"+(82,18); "DD"+(58,-6) **\dir{-};
"DD"+(48,4); "DD"+(58,-6) **\dir{-};
"DD"+(82,18)*{\bullet};
"DD"+(78,18)*{\scriptstyle \psi(\al)};
"DD"+(28,4)*{\bullet};
"DD"+(32,4)*{\scriptstyle \psi(\be)};
"DD"+(74,26); "DD"+(70,26) **\crv{"DD"+(72,28)};
"DD"+(72,30)*{\scriptstyle 2};
"DD"+(38,14)*{\bullet};
"DD"+(47,14)*{\scriptstyle \psi(\us=\al+\be)};
"DL"+(59,0)*{{\rm(iii)'}};
"DD"+(48,4); "DD"+(38,14) **\dir{.};
"DD"+(80,4); "DD"+(90,-6) **\dir{-};
"DD"+(96,0); "DD"+(90,-6) **\dir{-};
"DD"+(96,0); "DD"+(102,-6) **\dir{-};
"DD"+(122,16); "DD"+(102,-6) **\dir{-};
"DD"+(80,4); "DD"+(102,26) **\dir{.};
"DD"+(122,16); "DD"+(112,26) **\dir{.};
"DD"+(96,0); "DD"+(117,21) **\dir{.};
"DD"+(96,0); "DD"+(86,10) **\dir{.};
"DD"+(117,21)*{\bullet};
"DD"+(112,21)*{\scriptstyle \psi(s_1)};
"DD"+(86,10)*{\bullet};
"DD"+(91,10)*{\scriptstyle \psi(s_2)};
"DD"+(102,26); "DD"+(112,26) **\crv{"DD"+(107,28)};
"DD"+(107,29)*{\scriptstyle l > 2};
"DD"+(80,4)*{\bullet};
"DL"+(109,0)*{{\rm (iv)'}};
"DD"+(84,4)*{\scriptstyle \psi(\be)};
"DD"+(122,16)*{\bullet};
"DD"+(117,16)*{\scriptstyle \psi(\al)};
%
"DD"+(119,4); "DD"+(129,-6) **\dir{-};
"DD"+(135,0); "DD"+(129,-6) **\dir{-};
"DD"+(119,4); "DD"+(141,26) **\dir{-};
"DD"+(154,19); "DD"+(147,26) **\dir{-};
"DD"+(135,0); "DD"+(154,19) **\dir{-};
"DD"+(119,4)*{\bullet};
"DL"+(146,0)*{{\rm (v)'}};
"DD"+(123,4)*{\scriptstyle \psi(\be)};
"DD"+(141,26); "DD"+(147,26) **\crv{"DD"+(144,28)};
"DD"+(154,19)*{\bullet};
"DD"+(154,17)*{\scriptstyle \psi(\al)};
"DD"+(129,-6)*{\bullet};
"DD"+(129,-8)*{\scriptstyle 2 \psi(s_2)};
"DD"+(144,29)*{\scriptstyle 2};
\endxy}}
\allowdisplaybreaks \\
& \scalebox{0.79}{{\xy
(-20,0)*{}="DL";(-10,-10)*{}="DD";(0,20)*{}="DT";(10,10)*{}="DR";
"DT"+(-30,-4); "DT"+(120,-4)**\dir{.};
"DD"+(-20,-6); "DD"+(130,-6) **\dir{.};
"DT"+(-34,-4)*{\scriptstyle 1};
"DT"+(-34,-8)*{\scriptstyle 2};
"DT"+(-34,-12)*{\scriptstyle \vdots};
"DT"+(-34,-16)*{\scriptstyle \vdots};
"DT"+(-36,-36)*{\scriptstyle n};
"DD"+(-17,4); "DD"+(-7,-6) **\dir{-};
"DD"+(-1,0); "DD"+(-7,-6) **\dir{-};
"DD"+(-17,4); "DD"+(2,23) **\dir{-};
"DD"+(12,13); "DD"+(2,23) **\dir{-};
"DD"+(-1,0); "DD"+(12,13) **\dir{-};
"DD"+(-17,4)*{\bullet};
"DD"+(2,23)*{\bullet};
"DL"+(10,0)*{{\rm (vi)'}};
"DD"+(-13,4)*{\scriptstyle \psi(\be)};
"DD"+(2,25)*{\scriptstyle \psi(s_3)};
"DD"+(12,13)*{\bullet};
"DD"+(12,11)*{\scriptstyle \psi(\al)};
"DD"+(-7,-6)*{\bullet};
"DD"+(-7,-8)*{\scriptstyle 2 \psi(s_2)};
"DD"+(15,4); "DD"+(25,-6) **\dir{-};
"DD"+(15,4); "DD"+(5,-6) **\dir{-};
"DD"+(15,4); "DD"+(30,19) **\dir{.};
"DD"+(40,9);"DD"+(25,-6) **\dir{-};
"DD"+(40,9);"DD"+(30,19) **\dir{.};
"DD"+(40,9)*{\bullet};
"DD"+(30,19)*{\bullet};
"DD"+(30,21)*{\scriptstyle \psi(s_1)};
"DD"+(5,-6)*{\bullet};
"DD"+(5,-8)*{\scriptstyle \psi(\be)};
"DD"+(25,-6.2)*{\bullet};
"DD"+(25,-8.2)*{\scriptstyle \psi(s_2)};
"DD"+(35,9)*{\scriptstyle \psi(\al)};
"DD"+(25,4)*{{\rm (vii)'}};
"DD"+(45,4); "DD"+(55,-6) **\dir{-};
"DD"+(45,4); "DD"+(35,-6) **\dir{-};
"DD"+(45,4); "DD"+(67,26) **\dir{.};
"DD"+(80,19);"DD"+(55,-6) **\dir{-};
"DD"+(80,19);"DD"+(73,26) **\dir{.};
"DD"+(80,19)*{\bullet};
"DD"+(55,4)*{{\rm (viii)'}};
"DD"+(76,19)*{\scriptstyle \psi(\al)};
"DD"+(35,-6)*{\bullet};
"DD"+(55,-6.2)*{\bullet};
"DD"+(55,-8.2)*{\scriptstyle \psi(\us=\al+\be)};
"DD"+(35,-8)*{\scriptstyle \psi(\be)};
"DD"+(67,26); "DD"+(73,26) **\crv{"DD"+(70,28)};
"DD"+(70,30)*{\scriptstyle 2};
"DD"+(92,16); "DD"+(114,-6) **\dir{-};
"DD"+(92,16); "DD"+(70,-6) **\dir{-};
"DD"+(92,4)*{ {\rm (ix)'}};
"DD"+(114,-6)*{\bullet};
"DD"+(114,-8)*{\scriptstyle \psi(\al)};
"DD"+(70,-6)*{\bullet};
"DD"+(92,16)*{\bullet};
"DD"+(92,18)*{\scriptstyle\psi( \us=\al+\be)};
"DD"+(70,-8)*{\scriptstyle \psi(\be)};
\endxy}} 
\end{aligned}
\end{equation}

By taking  $\psi$ on~\eqref{eq: dist al be 2}, we have the following:

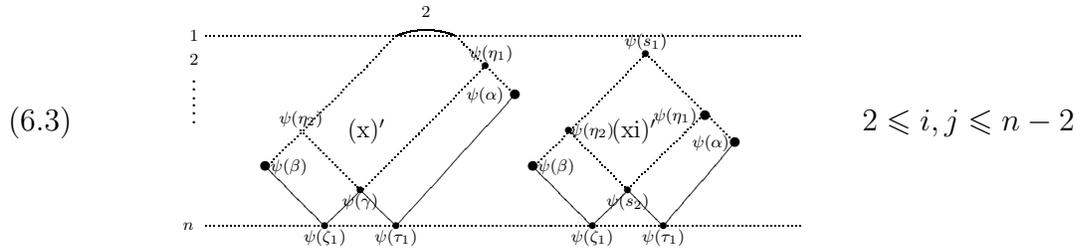
\begin{align} \label{eq: dist al be 2 p}
\scalebox{0.79}{{\xy
(-20,0)*{}="DL";(-10,-10)*{}="DD";(0,20)*{}="DT";(10,10)*{}="DR";
"DT"+(-30,-4); "DT"+(70,-4)**\dir{.};
"DD"+(-20,-6); "DD"+(80,-6) **\dir{.};
"DT"+(-32,-4)*{\scriptstyle 1};
"DT"+(-32,-8)*{\scriptstyle 2};
"DT"+(-32,-12)*{\scriptstyle \vdots};
"DT"+(-32,-16)*{\scriptstyle \vdots};
"DT"+(-33,-36)*{\scriptstyle n};
"DD"+(-10,4); "DD"+(0,-6) **\dir{-};
"DD"+(6,0); "DD"+(0,-6) **\dir{-};
"DD"+(6,0); "DD"+(12,-6) **\dir{-};
"DD"+(32,16); "DD"+(12,-6) **\dir{-};
"DD"+(-10,4); "DD"+(12,26) **\dir{.};
"DD"+(32,16); "DD"+(22,26) **\dir{.};
"DD"+(6,0); "DD"+(27,21) **\dir{.};
"DD"+(6,0); "DD"+(-4,10) **\dir{.};
"DD"+(12,26); "DD"+(22,26) **\crv{"DD"+(17,28)};
"DD"+(17,30)*{\scriptstyle 2};
"DD"+(-10,4)*{\bullet};
"DL"+(17,0)*{{\rm(x)'}};
"DD"+(-6,4)*{\scriptstyle \psi(\be)};
"DD"+(-4,12)*{\scriptstyle \psi(\eta_2)};
"DD"+(32,16)*{\bullet};
"DD"+(27,16)*{\scriptstyle \psi(\al)};
"DD"+(28,23)*{\scriptstyle \psi(\eta_1)};
"DD"+(27,21)*{\scriptstyle \bullet};
"DT"+(2,-36)*{\scriptstyle \bullet};
"DT"+(2,-38)*{\scriptstyle \psi(\tau_1)};
"DT"+(-10,-36)*{\scriptstyle \bullet};
"DT"+(-10,-38)*{\scriptstyle \psi(\zeta_1)};
"DT"+(-4,-30)*{\scriptstyle \bullet};
"DT"+(-4,-32)*{\scriptstyle \psi(\ga)};
"DD"+(35,4); "DD"+(45,-6) **\dir{-};
"DD"+(51,0); "DD"+(45,-6) **\dir{-};
"DD"+(51,0); "DD"+(57,-6) **\dir{-};
"DD"+(69,8); "DD"+(57,-6) **\dir{-};
"DD"+(35,4); "DD"+(54,23) **\dir{.};
"DD"+(69,8); "DD"+(54,23) **\dir{.};
"DD"+(51,0); "DD"+(64,13) **\dir{.};
"DD"+(51,0); "DD"+(41,10) **\dir{.};
"DD"+(35,4)*{\bullet};
"DL"+(62,0)*{{\rm(xi)'}};
"DD"+(39,4)*{\scriptstyle \psi(\be)};
"DD"+(41,10)*{\scriptstyle \bullet};
"DD"+(45,10)*{\scriptstyle \psi(\eta_2)};
"DD"+(69,8)*{\bullet};
"DD"+(64,12.5)*{\bullet};
"DD"+(59,12.5)*{\scriptstyle \psi(\eta_1)};
"DD"+(65,8)*{\scriptstyle \psi(\al)};
"DT"+(47,-36)*{\scriptstyle \bullet};
"DT"+(47,-38)*{\scriptstyle \psi(\tau_1)};
"DT"+(35,-36)*{\scriptstyle \bullet};
"DT"+(35,-38)*{\scriptstyle \psi(\zeta_1)};
"DT"+(41,-32)*{\scriptstyle \psi(s_2)};
"DT"+(44,-5)*{\scriptstyle \psi(s_1)};
"DT"+(44,-7)*{\scriptstyle \bullet};
"DT"+(41,-30)*{\scriptstyle \bullet};
\endxy}} \qquad
 2 \le i,j \le n-2
\end{align}

In each case, the value $(\upal,\upbe) = (\psi(\al),\psi(\be))$ does not depend on the choice of $Q$ and are given in Table~\ref{table: al,be B}:
\begin{table}[h]
\centering
{ \arraycolsep=1.6pt\def\arraystretch{1.5}
\begin{tabular}{|c||c|c|c|c|c|c|c|c|c|c|c|c|}
\hline
 &                         (i)         & (ii)  & (iii)  & (iv)   & (v)     & (vi)  & (vii)  & (viii)  & (ix)  & (x)  & (ix)  \\ \hline
$(\al,\be)$ &  $-2$ &  $0$ &  $-2$  &  $0$ &  $0$ &  $2$ &  $0$ &  $-2$ &  $0$ &  $-2$ &  $0$ \\ \hline
\end{tabular}
}\\[1.5ex]
    \caption{$(\al,\be)$ for non $[Q]$-simple \prqs of type $B_n$}
    \protect\label{table: al,be B}
\end{table}

For the cases when $\pair{\al,\be}$ is a $[Q]$-minimal \prq for $\al+\be \in \Phi^+$,
\begin{align}\label{eq: p be al B}
p_{\be,\al}=1 \quad \text{ if  $\pair{\al,\be}$ is of case ${\rm (ix)'}$} \quad \text{  and } \quad \text{$p_{\be,\al}=0$, otherwise. }
\end{align}

\begin{example} \label{ex: Uding B}
Using $\Gamma_Q$ in Example~\ref{ex: Label B4},
\ben
\item $ \srt{1,2} \prec_Q^\ttb \pair{\srt{1,-3},\srt{2,3}}$ corresponds to ${\rm (i)}'$,
\item $ \pair{\srt{1,-3},\srt{2,4}} \prec_Q^\ttb \pair{\srt{2,-3},\srt{1,4}}$ corresponds to ${\rm (ii)}'$,
\item $ \srt{2,-4}  \prec_Q^\ttb  \pair{\srt{2,-3},\srt{3,-4}}$ corresponds to ${\rm (iii)}'$,
\item $ \pair{\srt{1,-3},\srt{2,-4}} \prec_Q^\ttb  \pair{\srt{2,-3},\srt{1,-4}}$ corresponds to ${\rm (iv)}'$,
\item $ 2\srt{1} \prec_Q^\ttb  \pair{\srt{1,-3},\srt{1,3}}$ corresponds to ${\rm (v)}'$,
\item $ \seq{\srt{3,4},2\srt{1}} \prec_Q^\ttb  \pair{\srt{1,4},\srt{1,3}}$ corresponds to ${\rm (vi)}'$,
\item $ \pair{\srt{1},\srt{2,3}}  \prec_Q^\ttb  \pair{\srt{1,2},\srt{3}}$   correspond to ${\rm (vii)}'$,
\item $ \srt{1}  \prec_Q^\ttb  \pair{\srt{1,-3},\srt{3}}$   correspond to ${\rm (viii)}'$,
\item $ \srt{2,3}  \prec_Q^\ttb  \pair{\srt{2},\srt{3}}$ correspond to ${\rm (ix)'}$,
\item $ \srt{1,2} \prec_Q^\ttb    \pair{\srt{1,-3},\srt{2,3}} ,\pair{\srt{2},\srt{1}}  \prec_Q^\ttb  \pair{\srt{2,-3},\srt{1,3}}$ corresponds to ${\rm (x)}'$,
\item $  \pair{\srt{1,2},\srt{3,4}} \prec_Q^\ttb  \pair{\srt{1},\srt{2},\srt{3,4}} ,\pair{\srt{2,3},\srt{1,4}}  \prec_Q^\ttb  \pair{\srt{2,4},\srt{1,3}}$ corresponds to ${\rm (xi)}'$.
\ee
\end{example}

By the similarities between the combinatorics for $\Gamma_Q$ and $\Gamma_\oQ$, described in Proposition~\ref{prop: label1}, ~\ref{prop: label2} for $\Gamma_Q$ and Proposition~\ref{prop: label12 B} for $\Gamma_\oQ$, we can apply the same argument in
\cite{Oh18} and obtain the  following theorem:

\begin{theorem} \label{thm: soc B}
Let $Q=(\bDynkin,\xi)$ be a Dynkin quiver of type $B_n$.
\ben
\item For any \prq $\pair{\al,\be} \in (\Phi^+_{B_n})^2$, $\head_Q(\pair{\al,\be})$ is well-defined.
\item For a \prq $\pair{\al,\be}$, if $\dg_Q(\pair{\al,\be}) \ne 0$, then the relative position of $\pair{\al,\be}$ is one of ${\rm (i)}' \sim {\rm (xi)}'$ in~\eqref{eq: dist al, be 1 p} and~\eqref{eq: dist al be 2 p}.
\item Every \ex $\um \prec_Q^\ttb  \pair{\al,\be}$ is also described in ~\eqref{eq: dist al, be 1 p} and~\eqref{eq: dist al be 2 p}.
\item For $\pair{\al,\be} \in (\Phi^+_\blacktriangle)^2$, assume $\dg_Q(\pair{\al,\be}) \ne 0$. Then there exist $\al' \in \psi^{-1}(\al)$
and $\be' \in \psi^{-1}(\be)$ such that $\dg_{\uQ}( \pair{\al',\be'} ) \ne 0$ and $\head_Q(\pair{\al,\be})=\psi(\head_{\uQ}( \pair{\al',\be'} ))$.   Furthermore, in this case, we have
$$   \dg_Q(\pair{\al,\be})=\dg_{\uQ}( \pair{\al',\be'} ).$$
\ee
\end{theorem}

Notice that, in cases (v) and (vi), the head of $\pair{\al,\be}$
has component $2$ in which phenomena does not happen in
$D_{n+1}$-case.

\begin{proof}[Proof of Theorem~\ref{thm: B_n distance polynomial}]
Note that for $1 \le i ,j  \le n$ with $\min(i,j)<n$, we have $\max(\sfd_i,\sfd_j)=2$. By Theorem~\ref{thm: soc B}, we have
$$\de^{\blacktriangle}_{i,j}(t) =   2 \times   \de^{\Dynkin}_{i,j}(t)$$
implying the assertion by the formulas~\eqref{eq: d poly D} and~\eqref{eq: d poly BCFG}.

In the case when $i=j=n$, by Theorem~\ref{thm: soc B}, and
the facts that $\psi(\ep_i \pm \ep_{n+1})=\ve_i$ and $\sfd_n=1$, we have
$$  \de^{\blacktriangle}_{n,n}(t) =  \left( \de^{\Dynkin}_{n,n+1}(t)+\de^{\Dynkin}_{n,n}(t) \right)  = \sum_{s=1}^{n} t^{2s-1},$$
which completes the assertion.
\end{proof}

By observing~\eqref{eq: dist al be 2 p}, we have the following proposition:

\begin{proposition}   \label{prop: many pairs}
For a Dynkin quiver $Q$ of type $B_n$, let $\pair{\al,\be}$ be a \prq such that $\dg_Q(\pair{\al,\be})=2$.
\ben
\item If $\al+\be =\ga \in \Phi^+_{B_n}$, then there are two \prqs $\pair{\al_i,\be_i}$ $(1 \le i \le 2)$ of $\ga$ such that
$$ \head_Q(\pair{\al,\be})   =\ga   \prec^\ttb_{[Q]}   \pair{\al_i,\be_i}   \prec^\ttb_{[Q]}  \pair{\al,\be}  $$
\bnum
\item $\pair{\al_1,\be_1}$, $\pair{\al_2,\be_2}$ are not comparable with respect to $\prec^\ttb_{[Q]}$,
\item $\{ \al_2,\be_2\} = \{ \lan a  \ran , \lan b  \ran \}$  and $\res_Q(\al_1), \res_Q(\be_1 )<n$.
\ee
\item If $\al+\be \not\in \Phi^+_{B_n}$, then there are \exs $\pair{\al_1,\be_1}$, $\seq{\tau,\zeta,\ga}$ such that
$$  \head_Q(\pair{\al,\be})      \prec^\ttb_{[Q]}   \pair{\al_1,\be_1}  , \seq{\tau,\zeta,\ga} \prec^\ttb_{[Q]}  \pair{\al,\be}  $$
\bnum
\item $\pair{\al_1,\be_1}$, $\seq{\tau,\zeta,\ga}$ are not comparable with respect to $\prec^\ttb_{[Q]}$,
\item $\res^{[Q]}(\al_1), \res^{[Q]}(\ga),\res^{[Q]}(\be_1 )<n$ and  $\res^{[Q]}(\tau), \res^{[Q]}(\zeta)=n$.
\ee
\ee
\end{proposition}

\subsection{$C_n$-case}
Throughout this subsection, $\bDynkin$ usually denotes the Dynkin diagram of type $C_{n}$.
This subsection is devoted to prove the following theorem:
\begin{theorem} \label{thm: distance C_n}
Let $\bDynkin$ be the Dynkin diagram of type $C_n$. Then we have
\begin{align*}
\de_{i,j}^\blacktriangle(t) +\delta_{i,j} \max(\sfd_i,\sfd_j)t^{\sfh-1}
 = \tde_{i,j}^\blacktriangle(t).
\end{align*}
\end{theorem}

Recall the natural bijection $\psi^\tr\cl \Phi^+_{B_n} \to
\Phi^+_{C_n}$ and Proposition~\ref{prop: transpose}, which tells that
labeling of $\Gamma_Q$ and $\Gamma_{Q^\tr}$ for Dynkin quiver $Q$ of
type $B_n$ are almost the same, when we use the notations $\lan a,b \ran$, $\lan a\ran$ and $\lan a,a\ran$
for $1 \le a \le |b| \le n$.

For $\al \in \Phi^+_{B_n}$, we set $$\al^\tr \seteq \psi^\tr(\al).$$

Using Proposition~\ref{prop: label12 B} for $\Gamma_\oQ$ and the similar argument in \cite{Oh18},
we obtain the following as in $B_n$-cases:

\begin{theorem} \label{thm: soc C}
Let $Q^\tr$ a Dynkin quiver of type $C_n$.
\ben
\item For any \prq $\pair{\al,\be} \in (\Phi^+_{C_n})^2$, $  \head_{Q^\tr}  (\pair{\al,\be})$ is well-defined.
\item For $\pair{\al,\be} \in (\Phi^+_{B_n})^2$,   $\dg_Q(\pair{\al,\be}) \ne 0$ if and only if $\dg_{Q^\tr}(\pair{\al^\tr,\be^\tr}) \ne 0$. Moreover, we have
$$   \dg_Q(\pair{\al,\be})=\dg_{Q^\tr}(\pair{\al^\tr,\be^\tr}).$$
\ee
\end{theorem}

Note that
$$
\lan  i \ran =  \ve_i \  \text{ if  $\bDynkin$ is of type $B_n$} \text{ and }
\lan  i,i \ran = 2\ep_i \ \text{ if  $\bDynkin$ is of type $C_n$}.
$$
By this difference,  the descriptions for $  \head_Q(\pair{\al,\be})  $ for
\prqs $\pair{ \al,\be} \in \Phi^+_{C_n}$ and a Dynkin quiver $Q$ of
type $C_n$ with $\dg_Q(\pair{\al,\be}) \ne 0$ are sometimes different
from the corresponding ones in cases of type $B_n$. Here we present
the \prqs $\pair{ \al,\be} \in \Phi^+_{C_n}$ for $\dg_Q(\pair{\al,\be})
\ne 0$ with description of $ \head_Q(\pair{\al,\be})$ as follows:
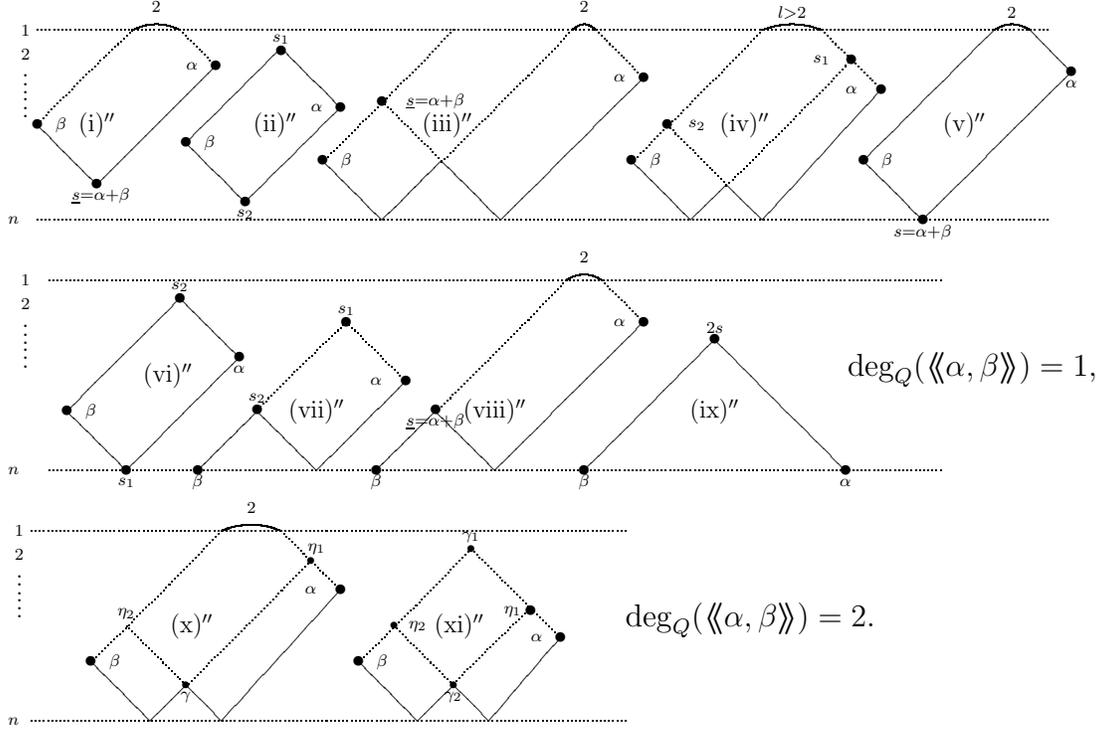
\begin{align*}
& \scalebox{0.79}{{\xy
(-20,0)*{}="DL";(-10,-10)*{}="DD";(0,20)*{}="DT";(10,10)*{}="DR";
"DT"+(-30,-4); "DT"+(140,-4)**\dir{.};
"DD"+(-20,-6); "DD"+(150,-6) **\dir{.};
"DT"+(-32,-4)*{\scriptstyle 1};
"DT"+(-32,-8)*{\scriptstyle 2};
"DT"+(-32,-12)*{\scriptstyle \vdots};
"DT"+(-32,-16)*{\scriptstyle \vdots};
"DT"+(-34,-36)*{\scriptstyle n};
"DL"+(-10,0); "DD"+(-10,0) **\dir{-};"DL"+(-10,0); "DT"+(-14,-4) **\dir{.};
"DT"+(-6,-4); "DR"+(-10,0) **\dir{.};"DR"+(-10,0); "DD"+(-10,0) **\dir{-};
"DL"+(-6,0)*{\scriptstyle \be};
"DL"+(0,0)*{{\rm(i)''}};
"DR"+(-14,0)*{\scriptstyle \al};
"DL"+(-10,0)*{\bullet};
"DR"+(-10,0)*{\bullet};
"DT"+(-6,-4); "DT"+(-14,-4) **\crv{"DT"+(-10,-2)};
"DT"+(-10,0)*{\scriptstyle 2};
"DD"+(-10,0)*{\bullet};
"DD"+(-10,-2)*{\scriptstyle \ \us=\al+\be};
"DL"+(15,-3); "DD"+(15,-3) **\dir{-};
"DR"+(11,-7); "DD"+(15,-3) **\dir{-};
"DT"+(11,-7); "DR"+(11,-7) **\dir{-};
"DT"+(11,-7); "DL"+(15,-3) **\dir{-};
"DL"+(19,-3)*{\scriptstyle \be};
"DR"+(7,-7)*{\scriptstyle \al};
"DL"+(30,0)*{{\rm(ii)''}};
"DL"+(15,-3)*{\bullet};
"DL"+(31,12.5)*{\bullet};
"DL"+(31,14.5)*{\scriptstyle  s_1};
"DL"+(25,-13)*{\bullet};
"DL"+(25,-15)*{\scriptstyle  s_2};
"DR"+(11,-7)*{\bullet};
"DD"+(28,4); "DD"+(38,-6) **\dir{-};
"DD"+(28,4); "DD"+(50,26) **\dir{.};
"DD"+(48,4); "DD"+(38,-6) **\dir{-};
"DD"+(48,4); "DD"+(70,26) **\dir{.};
"DD"+(82,18); "DD"+(74,26) **\dir{.};
"DD"+(82,18); "DD"+(58,-6) **\dir{-};
"DD"+(48,4); "DD"+(58,-6) **\dir{-};
"DD"+(82,18)*{\bullet};
"DD"+(78,18)*{\scriptstyle \al};
"DD"+(28,4)*{\bullet};
"DD"+(32,4)*{\scriptstyle  \be};
"DD"+(74,26); "DD"+(70,26) **\crv{"DD"+(72,28)};
"DD"+(72,30)*{\scriptstyle 2};
"DD"+(38,14)*{\bullet};
"DD"+(47,14)*{\scriptstyle  \us=\al+\be};
"DL"+(59,0)*{{\rm(iii)''}};
"DD"+(48,4); "DD"+(38,14) **\dir{.};
"DD"+(80,4); "DD"+(90,-6) **\dir{-};
"DD"+(96,0); "DD"+(90,-6) **\dir{-};
"DD"+(96,0); "DD"+(102,-6) **\dir{-};
"DD"+(122,16); "DD"+(102,-6) **\dir{-};
"DD"+(80,4); "DD"+(102,26) **\dir{.};
"DD"+(122,16); "DD"+(112,26) **\dir{.};
"DD"+(96,0); "DD"+(117,21) **\dir{.};
"DD"+(96,0); "DD"+(86,10) **\dir{.};
"DD"+(117,21)*{\bullet};
"DD"+(112,21)*{\scriptstyle  s_1};
"DD"+(86,10)*{\bullet};
"DD"+(91,10)*{\scriptstyle  s_2};
"DD"+(102,26); "DD"+(112,26) **\crv{"DD"+(107,28)};
"DD"+(107,29)*{\scriptstyle l > 2};
"DD"+(80,4)*{\bullet};
"DL"+(109,0)*{{\rm (iv)''}};
"DD"+(84,4)*{\scriptstyle  \be};
"DD"+(122,16)*{\bullet};
"DD"+(117,16)*{\scriptstyle  \al};
%
"DD"+(119,4); "DD"+(129,-6) **\dir{-};
"DD"+(135,0); "DD"+(129,-6) **\dir{-};
"DD"+(119,4); "DD"+(141,26) **\dir{-};
"DD"+(154,19); "DD"+(147,26) **\dir{-};
"DD"+(135,0); "DD"+(154,19) **\dir{-};
"DD"+(119,4)*{\bullet};
"DL"+(146,0)*{{\rm (v)''}};
"DD"+(123,4)*{\scriptstyle  \be};
"DD"+(141,26); "DD"+(147,26) **\crv{"DD"+(144,28)};
"DD"+(154,19)*{\bullet};
"DD"+(154,17)*{\scriptstyle  \al};
"DD"+(129,-6)*{\bullet};
"DD"+(129,-8)*{\scriptstyle  s=\al+\be};
"DD"+(144,29)*{\scriptstyle 2};
\endxy}}
\allowdisplaybreaks \\
& \scalebox{0.79}{{\xy
(-20,0)*{}="DL";(-10,-10)*{}="DD";(0,20)*{}="DT";(10,10)*{}="DR";
"DT"+(-30,-4); "DT"+(120,-4)**\dir{.};
"DD"+(-20,-6); "DD"+(130,-6) **\dir{.};
"DT"+(-34,-4)*{\scriptstyle 1};
"DT"+(-34,-8)*{\scriptstyle 2};
"DT"+(-34,-12)*{\scriptstyle \vdots};
"DT"+(-34,-16)*{\scriptstyle \vdots};
"DT"+(-36,-36)*{\scriptstyle n};
"DD"+(-17,4); "DD"+(-7,-6) **\dir{-};
"DD"+(-1,0); "DD"+(-7,-6) **\dir{-};
"DD"+(-17,4); "DD"+(2,23) **\dir{-};
"DD"+(12,13); "DD"+(2,23) **\dir{-};
"DD"+(-1,0); "DD"+(12,13) **\dir{-};
"DD"+(-17,4)*{\bullet};
"DD"+(2,23)*{\bullet};
"DL"+(10,0)*{{\rm (vi)''}};
"DD"+(-13,4)*{\scriptstyle  \be};
"DD"+(2,25)*{\scriptstyle  s_2};
"DD"+(12,13)*{\bullet};
"DD"+(12,11)*{\scriptstyle  \al};
"DD"+(-7,-6)*{\bullet};
"DD"+(-7,-8)*{\scriptstyle  s_1};
"DD"+(15,4); "DD"+(25,-6) **\dir{-};
"DD"+(15,4); "DD"+(5,-6) **\dir{-};
"DD"+(15,4); "DD"+(30,19) **\dir{.};
"DD"+(40,9);"DD"+(25,-6) **\dir{-};
"DD"+(40,9);"DD"+(30,19) **\dir{.};
"DD"+(40,9)*{\bullet};
"DD"+(30,19)*{\bullet};
"DD"+(30,21)*{\scriptstyle  s_1};
"DD"+(5,-6)*{\bullet};
"DD"+(5,-8)*{\scriptstyle  \be};
"DD"+(15,4.2)*{\bullet};
"DD"+(15,6.2)*{\scriptstyle s_2};
"DD"+(35,9)*{\scriptstyle \al};
"DD"+(25,4)*{{\rm (vii)''}};
"DD"+(45,4); "DD"+(55,-6) **\dir{-};
"DD"+(45,4); "DD"+(35,-6) **\dir{-};
"DD"+(45,4); "DD"+(67,26) **\dir{.};
"DD"+(80,19);"DD"+(55,-6) **\dir{-};
"DD"+(80,19);"DD"+(73,26) **\dir{.};
"DD"+(80,19)*{\bullet};
"DD"+(55,4)*{{\rm (viii)''}};
"DD"+(76,19)*{\scriptstyle  \al};
"DD"+(35,-6)*{\bullet};
"DD"+(45,4.2)*{\bullet};
"DD"+(45,2.2)*{\scriptstyle  \us=\al+\be};
"DD"+(35,-8)*{\scriptstyle \be};
"DD"+(67,26); "DD"+(73,26) **\crv{"DD"+(70,28)};
"DD"+(70,30)*{\scriptstyle 2};
"DD"+(92,16); "DD"+(114,-6) **\dir{-};
"DD"+(92,16); "DD"+(70,-6) **\dir{-};
"DD"+(92,4)*{ {\rm (ix)''}};
"DD"+(114,-6)*{\bullet};
"DD"+(114,-8)*{\scriptstyle \al};
"DD"+(70,-6)*{\bullet};
"DD"+(92,16)*{\bullet};
"DD"+(92,18)*{\scriptstyle 2s};
"DD"+(70,-8)*{\scriptstyle  \be};
\endxy}}  \hspace{-3em} \dg_Q(\pair{\al,\be})=1,
\allowdisplaybreaks \\
&\scalebox{0.79}{{\xy
(-20,0)*{}="DL";(-10,-10)*{}="DD";(0,20)*{}="DT";(10,10)*{}="DR";
"DT"+(-30,-4); "DT"+(70,-4)**\dir{.};
"DD"+(-20,-6); "DD"+(80,-6) **\dir{.};
"DT"+(-32,-4)*{\scriptstyle 1};
"DT"+(-32,-8)*{\scriptstyle 2};
"DT"+(-32,-12)*{\scriptstyle \vdots};
"DT"+(-32,-16)*{\scriptstyle \vdots};
"DT"+(-33,-36)*{\scriptstyle n};
"DD"+(-10,4); "DD"+(0,-6) **\dir{-};
"DD"+(6,0); "DD"+(0,-6) **\dir{-};
"DD"+(6,0); "DD"+(12,-6) **\dir{-};
"DD"+(32,16); "DD"+(12,-6) **\dir{-};
"DD"+(-10,4); "DD"+(12,26) **\dir{.};
"DD"+(32,16); "DD"+(22,26) **\dir{.};
"DD"+(6,0); "DD"+(27,21) **\dir{.};
"DD"+(6,0); "DD"+(-4,10) **\dir{.};
"DD"+(12,26); "DD"+(22,26) **\crv{"DD"+(17,28)};
"DD"+(17,30)*{\scriptstyle 2};
"DD"+(-10,4)*{\bullet};
"DL"+(17,0)*{{\rm(x)''}};
"DD"+(-6,4)*{\scriptstyle \be};
"DD"+(-4,12)*{\scriptstyle \eta_2};
"DD"+(32,16)*{\bullet};
"DD"+(27,16)*{\scriptstyle \al};
"DD"+(28,23)*{\scriptstyle \eta_1};
"DD"+(27,21)*{\scriptstyle \bullet};
"DT"+(-4,-30)*{\scriptstyle \bullet};
"DT"+(-4,-32)*{\scriptstyle \ga};
"DD"+(35,4); "DD"+(45,-6) **\dir{-};
"DD"+(51,0); "DD"+(45,-6) **\dir{-};
"DD"+(51,0); "DD"+(57,-6) **\dir{-};
"DD"+(69,8); "DD"+(57,-6) **\dir{-};
"DD"+(35,4); "DD"+(54,23) **\dir{.};
"DD"+(69,8); "DD"+(54,23) **\dir{.};
"DD"+(51,0); "DD"+(64,13) **\dir{.};
"DD"+(51,0); "DD"+(41,10) **\dir{.};
"DD"+(35,4)*{\bullet};
"DL"+(62,0)*{{\rm(xi)''}};
"DD"+(39,4)*{\scriptstyle \be};
"DD"+(41,10)*{\scriptstyle \bullet};
"DD"+(45,10)*{\scriptstyle \eta_2};
"DD"+(69,8)*{\bullet};
"DD"+(64,12.5)*{\bullet};
"DD"+(61,12.5)*{\scriptstyle \eta_1};
"DD"+(65,8)*{\scriptstyle \al};
"DT"+(41,-32)*{\scriptstyle \ga_2};
"DT"+(44,-5)*{\scriptstyle \ga_1};
"DT"+(44,-7)*{\scriptstyle \bullet};
"DT"+(41,-30)*{\scriptstyle \bullet};
\endxy}} \dg_Q(\pair{\al,\be})=2.
\end{align*}
Here one can recognize that the differences happen in cases (vii)$'' \sim$ (xi)$''$, when we compare the cases for $B_n$.

In each case, the value $(\al,\be) $ does not depend on the choice of $Q$ and are given  in Table~\ref{table: al,be C}:
\begin{table}[h]
\centering
{ \arraycolsep=1.6pt\def\arraystretch{1.5}
\begin{tabular}{|c||c|c|c|c|c|c|c|c|c|c|c|c|}
\hline
 &   (i) & (ii)  & (iii)  & (iv)  & (v)  & (vi)  & (vii)  & (viii)  & (ix)  & (x)  & (xi)  \\ \hline
$(\al,\be)$ &  $-1$ &  $0$ &  $-1$  &  $0$ &  $0$ &  $1$ &  $0$ &  $-1$ &  $0$ &  $-1$ &  $0$ \\ \hline
\end{tabular}
}\\[1.5ex]
    \caption{$(\al,\be)$ for non $[Q]$-simple \prqs of type $C_n$}
    \protect\label{table: al,be C}
\end{table}

For the cases when a \prq  $\pair{\al,\be}$ is  $[Q]$-minimal for $\al+\be \in \Phi^+$,
\begin{align}\label{eq: p be al C}
p_{\be,\al}=1 \quad \text{ if  $\pair{\al,\be}$ is of case ${\rm (v)''}$} \quad \text{  and } \quad \text{$p_{\be,\al}=0$, otherwise. }
\end{align}

\begin{example} \label{ex: Uding C}
Using $\Gamma_Q$ in Example~\ref{ex: Label C4},
\ben
\item[(7)] $ \pair{\srt{1,3},\srt{3,4}}  \prec_Q^\ttb  \pair{\srt{1,4},\srt{3,3}}$   correspond to ${\rm (vii)}''$,
\item[(8)] $ \srt{1,3}  \prec_Q^\ttb  \pair{\srt{1,4},\srt{3,-4}}$   correspond to ${\rm (viii)}''$,
\item[(9)] $ 2\srt{3,4}  \prec_Q^\ttb  \pair{\srt{3,3},\srt{4,4}}$ correspond to ${\rm (ix)''}$,
\item[(10)] $ \srt{1,3} \prec_Q^\ttb  \pair{\srt{3,4},\srt{1,-4}}  \prec_Q^\ttb  \pair{\srt{1,4},\srt{3,-4}}$ corresponds to ${\rm (x)}''$,
\item[(11)] $  \pair{\srt{1,2},\srt{3,4}} \prec_Q^\ttb  \pair{\srt{2,3},\srt{1,4}}  \prec_Q^\ttb  \pair{\srt{2,4},\srt{1,3}}$ corresponds to ${\rm (xi)}''$.
\ee
\end{example}

By observing (x)$''$ and (xi)$''$, we have the following proposition:

\begin{proposition}   For a Dynkin quiver $Q$ of type $C_n$, let $\pair{\al,\be}$ be a \prq such that $\dg_Q(\al,\be)=2$.
Then there exists a unique \prq $\up$ such that
$$   \head_Q(\pair{\al,\be})  \prec^\ttb_Q \up \prec_Q^\ttb \pair{\al,\be}.$$
\end{proposition}

\begin{proof} [Proof of Theorem~\ref{thm: distance C_n}]
By Theorem~\ref{thm: soc B}, we have
$$
 \sum_{k \in \Z_{\ge 0} } {\ttO_k^{B_n}(i,j)} t^{k-1} =  \sum_{k \in \Z_{\ge 0} } {\ttO_k^{C_n}(i,j)} t^{k-1}
$$
for any $1 \le i,j \le n$.
Since, when  $\max(i,j)=n$, we have $\max(\sfd_i,\sfd_j)=2$ and
$$
 \sum_{s=1}^{\min(i,j)} (t^{|i-j|+2s-1}+t^{2n-i-j+2s-1})
= 2 \sum_{s=1}^{i}t^{n-i+2s-1},
$$
our assertion follows from the definition of degree polynomial
\[  \de_{i,j}^\blacktriangle(t) \seteq \sum_{k \in \Z_{\ge 0} } \max(\sfd_i,\sfd_j) {\ttO_k^{\blacktriangle}(i,j)} t^{k-1}. \qedhere \]
\end{proof}

\subsection{$F_4$-case} In this case, we have only finite many Dynkin quivers. Recall that $\sfd_1=\sfd_2=2$ and $\sfd_3=\sfd_4=1$. Then one can check the results in this subsections and we exhibit several examples instead of giving proofs.

\smallskip

For  a Dynkin quiver $Q=(\bDynkin_{F_4}, \xi)$ with
$$\xymatrix@R=0.5ex@C=6ex{    *{\circled{$\circ$}}<3pt> \ar@{->}[r]^<{ _{\underline{4}} \ \  }_<{1} & *{\circled{$\circ$}}<3pt> \ar@{->}[r]^<{  _{\underline{3}} \ \  }_<{2}
&*{\circ}<3pt> \ar@{->}[r]^>{ \ \  _{\underline{1}}}_<{3 \ \ } &*{\circ}<3pt> \ar@{}[l]^<{ \ \ 4}_>{    _{\underline{2}} \ \ } },$$
the AR-quiver $\Gamma_Q$ can be described as follows:
\begin{align}\label{eq: F_4Q}
\Gamma^{F_4}_{Q}=  \hspace{-2ex} \raisebox{3.15em}{ \scalebox{0.63}{\xymatrix@!C=3.3ex@R=2ex{
(i\setminus p) & -9 & -8 & -7 & -6  & -5 & -4 & -3 & -2 & -1 & 0 & 1 & 2 & 3 & 4\\
1&&&& \sprt{1,0,-1,0}\ar[dr] && \sprt{0,0,1,1} \ar@{->}[dr]   &&   \sprt{1,0,0,-1} \ar@{->}[dr]  &&\sprt{0,1,0,1}\ar@{->}[dr]&& \sprt{0,0,1,-1}\ar@{->}[dr]&&  \sprt{0,1,-1,0} \\
2&&& \sprt{1,-1,0,0} \ar@{->}[dr]\ar@{->}[ur]   &&  \sprt{1,0,0,1} \ar@{->}[dr]\ar@{->}[ur]     &&  \sprt{1,0,1,0} \ar@{->}[dr]\ar@{->}[ur]&& \sprt{1,1,0,0}\ar@{->}[dr]\ar@{->}[ur]&& \sprt{0,1,1,0}\ar@{->}[dr]\ar@{->}[ur]&& \sprt{0,1,0,-1}  \ar@{->}[ul]\ar@{->}[ur] \\ 
3&&\sprt{\frac{1}{2},-\frac{1}{2},-\frac{1}{2},\frac{1}{2}} \ar@{->}[dr]\ar@{=>}[ur]   &&\sprt{\frac{1}{2},-\frac{1}{2},\frac{1}{2},\frac{1}{2}}\ar@{->}[dr]\ar@{=>}[ur]   &&  \sprt{1,0,0,0}\ar@{->}[dr]\ar@{=>}[ur] && \sprt{\frac{1}{2},\frac{1}{2},\frac{1}{2},\frac{1}{2}}\ar@{->}[dr]\ar@{=>}[ur] &&  \sprt{\frac{1}{2},\frac{1}{2},\frac{1}{2},-\frac{1}{2}}\ar@{->}[dr]\ar@{=>}[ur]   &&\sprt{0,1,0,0} \ar@{=>}[ur]\\
4& \sprt{\frac{1}{2},-\frac{1}{2},-\frac{1}{2},-\frac{1}{2}}   \ar@{->}[ur]  &&  \sprt{0,0,0,1}\ar@{->}[ur]      &&  \sprt{\frac{1}{2},-\frac{1}{2},\frac{1}{2},-\frac{1}{2}}\ar@{->}[ur]  && \sprt{\frac{1}{2},\frac{1}{2},-\frac{1}{2},\frac{1}{2}} \ar@{->}[ur]  &&\sprt{0,0,1,0}\ar@{->}[ur]    & & \sprt{\frac{1}{2},\frac{1}{2},-\frac{1}{2},-\frac{1}{2}} \ar@{->}[ur]
}}}
\end{align}
Here we use notation  the notation $\sprt{a,b,c,d}\seteq a\ve_1+b\ve_2+c\ve_3+d\ve_4$, $\ve_i =\sqrt{2}\ep_i$ and
$$  \al_1 = \sprt{0,1,-1,0}, \quad \al_2 = \sprt{0,0,1,-1}, \quad \al_3 = \sprt{0,0,0,1} \quad \text{ and } \quad \al_4 = \sprt{1/2,-1/2,-1/2,-1/2}.$$

\begin{theorem} For  a Dynkin quiver $Q$ of type $F_4$ and any \prq $\pair{\al,\be} \in (\Phi^+_{F_4})^2$, $\head_Q(\pair{\al,\be})$ is well-defined. In particular, if $\al+\be=\ga$,
then $\head_Q(\pair{\al,\be})=\ga$.
\end{theorem}

\begin{remark} \label{rmk: F4 do not hold}
In $F_4$-case, $\de_{i,j}(t)+\delta_{i,j}t^{\sfh-1}$ does not coincide with $\tde_{i,j}(t)$ any more. For instance,
$$   \tde_{2,3}[9]=4  \quad \text{ while } \quad   \de_{2,3}[9]=6.$$
Here the fact that $\de_{2,3}[9]=6$ can be computed by using
$\Gamma_Q$ in~\eqref{eq: F_4Q} as follows: Let $\al=\sprt{0,1,0,-1}$
and $\be=\sprt{\frac{1}{2},-\frac{1}{2},\frac{1}{2},\frac{1}{2}}$.
Then $\ga=\al+\be \in \Phi^+$ and
\begin{align}             \label{eq: F_4 de not tde}
\ga            \prec_Q^\ttb
\pair{\sprt{\frac{1}{2},\frac{1}{2},-\frac{1}{2},-\frac{1}{2}},\sprt{0,0,1,0}  }  \prec_Q^\ttb
\pair{\sprt{0,1,0,0},\sprt{\frac{1}{2},-\frac{1}{2},\frac{1}{2},-\frac{1}{2}}  }  \prec_Q^\ttb \pair{\al,\be}.
\end{align}
\end{remark}

\subsection{$G_2$-case} \label{subsec: G2 degree}
In this case, we have only finite many Dynkin quivers. Recall that $\sfd_1=1$ and $\sfd_2=3$. Then one can check the results in this subsections and we exhibit several examples instead of giving proofs.

For  a Dynkin quiver $Q=(\bDynkin_{G_2}, \xi)$ with
$$ \   
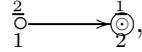
\xymatrix@R=0.5ex@C=6ex{  *{\circ}<3pt> \ar@{->}[r]_<{ 1 \ \  }^<{ _{\underline{2}} \ \  }  & *{\circled{$\odot$}}<3pt> \ar@{-}[l]^<{ \ \ 2  }_<{  \ \ _{\underline{1}}  } },$$
we have
$$
\Gamma^{G_2}_{Q}=  \raisebox{2.3em}{ \scalebox{0.9}{\xymatrix@!C=4ex@R=2ex{
(i\setminus p) &   -3 &  -2 & -1 & 0 & 1 & 2\\
1&   \ssrt{0,1,-1}  \ar@{=>}[dr]  && \ssrt{1,0,-1}  \ar@{=>}[dr]     &&\ssrt{1,-1,0} \ar@{=>}[dr] \\
2&&  \ssrt{1,1,-2}  \ar@{-}[ul]\ar@{->}[ur]  && \ssrt{2,-1,-1} \ar@{-}[ul]\ar@{->}[ur]  &&\ssrt{1,-2,1} \ar@{-}[ul]
}}}.
$$
Here  we use the orthonormal basis $\{ \ep_i \ | \ 1 \le i \le 3 \}$ of $\R^3$, the notation $\ssrt{a,b,c}\seteq a\ep_1+b\ep_2+c\ep_3$ and
$$ \al_1 = \ssrt{0,1,-1} \quad \text{ and }  \quad \al_2 = \ssrt{1,-2,1}.$$

\begin{theorem} 
For  a Dynkin quiver $Q$ of type $G_2$ and any \prq $\pair{\al,\be} \in (\Phi^+_{G_2})^2$, $\head_Q(\pair{\al,\be})$ is well-defined. In particular, if $\al+\be=\ga$,
then $\head_Q(\pair{\al,\be})=\ga$.
\end{theorem}

\begin{remark} \label{rmk: G2 do not hold}
As $F_4$-case, $\de_{i,j}(t)+\delta_{i,j}t^{\sfh-1}$ does not coincide with $\tde_{i,j}(t)$ any more in $G_2$-case either. For instance,
$$   \tde_{1,1}[4]=  2    \quad \text{ while } \quad   \de_{1,1}[4]=   1. $$
\end{remark}

Note that, for each Dynkin quiver $Q$ of type $G_2$, there exist only one \prq of positive roots $\pair{\al,\be}$ such that $\al+\be =\ga \in \Phi^+$ and $\dg_Q(\pair{\al,\be})=2$ $(i=1,2)$. In that case,
there exists a unique \prq $\up$ such that
$$\ga \prec_Q^\ttb \up  \prec_Q^\ttb  \pair{\al,\be}.$$

\section{Quiver Hecke algebras and $\rmR$-matrices} \label{Sec: quiver Hecke}

In this section, we briefly recall the definition of unipotent quantum coordinate ring and quiver Hecke algebra   $R$ associated to a finite Cartan datum $(\cm,\wl,\Pi,\cwl,\Pi^\vee)$. Then we review the definition, invariants and properties of affinizations of modules over quiver Hecke algebras which are introduced and investigated in \cite{KP18,KKOP19A}. In the final section, we will see that the $\de$-invariants  for pairs of cuspidal modules over $R$ \emph{associated with Dynkin quivers} coincide with $(\tde_{i,j}(t))_{i,j \in I}$.

\subsection{Unipotent quantum coordinate ring}
Let $q$ be an indeterminate. For $k \in \Z_{\ge 1}$ and $i \in I$, we set
$$  q_i =q^{\sfd_i}, \quad    [k]_i \seteq \dfrac{q_i^k -q_i^{-k}}{q_i-q_i^{-1}} \quad \text{ and } \quad [k]_i! \seteq \prod_{s=1}^k \; [s]_i.$$

We denote by $U_q(\g)$ the quantum group associated to a finite
Cartan datum $(\sfC,\wl,\Pi,\wl^\vee,\Pi^\vee)$, which is the
associative algebra over $\mathbb Q(q)$ generated by $e_i,f_i$ $(i
\in I)$ and $q^{h}$ ($h\in\wl^\vee$). We set $U_q^\pm(\g)$ be
the subalgebra generated by $e_i$ (resp.\ $f_i$) for $i \in I$.  Note
that $U_q(\g)$ admits the weight space decomposition:
$$  U_q(\g) = \soplus_{\be \in \rl} U_q(\g)_\be. $$

For any $i \in I$, there exists a unique $\Q(q)$-linear endomorphisms $e_ i'$ of $U_q^-(\g)$ such
that
$$   e_ i'(f_j)=\delta_{i,j},\ e_i'(xy) = e_i'(x)+   q^{(\al_i,\be)}x e'_i(y)\quad  (x \in U_q^-(\g)_\be, \  y \in U_q^-(\g)).$$

Then there exists a unique non-degenerate symmetric bilinear form $( , )_K$ on $U_q^-(\g)$ such that
$$
(  \mathbf{1},\mathbf{1})_K=1, \quad (f_iu,v)_K = (u,e_i'v)_K \text{ for $i \in I$, $u,v \in U_q^-(\g)$}.
$$

We set $\bbA=\Z[q^{\pm 1}]$ and let
$U_q^{+}(\g)_{\bbA}$ be the $\bbA$-subalgebra of
$U_q^{+}(\g)$ generated by $e_i^{(n)}\seteq e_i^n/[n]_i!$
($i \in I$, $n\in\Z_{>0}$).

Let $\Delta_\n$ be the algebra homomorphism $U_q^+(\g) \to U_q^+(\g) \tens U_q^+(\g)$ given by  $ \Delta_\n(e_i) = e_i
\tens 1 + 1 \tens e_i$,
where the algebra structure on $U_q^+(\g)
\tens U_q^+(\g)$ is defined by
$$(x_1 \tens x_2) \cdot (y_1 \tens y_2) = q^{-(\wt(x_2),\wt(y_1))}(x_1y_1 \tens x_2y_2).$$

Set
$$ A_q(\n) = \soplus_{\beta \in  \rl^-} A_q(\n)_\beta \quad \text{ where } A_q(\n)_\beta \seteq \Hom_{\Q(q)}(U^+_q(\g)_{-\beta}, \Q(q)).$$
Then  $A_q(\n)$ is an algebra with the multiplication given by
$(\psi \cdot \theta)(x)= \theta(x_{(1)})\psi(x_{(2)})$, when
$\Delta_\n(x)=x_{(1)} \tens x_{(2)}$ in  Sweedler's notation.

Let us denote by $A_\bbA(\n)$ the $\bbA$-submodule
of $A_q(\n)$
 consisting of $ \uppsi \in A_q(\n)$ such that
$ \uppsi \left( U_q^{+}(\g)_{\bbA} \right) \subset\bbA$. Then
it is an $\bbA$-subalgebra of $A_q(\n)$.
Note that Lusztig \cite{Lus90,Lus91} and Kashiwara \cite{K91} have constructed a specific $\bbA$-basis $\bfB^{{\rm up}}$ of $A_\bbA(\n)$.

\smallskip

Note that for a commutation class $\cc$,
the convex order
$\prec_\cc$ provides the \emph{dual PBW-vectors}
$$\st{\pbw(\beta)\mid\beta\in\Phi^+}$$ of
$A_q(\n)$  which satisfies the following conditions.

\begin{theorem} [\cite{BKM12}] \label{thm: minimal pair dual pbw}
The dual PBW vector $\pbw(\al_i)$ $(i\in I)$ is the element such that
$\bl f_j,\pbw(\al_i)\br_K=\delta_{i,j}$.
For a $\cc$-minimal \pr $\pair{\al,\be}$ for
$\ga \in \Phi^+ \setminus \Pi$, we have
\begin{align}\label{eq: BKMc}
 \pbw(\al)\pbw(\be) - q^{-(\al,\be)}\pbw(\be)\pbw(\al)
= q^{-p_{\be,\al}}(1-q^{2(p_{\be,\al}-(\al,\be))} )\pbw(\ga).
\end{align}
\end{theorem}

For a dominant weight $\la \in \wl^+$, let $V(\la)$ be the
irreducible highest weight $U_q(\g)$-module with highest weight
vector $u_\la$ of weight $\la$. Let $( \ , \ )_\la$ be the
non-degenerate symmetric bilinear form on $V(\la)$ such that
$(u_\la,u_\la)_\la=1$ and $(xu,v)_\la=(u,\upvarphi(x)v)_\la$ for $u,v
\in V(\la)$ and $x\in U_q(\g)$, where $\upvarphi$ is  the algebra
antiautomorphism on $U_q(\g)$  defined by $\upvarphi(e_i)=f_i$,
$\upvarphi(f_i)=e_i$  and $\upvarphi(q^h)=q^h$. For each $\mu, \zeta \in
\sfW_\g \la$, the \emph{unipotent quantum minor} $D(\mu, \zeta)$ is
an element in $A_q(\n)$ given by
 $D(\mu,\zeta)(x)=(x u_\mu, u_\zeta)_\la$
for $x\in U_q^+(\g)$, where $u_\mu$ and $u_{\zeta}$ are the
extremal weight vectors in $V(\la)$ of weight $\mu$ and $\zeta$,
respectively. Then we have $D(\mu,\zeta)\in \mathbf B^{{\rm up}}
\sqcup\{0\}$.
  Note that $D(\mu,\zeta)\in  \mathbf B^{{\rm up}}$ if and only if
$\mu\preceq \zeta$.
Here $\mu\preceq \zeta$ if there are $w,w'\in \sfW_\g$
such that $w'\preceq w$ and $\mu=w\la$, $\zeta=w'\la$.

\smallskip

For a reduced expression $\uw_0=s_{i_1}s_{i_2}\cdots s_{i_\ell}$ of the longest element
$w_0 \in \sfW_\g$, define
$\uw_{\le k}\seteq s_{i_1}\cdots s_{i_k}$ and
$ \la_k  \seteq \uw_{\le k}\varpi_{i_k}$ for $1 \le k \le l$.
Note that $\la_{k^-} =\uw_{\le k-1}\varpi_{i_k}$ if $k^- >0$. Here
$$k^-\seteq \max \left( \{ 1 \le s < k  \mid
i_s=i_k\}\sqcup \{0\} \right).
$$
For $0 \le t \le s \le \ell$, we set
$$
D_{\uw_0}(s,t) \seteq \begin{cases}
D( \uw_{\le s}\varpi_{i_s},\ \uw_{\le t}\varpi_{i_t}) & \text{ if } 1 \le  t \le s\le \ell \text{ and } i_s=i_t, \\
D( \uw_{\le s}\varpi_{i_s},\ \varpi_{i_s}) & \text{ if } 0 = t < s \le \ell, \\
\mathbf{1}  & \text{ if } 0 = t =s.
\end{cases}
$$
Then $D_{\uw_0}(s,t)$ belongs to $\mathbf B^{{\rm up}}$ and
\begin{align} \label{eq: D[s]}
  \pbw[{[\uw_0]}](\be^{\uw_0}_s) = D_{\uw_0}(s,s^-)  \qquad \text{ for } \ 1\le s \le \ell.
\end{align}

\subsection{Quiver Hecke algebras} Let $\bfk$ be a field. For $i,j \in I$, we choose polynomials $\calQ_{i,j}(u,v) \in \bfk[u,v]$ such that  $\calQ_{i,j}(u,v)=\calQ_{j,i}(v,u)$, which is of the form
$$
\calQ_{i,j}(u,v) = \bc
\sum_{p(\al_i,\al_i)+q(\al_j,\al_j)=-2(\al_i,\al_j)} t_{i,j;p,q}u^pv^p & \text{ if } i \ne j, \\
0  & \text{ if }  i =j,
\ec
$$
where $ t_{i,j;-c_{i,j},0} \in \bfk^\times$

For $\be \in \rl^+$ with $\het(\be)=m$, we set
$$ I^\be \seteq \Bigl\{ \nu =( \nu_1,\ldots, \nu_m) \in I^m \ \bigm| \ \sum_{k=1}^m \al_{ \nu_k} =\be   \Bigr\},$$
on which the symmetric group $\mathfrak{S}_m = \lan r_k \mid k =1,\ldots,n-1 \ran$ acts by place permutations.

\begin{definition}
For $\be \in \rl^+$, the \emph{quiver Hecke algebra} $R(\be)$ associated with $\cm$ and $(\calQ_{i,j}(u,v))_{i,j \in I}$ is the $\bfk$-algebra generated by
$$  \{ e( \nu) \ | \  \nu \in I^\be \}, \ \{ x_k \ | \ 1 \le  k \le m \}, \ \{ \tau_l \ | \ 1 \le l <m \} $$
satisfying the following defining relations:
\begin{align*}
& e( \nu)e( \nu') =  \delta_{\nu,\nu'}, \ \sum_{ \nu \in I^\beta} e( \nu)=1, \ x_ke(\nu)=e(\nu)x_k, x_kx_l=x_lx_k , \allowdisplaybreaks\\
&\tau_l e( \nu) = e\bl r_l( \nu)\br\tau_l, \ \tau_k\tau_l=\tau_l\tau_k \ \text{ if } |k-l|>1, \ \tau_k^2e( \nu) = \calQ_{ \nu_k, \nu_{k+1}}(x_k,x_{k+1}) e( \nu), \allowdisplaybreaks\\
& (\tau_kx_l-x_{r_k(l)}\tau_k)e( \nu) = \bc
-e( \nu) & \text{ if } l =k \text{ and }  \nu_k= \nu_{k+1}, \\
e( \nu) & \text{ if } l =k+1 \text{ and }  \nu_k= \nu_{k+1}, \\
0 & \text{ otherwise},
\ec \allowdisplaybreaks \\
& (\tau_{k+1}\tau_k\tau_{k+1}-\tau_k\tau_{k+1}\tau_k)e( \nu) = \bc
\overline{\calQ}_{ \nu_k, \nu_{k+1}}(x_k,x_{k+1},x_{k+2})e( \nu) & \text{ if }  \nu_k= \nu_{k+2}, \\
0 & \text{ otherwise},
\ec
\end{align*}
where
$$
\overline{\calQ}_{i,j}(u,v,w) \seteq \dfrac{\calQ_{i,j}(u,v)-\calQ_{i,j}(w,v)}{u-w} \in \bfk[u,v,w].
$$
\end{definition}
The algebra $R(\be)$ has the $\Z$-grading defined by
$$  \deg(e( \nu))=0, \quad  \deg(x_ke( \eta))=(\al_{ \eta_k},\al_{ \eta_k}), \quad \deg(\tau_le( \eta))=-(\al_{ \eta_l},\al_{ \eta_{l+1}}).$$

We denote by $R(\beta) \gmod$  the category of graded finite-dimensional $R(\beta)$-modules with degree preserving homomorphisms..
We set
$R\gmod \seteq \bigoplus_{\beta \in \rl^+} R(\beta)\gmod$.
The trivial $R(0)$-module of degree 0 is denoted by $\mathbf{1}$.
For simplicity, we write ``a module" instead of ``a graded module''.
We define the grading shift functor $q$
by $(q M)_k = M_{k-1}$ for a $ \Z$-graded module $M = \bigoplus_{k \in \Z} M_k $.
For $M, N \in R(\beta)\gmod $, $\Hom_{R(\beta)}(M,N)$ denotes the space of degree preserving module homomorphisms.
We define
\[
\HOM_{R(\beta)}( M,N ) \seteq \bigoplus_{k \in \Z} \Hom_{R(\beta)}(q^{k}M, N),
\]
and set $ \deg(f) \seteq k$ for $f \in \Hom_{R(\beta)}(q^{k}M, N)$.
We sometimes write $R$ for $R(\beta)$ in $\HOM_{R(\beta)}( M,N )$ for simplicity.

For $M \in R(\beta)\gmod$, we set $M^\star \seteq  \HOM_{R(\be)}(M, R(\be))$ with the $R(\beta)$-action given by
$$
(r \cdot f) (u) \seteq  f(\psi(r)u), \quad \text{for  $f\in M^\star$, $r \in R(\beta)$ and $u\in M$,}
$$
where $\psi$ is the antiautomorphism of $R(\beta)$ which fixes the generators.
We say that $M$ is \emph{self-dual} if $M \simeq M^\star$. For an $R(\beta)$-module $M$, we set $\wt(M) \seteq -\be$.

\smallskip
\emph{We sometimes ignore grading shifts in the rest of this paper.}

\smallskip
Let
$
e(\beta, \beta') \seteq \sum_{\nu \in I^\beta, \nu' \in I^{\beta'}} e(\nu, \nu'),
$
where $e(\nu, \nu')$ is the idempotent corresponding to the concatenation
$\nu\ast\nu'$ of
$\nu$ and $\nu'$.
Then there is an injective ring homomorphism
$$R(\beta)\otimes R(\beta')\to e(\beta,\beta')R(\beta+\beta')e(\beta,\beta')$$
given by
$e(\nu)\otimes e(\nu')\mapsto e(\nu,\nu')$,
$x_ke(\beta)\otimes 1\mapsto x_ke(\beta,\beta')$,
$1\otimes x_ke(\beta')\mapsto x_{k+\het(\beta)}e(\beta,\beta')$,
$\tau_ke(\beta)\otimes 1\mapsto \tau_ke(\beta,\beta')$ and
$1\otimes \tau_ke(\beta')\mapsto \tau_{k+\het(\beta)}e(\beta,\beta')$.

For $R(\beta)$-modules $M$ and $N$, we set
$$
M \conv N \seteq R(\beta+\beta') e(\beta, \beta') \otimes_{R(\beta) \otimes R(\beta')} (M \otimes N).
$$

We denote by $M \hconv N$ the head of $M \conv N$ and by $M \sconv N$ the socle of $M \conv N$.
We say that simple $R$-modules $M$ and $N$ \emph{strongly commute} if $M \conv N$ is simple.  A simple $R$-module
$L$ is \emph{real} if $L$ strongly commutes with itself. Note that if $M$ and $N$ strongly commute, then $M\conv N \simeq N \conv M$ up to a grading shift.

We denote by $K(R\gmod)$ the Grothendieck ring of $R\gmod$.

\begin{theorem} [{\cite{KL09, KL11, R08}}] \label{Thm: categorification}
There exists an isomorphism of $\bbA$-bialgebras
\begin{align}\label{eq: upomega}
\Upomega \colon K(R\gmod) \isoto  A_{\bbA}(\n),
\end{align}
which preserves weights.
\end{theorem}

\begin{proposition}[{\cite[Proposition 4.1]{KKOP18}}] \label{prop: cat of D}
For $\varpi \in\wl^+$ and $\mu,\zeta \in W\varpi$ with $\mu\preceq \zeta$, there exists a self-dual real simple $R$-module $\sfM(\mu,\zeta)$ such that
$$  \Upomega([\sfM(\mu,\zeta)]) = D(\mu,\zeta).$$
\end{proposition}

We call $\sfM(\mu,\zeta)$ the \emph{determinantial module} associated to $D(\mu,\zeta)$. Note that every
determinantial module is real.

\subsection{Affinization and $\rmR$-matrices} In this subsection, we recall the notion of affinization and $\rmR$-matrices for quiver Hecke algebras
which were mainly investigated in~\cite{KKK18A,KKKO18,KP18}.

\smallskip

For $\be \in \rl^+$ with $\het(\be)=m$ and $i \in I$, let
\begin{align} \label{eq: pibe}
\pibe = \sum_{ \eta\,  \in I^\be} \left( \prod_{a \in [1,m]; \  \eta_a  = i} x_a\right)e(\eta) \in R(\be),
\end{align}
which belongs to the center $Z(R(\be))$ of $R(\be)$. 

\begin{definition} \label{Def: aff}
Let $M$ be a simple $R(\beta)$-module. An \emph{affinization} of $M$ with degree $t_{\Ma}$ is an $R(\beta)$-module $\Ma$ with an endomorphism $z_{\Ma}$ of $\Ma$ with degree $t_{\Ma} \in \Z_{>0}$
and an isomorphism $\Ma / z_{\Ma} \Ma \simeq M$ such that
\bnum
\item $\Ma$ is a finitely generated free module over the polynomial ring
  $\cor[z_{\Ma}]$,
\item $\pibe \Ma \ne 0$ for all $i\in I$.
\ee
\end{definition}
Note that every affinization is \emph{essentially even}, i.e., $t_{\Ma} \in 2\Z_{>0}$ (\cite[Proposition 2.5]{KP18}).
Thus, from now on, we assume that every affinization has an even degree.

\begin{definition}\label{def: affreal}
We say that a simple $R$-module $M$ is \emph{affreal} if $M$ is real and admits an affinization.
If an affinization has degree $t \in 2\Z_{>0}$, we say that $M$ is \emph{$t$-affreal}.
\end{definition}

\begin{theorem} [{\cite[Theorem 3.26]{KKOP19A}}]
For $\varpi \in\wl^+$ and $\mu,\zeta \in W\varpi$ with $\mu\preceq \zeta$, the determinantial module $\sfM(\mu,\zeta)$
admits an affinization $\Ma(\mu,\zeta)$. When $\varpi=\varpi_i$, $\sfM(\mu,\zeta)$ admits an affinization of degree $(\al_i,\al_i)=2\sfd_i$.
\end{theorem}

Let $\beta \in \rl^+$ with $m =  \het(\beta)$. For  $k=1, \ldots, m-1$ and $ \nu  \in I^\beta$, the \emph{intertwiner} $\varphi_k$ is defined by
$$
\varphi_k e(\nu) =
\bc
 (\tau_k x_k - x_k \tau_k) e(\nu)
= (x_{k+1}\tau_k - \tau_kx_{k+1}) e(\nu) & \text{ if } \nu_k = \nu_{k+1},  \\
 \tau_k e(\nu) & \text{ otherwise.}
\ec
$$

\begin{lemma} [{\cite[Lemma 1.5]{KKK18A}}] \label{Lem: intertwiners}
Let $\sym_m$ be the symmetric group of degree $m\in\Z_{>0}$.
\begin{enumerate}[\rm (i)]
\item $\varphi_k^2 e( \nu ) = ( Q_{\nu_k, \nu_{k+1}} (x_k, x_{k+1} )+ \delta_{\nu_k, \nu_{k+1}} )\, e(\nu)$.
\item $\{  \varphi_k \}_{k=1, \ldots, m-1}$ satisfies the braid relation.
\item For a reduced expression $w = s_{i_1} \cdots s_{i_t} \in \mathfrak{S}_m$, we set $\varphi_w \seteq  \varphi_{i_1} \cdots \varphi_{i_t} $. Then
$\varphi_w$ does not depend on the choice of reduced expression of $w$.
\item For $w \in \mathfrak{S}_m$ and $1 \le k \le m$, we have $\varphi_w x_k = x_{w(k)} \varphi_w$.
\item For $w \in \mathfrak{S}_m$ and $1 \le k < m$, if $w(k+1)=w(k)+1$, then $\varphi_w \tau_k = \tau_{w(k)} \varphi_w$.
\end{enumerate}
\end{lemma}

For $m,n \in \Z_{\ge 0}$, we denote by $w[m,n]$ the element of $\mathfrak{S}_{m+n}$ defined by
$$
w[m,n](k) \seteq
\left\{
\begin{array}{ll}
 k+n & \text{ if } 1 \le k \le m,  \\
 k-m & \text{ if } m < k \le m+n.
\end{array}
\right.
$$
Let $\beta, \gamma \in \rl^+$ and set $m\seteq  \het(\beta)$ and $n\seteq \het(\gamma)$.
For $M \in R(\beta)\Mod$ and $N \in R(\gamma)\Mod$, the $R(\beta)\otimes R(\gamma)$-linear map $M \otimes N \rightarrow N \conv M$ defined by $u \otimes v \mapsto \varphi_{w[n,m]} (v \otimes u)$
can be extended to an $R(\beta+\gamma)$-module homomorphism (up to a grading shift)
$$
\mathrm{R}_{M,N}\colon  M\conv N \longrightarrow N \conv M.
$$

Let $\Ma$ be an affinization of a simple $R$-module $M$, and  let $N$ be a non-zero $R$-module. We define a homomorphism (up to a grading shift)
$$
\Rnorm_{\Ma, N} \seteq  \zm^{-s} \mathrm{R}_{\Ma, N}\colon  \Ma \conv N \longrightarrow N \conv \Ma,
$$
where $s$ is the largest integer such that $\mathrm{R}_{\Ma, N}(\Ma \conv N) \subset \zm^s (N \conv \Ma)$.
We define
$$
\Rr^\lt_{M,N} \colon M \conv N \longrightarrow N \conv M
$$
to be the homomorphism (up to a grading shift) induced from $\Rnorm_{\Ma, N}$ by specializing at $\zm=0$. By the definition, $\Rr^\lt_{M,N}$ never vanishes.

Similarly, for a simple module $N$ admitting an affinization $\Na$, we can define
$$  \Rr^\rt_{M,N} \colon M \conv N \longrightarrow N \conv M.$$

\begin{proposition}%
[{\cite{KKKO15}, \cite[Proposition 3.13]{KKOP19A}}]\label{prop: l=r}
Let $M$ and $N$ be simple modules.
We assume that one of them is affreal.
Then we have
$$\HOM_R(M \conv N,N \conv M)=\cor\, \Rr,$$
where $\Rr= \Rr^\lt_{M,N}$ or $\Rr^\rt_{M,N}$.
\end{proposition}

By the above proposition, we can write the homomorphism as  $\Rr_{M,N}$, called the \emph{$\rmR$-matrix},  without superscript.
Now we define
\begin{align*}
\La(M,N) &\seteq  \deg (\Rr_{M,N}) , \\
\tLa(M,N) &\seteq   \frac{1}{2} \big( \La(M,N) + (\wt(M), \wt(N)) \big) , \\
\de(M,N) &\seteq  \frac{1}{2} \big( \La(M,N) + \La(N,M)\big).
\end{align*}
for simple modules $M$, $N$ such that one of them   admits an
affinization.

\begin{proposition}%
[{\cite[Theorem 3.2]{KKKO15}, \cite[Proposition 2.10]{KP18}, \cite[Proposition 3.13, Lemma 3.17]{KKOP19A}}] \label{prop: simple head}
Let $M$ and $N$ be simple $R$-modules such that one of $M$ and $N$ is affreal.
Then we have
\bnum
\item   $M \conv N$ has simple socles and simple heads,
\item ${\rm Im}(\Rr_{M,N})$ is isomorphic to $M \hconv N$ and $N\sconv M$,
\item $M \hconv N$ and $M \sconv N$ appear once in the composition series of $M \conv N$, respectively.\label{it:once}
\ee
\end{proposition}

Now,
let us collect properties of $\Lambda(M,N)$, $\tLa(M,N)$ and $\de(M,N)$
for simple $R$-modules $M$ and $N$ such that one of $M$ and $N$ is affreal.
They are proved in
\cite{KKKO15,KKKO18,KKOP18,KKOP19A} or can be proved by using the arguments in those papers:

\begin{lemma}[{\cite[Lemma 3.11]{KKOP19A}}] \label{lem: d/2}
Let $M$ and $N$ be simple $R$-modules.
Assume that one of them admit affinization of degree $t$.
Then we have $$
\tLa(M,N) , \   \de(M,N) \in \frac{\;t\;}{2}\; \Z_{\ge0}.
$$
\end{lemma}

 By this lemma,  for $t_1$-affreal $M$ and $t_2$-affreal $N$, we have
\begin{align*} 
\tLa(M,N) , \   \de(M,N) \in \frac{\; {\rm lcm}(t_1,t_2)\;}{2}\; \Z_{\ge0}.
\end{align*}
Here ${\rm lcm}(a,b)$ denotes the least common multiple of $a,b \in \Z \setminus\{ 0 \}$.

\begin{proposition}[{\cite[Proposition 3.2.17]{KKKO18}}] \label{prop: length 2}  Let $M_i$ be $t_i$-affreal $(i=1,2)$.
 Assume that $$\de(M_1,M_2)=\dfrac{ \max(t_1,t_2)}{2}.$$ Then we have an exact sequence
$$ 0 \to M_1 \sconv M_2 \to M_1 \conv M_2 \to M_1 \hconv M_2 \to 0.$$
 \end{proposition}

\begin{lemma} [{\cite[Lemma 3.1.4]{KKKO18}}] \label{lem: tLa}
Let $M$ and $N$ be self-dual simple modules. If one of them is affreal, then
$$  q^{\tLa(M,N)} M \hconv N \text{ is a self-dual simple module}.  $$
\end{lemma}

\medskip

\begin{proposition}[{\cite[Corollary 3.11]{KKKO15}}] \label{prop: not vanishing}
Let $M_k$ be $R$-modules $(k=1,2,3)$
and assume that $M_2$ is simple.
If $\varphi_1\cl L \to M_1 \conv M_2$ and
$\varphi_2\cl M_2 \conv M_3 \to L'$ are non-zero homomorphisms, then the composition
\[
L\conv M_3\To[\varphi_1\circ M_3] M_1\conv M_2\conv M_3\To[M_1\circ\varphi_2]
M_1\conv L'
\]
does not vanish. Similarly,
if $\psi_1\cl L \to M_2 \conv M_3$ and
$\psi_2\cl M_1 \conv M_2 \to L'$ are non-zero homomorphisms, then
the composition
\[
M_1 \conv L \To[M_1\circ \varphi_1] M_1\conv M_2\conv M_3 \To[\varphi_2 \circ M_3]
L '\conv M_3.
\]
does not vanish.
\end{proposition}

The following proposition was incorrectly stated in \cite[Proposition 5.11]{Oh18}, and we present here its correct statement and its proof.

\begin{proposition} [{\cite[Proposition 5.11]{Oh18}}] \label{prop: inj}
Let $M_i$ and $N_i$ be simple $R$-modules such that one of them is affreal
and $M_i \conv N_i$ has composition length $2$
for $i=1,2$. Let $M$ and $N$ be simple $R$-modules
such that one of them is affreal.
Let $\varphi_i \colon M_i \conv N_i \to M \conv N$ $(i=1,2)$
be a non-zero homomorphisms.
We assume that
$$M_1 \sconv N_1 \simeq  M_2 \sconv N_2
\qtq M_1 \hconv N_1 \not\simeq  M_2 \hconv N_2.$$
Then we have
$$M_1 \sconv N_1 \simeq M_2 \sconv N_2  \simeq  M \sconv N  \quad \text{ and } \quad \varphi_i \colon M_i \conv N_i \monoto M \conv N.$$
\end{proposition}

\Proof
If $\vphi_i$ is not injective, then
$\Im(\vphi_i)\simeq\ M_i\hconv N_i\simeq M\sconv N$.
If $\vphi_i$ is injective, then
$M_i\sconv N_i\simeq M\sconv N$.

\snoi
(i)\ If $\vphi_1$ and $\vphi_2$ are not injective, then
$M_1 \hconv N_1 \simeq  M_2 \hconv N_2$, which is a contradiction.

\snoi
(ii) Assume that one of $\vphi_1$ and $\vphi_2$ (say $\vphi_1$) is injective and the other is
not injective. Then we have
$M\sconv N\simeq M_1\sconv N_1\simeq M_2\hconv N_2$.
Hence $ M_2\hconv N_2\simeq M_1\sconv N_1\simeq M_2\sconv N_2$,
which contradicts Proposition~\ref{prop: simple head}\;\eqref{it:once}.

\snoi(iii)\ Hence $\vphi_1$ and $\vphi_2$ are injective,
and $M\sconv N\simeq M_1\sconv N_1\simeq M_2\sconv N_2$.
\QED

\begin{definition}[{\cite[Definition 2.5]{KK19}}]
Let $L_1, L_2, \ldots, L_r$ be   affreal   modules.
The sequence $(L_1, \ldots, L_r)$ is called a \emph{normal sequence} if the composition of the R-matrices
\begin{align*}
\Rr_{L_1, \ldots, L_r} \seteq&  \prod_{1 \le i < j \le r} \Rr_{L_i, L_j}\allowdisplaybreaks \\
 = & (\Rr_{L_{r-1, L_r}}) \circ \cdots \circ ( \Rr_{L_2, L_r}  \circ \cdots \circ \Rr_{L_2, L_3} )  \circ ( \Rr_{L_1, L_r}  \circ \cdots \circ \Rr_{L_1, L_2} ) \allowdisplaybreaks\\
& \colon q\ms{4mu}\raisebox{1.5ex}{$\scriptstyle\sum_{1\le i  <  k \le r} \La(L_i,L_k)$} L_1 \conv L_2 \conv  \cdots \conv L_r \longrightarrow L_r \conv \cdots \conv L_2 \conv L_1.
\end{align*}
does not vanishes.
\end{definition}

The following lemmas are proved in \cite{KK19,KKOP21C} when $\g$ is
of symmetric type. However, by using the same arguments, we can
prove when $\g$ is of symmetrizable type.

\begin{lemma} [{\cite[Lemma 2.6]{KK19}}]\label{Lem: normal head socle}
Let $(L_1, \ldots, L_r)$ be a normal sequence of affreal $R$-modules.
Then ${\rm Im}(\Rr_{L_1,\ldots, L_r})$ is simple
and it coincides with the head of $L_1\conv\cdots \conv L_r$
and also with the socle of $L_r\conv\cdots \conv L_1$.
\end{lemma}

\begin{lemma}[{\cite[Lemma 2.8, Lemma 2.10]{KK19}}] \label{lem: normality} \hfill
\bnum
\item A sequence $(L_1 , \ldots , L_r)$ of affreal $R$-modules is a normal
sequence if and only if $(L_1 , \ldots , L_{r-1})$ is a normal sequence and
$$   \La(\hd(L_1 , \ldots , L_{r-1}),L_r ) = \sum_{1 \le j \le r-1}\La(L_j,L_r).$$
\item A sequence $(L_1 , \ldots , L_r)$ of affreal $R$-modules is a normal
sequence if and only if $(L_{ 2}, \ldots , L_{r})$ is a normal sequence and
$$   \La(L_1,\hd(L_2, \ldots , L_{r}) ) = \sum_{2 \le j \le r}\La(L_1,L_j).$$
\ee
\end{lemma}

\begin{lemma}[{\cite[Lemma 2.23]{KKOP21C}}] \label{lem: de additive}
Let $L, M, N$ be affreal $R$-modules. Then $\de(L, M \hconv N)$ = $\de(L, M) + \de(L, N)$, if and only if $(L, M, N)$ and $(M, N, L)$
are normal sequences.
\end{lemma}

\section{Cuspidal modules and $t$-quantized Cartan matrix} \label{sec: Cuspidal}
Let  $(\cm,\wl,\Pi,\cwl,\Pi^\vee)$ be an arbitrary finite Cartan datum.
In this section, we first recall the \emph{cuspidal modules}
over quiver Hecke algebras categorifying  the  dual PBW vectors. Then we
will prove that the $\de$-invariants among cuspidal modules
associated to Dynkin quivers are encoded in $(\tde_{i,j}(t))_{i,j \in I}$, and $(\de_{i,j}(t))_{i,j \in I}$ informs us of the information on the composition length of convolution product of
pair of the cuspidal modules. To prove them, we will use the combinatorial properties and statistics of AR-quivers in a crucial way.

\subsection{Cuspidal modules} \label{subsec: cuspidal}
Let $R$ be the quiver Hecke algebra of type of   $\cm$.
Then, in \cite{BKM12,McNa15,Kato12,KR09}, the following theorem is proved:

\begin{theorem} \label{thm: cuspidal}
For  any  reduced expression  $\uw_0=s_{i_1}\cdots s_{i_\ell}$ of $w_0 \in \sfW$, there exists a family of self-dual affreal simple modules
$$  \{ S_{\uw_0}(\beta) \mid\beta\in\Phi^+  \} $$
satisfying the following properties:
\ben
\item For each  $\be\in\Phi^+$,
we have $$\Upomega([S_{\uw_0}(\beta)]) = \pbw[{[\uw_0]}](\be) \quad \text{ and }\quad \wt( S_{\uw_0}(\be)) =-\be.$$
\item  For any \ex $\um=\seq{m_\beta}\in \Z_{\ge 0}^{\Phi^+}$, there exists a unique non-zero $R$-module homomorphism \ro up to grading shift\/\rfm
\begin{align*}
& \overset{\to}{\calS}_{\uw_0}(\um)\seteq S_{\uw_0}(\be_\ell)^{\conv m_{\beta_\ell}} \conv \cdots \conv  S_{\uw_0}(\beta_1)^{\conv m_{\beta_1}}  \\
& \hspace{20ex} \overset{\Rr_{\um}}{\To}
\overset{\gets}{\calS}_{\uw_0}(\um)\seteq   S_{\uw_0}(\beta_1)^{\conv m_{\beta_1}} \conv \cdots \conv  S_{\uw_0}(\beta_\ell)^{\conv m_{\beta_\ell}}
\end{align*}
where $\beta_k=\beta_k^{\uw_0}$.
Moreover, ${\rm Im}(\Rr_{\um})$ is simple and
${\rm Im}(\Rr_{\um}) \simeq \soc\left(  \overset{\gets}{\calS}_{\uw_0} ( \um )  \right) \simeq \head\left(  \overset{\to}{\calS}_{\uw_0}(\um)  \right)$ up to a grading shift.
 \item For any simple $R$-module $M$, there exists a unique \ex $\um \in \Z_{\ge 0}^{\Phi^+}$ such that
$$M \simeq    \hd\left(  \overset{\to}{\calS}_{\uw_0}(\um) \right)\qt{up to a grading shift.}$$
\ee
\end{theorem}

Note that, for each $\um=\seq{m_1,\ldots,m_\ell}\in \Z_{\ge 0}^\ell$,
$\head(  q^{s_\um} \overset{\to}{\calS}_{\uw_0}(\um)  )$ is a self-dual simple $R$-module, where
$$ s_\um  \seteq  \sum_{k=1}^\ell {\sfd_{\be^{\uw_0}_k}}\;\dfrac{ m_k (m_k-1)}{2}.$$
We denote
$$S_{\uw_0}(\um)\seteq\head(  q^{s_\um} \overset{\to}{\calS}_{\uw_0}(\um)).$$

We call $S_{\uw_0}(\be) $   the  \emph{cuspidal module} associated to $\uw_0$. It is known (see \cite[Theorem 5.8]{Oh18}) that the set $  \{ S_{\uw_0}(\be)
\mid \beta\in\Phi^+ \} $ does depend only on the commutation class $[\uw_0]$; i.e.,
\begin{align*}
S_{\uw_0}(\beta)   \simeq S_{\uw'_0}(\beta)  \quad \text{
if $\uw_0$ and $\uw'_0$ are commutation equivalent.}
\end{align*}
Thus, for any commutation class $\cc$, $S_{\cc}(\be)$ for $\be \in \Phi^+$ is well-defined.

\Lemma[{\cite[Proposition 5.7]{Oh18}}] \label{lem: non-com simple}
Let $\uw_0$ be a reduced expression of $w_0$.
If $\al$ and $\be$ are incomparable with respect to $\preceq_{[\uw_0]}$,
then
$S_{[\uw_0]}(\al)$ and $S_{[\uw_0]}(\be)$ commute.
\enlemma

By this lemma, $S_{\uw_0}(\um)$,
as well as $\overset{\gets}{\calS}_{\uw_0}(\um)$ and
$\overset{\to}{\calS}_{\uw_0}(\um)$, depends only on the commutation class $[\uw_0]$.
Hence,  $S_{[\uw_0]}(\um)$, $\overset{\gets}{\calS}_{[\uw_0]}(\um)$ and
$\overset{\to}{\calS}_{[\uw_0]}(\um)$ are well-defined.

\begin{proposition}[{\cite{BKM12,McNa15,Oh18}}] \label{prop: minimal,simple}
Let $\cc$ be a commutation class.
\bnum
\item  For any $\cc$-minimal \pr $\pair{\al,\be}$ for $\ga \in \Phi^+$, there exist exact sequences of $R$-modules
\begin{align*}
\ba{l}
0 \to S_\cc(\ga)\to S_\cc(\be) \conv S_\cc(\al)\to
S_\cc(\be)   \hconv   S_\cc(\al)\to 0,\\
0 \to S_\cc(\be)  \hconv    S_\cc(\al)\to S_\cc(\al) \conv S_\cc(\be)\to S_\cc(\ga)
\to 0
\ea
\end{align*}
and
$$ {[S_{{\cc}}(\al)]}\cdot{[S_{{\cc}}(\be)]} =  [S_{\cc}(\ga)] +    [S_{\cc}(\al) \sconv  S_{\cc}(\be)]    \qt{in $K(R\gmod)$}$$
up to grading shifts.
\item   For any \ex $\um$, we have
\begin{align}\label{eq: comp series}
[\overset{ \to  }{\calS}_{\cc}(\um)] \in [S_\cc(\um)] +
\sum_{\um' \prec^{\ttb}_{\cc} \um} \Z_{\ge 0} [S_\cc(\um')].
\end{align} up to  grading shifts.
\item   If  $\um$ is a $\cc$-simple sequence, then
$\overset{\to}{\calS}_\cc(\um)$ and $\overset{\gets}{\calS}_\cc(\um)$
are isomorphic up to a grading shift and they are simple.
\ee
\end{proposition}

\begin{proposition} [{\cite[Theorem 4.7]{BKM12}}] \label{prop: tLa}
For any commutation class $\cc$ and a \pr $\pair{\al,\be}$, we have
\begin{align} \label{eq: tla be al general}
\tLa( S_\cc(\be),S_\cc(\al)  ) =0 \qtq[and hence]\La( S_\cc(\be),S_\cc(\al)  ) = -(\al,\be).
\end{align}
In particular, if the \pr $\pair{\al,\be}$ is $\cc$-minimal  for a positive root $\ga \in \Phi^+$, we have short exact sequences \ro including grading shifts\rfm
\begin{equation}\label{eq: 2 sess}
\begin{aligned}
 & 0 \to q^{p_{\be,\al}- (\be,\al)} S_\cc(\ga) \to  S_\cc(\be) \conv S_\cc(\al)
\to S_\cc(\pair{\al,\be})  \to 0, \\
& 0 \to q^{-(\be,\al)} S_\cc(\pair{\al,\be}) \to  S_\cc(\al) \conv S_\cc(\be) \to q^{-p_{\be,\al}} S_\cc \to 0,
\end{aligned}
\end{equation}
which tells
\eqn
&&\tLa( S_\cc(\al),S_\cc(\be)  ) = p_{\be,\al},\quad
 \La( S_\cc(\al),S_\cc(\be)  ) =2 p_{\be,\al}  -(\al,\be),\qtq\\
&&\de( S_\cc(\al),S_\cc(\be) )= p_{\be,\al}  -(\al,\be).
\eneqn
\end{proposition}

\begin{proposition} [{\cite[Corollary 5.12]{Oh18}}] \label{prop: simple socle}
Let $\um^{(k)}$ $(k=1,2,3)$ be distinct sequences such that $\wt_{\cc}(\um_{k_1}^{(a)}) = \wt_{\cc}(\um^{(b)})$ $(a,b \in \{1,2,3\})$,
$\dg_{\cc}(\um^{(j)})=1$  and
$\dg_{\cc}(\um^{(3)})>1$ and $\um^{(j)} \prec_{\cc}^\ttb \um^{(3)}$ $(j=1,2)$. If there exist non-zero $R(\gamma)$-homomorphisms $($up to grading shift$)$
$$  \overset{ \to  }{\calS}_{\cc}(\um^{(j)}) \longrightarrow  \overset{ \to  }{\calS}_{\cc}(\um^{(3)})\qt{for $j=1,2$,}$$
and
$\overset{ \to  }{\calS}_{\cc}(\um^{(k)})$ $(k=1,2,3)$ have the common simple head,
then we have
$$ 
 \overset{ \to  }{\calS}_{\cc}(\um^{(j)})  \hookrightarrow   \overset{ \to  }{\calS}_{\cc}(\um^{(3)})  \qquad \text{
for $j=1,2$.} $$
\end{proposition}

For a $\rmQ$-datum $Q=(\Dynkin,\xi)$
and $(i,p)\in(\Gamma_Q)_0$, we set
$$     S_Q(i;p) \seteq S_{[Q]}(\beta)\qt{where $\phi_Q(i,p)=(\be,0)$.}$$

The following is a main theorem of this paper:

\begin{theorem} \label{thm: determimant}
Let $Q = (\Dynkin,  \xi)$ be a Dynkin quiver of an arbitrary type. Then we have
\eq
&&\de( S_Q(i;p) , S_Q(j;s)) =  \tde^\triangle_{i,j}[\;|p-s|\;]
\qt{for any $(i,p)$, $(j,s)\in(\Gamma_Q)_0$.}
\label{eq: de ADE}
\eneq
Moreover, for a \prq $\pair{\al,\be}$ with $\phi_Q(i,p)=(\al,0)$ and $\phi_Q(j,s)=(\be,0)$, we have the following statements:
\bna
\item \label{it: head} $S_Q(\al) \hconv S_Q(\be) \simeq S_Q(\hd_Q(\pair{\al,\be}))$
up to  a grading shift.
\item \label{it: length2} If $ \de^\triangle_{i,j}[\;|p-s|\;]  = \max(\sfd_i,\sfd_j)$, we have $\ell(S_Q(\al)  \conv  S_Q(\be)) =2$.
\item \label{it: length>2} If $ \de^\triangle_{i,j}[\;|p-s|\;]  > \max(\sfd_i,\sfd_j)$, we have $\ell(S_Q(\al)  \conv  S_Q(\be))>2$.
\item label{it: length1} If $ \de^\triangle_{i,j}[\;|p-s|\;] =0$, then  $S_Q(\al)  \conv  S_Q(\be)$ is simple.
\ee
Here, $\ell(M)$ denotes the composition length of $M$ for  $M \in R\gmod$.
\end{theorem}

\begin{remark} \label{rmk: simple product}
In Theorem~\ref{thm: determimant}~\eqref{it: head}, $S_Q(\hd_Q(\pair{\al,\be}))$ is isomorphic to a convolution product of mutually commuting cuspidal modules, since $\hd_Q(\pair{\al,\be}$ is $[Q]$-simple.
For $\g$ of simply-laced type, we refer the reader to \cite[Theorem 5.18]{Oh18}.
\end{remark}

The simply-laced case of this theorem is obtained in
\cite[Theorem 4.12]{KKOP21C}, \cite[Theorem 1.2]{Fuj19}, \cite{Oh18}.
In the next subsection, we will give a proof of this theorem  in BCFG case.

\smallskip

\begin{corollary}  Recall~\eqref{eq: double pole}.
For $\g$ of type $B_n$, $C_n$ or $D_{n+1}$ and $(i;p),(j;s)\in(\Gamma_Q)_0$,
the $R$-module $S_Q(i;p) \conv S_Q(j;s)$ is of composition length $>2$ if and only if  %
\begin{align*}  
2 \le i,j \le n-1,  \ i+j \ge n+1, \   \sfh+2 -i-j \le  |p-s|   \le  i+j   \text{ and }    |p-s|  \equiv_2 i+j.
\end{align*}
\end{corollary}

The corollary below for ADE-case was also proved in  \cite[Lemma 9.12]{FHOO}  (see also \cite{KKOP21C}).

\begin{corollary} \label{cor: dual-phen}
For $\phi_Q(\al_a,0)=(i,p)$ and $\phi_Q(\al_b,0)=(j,s)$ with $a\ne b$, we have
\begin{align}\label{eq: dual-phenomenon}
\tde_{i^*,j}[\sfh-|p-s|] =\tde_{i,j^*}[\sfh-|p-s|] =0.
\end{align}
\end{corollary}
 \begin{proof}
Note that
\begin{eqnarray} &&
\parbox{80ex}{
\bnA
\item \label{it: D n inv} $\tde^{D_{n+1}}_{i,n}(t) =\tde^{D_{n+1}}_{i,n+1}(t)$ for $i<n$ and
\item \label{it: B C cor} $\tde^{C_n}_{i,j}(t) =  \kappa_{i,j} \tde^{B_n}_{i,j}(t)$, where
$\kappa_{i,j} = \bc
1/2 & \text{ if } i,j<n, \\
1 & \text{ if }  \min(i,j)<\max(i,j)=n,\\
2 & \text{ if }  i=j =n.
\ec
$
\ee
}\label{eq: simple obervation}
\end{eqnarray}

\noindent
(1) Let us assume that  $\max(a,b)<n$ and $\Dynkin$ is of type $B_n$ or $C_n$.
In this case, we have $\max(i,j)<n$ since $\sfd_{\al_a}=\sfd_{\al_a} \ne \sfd_{\al_n}$.
By regarding $\al_a \in \Phi^+_{B_n}$, we have $\al_a=\psi^{-1}(\al_a) \in \Phi_{D_{n+1}}^+$ and $\phi_{\uQ}^{-1}(\al_a,0) = (\im,p)$ for $\im =\pi^{-1}(i)$ by Proposition~\ref{prop: surgery D to B}.   
Similar observations hold for  $\al_b = \psi^{-1}(\al_b)$.  Since~\eqref{eq: dual-phenomenon} holds for $D_{n+1}$-case, we have
$$   \dfrac{1}{2} \tde^{B_n}_{i,j}[\sfh-|p-s|]=\tde^{C_n}_{i,j}[\sfh-|p-s|] =  \tde^{D_{n+1}}_{\im,\jm}[\sfh-|p-s|] =0. $$
Here the second equation follows from Theorem~\ref{thm: tbij folding}. Thus the assertion for this case follows.

\noindent
(2)  Let us assume that  $\max(a,b)=n$ and $\Dynkin$ is of type $B_n$ or $C_n$. Without loss of generality, we set $\al_a=\al_n$. Note that $n$ is a sink or source in $Q$ and $i=n$.
By regarding $\al_n\in \Phi^+_{B_n}$, we have $\{\al_n,\al_{n+1} \}=\psi^{-1}(\al_n) \subset \Phi_{D_{n+1}}^+$ and $\{ \phi^{-1}_{\uQ}(\al_{n},0),\phi^{-1}_{\uQ}(\al_{n+1},0) \}= \{(n,p),(n+1,p)\}$. Then, as in (1)-case, we have
\[  \tde^{B_n}_{n,j}[\sfh-|p-s|]= \tde^{C_n}_{n,j}[\sfh-|p-s|]  =  \tde^{D_{n+1}}_{n,\jm}[\sfh-|p-s|] + \tde^{D_{n+1}}_{n+1,\jm}[\sfh-|p-s|] =0.  \]

\noindent
(3) The cases for $F_4$ and $G_2$ can be directly checked by considering all Dynkin quivers of those types.
\end{proof}

\subsection{BCFG cases} Throughout this subsection, we consider the convolution products of $[Q]$-cuspidal modules, when
$Q$ is a Dynkin quiver of type BCFG.

\smallskip

We set $\phi_{Q}(i,p)=(\al,0)$ and $\phi_{Q}(j,s)=(\be,0)$.

\begin{proposition} \label{prop: minimal pair ga de}
Let $Q$ be any Dynkin quiver of types $B_n$ and $C_n$.
 For a Dynkin quiver $Q$, let $\pair{\al,\be}$ be a $[Q]$-minimal \prq for $\al+\be=\ga \in \Phi^+$. Then we have
$$
\de( S_Q(\al),S_Q(\be) ) = \max(\sfd_i,\sfd_j)  =  \tde_{i,j}^\Dynkin[\;|p-s|\;].$$
\end{proposition}

\begin{proof}
The second equality follows from Theorem~\ref{thm:Main}. 

\smallskip

\noindent
(1) Consider the case when $ \sfd_\ga=2$, $\sfd_{\al}= \sfd_{\be} =1$ and $\g$ is of type $B_n$. Then we have $\res^{[Q]}(\al)=\res^{[Q]}(\be)=n$ and $\res^{[Q]}(\ga) <n$. Hence
$\ga=\ve_a+\ve_b$ and $\{ \al,\be \} =\{ \ep_a,\ep_b \}$ for $1 \le a \ne b \le n$.  Thus we have
$$   \de( S_Q(\al),S_Q(\be) )  = 1 - (\ve_a,\ve_b) =1 = \max(\sfd_\al,\sfd_\be).$$

\noindent
(2) For the remained cases of type $B_n$ or $C_n$, we need to show that
$$   \de( S_Q(\al),S_Q(\be) ) = -(\al,\be) = \max(\sfd_\al,\sfd_\be).$$
Then we have
\bna
\item $\{\al,\be\} = \{  \ve_a - \ve_b ,   \ve_b \pm \ve_c \}$ or $ \{  \ve_a - \ve_b ,   \ve_b  \}$ if $\g$ is of type $B_n$,
\item $\{\al,\be\} = \{  \ep_a - \ep_b ,   \ep_b \pm \ep_c \}$, $ \{  \ep_a - \ep_b , 2\ep_b\}$ or $ \{  \ep_a + \ep_b ,   \ep_a -\ep_b  \}$ if $\g$ is of type $C_n$,
\ee
for proper distinct $1 \le a,b,c \le n$. Then one can easily see that
$$-(\al,\be) = \max(\sfd_i,\sfd_j),$$
which completes our assertion. 
\end{proof}

Note that the arguments in Proposition~\ref{prop: minimal pair ga de} also works for $ADE$-cases.

\subsubsection{$B_n$ case} Let $Q$ be \emph{any} Dynkin quiver of type $B_n$.
 Under the assumption that
$$\text{a \prq $\pair{\al,\be}$ is neither $[Q]$-simple nor  $[Q]$-minimal  for $\al+\be =\ga \in \Phi^+$,}$$ it suffices to consider the cases
${\rm (ii)}'$, ${\rm (iv)}'$, ${\rm (v)}'$, ${\rm (vi)}'$, ${\rm (vii)}'$ in~\eqref{eq: dist al, be 1 p}, and ${\rm (x)}'$, ${\rm (xi)}'$ in~\eqref{eq: dist al be 2 p}. Recall the values $(\al,\be)$
and $p_{\be,\al}$ are recorded in Table~\ref{table: al,be B} and~\eqref{eq: p be al B}, and do not depend on the choice of $Q$.

\smallskip

Throughout this section, we set
$$ \phi_Q(i,p)=(\al,0)  \quad \text{ and } \quad   \phi_Q(j,s)=(\be,0).$$

\smallskip

In this subsection, we will prove Theorem~\ref{thm: determimant} by considering each case.

\smallskip

Consider the case  ${\rm (ii)}'$.
By Proposition~\ref{prop: dir Q cnt} and the fact that $\psi$ is additive, there exists $\eta \in \Phi^+$ such that
\bna
\item \label{it: ii a} the \prq $\pair{\eta,s_1}$ is   $[Q]$-minimal  for $\al$ and of the form ${\rm (i)}'$, or
\item \label{it: ii b} the \prq $\pair{\al,\eta}$ is a $[Q]$-minimal  for $s_1$ and of the form ${\rm (iii)}'$.
\ee

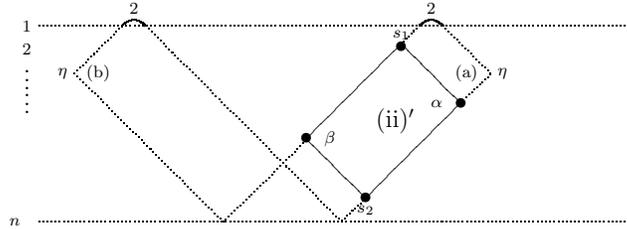
\begin{equation*}
\begin{aligned}
& \scalebox{0.79}{{\xy
(-20,0)*{}="DL";(-10,-10)*{}="DD";(0,20)*{}="DT";(10,10)*{}="DR";
"DT"+(-50,-4); "DT"+(50,-4)**\dir{.};
"DD"+(-40,-7); "DD"+(60,-7) **\dir{.};
"DT"+(-52,-4)*{\scriptstyle 1};
"DT"+(-52,-8)*{\scriptstyle 2};
"DT"+(-52,-12)*{\scriptstyle \vdots};
"DT"+(-52,-16)*{\scriptstyle \vdots};
"DT"+(-54,-37)*{\scriptstyle n};
"DL"+(15,-3); "DD"+(15,-3) **\dir{-};
"DR"+(11,-7); "DD"+(15,-3) **\dir{-};
"DT"+(11,-7); "DR"+(11,-7) **\dir{-};
"DT"+(11,-7); "DL"+(15,-3) **\dir{-};
"DT"+(14,-4); "DT"+(18,-4) **\crv{"DT"+(16,-2)};
"DT"+(16,-1)*{\scriptstyle 2};
"DT"+(11,-7); "DT"+(14,-4) **\dir{.};
"DT"+(18,-4); "DT"+(26,-12) **\dir{.};
"DT"+(26,-12); "DT"+(22,-16) **\dir{.};
"DT"+(28,-12)*{\scriptstyle \eta};
"DT"+(22,-12)*{\scriptstyle{{\rm (a)}}};
"DD"+(-22,26); "DD"+(-26,26) **\crv{"DD"+(-24,28)};
"DD"+(-24,29)*{\scriptstyle 2};
"DD"+(11,-7); "DD"+(-22,26) **\dir{.};
"DD"+(-26,26); "DD"+(-34,18) **\dir{.};
"DD"+(-36,18)*{\scriptstyle \eta};
"DD"+(-30,18)*{\scriptstyle{{\rm (b)}}};
"DD"+(11,-7); "DD"+(15,-3) **\dir{.};
"DL"+(15,-3); "DL"+(1,-17) **\dir{.};
"DL"+(-24,8); "DL"+(1,-17) **\dir{.};
"DL"+(19,-3)*{\scriptstyle \be};
"DR"+(7,-7)*{\scriptstyle \al};
"DL"+(30,0)*{{\rm(ii)'}};
"DL"+(15,-3)*{\bullet};
"DL"+(31,12.5)*{\bullet};
"DL"+(31,14.5)*{\scriptstyle s_1};
"DL"+(25,-13)*{\bullet};
"DL"+(25,-15)*{\scriptstyle  s_2 };
"DR"+(11,-7)*{\bullet};
\endxy}}
\end{aligned}
\end{equation*}

Assume ~\eqref{it: ii a} first. Then ~\eqref{eq: 2 sess} tells that there exists a homogeneous $R$-module monomorphism
$$    q^{p_{s_1,\eta} - (s_1,\eta)} S_Q(\al) \conv  S_Q(\be)  \hookrightarrow  S_Q(s_1) \conv  S_Q(\eta) \conv S_{Q}(\be).   $$
Note that $\pair{\eta,\be}$ is a $[Q]$-minimal \prq of the form ${\rm (i)}'$ for $s_2$. Thus we have
$$   S_Q(s_1) \conv  S_Q(\eta) \conv S_{Q}(\be)\twoheadrightarrow q^{-p_{\be,\eta}} S_{Q}(s_1) \conv S_{Q}(s_2)   $$
by~\eqref{eq: 2 sess}  again. Then the composition
\begin{align} \label{eq: head comp}
  q^{p_{s_1,\eta} - (s_1,\eta)}  S_Q(\al) \conv  S_Q(\be)  \hookrightarrow  S_Q(s_1) \conv  S_Q(\eta) \conv S_{Q}(\be)
 \twoheadrightarrow q^{-p_{\be,\eta}}  S_{Q}(s_1) \conv S_{Q}(s_2)
\end{align}
does not vanish by Proposition~\ref{prop: not vanishing}.
 Since $  S_{Q}(s_1) \conv S_{Q}(s_2)$ is self-dual simple by Lemma~\ref{lem: non-com simple},
we have $S_Q(\al) \hconv  S_Q(\be)  \simeq S_{Q}(s_1) \conv S_{Q}(s_2)$ up to a grading shift, and
$$
S_Q(\al) \conv  S_Q(\be)  \twoheadrightarrow q^{-p_{\be,\eta}-p_{s_1,\eta} +(s_1,\eta)}  S_{Q}(s_1) \conv S_{Q}(s_2),
$$
which tells
\begin{align*}
\tLa( S_Q(\al) ,  S_Q(\be)  ) = p_{\be,\eta}+p_{s_1,\eta} -(s_1,\eta),
\end{align*}
by Lemma~\ref{lem: tLa}. Thus~\eqref{eq: tla be al general} tells that
\begin{align*}
\de( S_Q(\al) ,  S_Q(\be)  )  = p_{\be,\eta}+p_{s_1,\eta} -(s_1,\eta)-(\al,\be).
\end{align*}
Since $ p_{\be,\eta}=p_{s_1,\eta} =0 $, $(s_1,\eta)=-2$ and $(\al,\be)=0$, we have
\begin{align*}
 \de( S_Q(\al) ,  S_Q(\be)  )  =2 = \max(\sfd_\al,\sfd_\be)=\de_{i,j}[p-s].
\end{align*}

Now let us assume ~\eqref{it: ii b}. Similarly, we have a homogeneous $R$-module homomorphism,
$$
q^{p_{\eta,\al}-(\eta,\al)} S_Q(s_2) \conv S_Q(s_1) \hookrightarrow S_Q(s_2) \conv S_Q(\eta) \conv S_Q(\al)
$$
Note that the \prq $\pair{s_2,\eta}$ is $[Q]$-minimal for $\be$ and of the form ${\rm (iii)}'$.  Thus we have
a non-zero composition of $R$-module homomorphism,
\begin{align} \label{eq: head comp2}
q^{p_{\eta,\al}-(\eta,\al)} S_Q(s_2) \conv S_Q(s_1) \hookrightarrow S_Q(s_2) \conv S_Q(\eta) \conv S_Q(\al) \twoheadrightarrow  q^{-p_{\eta,s_2}} S_Q(\beta) \conv S_Q(\al),
\end{align}
which tells $ S_Q(\beta) \sconv S_Q(\al) \simeq S_Q(s_1) \conv S_Q(s_2)$ up to a grading shift and
\begin{align*}
\de( S_Q(\al) ,  S_Q(\be)  ) & = p_{\eta,\al}+p_{\eta,s_2}-(s_2,s_1)-(\eta,\al) =2= \max(\sfd_\al,\sfd_\be)=\de_{i,j}[p-s].
\end{align*}

From~\eqref{eq: head comp} and~\eqref{eq: head comp2}, we have a non-zero $R$-module monomorphism (up to grading shifts)
$$ S_{Q}(s_2) \conv S_{Q}(s_1) \hookrightarrow S_Q(\be) \conv   S_Q(\al).$$
We can
apply the similar argument to  the length $1$ chains
of $\prec_Q^\ttb$-sequence, which are of the form   $\um \prec^\ttb_Q \um'$ and depicted in the
cases ${\rm (iv)}'$, ${\rm (vii)}'$, ${\rm (x)}'$ and ${\rm (xi)}'$,
for obtaining  non-zero $R$-module homomorphisms  (up to grading shifts)
\begin{align}\label{eq: non-zero}
 \overset{\gets}{\calS}_{[\uw_0]}(\um)  \to \overset{\gets}{\calS}_{[\uw_0]}(\um').
\end{align}
More precisely, those length $1$-chains are listed below:
$$
\bc
\pair{s_1,s_2} \prec_{Q}^\ttb \pair{\al,\be}  & \text{  ${\rm (iv)}'$ and ${\rm (vii)}'$ cases }, \\
\pair{\eta_1,\eta_2},\pair{\tau_1,\zeta_1} \prec_{Q}^\ttb \pair{\al,\be} & \text{  ${\rm (x)}'$ case}, \\
\pair{\eta_1,\eta_2} \prec_{Q}^\ttb \pair{\al,\be} ,\seq{s_1,\tau_1,\zeta_1} \prec_{Q}^\ttb \pair{\al,\be}, & \text{  ${\rm (xi)}'$ case}, \\
\pair{s_1,s_2} \prec_{Q}^\ttb  \pair{\eta_1,\eta_2}, \pair{s_1,s_2} \prec_{Q}^\ttb \seq{s_1,\tau_1,\zeta_1}, & \text{  ${\rm (xi)}'$ case continued}.
\ec
$$

In cases ${\rm (iv)}'$ and ${\rm (vii)}'$,
one can prove
$$    S_Q(\al) \hconv  S_Q(\be) \simeq    S_Q(s_1) \conv S_{Q}(s_2)  \quad \text{ (up to grading shifts)} $$
and
$$\de(S_Q(\al),S_Q(\be))=2$$ as in ${\rm (ii)}'$-case.
Moreover, by~\eqref{eq: comp series}
together with Proposition~\ref{prop: simple head},
we have
$$[S_Q(\al) \conv  S_Q(\be)]=   [S_Q(\al) \sconv  S_Q(\be)]  +[S_Q(s_2) \conv S_{Q}(s_1)],$$
for the cases ${\rm (ii)}'$, ${\rm (iv)}'$ and ${\rm (vii)}'$.  Thus the assertions in Theorem~\ref{thm: determimant} hold for those cases.

\begin{example}
Using $\Gamma_Q$ in Example~\ref{ex: Label B4}  and the \prqs corresponding to ${\rm (iv)}'$ and ${\rm (vii)}'$ in Example~\ref{ex: Uding B}, we can compute
$\de(S_Q(\al),S_Q(\be))$ as follows:
\ben
\item Since $p_{\srt{1,-2},\srt{2,-4}}=0$, $p_{\srt{1,-2},\srt{2,-3}}=0$, $(\srt{1,-2},\srt{2,-4})=-2$ and $(\srt{2,-3},\srt{1,-4})=0$, we have
$$\de(S_Q(\srt{2,-3}),S_Q(\srt{1,-4}) ) = 2.$$
\item Since $p_{\srt{1},\srt{2}}=1$, $p_{\srt{3},\srt{2}}=1$, $(\srt{1},\srt{2})=0$ and $(\srt{2},\srt{3})=0$, we have $$\de(S_Q(\srt{1,2}),S_Q(\srt{3}) ) = 2.$$
\ee
\end{example}

Now let us consider the case ${\rm (x)}'$.
In this case, we have non-zero $R$-module homomorphisms (up to grading shifts)
\begin{align}\label{eq: two non-zero}
S_Q(\eta_2) \conv S_Q(\eta_1)  \to S_Q(\be) \conv S_Q(\al)  \quad \text{ and } \quad S_Q(\zeta_1) \conv S_Q(\tau_1)\to  S_Q(\be) \conv S_Q(\al)
\end{align}
by~\eqref{eq: non-zero}
and hence the homomorphisms are injective, and
\begin{align} \label{eq: common socle}
S_Q(\al) \hconv S_Q(\be) & \simeq S_Q(\hd_Q(\al,\be))  \simeq S_Q(\ga)  \simeq S_Q(\eta_1) \hconv S_Q(\eta_2)  \simeq   S_Q(\tau_1) \hconv S_Q(\zeta_1)
\end{align}
(up to grading shifts) by Proposition~\ref{prop: inj} and Proposition~\ref{prop: simple socle}.  Moreover, by Proposition~\ref{prop:
minimal,simple}  and  Proposition~\ref{prop: inj}, we can conclude
that  $$\ell(S_Q(\al) \conv S_Q(\be) ) >2.$$
Similarly, we can conclude the same result for case ${\rm (xi)}'$.

Let us compute $\de(S_Q(\al),S_Q(\be))$ in case ${\rm (x)'}$: Since $(\al,\eta_1)$ is sectional,
 there exists $\zeta \in \Phi^+$ such that
\bna
\item \label{it: x a} the \prq $\pair{\kappa,\eta_1}$ is $[Q]$-minimal  for $\al$ and of the form ${\rm (i)'}$, or
\item \label{it: x b} the \prq $\pair{\al,\kappa}$ is  $[Q]$-minimal for $\eta_1$ and of the form ${\rm (iii)'}$.
\ee

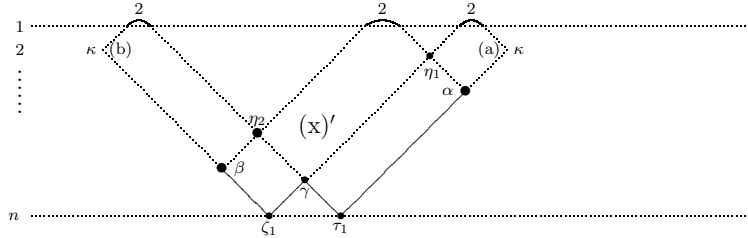
\begin{align*} 
\scalebox{0.79}{{\xy
(-20,0)*{}="DL";(-10,-10)*{}="DD";(0,20)*{}="DT";(10,10)*{}="DR";
"DT"+(-50,-4); "DT"+(70,-4)**\dir{.};
"DD"+(-40,-6); "DD"+(80,-6) **\dir{.};
"DT"+(-52,-4)*{\scriptstyle 1};
"DT"+(-52,-8)*{\scriptstyle 2};
"DT"+(-52,-12)*{\scriptstyle \vdots};
"DT"+(-52,-16)*{\scriptstyle \vdots};
"DT"+(-53,-36)*{\scriptstyle n};
"DD"+(-8,2); "DD"+(0,-6) **\dir{-};
"DD"+(6,0); "DD"+(0,-6) **\dir{-};
"DD"+(6,0); "DD"+(12,-6) **\dir{-};
"DD"+(33,15); "DD"+(12,-6) **\dir{-};
"DD"+(-8,2); "DD"+(16,26) **\dir{.};
"DD"+(32,16); "DD"+(22,26) **\dir{.};
"DD"+(6,0); "DD"+(27,21) **\dir{.};
"DD"+(32,26); "DD"+(27,21) **\dir{.};
"DD"+(32,26); "DD"+(36,26) **\crv{"DD"+(34,28)};
"DD"+(34,29)*{\scriptstyle 2};
"DD"+(36,26); "DD"+(40,22) **\dir{.};
"DD"+(33,15); "DD"+(40,22) **\dir{.};
"DD"+(42,22)*{\scriptstyle \kappa};
"DD"+(37,22)*{\scriptstyle {\rm (a)}  };
"DD"+(6,0); "DD"+(-2,8) **\dir{.};
"DD"+(-20,26); "DD"+(-2,8) **\dir{.};
"DD"+(-20,26); "DD"+(-24,26) **\crv{"DD"+(-22,28)};
"DD"+(-22,29)*{\scriptstyle 2};
"DD"+(-30,22)*{\scriptstyle \kappa};
"DD"+(-25,22)*{\scriptstyle {\rm (b)}  };
"DD"+(-24,26); "DD"+(-28,22) **\dir{.};
"DD"+(-28,22); "DD"+(-8,2) **\dir{.};
"DD"+(16,26); "DD"+(22,26) **\crv{"DD"+(19,28)};
"DD"+(19,29)*{\scriptstyle 2};
"DD"+(-8,2)*{\bullet};
"DL"+(18,-1)*{{\rm(x)'}};
"DD"+(-5,2)*{\scriptstyle \be};
"DD"+(-2,8)*{\bullet};
"DD"+(-2,10)*{\scriptstyle \eta_2};
"DD"+(33,15)*{\bullet};
"DD"+(30,15)*{\scriptstyle \al};
"DD"+(27.5,18)*{\scriptstyle  \eta_1 };
"DD"+(27,21)*{\scriptstyle \bullet};
"DT"+(2,-36)*{\scriptstyle \bullet};
"DT"+(2,-38)*{\scriptstyle \tau_1};
"DT"+(-10,-36)*{\scriptstyle \bullet};
"DT"+(-10,-38)*{\scriptstyle  \zeta_1};
"DT"+(-4,-30)*{\scriptstyle \bullet};
"DT"+(-4,-32)*{\scriptstyle \ga};
\endxy}}
\end{align*}

Assume~\eqref{it: x a}. Then as in the previous cases, we have a homogeneous $R$-module homomorphism
$$
q^{p_{\eta_1,\kappa}-(\eta_1,\kappa)} S_Q(\al) \conv S_Q(\be) \hookrightarrow S_Q(\eta_1) \conv S_Q(\kappa) \conv S_Q(\be).
$$
Since the \prq $\pair{\kappa ,\be}$ is $[Q]$-minimal  for $\eta_2$ and of the form ${\rm (iii)'}$, we have a non-zero composition of $R$-module homomorphisms
$$
q^{p_{\eta_1,\kappa}-(\eta_1,\kappa)} S_Q(\al) \conv S_Q(\be) \hookrightarrow S_Q(\eta_1) \conv S_Q(\kappa) \conv S_Q(\be) \twoheadrightarrow
q^{-p_{\be,\kappa}} S_Q(\eta_1)  \conv   S_Q(\eta_2).
$$
Note that the homomorphism is surjective by Proposition~\ref{prop: simple socle} and~\eqref{eq: two non-zero}. Hence
we can obtain a non-zero  $R$-module homomorphism
$$
q^{p_{\eta_1,\kappa}-(\eta_1,\kappa)}  S_Q(\al) \conv S_Q(\be)  \twoheadrightarrow  q^{-p_{\be,\kappa}-p_{\eta_2,\eta_1}} S_Q(\ga),
$$
since $S_Q(\eta_1)  \hconv   S_Q(\eta_2) \simeq  q^{-p_{\eta_2,\eta_1}} S_Q(\ga)$.
Thus we have
\begin{align*}
\tLa(S_Q(\al),S_Q(\be)) = p_{\be,\kappa}+p_{\eta_2,\eta_1} + p_{\eta_1,\kappa}-(\eta_1,\kappa)
\end{align*}
and
\begin{align*}
\de(S_Q(\al),S_Q(\be)) = p_{\be,\kappa}+p_{\eta_2,\eta_1} + p_{\eta_1,\kappa}-(\eta_1,\kappa)-(\al,\be).
\end{align*}

Since $p_{\be,\kappa}=p_{\eta_2,\eta_1} = p_{\eta_1,\kappa}=0$ and $(\eta_1,\kappa)=(\al,\be)=-2$, we have
\begin{align}
\de(S_Q(\al),S_Q(\be)) = 4 = \dg_Q(\al,\be) \times \max(\sfd_\al,\sfd_\be)=\de_{i,j}[p-s].
\end{align}
Similarly, we have the same results for $\tLa(S_Q(\al),S_Q(\be)) $ and $\de(S_Q(\al),S_Q(\be))$ in  ~\eqref{it: x b} case.

Let us compute $\de(S_Q(\al),S_Q(\be))$ in case ${\rm (xi)}'$: As in case ${\rm (x)}'$, we have
\bna
\item \label{it: xi a} the \prq $\pair{\zeta,\eta_1}$ is  $[Q]$-minimal for $\al$ and of the form ${\rm (i)'}$, or
\item \label{it: xi b} the \prq $\pair{\al,\zeta}$ is $[Q]$-minimal  for $\eta_1$ and of the form ${\rm (iii)'}$.
\ee
Assume~\eqref{it: x a}. As in case ${\rm (x)'}$, we have a non-zero composition of $R$-module homomorphisms
$$
q^{p_{\eta_1,\zeta}-(\eta_1,\zeta)} S_Q(\al) \conv S_Q(\be) \hookrightarrow S_Q(\eta_1) \conv S_Q(\zeta) \conv S_Q(\be) \twoheadrightarrow
q^{-p_{\be,\zeta}} S_Q(\eta_1)  \conv   S_Q(\eta_2).
$$
Since the \prq $\pair{\eta_1,\eta_2}$ corresponds to case ${\rm (ii)'}$, \eqref{eq: common socle} tells that we can obtain a non-zero  $R$-module homomorphism
$$
q^{p_{\eta_1,\zeta}-(\eta_1,\zeta)} S_Q(\al) \conv S_Q(\be)  \twoheadrightarrow  q^{-p_{\be,\zeta}-\tLa(S_Q(\eta_1),S_Q(\eta_2))} S_Q(s_1) \conv S_Q(s_2),
$$
which is surjective. Thus we have
\begin{align*}
\tLa(S_Q(\al),S_Q(\be)) = p_{\be,\zeta}+\tLa(S_Q(\eta_1),S_Q(\eta_2) )+ p_{\eta_1,\zeta}-(\eta_1,\zeta)
\end{align*}
and hence
\begin{align*}
\de(S_Q(\al),S_Q(\be)) = p_{\be,\zeta}+\tLa(S_Q(\eta_1),S_Q(\eta_2) )+ p_{\eta_1,\zeta}-(\eta_1,\zeta)-(\al,\be).
\end{align*}

Since the \prq $\pair{\eta_1,\eta_2}$ corresponds to case ${\rm (ii)'}$, we have
$$\tLa(S_Q(\eta_1),S_Q(\eta_2) ) =\de(S_Q(\eta_1),S_Q(\eta_2) ) +(\eta_1,\eta_2) =2.$$
By Table~\ref{table: al,be B}, and the facts that $p_{\be,\zeta}=p_{\eta_1,\zeta}=0$ $(\eta_1,\zeta)=-2$ and $(\al,\be)=0$, we have
$$
\de(S_Q(\al),S_Q(\be))  = 4 = \dg_Q(\al,\be) \times \max(\sfd_\al,\sfd_\be)=\de_{i,j}[p-s].
$$
Thus the assertions in Theorem~\ref{thm: determimant} hold for cases ${\rm (x)}'$ and ${\rm (xi)}'$.

\smallskip

Let us consider the cases ${\rm (v)}'$ and ${\rm (vi)}'$, in which contain $m_i =2$ in the head $\um$ of $\pair{\al,\be}$.
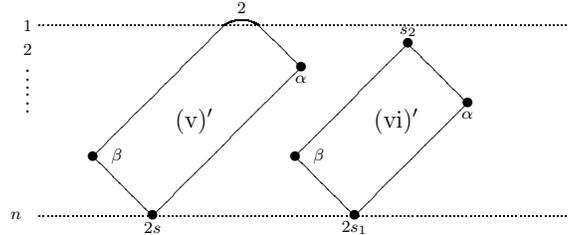
\begin{equation*}
\begin{aligned}
& \scalebox{0.79}{{\xy
(-20,0)*{}="DL";(-10,-10)*{}="DD";(0,20)*{}="DT";(10,10)*{}="DR";
"DT"+(-30,-4); "DT"+(60,-4)**\dir{.};
"DD"+(-20,-6); "DD"+(70,-6) **\dir{.};
"DT"+(-32,-4)*{\scriptstyle 1};
"DT"+(-32,-8)*{\scriptstyle 2};
"DT"+(-32,-12)*{\scriptstyle \vdots};
"DT"+(-32,-16)*{\scriptstyle \vdots};
"DT"+(-34,-36)*{\scriptstyle n};
"DD"+(-11,4); "DD"+(-1,-6) **\dir{-};
"DD"+(5,0); "DD"+(-1,-6) **\dir{-};
"DD"+(-11,4); "DD"+(11,26) **\dir{-};
"DD"+(24,19); "DD"+(17,26) **\dir{-};
"DD"+(5,0); "DD"+(24,19) **\dir{-};
"DD"+(-11,4)*{\bullet};
"DL"+(16,0)*{{\rm (v)'}};
"DD"+(-7,4)*{\scriptstyle \be};
"DD"+(11,26); "DD"+(17,26) **\crv{"DD"+(14,28)};
"DD"+(24,19)*{\bullet};
"DD"+(24,17)*{\scriptstyle \al};
"DD"+(-1,-6)*{\bullet};
"DD"+(-1,-8)*{\scriptstyle 2 s};
"DD"+(14,29)*{\scriptstyle 2};
"DD"+(23,4); "DD"+(33,-6) **\dir{-};
"DD"+(39,0); "DD"+(33,-6) **\dir{-};
"DD"+(23,4); "DD"+(42,23) **\dir{-};
"DD"+(52,13); "DD"+(42,23) **\dir{-};
"DD"+(39,0); "DD"+(52,13) **\dir{-};
"DD"+(23,4)*{\bullet};
"DD"+(42,23)*{\bullet};
"DL"+(50,0)*{{\rm (vi)'}};
"DD"+(27,4)*{\scriptstyle \be};
"DD"+(42,25)*{\scriptstyle s_2};
"DD"+(52,13)*{\bullet};
"DD"+(52,11)*{\scriptstyle \al};
"DD"+(33,-6)*{\bullet};
"DD"+(33,-8)*{\scriptstyle 2 s_1};
\endxy}} 
\end{aligned}
\end{equation*}

Consider ${\rm (v)}'$ case. As in the previous cases,
there exists $\eta \in \Phi^+$ such that
\bna
\item \label{it: va} the \prq $\pair{\eta,s}$ is $[Q]$-minimal for $\al$ and of the form ${\rm (ix)}'$ or
\item \label{it: vb} the \prq $\pair{s,\eta}$ is $[Q]$-minimal for $\be$ and of the form ${\rm (ix)}'$. 
\ee
Assume~\eqref{it: va}. Then
we have a homogeneous $R$-module monomorphism
$$  q^{p_{s,\eta} - (s,\eta)} S_Q(\al)\conv  S_Q(\be)   \hookrightarrow  S_{Q}(s) \conv  S_Q(\eta) \conv    S_Q(\be).  $$
Note that $\pair{\eta,\be}$ is a $[Q]$-minimal \prq for $s$ and of the form ${\rm (viii)}'$. Thus the composition
$$
 q^{p_{s,\eta} - (s,\eta)}  S_Q(\al) \conv  S_Q(\be)  \hookrightarrow S_{Q}(s) \conv  S_Q(\eta) \conv    S_Q(\be)
 \twoheadrightarrow q^{-p_{\be,\eta}}   S_{Q}(s)^{ \conv 2}
$$
does not vanish by Proposition~\ref{prop: not vanishing} and hence $S_Q(\al) \hconv  S_Q(\be) \simeq S_{Q}(s)^{ \conv 2}$ up to a grading shift.
Note that  $ S_{Q}(s)^{ \conv 2}  $ is not self-dual simple, while $ S_{Q}(2s) \seteq q S_{Q}(s)^{ \conv 2} $ is  self-dual simple. Thus we have
$$
S_Q(\al) \conv  S_Q(\be)  \twoheadrightarrow q^{-p_{\be,\eta}-p_{s,\eta} + (s,\eta)-1}   S_{Q}(2s)
$$
which tells
\begin{align*}
 \tLa(S_Q(\al), S_Q(\be) ) =p_{\be,\eta}+p_{s,\eta} - (s,\eta)+1
\end{align*}
and hence
\begin{align*}
\de(S_Q(\al), S_Q(\be) ) =p_{\be,\eta}+p_{s,\eta} - (s,\eta)-(\al,\be)+1.
\end{align*}
Since $p_{\be,\eta}=1$ and $p_{s,\eta}= (s,\eta)=(\al,\be)=0$, we have
\begin{align}\label{eq: tLa de v B}
\de(S_Q(\al), S_Q(\be) ) =2 = \max(\sfd_\al,\sfd_\be)=\de_{i,j}[p-s].
\end{align}
By Proposition~\ref{prop: length 2}, $\ell(S_Q(\al) \conv  S_Q(\be)) =2$.
We can
apply the similar arguments to~\eqref{it: vb}.

Consider ${\rm (vi)}'$ case. As in ${\rm (v)}'$-case, there exists $\eta \in \Phi^+$ such that
\bna
\item \label{it: vi a} the \prq $\pair{\eta,s_2}$ is $[Q]$-minimal for $\al$ and of the form ${\rm (ix)}'$,  or
\item \label{it: vi b} the \prq $\pair{s_2,\eta}$ is $[Q]$-minimal for $\be$ and of the form ${\rm (ix)}'$. 
\ee

Assume~\eqref{it: vi a}. Then we have
$$  q^{p_{s_2,\eta} - (s_2,\eta)} S_Q(\al)\conv  S_Q(\be)   \hookrightarrow  S_{Q}(s_2) \conv  S_Q(\eta) \conv    S_Q(\be).  $$
Then  the \prq $\pair{\eta,\be}$ corresponds to ${\rm (vii)}'$. Thus we have a non-zero  composition of homogeneous $R$-module homomorphisms
\begin{align*}
& q^{p_{s_2,\eta} - (s_2,\eta)} S_Q(\al)\conv  S_Q(\be)     \hookrightarrow  S_{Q}(s_2) \conv  S_Q(\eta) \conv    S_Q(\be) \\
& \ \twoheadrightarrow q^{-\tLa( S_Q(\eta), S_Q(\be))}  S_{Q}(s_2) \conv S_{Q}(s_2) \conv S_Q(s_3) \simeq   q^{-\tLa( S_Q(\eta), S_Q(\be))-1}   S_{Q}(2s_2) \conv S_Q(s_3),
\end{align*}
and hence $ S_Q(\al)\conv  S_Q(\be)  \simeq  S_{Q}(2s_2) \conv S_Q(s_3)$ up to a grading shift.
Furthermore, we have
$$
S_Q(\al) \conv  S_Q(\be)  \twoheadrightarrow q^{-p_{s_2,\eta} +(s_2,\eta) -\tLa( S_Q(\eta), S_Q(\be))-1}   S_{Q}(2s_2) \conv S_Q(s_3)
$$
which tells
\begin{align*}
 \tLa(S_Q(\al), S_Q(\be) ) = p_{s_2,\eta} -(s_2,\eta) +\tLa( S_Q(\eta), S_Q(\be))+1
\end{align*}
and hence
\begin{align*}
\de(S_Q(\al), S_Q(\be) ) =p_{s_2,\eta} +\tLa( S_Q(\eta), S_Q(\be))-(s_2,\eta) -(\al,\be)+1.
\end{align*}
Since the \prq $\pair{\eta,\be}$ corresponds to ${\rm (vii)}'$, we have
$$
\tLa( S_Q(\eta), S_Q(\be)) = \tLa( S_Q(\eta), S_Q(\be)) +(\eta,\be) = 2 + 0= 2.
$$
Then the facts that $p_{s_2,\eta}=1$, $(s_2,\eta)=0$ and
$(\al,\be)=2$ tell that we have
\begin{align}
\de(S_Q(\al), S_Q(\be) ) = 2 = \max(\sfd_\al,\sfd_\be)=\de_{i,j}[ p-s].
\end{align}
We can
apply the similar arguments to~\eqref{it: vi b}.  Thus the assertions  in Theorem~\ref{thm: determimant} hold for cases ${\rm (v)}'$ and ${\rm (vi)}'$.

\subsubsection{$C_n$ case} Now we shall consider the $C_n$-cases.
Comparing with $B_n$-case, there is no $C_n$-analogue of Proposition~\ref{prop: many pairs}; i.e., for a \prq $\up=\pair{\al,\be}$ with $\dg_Q(\al,\be)=2$,
there exists a unique sequence $\um$ such that
$$   \hd_Q(\up)  \prec_Q^{\ttb} \um  \prec_Q^{\ttb} \up. $$
Thus we can not apply Proposition~\ref{prop: inj} for the cases ${\rm (x)}''$ and ${\rm (xi)}''$.  For the cases except ${\rm (x)}''$ and ${\rm (xi)}''$, we can apply the same strategies of $B_n$ type to show Theorem~\ref{thm: determimant}.

\smallskip

Let us consider the ${\rm (x)}''$-case. Before we consider the general ${\rm (x)}''$-case, let us observe the AR-quiver $\Gamma_Q$ of type $C_3$ described in Example~\ref{ex: AR Q BCFG}.
In this case, $\pair{\al,\be}\seteq \pair{\srt{1,-3},\srt{2,3}}$ is the unique \prq corresponding to ${\rm (x)}''$-case.

 Now we want to show
$$\de(S_Q(\srt{1,-3}),S_Q(\srt{2,3})) =2.$$
Since \prq $\pair{\srt{2,-3} ,\srt{3,3} }$ is $[Q]$-minimal  for $\srt{2,3}$, we have
$$S_Q(\srt{2,3}) \simeq S_Q(\srt{2,-3}) \hconv S_Q(\srt{3,3} ) \quad \text{(up to grading shift)}.$$

Let $L=S_Q(\srt{1,-3})$, $M=S_Q(\srt{2,-3})$ and $N=S_Q(\srt{3,3})$. Since $L$ and $M$ commute, the sequence
$(L,M,N)$ is a normal sequence. On the other hand,
we can compute
$$ \La(M,L)=-1 \quad \La(N,L) = 2 \quad \text{ and }  \quad \La(M \hconv N,L)= 1  $$
since $$\srt{1,-3} \prec_Q \srt{2,-3}, \quad \srt{1,-3} \prec_Q \srt{3,3} \quad \text{ and } \quad \srt{1,-3} \prec_Q \srt{2,3}.$$
Then Lemma~\ref{lem: normality} implies $(M,N,L)$ is normal and hence
$$\de(L,M\hconv N) = \de(L,M)+\de(L,N) = 0+2 =2,$$
by Lemma~\ref{lem: de additive}. Here we use the fact  that $\de(L,N)=2$ followed by Proposition~\ref{prop: minimal pair ga de}.
Moreover, we have
\begin{align*}
L \hconv (M\hconv N) \simeq \head(L \conv M \conv N) &\simeq \hd(M \conv L \conv N) \simeq  \hd(M \conv (L \hconv N) ) \\
& \simeq S_Q(\srt{2,-3}) \hconv S_Q(\srt{1,3}) \simeq S_Q(\srt{1,2})
\end{align*}
as we desired. Since  there exists a non-zero $R$-module homomorphism
$$    S_Q(\srt{1,3}) \conv  S_Q(\srt{2,-3})  \to S_Q(\srt{2,3}) \conv S_Q(1,-3),$$
$\ell(S_Q(\srt{1,3}) \conv  S_Q(\srt{2,-3}))=2$ and they have common simple socle, the $R$-module homomorphism
is injective. Thus  we have
$$\ell( S_Q(\srt{2,3}) \conv S_Q(1,-3)) > 2.$$
Thus we can check the assertions in Theorem~\ref{thm: determimant}  hold for the \prq.

\medskip

Let us consider the general ${\rm (x)}''$-case. Then there exists $\kappa \in \Phi^+$ such that
\bna
\item \label{it: Cx a} $\pair{\kappa,\eta_1}$  is a $[Q]$-minimal \prq for $\al$ and of the form ${\rm (i)}''$, or     
\item \label{it: Cx b} $\pair{\al,\kappa}$  is a $[Q]$-minimal \prq for $\eta_1$ and of the form ${\rm (iii)}''$.  
\ee

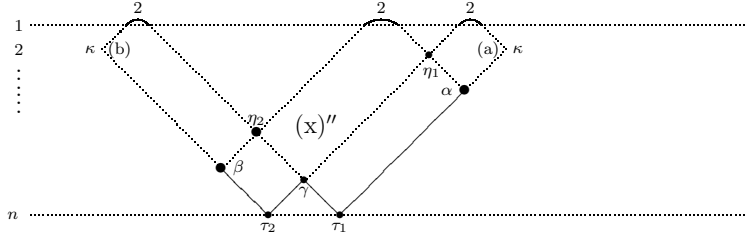
\begin{align*} 
\scalebox{0.79}{{\xy
(-20,0)*{}="DL";(-10,-10)*{}="DD";(0,20)*{}="DT";(10,10)*{}="DR";
"DT"+(-50,-4); "DT"+(70,-4)**\dir{.};
"DD"+(-40,-6); "DD"+(80,-6) **\dir{.};
"DT"+(-52,-4)*{\scriptstyle 1};
"DT"+(-52,-8)*{\scriptstyle 2};
"DT"+(-52,-12)*{\scriptstyle \vdots};
"DT"+(-52,-16)*{\scriptstyle \vdots};
"DT"+(-53,-36)*{\scriptstyle n};
"DD"+(-8,2); "DD"+(0,-6) **\dir{-};
"DD"+(6,0); "DD"+(0,-6) **\dir{-};
"DD"+(6,0); "DD"+(12,-6) **\dir{-};
"DD"+(33,15); "DD"+(12,-6) **\dir{-};
"DD"+(-8,2); "DD"+(16,26) **\dir{.};
"DD"+(32,16); "DD"+(22,26) **\dir{.};
"DD"+(6,0); "DD"+(27,21) **\dir{.};
"DD"+(32,26); "DD"+(27,21) **\dir{.};
"DD"+(32,26); "DD"+(36,26) **\crv{"DD"+(34,28)};
"DD"+(34,29)*{\scriptstyle 2};
"DD"+(36,26); "DD"+(40,22) **\dir{.};
"DD"+(33,15); "DD"+(40,22) **\dir{.};
"DD"+(42,22)*{\scriptstyle \kappa};
"DD"+(37,22)*{\scriptstyle {\rm (a)}  };
"DD"+(6,0); "DD"+(-2,8) **\dir{.};
"DD"+(-20,26); "DD"+(-2,8) **\dir{.};
"DD"+(-20,26); "DD"+(-24,26) **\crv{"DD"+(-22,28)};
"DD"+(-22,29)*{\scriptstyle 2};
"DD"+(-30,22)*{\scriptstyle \kappa};
"DD"+(-25,22)*{\scriptstyle {\rm (b)}  };
"DD"+(-24,26); "DD"+(-28,22) **\dir{.};
"DD"+(-28,22); "DD"+(-8,2) **\dir{.};
"DD"+(16,26); "DD"+(22,26) **\crv{"DD"+(19,28)};
"DD"+(19,29)*{\scriptstyle 2};
"DD"+(-8,2)*{\bullet};
"DL"+(18,-1)*{{\rm(x)''}};
"DD"+(-5,2)*{\scriptstyle \be};
"DD"+(-2,8)*{\bullet};
"DD"+(-2,10)*{\scriptstyle \eta_2};
"DD"+(33,15)*{\bullet};
"DD"+(30,15)*{\scriptstyle \al};
"DD"+(27.5,18)*{\scriptstyle  \eta_1 };
"DD"+(27,21)*{\scriptstyle \bullet};
"DT"+(2,-36)*{\scriptstyle \bullet};
"DT"+(-10,-36)*{\scriptstyle \bullet};
"DT"+(-4,-30)*{\scriptstyle \bullet};
"DT"+(-4,-32)*{\scriptstyle \ga};
"DT"+(2,-38)*{\scriptstyle \tau_1};
"DT"+(-10,-38)*{\scriptstyle  \tau_2};
\endxy}}
\end{align*}
Here $\tau_1+\tau_2 = 2\ga = 2(\al+\be)$.

\smallskip

Assume~\eqref{it: Cx a}.  Then we have
\begin{eqnarray} &&
\parbox{75ex}{
\ben
\item $S_Q(\kappa) \hconv S_Q(\eta_1) \simeq S_Q(\al)$,
\item the \prq $\pair{\kappa,\be}$ is  $[Q]$-minimal for $\eta_2$ of the form ${\rm (iii)}''$,
\item the \prq $\pair{\alpha,\eta_2}$ is $[Q]$-minimal for $\tau_1$ of the form ${\rm (v)}''$,
\item the \prq $\pair{\eta_1,\beta}$ is $[Q]$-minimal for $\tau_2$ of the form ${\rm (v)}''$,
\item the \prq $\pair{\kappa,\tau_2}$ is $[Q]$-minimal for $\ga$ of the form ${\rm (viii)}''$.
\ee
}\label{eq: C X a}
\end{eqnarray}

Now we shall prove the following claims:
$$
\text{ The sequences $(S_Q(\be),S_Q(\kappa),S_Q(\eta_1))$ and $(S_Q(\kappa),S_Q(\eta_1),S_Q(\be))$ are normal.}
$$

\noindent
(I)  Note that
$$\text{$\La( S_Q(\be),S_Q(\kappa) )=-(\be,\kappa)$, $\La( S_Q(\be),S_Q(\eta_1) )=-(\be,\eta_1)$ and $\La( S_Q(\be),S_Q(\al) )=-(\be,\al)$,}$$
since $ \al,\eta_1,\kappa \prec_Q \be$. Thus we have
$$  \La( S_Q(\be),  S_Q(\kappa) \hconv S_Q(\eta_1) )=  \La( S_Q(\be), S_Q(\kappa) )+ \La( S_Q(\be), S_Q(\eta_1) )$$
which implies that the sequence $(S_Q(\be),S_Q(\kappa),S_Q(\eta_1))$ is normal.

\noindent
(II) We have
\begin{align*}
 \La( S_Q(\kappa),S_Q(\eta_1) ) & = 2\de( S_Q(\kappa),S_Q(\eta_1) ) + (\kappa,\eta_1)  \overset{{\rm (A)}}{=} 2+ (\kappa,\eta_1),  \\
 \La( S_Q(\kappa),S_Q(\be) ) & = 2\de( S_Q(\kappa),S_Q(\beta) ) + (\kappa,\be)  \overset{{\rm (B)}}{=}  2+ (\kappa,\be), \\
 \La( S_Q(\kappa),S_Q(\tau_2) ) & = 2\de( S_Q(\kappa),S_Q(\tau_2) ) + (\kappa,\tau_2)  \overset{{\rm (C)}}{=}  4+ (\kappa,\tau_2).
\end{align*}
Here we use the facts $$\text{$\de( S_Q(\kappa),S_Q(\eta_1) ) \overset{{\rm (A)}}{=}  1$, $\de( S_Q(\kappa),S_Q(\beta) ) \overset{{\rm (B)}}{=}  1$ and $\de( S_Q(\kappa),S_Q(\tau_2) ) \overset{{\rm (C)}}{=}  2$}$$  that follow from Proposition~\ref{prop: minimal pair ga de}
and~\eqref{eq: C X a}.
Thus we have
$$
 \La( S_Q(\kappa),S_Q(\eta_1)  \hconv S_Q(\be)) =  \La( S_Q(\kappa),S_Q(\tau_2) ) =  \La( S_Q(\kappa),S_Q(\eta_1) )+ \La( S_Q(\kappa),S_Q(\be) )
$$
which implies that the sequence $(S_Q(\kappa),S_Q(\eta_1),S_Q(\be))$ is normal.

\medskip

By (I) and (II), Lemma~\ref{lem: de additive} tells that
$$
\de(S_Q(\al),S_Q(\be) ) = \de (S_Q(\kappa) \hconv  S_Q(\eta_1) ,S_Q(\be)) =  \de (S_Q(\kappa)  ,S_Q(\be)) + \de (S_Q(\eta_1)  ,S_Q(\be))  =2,
$$
as we desired.  Moreover, we have
\begin{align*}
S_Q(\al) \hconv S_Q(\be) & \simeq \hd( S_Q(\kappa) \conv   S_Q(\eta_1) \conv S_Q(\be) )  \\
& \simeq S_Q(\kappa) \hconv (S_Q(\eta_1) \hconv S_Q(\be) ) \simeq S_Q(\kappa) \hconv S_Q(\tau_2) \\
& \simeq S_Q(\ga),
\end{align*}
up to grading shifts.  Finally we have
$$  \ell( S_Q(\al) \conv S_Q(\be))  >2,$$
as in the particular case of type $C_3$.  For the case~\eqref{it: Cx b} and ${\rm (xi)}''$, we can apply the similar argument.
Thus the assertions in Theorem~\ref{thm: determimant} hold for cases ${\rm (x)}''$ and ${\rm (xi)}''$.

\subsubsection{$G_2$ case}
Note that there exist only $2$-reduced expressions of $w_0$ of type $G_2$, which are all adapted to the two Dynkin quivers (up to constant).
We shall only consider $Q$ and $\Gamma_Q$ in Section~\ref{subsec: G2 degree}. Then we have
\begin{align*}
&\be_1 = \al_2 =  (1,-2,1),  && \be_2 = \al_1+\al_2 =  (1,-1,0), && \be_3 = 3\al_1+\al_2 = (2,-1,-1), \\
&\be_4 = 2\al_1+\al_2 =  (1,0,-1),  && \be_5 = 3\al_1+\al_2 =  (1,1,-2), && \be_6 = \al_1 = (0,1,-1).
\end{align*}
Note that non $[Q]$-simple \prqs are given as follows:
\begin{equation} \label{eq: G2 regular}
\begin{aligned}
&\be_3 \prec_Q^\ttb \pair{\be_2,\be_4}, && \be_5 \prec_Q^\ttb \pair{\be_4,\be_6}, && \be_4 \prec_Q^\ttb \pair{\be_2,\be_6}, \\
&2\be_2 \prec_Q^\ttb \pair{\be_1,\be_4}, && 2\be_4 \prec_Q^\ttb \pair{\be_3,\be_6}, &&2\be_3 \prec_Q^\ttb \pair{\be_2,\be_5},
\end{aligned}
\end{equation}
and
\begin{align}\label{eq: G2 special}
&{\rm (i)} \ 3\be_2 \prec_Q^\ttb \pair{\be_1,\be_3}, && {\rm (ii)} \  3\be_4 \prec_Q^\ttb \pair{\be_3,\be_5},  &&{\rm (iii)} \  \be_3 \prec_Q^\ttb \pair{\be_2,\be_4}  \prec_Q^\ttb \pair{\be_5,\be_1}.
\end{align}

The cases in~\eqref{eq: G2 regular} can be proved as in $B_n$ or $C_n$-cases. So let us consider the cases in~\eqref{eq: G2 special}.

At first, we shall deal with (iii) in \eqref{eq: G2 special}.  Take
$L= S_Q(\be_1)$, $M= S_Q(\be_4)$, $N= S_Q(\be_6)$ and hence $M
\hconv N \simeq  S_Q(\be_5)$ and $L \hconv M \simeq
S_Q(\be_2)^{\circ 2}$ up to grading shifts. Since we know
$\de(S_Q(\be_2),N)=2$,  $\de(L,N)=3$ and  $\de(M,N)=1$ by
Proposition~\ref{prop: tLa}, we know
\begin{align*}
2 \times 3 = 2 \La(S_Q(\be_2),N)& = \La(L\hconv M,N)  \\
&= \La(L,N)+\La(M,N) =3+3,
\end{align*}
which implies that the sequence $(L,M,N)$ is normal by Lemma~\ref{lem: normality}. By the convex order $\prec_Q$, the sequence $(M,N,L)$ is also normal. Thus
$$\de(L,M\hconv N) =\de(L,M)+\de(L,N)=6,$$
by Lemma~\ref{lem: de additive}.  Furthermore, we have
\begin{align*}
  L \hconv (M \hconv N) & \simeq (L \hconv M)\hconv N \simeq  S_Q(\be_2)^{\circ 2} \hconv N \\
& \simeq S_Q(\be_2) \hconv (S_Q(\be_2) \hconv N) \simeq S_Q(\be_2) \hconv S_Q(\be_4)  \simeq S_Q(\be_3).
\end{align*}
Since we can construct a non-zero homomorphism
$$  S_Q(\be_4) \conv  S_Q(\be_2)  \to  S_Q(\be_5)  \conv  S_Q(\be_1) $$
and they have common simple socle, we have
$$  S_Q(\be_4) \conv  S_Q(\be_2)  \hookrightarrow  S_Q(\be_5)  \conv  S_Q(\be_1) $$
and hence
$$  \ell(S_Q(\be_5)  \conv  S_Q(\be_1)) >2.$$

Now let us consider (i) in~\eqref{eq: G2 special}. Since $\pair{\be_2,\be_4}$ is a minimal
\prq for $\be_5$, we have a non-zero composition of  $R$-module
homomorphisms
$$
S_Q(\be_1) \conv S_Q(\be_3) \hookrightarrow S_Q(\be_1) \conv S_Q(\be_4) \conv S_Q(\be_2)  \twoheadrightarrow  S_Q(\be_2)^{\conv 3},
$$
which implies
$$
S_Q(\be_1) \conv S_Q(\be_3) \simeq   S_Q(\be_2)^{\conv 3}.
$$

The assertion (ii) in~\eqref{eq: G2 special} can be proved in the same way.

\subsubsection{$F_4$ case}
 In $F_4$-case, we can prove Theorem~\ref{thm: determimant} by using techniques we have used in the previous subsections to each \prq $\pair{\al,\be}$.
In this subsection, we shall present several cases, instead of
considering all cases.  Let us use $\Gamma_Q$  in~\eqref{eq: F_4Q}
and assume that  Theorem~\ref{thm: determimant} holds for \prqs
$\pair{\theta,\mu}$ with $\dg_Q(\theta,\mu) \le 1$.

\smallskip

\noindent
(1) Let us compute $\de(S_Q(\al'),S_Q(\be') )$ when $\al'=\sprt{0,1,0,0}$ and $\be'= \sprt{\frac{1}{2},-\frac{1}{2},\frac{1}{2},-\frac{1}{2}}$. In this case, we have
\bna
\item $\phi_Q(3,2)=(\al',0)$, $\phi_Q(4,-5)=(\be',0)$, $\tde_{3,4}[7]= 2$,
\item there exists a unique \prq $\up = \pair{\sprt{\frac{1}{2},\frac{1}{2},-\frac{1}{2},-\frac{1}{2}},\sprt{0,0,1,0}  }=\pair{\al^{(1)},\be^{(1)}} $ such that
$$\ga\seteq \al'+\be' \prec_Q^\ttb \up   \prec_Q^\ttb \pair{\al',\be'}.$$
\ee

Note that $\pair{\be^{(1)},\eta \seteq \sprt{\frac{1}{2},-\frac{1}{2},-\frac{1}{2},-\frac{1}{2}}}$ is a minimal \prq for $\be'$ and $\pair{\al',\be^{(1)}}$ is a minimal \prq for $\zeta = \sprt{0,1,1,0}$. Setting $L=S_Q(\al')$, $M=S_Q(\be^{(1)})$, $N= S_Q(\eta)$ and $T=S_Q(\zeta)$,
we have $ \de(S_Q(\al'),S_Q(\be')) = \de(L,M\hconv N)$ and $L \hconv M \simeq T$. Since $\al' \prec_Q  \be^{(1)}  \prec_Q   \be \prec_Q \eta$, the sequence $(M,N,L)$ is normal.

On the other hand, we have
$$ \La(L \hconv M,N) = \La(T,N) =  2\;\de(T,N) +(\zeta,\eta) = 2\time 2 -2 =2 $$
and
$$ \La(L,N) = 2 \de(L,N) +  (\al',\eta) = 2\times 1 -1 =1 , \ \  \La(M,N) = 2  \de(M,N) + (\be^{(1)},\eta) =2 \times 1 -1 =1.$$
Thus the sequence $(L,M,N)$ is also normal. Thus we have
$$\de(L,M\hconv N)=\de(L,M)+\de(L,N)= 1 + 1 =2.$$

Finally, we have
\begin{align*}
L \hconv(M\hconv N) \simeq (L\hconv M) \hconv N \simeq T \hconv N \simeq S_Q(\ga)
\end{align*}
and hence $\ell(S_Q(\al')\conv S_Q(\be'))>2$.

\smallskip

\noindent
(2)  Now let us keep the notations in (1) and  compute $\de(S_Q(\al),S_Q(\be) )$ when $\al=\sprt{0,1,0,-1}$ and $\be= \sprt{\frac{1}{2},-\frac{1}{2},\frac{1}{2},\frac{1}{2}}$ in~\eqref{eq: F_4 de not tde}.
As we have seen in~\eqref{eq: F_4 de not tde}, we have
$$
\ga =\al+\be             \prec_Q^\ttb
\pair{\al^{(1)},\be^{(1)} }  \prec_Q^\ttb
\pair{\al',\be' }  \prec_Q^\ttb \pair{\al,\be}.
$$
However, we have another chain given as follows:
$$
\ga =\al+\be             \prec_Q^\ttb
\pair{\al^{(1)},\be^{(1)} }  \prec_Q^\ttb
\pair{\al'',\be'' }  \prec_Q^\ttb \pair{\al,\be},
$$
where $\al''=\sprt{0,0,1,-1}$ and $\be''= \sprt{\frac{1}{2},\frac{1}{2},-\frac{1}{2},\frac{1}{2}}$. As in (1), we can prove that $\de(J,K)=2$ and $J \hconv K \simeq S_Q(\ga)$, where
$J \seteq S_Q(\al'')$ and $K \seteq S_Q(\be'')$. Let $\kappa = \sprt{0,1,-1,0}$, and set $P \seteq S_Q(\kappa)$, $R=S_Q(\al)$, $X=S_Q(\be')$  and $U=S_Q(\be)$. Then
$ \de(R,U) = \de(P \hconv J, U)  $, since $\pair{\al'',\kappa}$ is a minimal \prq of $\al$.  Note that $J \hconv U \simeq M \conv X \simeq X \conv M$. Since
$$  \kappa \prec_Q  \al'' \prec_Q \be^{(1)} \prec_Q \be' \prec_Q \be,$$
the sequence $(J,U,P)$ is normal.

On the other hand, we have
\begin{align*}
\La(P,J\hconv U) & = \La(P, X \conv M) = \La(P,X)+\La(P,M) \\
& = (2\de(P,X)+(\kappa,\be')) + (2\de(P,M)+(\kappa,\be^{(1)})) = (2 \times 2 - 2 ) +  (2 \times 2 -2) =4
\end{align*}
and
\begin{align*}
\La(P,J) &=   2\de(P,J)+(\kappa,\al'') = 2 \times 2 -2 =2,\\
\La(P,U) &= 2\de(P,U)+(\kappa,\be) =2 \times 2 -2 =2.
\end{align*}
Thus the sequence $(P,J,U)$ is also normal. Hence
$$ \de(P\hconv J, U) = \de(P,U)+\de(J,U) =2+2=4 = \tde_{1,3}[9].$$

Consequently, we have
\begin{align*}
(P \hconv J) \hconv U &\simeq P \hconv (J\hconv U)  \simeq  P \hconv (X\conv M) \simeq (P \hconv X) \hconv M \\
& \simeq L \hconv M \simeq S_Q(\ga).
\end{align*}
Thus we have injective $R$-module homomorphisms
$$  S_Q(\be'') \conv  S_Q(\al'')  \hookrightarrow S_Q(\be) \conv S_Q(\al) \quad \text{ and }  \quad  S_Q(\be') \conv  S_Q(\al')  \hookrightarrow S_Q(\be) \conv S_Q(\al),$$
since  they have the common socle $S_Q(\ga)$, which appears once in their composition series.
Thus we have $\ell(S_Q(\al) \conv S_Q(\be)) >2$.

\subsection{Ending Remark}
We can also  apply the techniques in the previous subsections to prove \eqref{eq: de ADE} for $ADE$-types. 
Also the result for $\dg_Q(\al,\be)=1$ can be written in the form of~\eqref{eq: BKMc}.

\begin{example}
For a reduced expression $s_1s_2s_3s_1s_2s_3s_1s_2s_3$ of $w_0$ of type $C_3$, adapted to $Q$ in Example~\ref{ex: AR Q BCFG},  the fact that
$$ S_Q(\srt{1,1} ) \hconv S_Q(\srt{2,2} )  \simeq S_Q(\srt{1,2} )^{\conv 2} \quad \text{ (up to  a grading shift)}$$
corresponds to
$$   r_{[Q]}(\srt{1,1}) r_{[Q]}(\srt{2,2}) - r_{[Q]}(\srt{2,2}) r_{[Q]}(\srt{1,1}) = (q^{-1}-q^3) r_{[Q]}(\srt{1,2})^2.$$
\end{example}

\appendix

\section{$\tde_{i,j}(t)$ for $E_7$ and $E_8$} \label{appeA: tde}

\subsection{$E_7$}   Here is the list of $\tde_{i,j}(t)$ for $E_7$.
\begin{align*}
&\tde_{1,1}(t) = t^1+t^{7}+t^{11}+t^{17},  \hspace{17.4ex}\tde_{1,2}(t) = t^4+t^{8}+t^{10}+t^{14}, \allowdisplaybreaks \\
&\tde_{1,3}(t)=t^2+t^{6}+t^{8}+t^{10}+t^{12}+t^{16}, \ \hspace{7ex} \tde_{1,4}(t)=t^3+t^{5}+t^{7}+2t^{9}+t^{11}+t^{13}+t^{15},\allowdisplaybreaks \\
&\tde_{1,6}(t)=t^5+t^{7}+t^{11}+t^{13}, \hspace{17.8ex}  \tde_{1,7}(t)=t^6+t^{12}, \allowdisplaybreaks\\
&\tde_{2,2}(t)=t^1+t^{5}+t^{7}+t^{9}+t^{11}+t^{13}+t^{17}, \hspace{3.2ex}  \tde_{2,3}(t)=\tde_{1,4}(t)  \allowdisplaybreaks\\
&\tde_{2,4}(t)=t^2+t^{4}+2t^{6}+2t^{8}+2t^{10}+2t^{12}+t^{14}+t^{16}, \allowdisplaybreaks\\
&\tde_{2,5}(t)=t^3+t^{5}+2t^{7}+t^{9}+2t^{11}+t^{13}+t^{15}, \  \tde_{2,7}(t)=t^{5}+t^{9}+t^{13}, \allowdisplaybreaks\\
&\tde_{2,6}(t)=t^{4}+t^{6}+t^{8}+t^{10}+t^{12}+t^{14}, \allowdisplaybreaks\\
&\tde_{3,3}(t) = t^1+t^{3}+t^{5}+2t^{7}+2t^{9}+2t^{11}+t^{13}+t^{15}+t^{17}, \allowdisplaybreaks\\
&\tde_{3,4}(t) = t^2+2t^{4}+2t^{6}+3t^{8}+3t^{10}+2t^{12}+2t^{14}+t^{16}, \allowdisplaybreaks\\
&\tde_{3,5}(t) = t^{3}+2t^{5}+2t^{7}+2t^{9}+2t^{11}+2t^{13}+t^{15}, \allowdisplaybreaks\\
&\tde_{3,6}(t) = t^{4}+2t^{6}+t^{8}+t^{10}+2t^{12}+t^{14}, \hspace{6ex}   \tde_{3,7}(t) = t^{5}+t^{7}+t^{11}+t^{13}, \allowdisplaybreaks\\
&\tde_{4,4}(t) = t^1+2t^{3}+3t^{5}+4t^{7}+4t^{9}+4t^{11}+3t^{13}+2t^{15}+t^{17}, \allowdisplaybreaks\\
&\tde_{4,5}(t) = t^2+2t^{4}+3t^{6}+3t^{8}+3t^{10}+3t^{12}+2t^{14}+t^{16}, \allowdisplaybreaks\\
&\tde_{4,6}(t) =\tde_{3,5}(t),  \hspace{29ex}   \tde_{4,7}(t) = t^{4}+t^{6}+t^{8}+t^{10}+ t^{12}+ t^{14}, \allowdisplaybreaks\\
&\tde_{5,5}(t) = t^1+t^{3}+2t^{5}+2t^{7}+3t^{9}+2t^{11}+2t^{13}+t^{15}+t^{17}, \allowdisplaybreaks\\
&\tde_{5,6}(t) = t^2+t^{4}+t^{6}+2t^{8}+2t^{10}+t^{12}+t^{14}+t^{16}, \quad \tde_{5,7}(t) = t^{3}+t^{7}+t^{11}+t^{15}, \allowdisplaybreaks\\
&\tde_{6,6}(t) = t^1+t^{3}+t^{7}+2t^{9}+t^{11}+t^{15}+t^{17},  \hspace{4ex}   \tde_{6,7}(t) = t^2+t^{8}+t^{10}+t^{16}, \allowdisplaybreaks\\
&\tde_{7,7}(t) = t^1+t^{9}+t^{17} \hspace{10.65ex}  \text{ and }  \hspace{10.45ex}  \tde_{i,j}(t) =   \tde_{j,i}(t).
\end{align*}

\subsection{$E_8$}   Here is the list of $\tde_{i,j}(t)$ for $E_8$.
\begin{align*}
\tde_{1,1}(t) &= t^1+t^{7}+t^{11}+t^{13}+t^{17}+t^{19}+t^{23}+t^{29}, \allowdisplaybreaks\\
\tde_{1,2}(t) &= t^4+t^{8}+t^{10}+t^{12}+t^{14}+t^{16}+t^{18}+t^{20}+t^{22}+t^{26}, \allowdisplaybreaks\\
\tde_{1,3}(t) &= t^2+t^{6}+t^{8}+t^{10}+2t^{12}+t^{14}+t^{16}+2t^{18}+t^{20}+t^{22}+t^{24}+t^{28}, \allowdisplaybreaks\\
\tde_{1,4}(t) &= t^3+t^{5}+t^{7}+2t^{9}+2t^{11}+2t^{13}+2t^{15}+2t^{17}+2t^{19}+2t^{21}+t^{23}+t^{25}+t^{27}, \allowdisplaybreaks\\
\tde_{1,5}(t) &= t^4+t^{6}+t^{8}+2t^{10}+t^{12}+2t^{14}+2t^{16}+t^{18}+2t^{20}+t^{22}+t^{24}+t^{26}, \allowdisplaybreaks\\
\tde_{1,6}(t) &= t^{5}+t^{7}+t^{9}+t^{11}+t^{13}+2t^{15}+t^{17}+t^{19}+t^{21}+t^{23}+t^{25}, \allowdisplaybreaks\\
\tde_{1,7}(t) &= t^{6}+t^{8}+t^{12}+t^{14}+t^{16}+t^{18}+t^{22}+t^{24},  \allowdisplaybreaks\\
\tde_{1,8}(t) &= t^{7}+t^{13}+t^{17}+t^{23}, \allowdisplaybreaks\\
\tde_{2,2}(t) &= t^1+t^{5}+t^{7}+t^{9}+2t^{11}+t^{13}+2t^{15}+t^{17}+2t^{19}+t^{21}+t^{23}+t^{25}+t^{29}, \allowdisplaybreaks\\
\tde_{2,3}(t) &= t^3+t^{5}+t^{7}+2(t^{9}+t^{11}+t^{13}+t^{15}+t^{17}+t^{19}+t^{21})+t^{23}+t^{25}+t^{27}, \allowdisplaybreaks\\
\tde_{2,4}(t) &= t^2+t^{4}+2t^{6}+2t^{8}+3(t^{10}+t^{12}+t^{14}+t^{16}+t^{18}+t^{20})+2t^{22}+2t^{24}+t^{26}+t^{28}, \allowdisplaybreaks\\
\tde_{2,5}(t) &= t^3+t^{5}+2t^{7}+2t^{9}+2t^{11}+3t^{13}+3t^{15}+3t^{17}+2t^{19}+2t^{21}+2t^{23}+t^{25}+t^{27}, \allowdisplaybreaks\\
\tde_{2,6}(t) &= t^4+t^{6}+2t^{8}+t^{10}+2t^{12}+2t^{14}+2t^{16}+2t^{18}+t^{20}+2t^{22}+t^{24}+t^{26}, \allowdisplaybreaks\\
\tde_{2,7}(t) &=\tde_{1,6}(t), \qquad  \qquad  \qquad \tde_{2,8}(t) = t^{6}+t^{10}+t^{14}+t^{16}+t^{20}+t^{24},  \allowdisplaybreaks\\
\tde_{3,3}(t) &= t^1+t^{3}+t^{5}+2t^{7}+2t^{9}+3t^{11}+3t^{13}+2t^{15}+3t^{17}+3t^{19}+2t^{21}+2t^{23} \allowdisplaybreaks\\
                         & \hspace{2ex}+t^{25}+t^{27}+t^{29}, \allowdisplaybreaks\\
\tde_{3,4}(t) &= t^2+2(t^{4}+t^{6})+3t^{8}+4(t^{10}+t^{12}+t^{14}+t^{16}+t^{18}+t^{20})+3t^{22}+2(t^{24}+t^{26})+t^{28}, \allowdisplaybreaks\\
\tde_{3,5}(t) &= t^3+2t^{5}+2t^{7}+3t^{9}+3t^{11}+3t^{13}+4t^{15}+3t^{17}+3t^{19}+3t^{21}+2t^{23}+2t^{25}+t^{27}, \allowdisplaybreaks\\
\tde_{3,6}(t) &= t^4+2t^{6}+2t^{8}+2t^{10}+2t^{12}+3t^{14}+3t^{16}+2t^{18}+2t^{20}+2t^{22}+2t^{24}+t^{26}, \allowdisplaybreaks\\
\tde_{3,7}(t) &= t^{5}+2t^{7}+t^{9}+t^{11}+2t^{13}+2t^{15}+2t^{17}+t^{19}+t^{21}+2t^{23}+t^{25}, \allowdisplaybreaks\\
\tde_{3,8}(t) &= t^{6}+t^{8}+t^{12}+t^{14}+t^{16}+t^{18}+t^{22}+t^{24}, \allowdisplaybreaks\\
\tde_{4,4}(t) &= t^1+2t^{3}+3t^{5}+4t^{7}+5t^{9}+6t^{11}+6t^{13}+6t^{15}+6t^{17}+6t^{19}+5t^{21}\allowdisplaybreaks\\
                         & \hspace{2ex}+4t^{23} +3t^{25}+2t^{27}+t^{29}, \allowdisplaybreaks\\
\tde_{4,5}(t) &= t^2+2t^{4}+3t^{6}+4t^{8}+4t^{10}+5t^{12}+5t^{14}+5t^{16}+5t^{18}+4t^{20}+4t^{22}\allowdisplaybreaks\\
                         & \hspace{2ex}+3t^{24}+2t^{26}+t^{28}, \allowdisplaybreaks\\
\tde_{4,6}(t) &= t^3+2t^{5}+3t^{7}+3t^{9}+3t^{11}+4t^{13}+4t^{15}+4t^{17}+3t^{19}+3t^{21}+3t^{23}+2t^{25}+t^{27}, \allowdisplaybreaks\\
\tde_{4,7}(t) &= \tde_{3,6}(t), \ \
\tde_{4,8}(t) = t^{5}+t^{7}+t^{9}+t^{11}+t^{13}+2t^{15}+t^{17}+t^{19}+t^{21}+t^{23}+t^{25}, \allowdisplaybreaks\\
\tde_{5,5}(t) &= t^1+t^{3}+2t^{5}+3t^{7}+3t^{9}+4t^{11}+4t^{13}+4t^{15}+4t^{17}+4t^{19}+3t^{21}\allowdisplaybreaks\\
                         & \hspace{2ex}+3t^{23}+2t^{25}+t^{27}+t^{29}, \allowdisplaybreaks\\
\tde_{5,6}(t) &= t^2+t^{4}+2t^{6}+2t^{8}+3t^{10}+3t^{12}+3t^{14}+3t^{16}+3t^{18}+3t^{20}+2t^{22}\allowdisplaybreaks\\
                         & \hspace{2ex}+2t^{24}+t^{26}+t^{28}, \allowdisplaybreaks\\
\tde_{5,7}(t) &= t^3+t^{5}+t^{7}+2t^{9}+2t^{11}+2t^{13}+2t^{15}+2t^{17}+2t^{19}+2t^{21}+t^{23}+t^{25}+t^{27}, \allowdisplaybreaks\\
\tde_{5,8}(t) &= t^{4}+t^{8}+t^{10}+t^{12}+t^{14}+t^{16}+t^{18}+t^{20}+t^{22}+t^{26}, \allowdisplaybreaks\\
\tde_{6,6}(t) &= t^1+t^{3}+t^{5}+t^{7}+2t^{9}+3t^{11}+2t^{13}+2t^{15}+2t^{17}+3t^{19}+2t^{21}\allowdisplaybreaks\\
                         & \hspace{2ex}+t^{23}+t^{25}+t^{27}+t^{29}, \allowdisplaybreaks\\
\tde_{6,7}(t) &= t^2+t^{4}+t^{8}+2t^{10}+2t^{12}+t^{14}+t^{16}+2t^{18}+2t^{20}+t^{22}+t^{26}+t^{28}, \allowdisplaybreaks\\
\tde_{6,8}(t) &= t^3+t^{9}+t^{11}+t^{13}+t^{17}+t^{19}+t^{21}+t^{27}, \allowdisplaybreaks\\
\tde_{7,7}(t) &= t^1+t^{3}+t^{9}+2t^{11}+t^{13}+t^{17}+2t^{19}+t^{21}+t^{27}+t^{29}, \allowdisplaybreaks\\
\tde_{7,8}(t) &= t^2+t^{10}+t^{12}+t^{18}+t^{20}+t^{28}, \allowdisplaybreaks\\
\tde_{8,8}(t) &= t^1+t^{11}+t^{19}+t^{29} \qquad\qquad
\qt{and} \qquad\qquad\quad
 \tde_{i,j}(t) =   \tde_{j,i}(t).
\end{align*}

\end{document}